\documentclass[10pt]{amsart}
\input epsf
\address{Simons Center for Geometry and Physics,
State University of New York, Stony Brook, NY 11794-3636 U.S.A. \&
Center for Geometry and Physics, Institute for Basic Sciences (IBS), Pohang, Korea} \email{kfukaya@scgp.stonybrook.edu}

\usepackage{graphicx}
\usepackage{amsmath}
\usepackage{amscd}
\usepackage{amssymb}
\usepackage{amstext}
\usepackage{amsmath}
\usepackage[all]{xy}
\usepackage{mathrsfs}

\usepackage{color}
\usepackage[dvipdfmx]{hyperref}

\def\R{\ifmmode{\mathbb R}\else{$\mathbb R$}\fi} %real numbers
\def\C{\ifmmode{\mathbb C}\else{$\mathbb C$}\fi} %complex numbers
\def\Z{\ifmmode{\mathbb Z}\else{$\mathbb Z$}\fi} %integers

\theoremstyle{theorem}
\newtheorem{thm}{Theorem}[section]
\newtheorem{cor}[thm]{Corollary}
\newtheorem{lem}[thm]{Lemma}
\newtheorem{sublem}[thm]{Sublemma}
\newtheorem{subsublemma}[thm]{Subsublemma}
\newtheorem{prop}[thm]{Proposition}

\newtheorem{lemdef}[thm]{Lemma-Definition}

\theoremstyle{definition}
\newtheorem{defn}[thm]{Definition}
\newtheorem{rem}[thm]{Remark}
\newtheorem{cons}[thm]{\rm\bfseries{Construction}}
\newtheorem{exm}[thm]{Example}
\newtheorem{conds}[thm]{Condition}
\newtheorem{assump}[thm]{Assumption}

\numberwithin{equation}{section}
\makeindex
\begin{document}

\title[Equivariant Kuranishi charts]{Lie groupoid, deformation of 
unstable curves, and construction of equivariant Kuranishi charts}
\author{Kenji Fukaya}

\thanks{Supported partially by NSF Grant No. 1406423,
and Simons Collaboration on Homological Mirror Symmetry.}

\begin{abstract}
In this paper we give detailed construction 
of $G$-equivariant Kuranishi chart of 
moduli spaces of pseudo holomorphic curves 
to a symplectic manifold with $G$-action,
for an arbitrary compact Lie group $G$.
\par
The proof is based on the deformation theory 
of {\it unstable} marked curves using the language 
of Lie groupoid 
(which is {\it not} necessarily \'etale) and the Riemannian center of mass technique.
\par
This proof is actually similar to \cite[Sections 13 and 15]{FO}
except the usage of the language of Lie groupoid makes the argument 
more transparent.
\end{abstract}

\par
\date{Octorbor 18th, 2017}

\keywords{pseudo holomorphic curve, Gromov-Witten invariant, 
Kuranishi structure, Lie groupoid, unstable curve}
\subjclass[2010]{53D35,	53D45,  14D23}
\maketitle

\tableofcontents

\section{Introduction}
\label{sec:intro}

Let $(X,\omega)$ be a symplectic manifold 
which is compact or convex at infinity.
We assume that a compact Lie group $G$ acts on $X$ 
preserving the symplectic form $\omega$.
We take a $G$-invariant almost complex structure $J$ which is 
compatible with $\omega$.
We consider the moduli space 
$\mathcal M_{g,\ell}((X,\omega),\alpha)$ of $J$-stable maps 
with given genus $g$ and $\ell$ marked points 
and of homology class $\alpha \in H_2(X;\Z)$.
This space has an obvious $G$ action.
\par
The problem we address in this paper is 
to associate an equivariant 
virtual fundamental class to this 
moduli space.
It then gives an equivariant version of 
Gromov-Witten invariant.
(The corresponding problem was solved in the case when 
$(X,J,\omega)$ is projective algebraic variety.
(See \cite{Gi,GP}.))
\par
In the symplectic case, the virtual fundamental 
class $[\mathcal M_{g,\ell}((X,\omega),\alpha)]$
was 
established in the year 1996 by several 
groups of mathematicians
(\cite{FO,LiTi98,Rua99,Siebert,LiuTi98}.) 
Its equivariant version is discussed by various people.
However the foundation of such equivariant version is not 
so much transparent in the literature.
\par
In case $L$ is a Lagrangian submanifold 
which is $G$-invariant, 
we can discuss a similar problem 
to define a virtual fundamental chain 
of the moduli space
of bordered $J$-holomorphic curves, 
especially disks.
Equivariant virtual fundamental chain is used to define an equivariant version of Lagrangian 
Floer theory.
Equivariant Kuranishi structures on the moduli space 
of pseudo holomorphic curve in a manifold with group 
action, have been used in several 
places already. 
For example it is used in a series of 
papers the author wrote with 
joint authors 
\cite{fooo08},\cite{foootoric32}
and etc. which study the case when 
$(X,\omega)$ is a toric manifold 
and $G$ is the torus.
See also \cite{Liu}.
The construction of equivariant Kuranishi structure 
in such a situation is written 
in detail in \cite[Sections 4-3,4-4,4-5]{foootoric32}.
The construction there uses the fact 
that the Lagrangian submanifold $L$ is a single orbit of the group action, which is free
on $L$, and also the fact that the group $G$ is abelian.
The argument there is rather ad-hoc and by this reason seems 
to be rather complicated, though it is correct.
\par
In this paper the author provides a result 
which is the most important part of 
the construction of $G$-equivariant virtual fundamental 
cycle and chain on the moduli space $\mathcal M_{g,\ell}((X,\omega),\alpha)$.
\par
We will prove the following:
\begin{thm}\label{thm11}
For each $\frak p \in \mathcal M_{g,\ell}((X,\omega),\alpha)$ 
there exists $(V_{\frak p},\mathcal E_{\frak p},s_{\frak p},\psi_{\frak p})$
such that:
\begin{enumerate}
\item
$V_{\frak p}$ is a finite dimensional smooth and effective orbifold.
The group $G$ has a smooth action  on it.
\item
$\mathcal E_{\frak p}$ is a smooth vector bundle (orbibundle) on $V_{\frak p}$.
The $G$ action on $V_{\frak p}$ lifts to a $G$ action on the vector bundle $\mathcal E_{\frak p}$.
\item
$s_{\frak p}$ is a $G$-invariant section of $\mathcal E_{\frak p}$.
\item
$\psi_{\frak p}$ is a $G$-equivariant homeomorphism 
from $s_{\frak p}^{-1}(0)$ onto an open neighborhood of the $G$ orbit of 
$\frak p$.
\end{enumerate}
\end{thm}
In short $(V_{\frak p},\mathcal E_{\frak p},s_{\frak p},\psi_{\frak p})$ is 
a $G$-equivariant Kuranishi chart of $\mathcal M_{g,\ell}((X,\omega),\alpha)$ 
at $\frak p$.
See Section \ref{sec:eqkurachar} Theorem \ref{thm54} for the precise statement.
\par
We can glue those charts and obtain a $G$-equivariant Kuranishi structure.
We can also prove a similar result in the case of the 
moduli space of pseudo holomorphic maps from bordered curves.
However in this paper we focus on the construction of the 
$G$-equivariant Kuranishi chart on $\mathcal M_{g,\ell}((X,J),\alpha)$. In fact 
this is the part where we need something novel compared to the 
case without $G$ action. 
Once we obtain a $G$-equivariant Kuranishi chart at each 
point, the rest of the construction is fairly analogous to the 
case without $G$ action.
(See for example \cite{foooconstr}.)
So to reduce the length of this paper we do not address the problem 
of constructing global $G$-equivariant Kuranishi structure but restrict 
ourselves to the construction of a $G$-equivariant Kuranishi  
chart.
(Actually the argument of Subsection \ref{indsmoothchart} 
contains a large portion of the arguments needed for the construction of 
global Kuranishi structure.)
\begin{rem}
Joyce's approach \cite{joyce2} on virtual
fundamental chain, especially the idea using 
certain kinds of universality to 
construct finite dimensional reduction, which Joyce explained 
in his talk \cite{joyce4}, when it will be worked out 
successfully, has advantage
in establishing the equivariant version, (since in this approach 
the Kuranishi chart obtained is `canonical' in certain 
sense and so its $G$-equivariance could be automatic.)
\par
If one takes infinite dimensional approach 
for virtual
fundamental chain
such as those in \cite{LiTi98,hwze}, one does not need the 
process to take finite dimensional reduction.
So the main issue of this paper (to perform finite dimensional
reduction in a $G$-equivariant way) may be absent.
On the other hand, then one needs to develop certain 
frame work to study equivariant cohomology in such infinite 
dimensional situation.
\end{rem}
The main problem to resolve to construct $G$-equivariant Kuranishi
charts is the following.
Let $\frak p = [(\Sigma,\vec z),u]$ be an element of 
$\mathcal M_{g,\ell}((X,\omega),\alpha)$.
In other words, $(\Sigma,\vec z)$ is a marked 
pre-stable curve and $u : \Sigma \to X$ is a 
$J$-holomorphic map.
We want to find an orbifold $U_{\frak p}$
on which $G$ acts and such that the $G$ orbit $G\frak p$ 
is contained in $U_{\frak p}$.
$U_{\frak p}$ is obtained as the set of isomorphism classes 
of the solutions 
of certain differential equation
$$
\overline{\partial}u' \in E((\Sigma',\vec z^{\,\prime}),u')
$$
where $((\Sigma',\vec z^{\,\prime}),u')$ is an object 
which is `close' to $((\Sigma,\vec z),u)$ in certain 
sense (See Definition \ref{defn41}.) and 
$E((\Sigma',\vec z^{\,\prime}),u')$ is a finite dimensional 
vector subspace of
$$
C^{\infty}(\Sigma';(u')^*TX \otimes \Lambda^{01}).
$$
We want our space of solutions $U_{\frak p}$
has a $G$ action. For this purpose we need the family 
of vector spaces $E((\Sigma',\vec z^{\,\prime}),u')$ 
to be $G$-equivariant, that is,
\begin{equation}\label{form11}
g_*E((\Sigma',\vec z^{\,\prime}),u') = E((\Sigma',\vec z^{\,\prime}),gu').
\end{equation}
A possible way to construct such a family 
$E((\Sigma',\vec z^{\,\prime}),u')$ is as follows.
\begin{enumerate}
\item
We first 
take a subspace 
$$
E((\Sigma,\vec z),u)
\subset C^{\infty}(\Sigma;u^*TX \otimes \Lambda^{01}).
$$
which is invariant under the action of the isotropy group at 
$[(\Sigma,\vec z),u]$
of $G$ action on 
$\mathcal M_{g,\ell}((X,\omega),\alpha)$.
\item
For each $((\Sigma',\vec z^{\,\prime}),u')$ 
which is `close' to the $G$-orbit of 
$((\Sigma,\vec z),u)$ we find $g\in G$ such that 
the distance between $u'$ and $gu$ is smallest.
\item
We move $E((\Sigma,\vec z),u)$ to a subspace of 
$C^{\infty}(\Sigma;(gu)^*TX \otimes \Lambda^{01})$
by $g$ action and then move it to
$
C^{\infty}(\Sigma';(u')^*TX \otimes \Lambda^{01})
$
by an appropriate parallel transportation.
\end{enumerate}
There are problems to carry out Step (2) and Step (3).
Note that we need to consider the equivalence class of  
$((\Sigma,\vec z),u)$ with respect to an appropriate 
isomorphisms.
By this reason the parametrization of the source curve 
$\Sigma$ is well-defined only up to a certain isomorphism
group. This causes a problem in defining the notion 
of closeness in (2) and defining the way how to move 
our obstruction bundle $E((\Sigma,\vec z),u)$
by a parallel transportation in (3).
\par
In case $(\Sigma,\vec z)$ is stable, 
the ambiguity, that is, the group of automorphisms 
of this marked curve, is a finite group.
Using the notion of multisection (or multivalued 
perturbation) which was introduced in \cite{FO},
we can go around the problem of this ambiguity 
of the identification of the source curve.
\par
In the case when $(\Sigma,\vec z)$ is unstable
(but $((\Sigma,\vec z),u)$ is stable), 
the problem is more nontrivial.
In \cite{FO}, Fukaya-Ono provide two methods 
to resolve this problem.
One of the methods, which is discussed in \cite[appendix]{FO},
uses additional marked points $\vec w$ so that 
$(\Sigma,\vec z \cup \vec w)$ becomes stable.
The moduli space (including $\vec w$) 
does not have a correct dimension, because of the extra 
parameter to move $\vec w$.
Then \cite[appendix]{FO} uses a codimension 2 submanifold 
${\mathcal N}_i$ and require that $u(w_i) \in {\mathcal N}_i$ to kill 
this extra dimension.
\par
In our situation where we have $G$ action,
including extra marked points  $\vec w$ breaks the symmetry 
of $G$ action.
For example suppose there is $S^1 \subset G$
and a $g \in S^1$ parametrized family of automorphisms $\gamma_g$
of $(\Sigma,\vec z)$
such that 
$$
h(\gamma_g(z)) = gu(z).
$$
Then we can not take $\vec w$ which is invariant 
under this $S^1$ action.
This causes a trouble to define obstruction spaces 
$E((\Sigma,\vec z),u)$ satisfying
(\ref{form11}).
\par
In this paper we use a different way to 
resolve the problem appearing in the case when 
$(\Sigma,\vec z)$ is unstable.
This method was written in \cite{FO}
especially in its Sections 13 and 15.
During these 20 years after \cite{FO} had been written
the authors of \cite{FO} did not use this method so much 
since it seems easier to use the method of
\cite[appendix]{FO}.
The author of this paper however recently realized that 
for the purpose of constructing a family of 
obstruction spaces $E((\Sigma,\vec z),u)$
in a $G$-equivariant way, 
the method of \cite[Sections 13 and 15]{FO}
is useful.
\par
Let us briefly explain this second method.
We fix $\Sigma$ and take an 
obstruction space $E((\Sigma,\vec z),u)$
on it.
Let $((\Sigma',\vec z^{\,\prime}),u')$ 
be an element which is `close' to 
$((\Sigma,\vec z),gu)$ for some $g \in G$.
To carry out steps (2)(3) we need to 
find a way to fix a diffeomorphism
$\Sigma \cong \Sigma'$ at least on the 
support of $E((\Sigma,\vec z),u)$.
If $(\Sigma,\vec z)$ is stable 
we can find such identification 
$\Sigma \cong \Sigma'$ up to finite ambiguity.
In case $(\Sigma,\vec z)$ is unstable
the ambiguity is actually controlled 
by the group of automorphisms of 
$(\Sigma,\vec z)$, which has positive dimension.
The idea is to choose 
certain identification $\Sigma \cong \Sigma'$
together with $g$ such that 
the distance between 
$u'$ with this identification and $gu$ is smallest
among all the choices of  the 
identification $\Sigma \cong \Sigma'$ and $g \in G$.
\par
To work out this idea, we need to make precise
what we mean by 
`the ambiguity is controlled 
by the group of automorphisms'.
In \cite{FO} certain `action' of a group germ 
is used for this purpose.
Here `action' is in a quote since it is not 
actually an action.
(In fact, $g_1(g_2x) = (g_1g_2)(x)$ does {\it not} hold.
See  \cite[3 lines above Lemma 13.22]{FO}.)
Though the statements and the proofs 
(of \cite[Lemmata 13.18 and 13.22]{FO}) provided there 
are rigorous and correct, as was written there, 
the notion of 
```action' of group germ''
is rather confusing.
Recently the author realized that 
the notion of ```action' of group germ''
can be nicely reformulated by using the language of 
Lie groupoid.
In our generalization to the $G$-equivariant case, 
which is related to a rather delicate problem 
of equivariant  transversality, 
rewriting the method of \cite[Sections 13 and 15]{FO}
using the language of Lie groupoid seems meaningful
for the author.
\par
In Section \ref{sec:groupoid} we review 
the notion of Lie groupoid in the form we use.
Then in Section \ref{sec:universal}
we construct a `universal family of deformation 
of a marked curve' including the case when 
the marked curve is unstable.
Such universal family does not exist in the usual sense
for an unstable curve.
However we can still show the unique existence 
of such a universal family in the sense 
of deformation parametrized by a Lie groupoid.
\begin{thm}\label{thn12}
For any marked nodal curve $(\Sigma,\vec z)$
(which is not necessarily stable)
there exists uniquely a universal family of  deformations of 
$(\Sigma,\vec z)$
parametrized by a Lie groupoid.
\end{thm}
See Section \ref{sec:universal} Theorem \ref{them35} 
for the precise statement.
This result may have independent interest other than its 
application to the proof of Theorem \ref{thm11}.
We remark that the Lie groupoid appearing in Theorem \ref{thn12}
is \'etale if and only if  $(\Sigma,\vec z)$ is stable.
So in the case of our main interest where  $(\Sigma,\vec z)$ is not 
stable, the Lie groupoid we study is {\it not} an \'etale groupoid
or an orbifold.
\par
The universal family in Theorem \ref{thn12}
should be related to a similar universal family defined
in algebraic geometry
based on Artin stack.
\par
Theorem \ref{thn12}
provides the precise formulation of the fact
that `identification of $\Sigma$ with $\Sigma'$
is well defined up to automorphism group of
$(\Sigma,\vec z)$'.
\par
Using Theorem \ref{thn12} we carry out the idea mentioned above 
and construct a family of 
obstruction spaces $E((\Sigma',\vec z^{\,\prime}),u')$
satisfying (\ref{form11}) in Sections \ref{sec:obstruction} and 
\ref{sec:main}.
\par
Once we obtain $E((\Sigma',\vec z^{\,\prime}),u')$
the rest of the construction is similar 
to the case without $G$ action.
However since the problem of constructing 
equivariant Kuranishi chart is rather delicate 
one, 
we provide detail of the process of constructing 
equivariant Kuranishi chart 
in Section \ref{sec:decay}.
Most of the argument of Section \ref{sec:decay}
is taken from  \cite{foootech}.
Certain exponential decay estimate of the 
solution of pseudo holomorphic curve equation
(especially the exponential decay estimate
of its derivative with respect to the 
gluing parameter) is 
crucial to obtain a smooth Kuranishi structure.
(In our equivariant situation, obtaining 
{\it smooth} Kuranishi structure is more 
essential than the case without group action.
This is because in the $G$-equivariant case it is 
harder to apply certain tricks 
of algebraic or differential topology to reduce 
the construction to the study of $0$ or $1$ dimensional 
moduli spaces.)
This exponential decay estimate is proved in detail in 
\cite{foooexp}.
Other than this point, our discussion is 
independent of the papers we have written 
on the foundation of virtual fundamental chain 
technique and is selfcontained.
\par
The author is planning to apply the result of this 
paper to several problems in the papers \cite{Fu2, Fu3, Fu4} in preparation.
It includes, the definition of equivariant Lagrangian 
Floer homology and of equivariant Gromov-Witten invariant,
relation of equivariant Lagrangian 
Floer theory
to the Lagrangian 
Floer theory of the symplectic quotient.
The author also plan to apply it to study 
some gauge theory related problems,
especially it is likely that we can use it 
to provide a rigorous mathematical definition of the
symplectic geometry side of Atiyah-Floer conjecture.
(Note Atiyah-Floer conjecture concerns a relation between 
Lagrangian Floer homology and Instanton (gauge theory) Floer homology.)
See \cite{DF}.
However in this paper we do not discuss those applications 
but concentrate on establishing the foundation of such 
study.
\par
Several material of this paper is taken from joint works 
of the author with other mathematicians.
Especially Section \ref{sec:decay} and several related places 
are taken from a joint work with Oh-Ohta-Ono such as 
\cite{foootech}.
Also the main novel part of this paper (the contents of Sections 
\ref{sec:universal} and  \ref{sec:main} and related places) are 
$G$-equivariant version of a
rewritted version of a part 
(Sections 13 and 15) of a joint paper \cite{FO} with Ono. 
\par
The author thanks anonymous refere for careful reading, pointing out several 
errors in the earlier version of this paper and many useful comments to improve
presentations.

\section{Lie groupoid and deformation of 
complex structure}
\label{sec:groupoid}

\subsection{Lie groupoid: Review}
\label{Liegroupoidrev}

The notion of Lie groupoid has been used in symplectic 
and Poisson geometry. (See for example \cite{cdw}.)  
We use the notion of Lie groupoid to formulate deformation theory 
of marked (unstable) curves.
Usage of the language of groupoid to study moduli problem 
is well established in algebraic geometry. (See for example \cite{keelmori}.)
To fix the notation etc. 
we start with defining a version of Lie groupoid which we use in this 
paper.
We work in complex analytic category.
So in this and the next sections manifolds are complex manifolds and 
maps between them are holomorphic maps,
unless otherwise mentioned.
(In later sections we study real $C^{\infty}$ manifolds.)
We assume all the manifolds are Hausdorff and paracompact
in this paper.
In the next definition the sentence in the [\dots] is an 
explanation of the condition and is not a part 
of the condition.

\begin{defn}\label{defn2121}
A {\it Lie groupoid} \index{Lie groupoid} is a system 
$\frak G = (\mathcal{OB},\mathcal{MOR},{\rm Pr}_s,{\rm Pr}_t,\frak{comp},\frak{inv},
\mathcal{ID})$ with the following properties.
\begin{enumerate}
\item
$\mathcal{OB}$ \index{00O3B3frak@$\mathcal{OB}$} is a complex manifold, which we call the {\it space of 
objects}.\index{space of 
objects}
\item
$\mathcal{MOR}$ \index{00M33OR@$\mathcal{MOR}$} is a complex manifold, which we call the {\it space of 
morphisms}.\index{space of 
morphisms}
\item
${\rm Pr}_s$ (resp. ${\rm Pr}_t$) is a map
\index{00P11Rs@${\rm Pr}_s$} \index{00P11rt@${\rm Pr}_t$} 
$$
{\rm Pr}_s : \mathcal{MOR} \to \mathcal{OB}
$$
(resp.
$$
{\rm Pr}_t : \mathcal{MOR} \to \mathcal{OB})
$$
which we call the {\it source projection}, (resp. the 
{\it target projection}).\index{source projection}\index{target projection}
[This is a map which assigns the source and the target to a 
morphism.]
\item
We require ${\rm Pr}_s$ and ${\rm Pr}_t$ are both submersions.
(We however do {\it not} assume the map
$
({\rm Pr}_s,{\rm Pr}_t) : \mathcal{MOR} \to \mathcal{OB}^2
$
is a submersion.)
\item
The {\it composition map}, 
\index{composition map}
$\frak{comp}$ is a map 
\index{00C44OMP@$\frak{comp}$}
\begin{equation}\label{comp}
\frak{comp} : \mathcal{MOR} \,\,{}_{{\rm Pr}_t}\times_{{\rm Pr}_s}
\mathcal{MOR}
\to \mathcal{MOR}.
\end{equation}
We remark that the fiber product in (\ref{comp}) is 
transversal and gives a smooth (complex) manifold, 
because of Item (3).
[This is a map which defines the composition of morphisms.]
\item
The next diagram commutes.
\begin{equation}
\begin{CD}
\mathcal{MOR} \,\,{}_{{\rm Pr}_t}\times_{{\rm Pr}_s}\mathcal{MOR} @ >\frak{comp}>>
\mathcal{MOR} \\
@ V({\rm Pr}_s,{\rm Pr}_t)VV @ VV({\rm Pr}_s,{\rm Pr}_t)V\\
\mathcal{OB}^2 @ >{=}>> \mathcal{OB}^2.
\end{CD}
\end{equation}
Here ${\rm Pr}_t$ (resp. ${\rm Pr}_s$) in the left vertical 
arrow is ${\rm Pr}_t$ (resp. ${\rm Pr}_s$) of the 
second factor (resp. the first factor).
\item
The next diagram commutes
\begin{equation}
\begin{CD}
\mathcal{MOR} \,\,{}_{{\rm Pr}_t}\times_{{\rm Pr}_s}\mathcal{MOR} 
\,\,{}_{{\rm Pr}_t}\times_{{\rm Pr}_s}\mathcal{MOR} @ >{{\rm id} \times \frak{comp}}>>
\mathcal{MOR} \,\,{}_{{\rm Pr}_t}\times_{{\rm Pr}_s}\mathcal{MOR} \\
@ V{\frak{comp} \times {\rm id}}VV @ VV\frak{comp}V\\
\mathcal{MOR} \,\,{}_{{\rm Pr}_t}\times_{{\rm Pr}_s}\mathcal{MOR} 
@>{\frak{comp}}>> \mathcal{MOR}.
\end{CD}
\end{equation}
[This means that the composition of morphisms is associative.]
\item
The {\it identity section}
\index{identity section}
\index{00I3D@$\mathcal{ID}$}
$\mathcal{ID}$ is a map
\begin{equation}\label{identity}
\mathcal{ID} : \mathcal{OB}\to \mathcal{MOR}.
\end{equation}
[This is a map which assigns the identity morphism to each object.]
\item
The next diagram commutes.
\begin{equation}
\xymatrix{
&&\mathcal{MOR}  \ar[dd]^{({\rm Pr}_t,{\rm Pr}_s)}   \\
& 
\mathcal{OB} \ar[ru]^{\mathcal{ID}}\ar[rd]_{{\Delta}}
\\
&& 
\mathcal{OB}^2 
}
\end{equation}
Here $\Delta$ is the diagonal embedding.
\item
The next diagram commutes.
\begin{equation}
\begin{CD}
\mathcal{MOR}\,\,{}_{{\rm Pr}_t}\times_{{\rm Pr}_s} \mathcal{MOR}
@ <{({\rm id},\mathcal{ID}\circ {\rm Pr}_t)}<<
\mathcal{MOR} @ >{(\mathcal{ID}\circ {\rm Pr}_s,{\rm id})}>>
\mathcal{MOR}\,\,{}_{{\rm Pr}_t}\times_{{\rm Pr}_s} \mathcal{MOR} \\
@ VV{\frak{comp}}V @ V{{\rm id}}VV @ VV{\frak{comp}}V\\
\mathcal{MOR}  @ <{\rm id}<< \mathcal{MOR} @ >{\rm id}>> \mathcal{MOR}.
\end{CD}
\end{equation}
[This means that the composition with the identity morphism gives the 
identity map.]
\item
The {\it inversion map}
$\frak{inv}$ is a map
\index{inversion map}
\index{00I4NV@$\frak{inv}$}
\begin{equation}\label{inv}
\frak{inv} : \mathcal{MOR}
\to \mathcal{MOR}.
\end{equation}
such that
$
\frak{inv} \circ \frak{inv} = {\rm id}.
$
[This map assigns an inverse to a morphism.
In particular all the morphisms are invertible.]
\item
The next diagram commutes.
\begin{equation}
\begin{CD}
\mathcal{MOR} @ >\frak{inv}>>
\mathcal{MOR} \\
@ V{{\rm Pr}_t}VV @ VV{{\rm Pr}_s}V\\
\mathcal{OB} @ >{=}>> \mathcal{OB}.
\end{CD}
\end{equation}
\item
The next diagrams commute
\begin{equation}
\begin{CD}
\mathcal{MOR} @ >{({\rm id},\frak{inv})}>>
\mathcal{MOR}\,\,{}_{{\rm Pr}_t}\times_{{\rm Pr}_s} \mathcal{MOR}
 \\
@ V{{\rm Pr}_s}VV @ VV{\frak{comp}}V\\
\mathcal{OB} @ >{\mathcal{ID}}>> \mathcal{MOR}.
\end{CD}
\end{equation}
\begin{equation}
\begin{CD}
\mathcal{MOR} @ >{(\frak{inv},{\rm id})}>>
\mathcal{MOR}\,\,{}_{{\rm Pr}_t}\times_{{\rm Pr}_s} \mathcal{MOR}
 \\
@ V{{\rm Pr}_t}VV @ VV{\frak{comp}}V\\
\mathcal{OB} @ >{\mathcal{ID}}>> \mathcal{MOR}.
\end{CD}
\end{equation}
[This means that the composition with the inverse becomes an identity map.]
\end{enumerate}
\end{defn}
Note that we assume all the maps in Definition \ref{defn2121}
are holomorphic. (We do not repeat this remark from 
now on.)
\begin{exm}\label{exaGaction}
Let $\frak X$ be a complex manifold and $G$ a complex Lie group
which has a holomorphic action on $\frak X$.
(We use right action for the consistency of notation.)
\par
We define
$\mathcal{OB} = \frak X$, $\mathcal{MOR} = \frak X \times G$, 
${\rm Pr}_s(x,g) = x$, ${\rm Pr}_t(x,g) = xg$,
$\frak{comp}((x,g),(y,h)) = (x,gh)$,
$\mathcal{ID}(x) = (x,e)$ (where $e$ is the unit of $G$),
$\frak{inv}(x,g) = (xg,g^{-1})$.
\par
It is easy to see that they satisfy the axiom of Lie groupoid.
\end{exm}
\begin{defn}
Let $\frak G^{(i)} = (\mathcal{OB}^{(i)},\mathcal{MOR}^{(i)},{\rm Pr}^{(i)}_s,{\rm Pr}^{(i)}_t,\frak{comp}^{(i)},\frak{inv}^{(i)},
\mathcal{ID}^{(i)})$  be a Lie groupoid for $i=1,2$.
A {\it morphism} 
\index{morphism of Lie groupoid} $\mathscr F$ from $\frak G^{(1)}$ to $\frak G^{(2)}$
is a pair $(\mathscr F_{ob},\mathscr F_{mor})$ such that
the maps
$$
\mathscr F_{ob} 
: \mathcal{OB}^{(1)} \to \mathcal{OB}^{(2)},
\quad
\mathscr F_{mor}: \mathcal{MOR}^{(1)}
\to \mathcal{MOR}^{(2)}
$$
\index{00F6OB@$\mathscr F_{ob}$}  \index{00F6MOR@$\mathscr F_{mor}$}
are holomorphic and commute with 
${\rm Pr}^{(i)}_s,{\rm Pr}^{(i)}_t,\frak{comp}^{(i)},\frak{inv}^{(i)},
\mathcal{ID}^{(i)}$ in an obvious sense.
We call $\mathscr F_{ob}$ (resp. $\mathscr F_{mor}$) 
the {\it object part}\index{object part} (resp. the {\it morphism part}) 
\index{morphism part} of the morphism.
\par
We can compose two morphisms in an obvious way.
The pair of identity maps defines a morphism 
from $\frak G = (\mathcal{OB},\mathcal{MOR},{\rm Pr}_s,{\rm Pr}_t,\frak{comp},\frak{inv},
\mathcal{ID})$ to itself, which we call the {\it identity morphism}.
\index{identity morphism}
\par
Thus the set of all Lie groupoids consists a category.
Therefore the notion of isomorphism and the two Lie groupoids being
isomorphic are defined.
\end{defn}
\begin{defn}
Let $\frak G = (\mathcal{OB},\mathcal{MOR},{\rm Pr}_s,{\rm Pr}_t,\frak{comp},\frak{inv},
\mathcal{ID})$ be a Lie groupoid and 
$\mathcal U \subset \mathcal{OB}$ an open subset.
We define the {\it restriction} 
\index{restriction of Lie gropoid} $\frak G\vert_{\mathcal U}$ of $\frak G$ to $\mathcal U$ 
as follows.
\par
The space of objects is $\mathcal U$.
The space of morphisms is 
${\rm Pr}_s^{-1}(\mathcal U) \cap {\rm Pr}_t^{-1}(\mathcal U)$.
${\rm Pr}_s,{\rm Pr}_t,\frak{comp},\frak{inv},
\mathcal{ID}$ of $\frak G\vert_{\mathcal U}$ are restrictions of corresponding objects of 
$\frak G$.
\par
It is easy to see that axioms are satisfied.
\par
The inclusions  $\mathcal U \to \mathcal{OB}$, 
${\rm Pr}_s^{-1}(\mathcal U) \cap {\rm Pr}_t^{-1}(\mathcal U) 
\to \mathcal{MOR}$ defines a morphism 
$\frak G\vert_{\mathcal U} \to \frak G$.
We call it an {\it open embedding}.
\index{open embedding of Lie groupoid}
\end{defn}
\begin{lemdef}\label{lemdef25}
Let $\frak G = (\mathcal{OB},\mathcal{MOR},{\rm Pr}_s,{\rm Pr}_t,\frak{comp},\frak{inv},
\mathcal{ID})$ be a Lie groupoid and 
$\mathcal T : \mathcal{OB} \to \mathcal{MOR}$  a (holomorphic) map 
with ${\rm Pr}_t\circ \mathcal T = {\rm id}$.
\par
It defines a morphism $\frak{conj}^{\mathcal T}$ from $\frak G$ to itself as follows.
\begin{enumerate}
\item
$\frak{conj}^{\mathcal T}_{ob} = {\rm Pr}_s\circ \mathcal T : 
\mathcal{OB} \to \mathcal{OB}$.
\item
We write $\varphi \circ \psi = \frak{comp}(\varphi,\psi)$
in case ${\rm Pr}_s(\varphi) = {\rm Pr}_t(\psi)$.
Now for $\varphi \in \mathcal{MOR}$ with 
${\rm Pr}_s(\varphi) = x$, ${\rm Pr}_t(\varphi) = y$,
we define
$$
\frak{conj}^{\mathcal T}_{mor}(\varphi)
= \frak{inv}(\mathcal T(y)) \circ \varphi \circ \mathcal T(x).
$$
\end{enumerate}
It is easy to see that $(\frak{conj}^{\mathcal T}_{ob},
\frak{conj}^{\mathcal T}_{mor})$ is a morphism from $\frak G$ to $\frak G$.
\end{lemdef}
We can generalize this construction as follows.
\begin{defn}
Let $\frak G^{(i)} = (\mathcal{OB}^{(i)},\mathcal{MOR}^{(i)},{\rm Pr}^{(i)}_s,{\rm Pr}^{(i)}_t,\frak{comp}^{(i)},\frak{inv}^{(i)},
\mathcal{ID}^{(i)})$  be a Lie groupoid for $i=1,2$
and $\mathscr F^{(j)} = (\mathscr F^{(j)}_{ob},\mathscr F^{(j)}_{mor})$  
a morphism from $\frak G^{(1)}$ to $\frak G^{(2)}$, 
for $j=1,2$.
\par
A {\it natural transformation} \index{natural transformation} 
from $\mathscr F^{(1)}$ to 
$\mathscr F^{(2)}$ is a (holomorphic) map:
$\mathcal T : \mathcal{OB}^{(1)} \to \mathcal{MOR}^{(2)}$
with the following properties.
\begin{enumerate}
\item
${\rm Pr}^{(2)}_s \circ \mathcal T = \mathscr F^{(1)}_{ob}$
and ${\rm Pr}^{(2)}_t \circ \mathcal T = \mathscr F^{(2)}_{ob}$.
\item
$
\frak{comp}(\mathcal T\circ {\rm Pr}^{(1)}_t,\mathscr F^{(1)}_{mor})
=
\frak{comp}(\mathscr F^{(2)}_{mor},\mathcal T\circ {\rm Pr}^{(1)}_s).
$
In other words the next diagram commutes for $\varphi \in \mathcal{MOR}^{(1)}$
with ${\rm Pr}^{(1)}_s(\varphi) = x$, ${\rm Pr}^{(1)}_t(\varphi) = y$.
\begin{equation}
\begin{CD}
\mathscr F^{(2)}_{ob}(x) @ >{\mathscr F^{(2)}_{mor}(\varphi)}>> \mathscr F^{(2)}_{ob}(y)
 \\
@ A{\mathcal T(x)}AA @ AA{\mathcal T(y)}A\\
\mathscr F^{(1)}_{ob}(x) @ >{\mathscr F^{(1)}_{mor}(\varphi)}>> \mathscr F^{(1)}_{ob}(y).
\end{CD}
\end{equation}
\end{enumerate}
We say $\mathscr F^{(2)}$ is {\it conjugate}\index{conjugate} to $\mathscr F^{(1)}$, 
if there is a natural transformation from  $\mathscr F^{(1)}$ to 
$\mathscr F^{(2)}$.
\end{defn}
\begin{lem}
\begin{enumerate}
\item
If $\mathcal T$ is a natural transformation from $\mathscr F^{(1)}$
to $\mathscr F^{(2)}$ then $\frak{inv} \circ \mathcal T$ 
is a  natural transformation from $\mathscr F^{(2)}$
to $\mathscr F^{(1)}$.
\item
If $\mathcal T$ (resp. $\mathcal S$) is a 
natural transformation from $\mathscr F^{(1)}$
to $\mathscr F^{(2)}$ (resp. $\mathscr F^{(2)}$
to $\mathscr F^{(3)}$) then
$\frak{comp} \circ (\mathcal T,\mathcal S)$ is a 
natural transformation from $\mathscr F^{(1)}$
to $\mathscr F^{(3)}$.
\item
Being conjugate is an equivalence relation.
\end{enumerate}
\end{lem}
\begin{proof}
(1)(2) are obvious from definition. (3) follows from (1) and (2).
\end{proof}
\begin{lem}\label{lem28}
A morphism $\mathscr F$ from $\frak G$ to itself is 
conjugate to the identity morphism 
if and only if it is $\frak{conj}^{\mathcal T}$
for some $\mathcal T$ as in Lemma-Definition \ref{lemdef25}.
\end{lem}
This is obvious from the definition.
\begin{lem}\label{lemdef29}
Let $\frak G^{(i)} = (\mathcal{OB}^{(i)},\mathcal{MOR}^{(i)},{\rm Pr}^{(i)}_s,{\rm Pr}^{(i)}_t,\frak{comp}^{(i)},\frak{inv}^{(i)},
\mathcal{ID}^{(i)})$  be a Lie groupoid for $i=1,2,3$
and $\mathscr F = (\mathscr F_{ob},\mathscr F_{mor})$,
$\mathscr F^{(j)} = (\mathscr F^{(j)}_{ob},\mathscr F^{(j)}_{mor})$
a morphism from $\frak G^{(1)}$ to $\frak G^{(2)}$, for $j=1,2$ .
Let $\mathscr G = (\mathscr G_{ob},\mathscr G_{mor})$,
$\mathscr G^{(j)} = (\mathscr G^{(j)}_{ob},\mathscr G^{(j)}_{mor})$
be a
morphism from $\frak G^{(2)}$ to $\frak G^{(3)}$, for $j=1,2$.
\begin{enumerate}
\item If $\mathscr F^{(1)}$ is conjugate to $\mathscr F^{(2)}$
then $\mathscr G\circ \mathscr F^{(1)}$ is conjugate to $\mathscr G\circ \mathscr F^{(2)}$.
\item
If $\mathscr G^{(1)}$ is conjugate to $\mathscr G^{(2)}$
then $\mathscr G^{(1)}\circ \mathscr F$ is conjugate to $\mathscr G^{(2)}\circ \mathscr F$.
\end{enumerate}
\end{lem}
\begin{proof}
If $\mathcal T$ is a natural transformation from $\mathscr F^{(1)}$
to $\mathscr F^{(2)}$ then 
$\mathscr G_{mor} \circ \mathcal T$ is a natural transformation from 
$\mathscr G\circ \mathscr F^{(1)}$ is conjugate to $\mathscr G\circ \mathscr F^{(2)}$.
\par
If $\mathcal S$ is a natural transformation from $\mathscr G^{(1)}$
to $\mathscr G^{(2)}$ then
$\mathcal S \circ \mathscr F_{ob}$ is a natural transformation from 
$\mathscr G^{(1)}\circ \mathscr F$ to $\mathscr G^{(2)}\circ \mathscr F$.
\end{proof}

\subsection{Family of complex varieties parametrized by a Lie groupoid}
\label{familypara}

\begin{defn}\label{defn210}
Let $\frak G = (\mathcal{OB},\mathcal{MOR},{\rm Pr}_s,{\rm Pr}_t,\frak{comp},\frak{inv},
\mathcal{ID})$  be a Lie groupoid.
A {\it family of complex analytic spaces parametrized by  $\frak G$}, is 
\index{family of complex analytic spaces parametrized by  $\frak G$}
a pair $(\widetilde{\frak G},\mathscr F)$ of a Lie groupoid 
$\widetilde{\frak G} = (\widetilde{\mathcal{OB}},\widetilde{\mathcal{MOR}},
\widetilde{{\rm Pr}}_s,\widetilde{{\rm Pr}}_t,\widetilde{\frak{comp}},\widetilde{\frak{inv}},
\widetilde{\mathcal{ID}})$ 
and a morphism $\mathscr F : \widetilde{\frak G} \to \frak G$,
such that next two diagrams are cartesian squares, and $\mathscr F_{ob}$, $\mathscr F_{mor}$ 
are flat and surjective.
\begin{equation}\label{diagram212}
\begin{CD}
\widetilde{\mathcal{MOR}} @ >{\widetilde{\rm Pr}_t}>> \widetilde{\mathcal{OB}}
\\
@ V{\mathscr F_{mor}}VV @ V{\mathscr F_{ob}}VV\\
\mathcal{MOR} @ >{{\rm Pr}_t}>> \mathcal{OB}.
\end{CD}
\qquad
\begin{CD}
\widetilde{\mathcal{MOR}} @ >{\widetilde{\rm Pr}_s}>> \widetilde{\mathcal{OB}}
\\
@ V{\mathscr F_{mor}}VV @ V{\mathscr F_{ob}}VV\\
\mathcal{MOR} @ >{{\rm Pr}_s}>> \mathcal{OB}.
\end{CD}
\end{equation}
\end{defn}
\begin{rem}
Note a diagram 
$$
\begin{CD}
A @ >{f}>> B
\\
@ V{g}VV @ V{g'}VV\\
C @ >{f'}>> D.
\end{CD}
$$
is said to be a {\it cartesian square} 
\index{cartesian square} if it commutes and the induced morphism
$A \to B \times_{D} C$ is an isomorphism.
\end{rem}

We elaborate on this definition below.
For $x \in \mathcal{OB}$ we write
$X_x = \mathscr F_{ob}^{-1}(x)$.
It is a complex analytic space, which is in general singular.
Let $\varphi \in \mathcal{MOR}$ and 
$x =  {\rm Pr}_s(\varphi)$ and $y = {\rm Pr}_t(\varphi)$.
Since (\ref{diagram212}) is a cartesian square 
we have isomorphisms:
\begin{equation}\label{form213}
\begin{CD}
X_x @<{{\rm Pr}_s}<< \mathscr F_{mor}^{-1}(\varphi) @>{{\rm Pr}_t}>> X_y.
\end{CD}
\end{equation}
Here the  arrows are restrictions of ${\rm Pr}_s$ and
${\rm Pr}_t$. They are isomorphisms.
Thus $\varphi$ induces an isomorphism $X_x \cong X_y$, 
which we write $\widetilde{\varphi}$.
Then using the compatibility of $\mathscr F_{mor}$ with compositions 
we can easily show
\begin{equation}
\widetilde{\varphi} \circ \widetilde{\psi}
= \widetilde{\varphi\circ \psi},
\end{equation}
if 
${\rm Pr}_s(\varphi) = {\rm Pr}_t(\psi)$.
(Here the right hand side is $\widetilde{\frak{comp}(\varphi,\psi)}$.)
\begin{exm}
Let $\frak X,\frak Y$ be complex manifolds 
on which a complex Lie group $G$ acts.
Let $\pi : \frak Y \to \frak X$ be a 
holomorphic map which is $G$-equivariant.
By Example \ref{exaGaction} we have
Lie groupoids whose  spaces of objects are 
$\frak X$ and $\frak Y$,
and whose  spaces of morphisms are 
$\frak X \times G$ and $\frak Y \times G$ respectively.
We denote them by $(\frak X,G)$ and $(\frak Y,G)$
\par
The projections define a morphism $(\frak Y,G)
\to (\frak X,G)$.
It is easy to see that by this morphism 
$(\frak Y,G)$ becomes a family of complex analytic spaces parametrized by  $(\frak X,G)$.
\end{exm}
\begin{cons}\label{const212}
Let $\pi : \frak Y \to \frak X$ be a proper, surjective and flat
holomorphic map between complex manifolds.
We put $X_x = \pi^{-1}(x)$ for $x \in X$.
We consider the set of triples:
\begin{equation}\label{mortal}
\aligned
\{ (x,y,\varphi) \mid x,y \in X, &\,\,\,\varphi : X_x \to X_y \\
&\text{is an isomorphism of complex analytic spaces.}\}
\endaligned
\end{equation}

We {\it assume} the space 
(\ref{mortal}) is a complex manifold and write it as 
$\mathcal{MOR}$.
We {\it assume} moreover the maps 
$\mathcal{MOR} \to \frak X$, $(x,y,\varphi) \mapsto x$ 
and $\mathcal{MOR} \to \frak X$, $(x,y,\varphi) \mapsto y$
are both submersions.
(See also Remark \ref{rem214}.)
We then define a Lie groupoid  
$$
\frak G = (\mathcal{OB},\mathcal{MOR},{\rm Pr}_s,{\rm Pr}_t,\frak{comp},\frak{inv},
\mathcal{ID})$$ and a family of complex analytic spaces parametrized by  $\frak G$
as follows.
\par
We first put $\mathcal{OB} = \frak X$,
$\mathcal{MOR} = (\ref{mortal})$, 
${\rm Pr}_s(x,y,\varphi) = x$, ${\rm Pr}_t(x,y,\varphi) = y$,
$\frak{comp}((x,y,\varphi),(y,z,\psi)) = (x,z,\psi\circ\varphi)$,
$\mathcal{ID}(x) = (x,x,{\rm id})$,
$\frak{inv}(x,y,\varphi) = (y,x,\varphi^{-1})$.
It is easy to see that we obtain Lie groupoid $\frak G$ in this way.
\par 
We next define $\widetilde{\frak G}$ as follows.
We put $\widetilde{\mathcal{OB}} = \frak Y$, 
$$
\aligned
\widetilde{\mathcal{MOR}}
=
\{(\tilde x,\tilde y,\varphi) \mid \tilde x,\tilde y \in \frak Y,
\,\,
&
\varphi(\tilde x) = \tilde y, \,\text{and}\,\,\varphi : \pi^{-1}(\pi(\tilde x)) \to \pi^{-1}(\pi(\tilde y))
\\
&\text{is an isomorphism of complex analytic spaces.}\}
\endaligned
$$
${\rm Pr}_s(\tilde x,\tilde y,\varphi) = \tilde x$, ${\rm Pr}_t(\tilde x,\tilde y,\varphi) = \tilde y$,
$\frak{comp}((\tilde x,\tilde y,\varphi),
(\tilde y,\tilde z,\psi)) = (\tilde x,\tilde z,\psi\circ\varphi)$
$\mathcal{ID}(\tilde x) = (\tilde x,\tilde x,{\rm id})$,
$\frak{inv}(\tilde x,\tilde y,\varphi) = (\tilde y,\tilde x,\varphi^{-1})$.
It is easy to see that we obtain a Lie groupoid $\widetilde{\frak G}$ in this way.
\par
The map $\pi : \frak Y \to \frak X$ together with 
$(\tilde x,\tilde y,\varphi) \mapsto (\pi(\tilde x),\pi(\tilde y),\varphi)$
defines a morphism $\mathscr F : \widetilde{\frak G} \to {\frak G}$.
\par
It is easy to check that (\ref{diagram212}) is a cartesian square in this case.
\par
We call $(\widetilde{\frak G},{\frak G},\mathscr F)$ the family associated to the map 
$\pi : \frak Y \to \frak X$.
\end{cons}
\begin{rem}\label{rem214}
Note by assumption ${\rm Pr}_s^{-1}(x) \cap {\rm Pr}_t^{-1}(y) \subset 
\mathcal{MOR}$ has a structure of complex variety.
For the construction to work we  need certain 
compatibility condition for this structure with 
one on $\{\varphi \mid \varphi : \pi^{-1}(x) \to \pi^{-1}(y),
\,\,\, \text{bi-holomorphic}\}.$ We do not discuss this point here.
We will discuss this point in the situation we use 
Construction \ref{const212} during the proof of Theorem 
\ref{them35}.
(See Remark \ref{rem321333}.)
\end{rem}

The assumptions that (\ref{mortal}) is a complex manifold 
and ${\rm Pr}_s$, ${\rm Pr}_t$ are submersions, are not 
necessarily satisfied in general. Here is a counter example.
Let $\Sigma = \Sigma_2 \cup_p S^2$. In other words,  we glue 
a genus 2 Riemann surface and $S^2$ at one point $p$.
We take coordinates of a neighborhood of  $p$ in $\Sigma_2$ and 
in $S^2$ and denote them by $z$ and $w$ respectively.
We assume $w^{-1} : D^2 \to S^2$ is a holomorphic map 
which extends to a bi-holomorphic map $S^2 \to S^2$.
We smooth the node by equating $zw = \rho$ for each 
$\rho \in D^2({\epsilon})$.
In this way we obtain a $D^2({\epsilon})$ parametrized family 
of nodal curves which gives a map $\pi : \mathcal C \to D^2({\epsilon})$
such that $\pi^{-1}(0) = \Sigma$ and $\pi^{-1}(\rho)$ is 
isomorphic to $\Sigma_2$ for $\rho \ne 0$.
(This is a consequence of our choice of the coordinate $w$.)
\begin{rem}
Here and hereafter we put \index{00D22r@$D^2(r)$}
$$
D^2(r) = \{ z \in \C \mid \vert z\vert < r\}. 
$$
\end{rem}
\par
We may take $\mathcal C$ to be a complex manifold of dimension $2$.
(See Subsection \ref{unideformationexi}.)
\begin{figure}[h]
\centering
\includegraphics[scale=0.3]{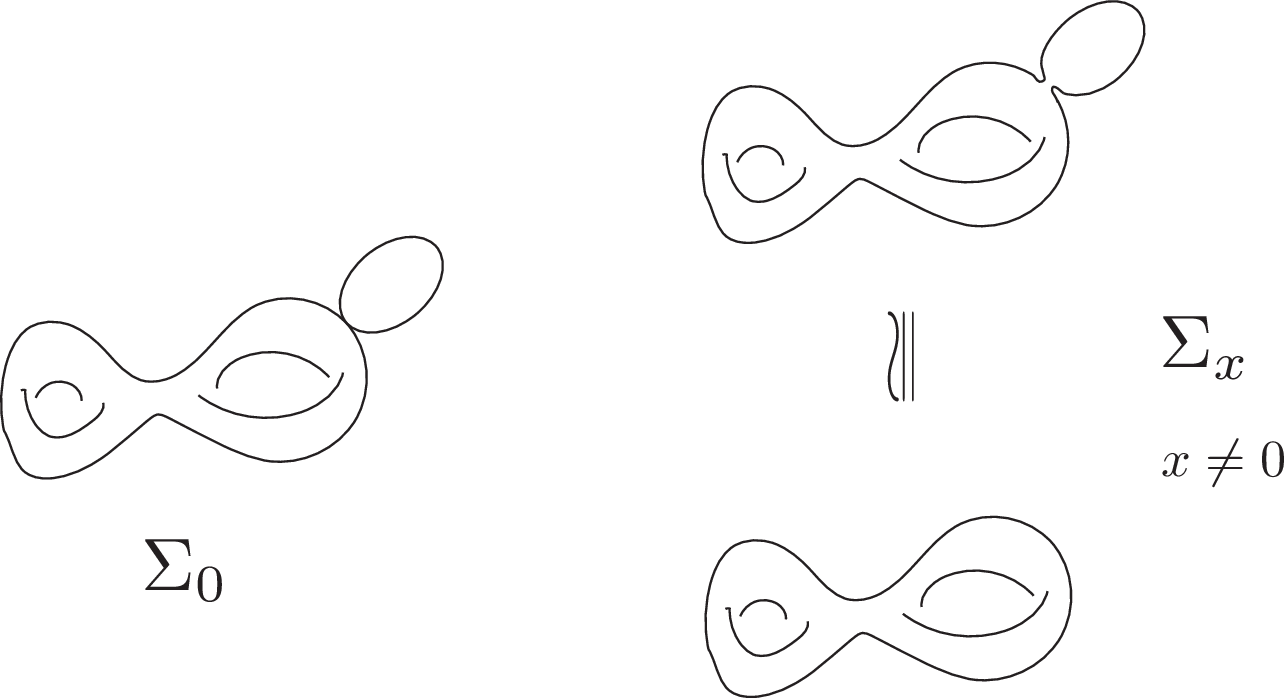}
\vskip1cm
\caption{$\Sigma_x$.}
\label{Figure1}
\end{figure}
Let up take $\frak Y = \mathcal C$ and $\frak X = D^2({\epsilon})$.
For $x \in \frak X$ we put $X_x = \pi^{-1}(x) \subset \mathcal C$.
Note:
\begin{enumerate}
\item
If $x,y \ne 0$ then there exists a unique bi-holomorphic map 
$X_x \to X_y$.
\item
If $x=y=0$ then the set of bi-holomorphic maps 
$X_x \to X_y$ is identified with the set of all affine 
transformations of $\C$, (that is, the maps of the form 
$z \to az +b$).
\item
If $x=0$, $y\ne 0$ then there exist no 
bi-holomorphic map 
$X_x \to X_y$.
\end{enumerate}
\par
For $x \in \frak X$ we consider the set of 
the pairs $(\varphi,y)$ such that $y \in \frak X$
and $\varphi : X_x \to X_y$ is a bi-holomorphic map.
(1)(2)(3) above imply that the complex dimension of 
the space of such 
pairs is $2$ if $x = 0$ and $1$ if $x\ne 0$.
Therefore in this case the map ${\rm Pr}_s$ 
can not be a submersion from a complex manifold.
\par\medskip
We will study moduli spaces of marked curves. 
So we include marking to Definition \ref{defn210}
as follows.

\begin{defn}\label{defn21022}
A {\it marked family of complex analytic spaces parametrized by  $\frak G$}, 
\index{marked family of complex analytic spaces parametrized by  $\frak G$}
is 
a triple $(\widetilde{\frak G},\mathscr F,\vec{\frak T})$,
where $(\widetilde{\frak G},\mathscr F)$ is a family of complex analytic spaces parametrized by  $\frak G$
and 
$\vec{\frak T} = (\frak T_1,\dots,\frak T_{\ell})$
such that $\frak T_i : \mathcal{OB} \to \widetilde{\mathcal{OB}}$ 
\index{00T4I@$\frak T_i$}
are holomorphic maps with the following properties.
\begin{enumerate}
\item
$\mathscr F_{ob}\circ\frak T_i = {\rm id}$.
\item
Let $\tilde\varphi \in \widetilde{\mathcal{MOR}}$
and $\tilde x = \widetilde{\rm Pr}_s(\tilde\varphi)$,
$x = \mathscr F_{ob}(\tilde x)$.
Suppose
$\tilde x = \frak T_i(x)$. 
Then 
$$
\widetilde{\rm Pr}_t(\tilde\varphi) = \frak T_i({\rm Pr}_t(\varphi)).
$$
\end{enumerate}
\end{defn}
Condition (2) is rephrased as the commutativity of the next diagram.
\begin{equation}
\begin{CD}
\mathscr F^{-1}_{ob}(x) @ >{\tilde\varphi}>> \mathscr F^{-1}_{ob}(y)
 \\
@ A{\frak T_i}AA @ AA{\frak T_i}A\\
x @ >{\varphi}>> y.
\end{CD}
\end{equation}
\begin{cons}\label{const2124}
Let $\pi : \frak Y \to \frak X$ be a proper, surjective and flat
holomorphic map between complex manifolds
and $\frak T_i : \frak X \to \frak Y$ 
holomorphic sections for $i=1,\dots,\ell$.
We put $X_x = \pi^{-1}(x)$ for $x \in X$.
We replace (\ref{mortal}) by 
\begin{equation}\label{mortal2}
\aligned
\{ (x,y,\varphi) \mid x,y \in X, &\,\,\,\varphi : X_x \to X_y,
\varphi(\frak T_i(x)) = \frak T_i(y), i=1,\dots,\ell\,\, \text{and} \\
&\text{$\varphi$ is an isomorphism of complex analytic spaces.}\}
\endaligned
\end{equation}
We define $\mathcal{MOR} = (\ref{mortal2})$.
We assume that it is a complex manifold. 
The maps ${\rm Pr}_s$, ${\rm Pr}_t$, which are defined by the same formula 
as  Construction \ref{const212}, are 
assumed to be submersions.
We then obtain $\frak G$, $\widetilde{\frak G}$ 
and $\mathscr F$ in the same way.
\par
Then together with $\frak T_i$, the pair
$(\widetilde{\frak G},\mathscr F)$ defines a 
marked family of complex analytic spaces parametrized by  $\frak G$.
\end{cons}
We next define a morphism between families of complex analytic spaces.
\begin{defn}
Let $(\widetilde{\frak G}^{(j)},\mathscr F^{(j)})$
be a family of complex analytic spaces
parametrized by ${\frak G}^{(j)}$ for $j=1,2$.
A morphism from 
$(\widetilde{\frak G}^{(1)},\mathscr F^{(1)},{\frak G}^{(1)})$
to
$(\widetilde{\frak G}^{(2)},\mathscr F^{(2)},{\frak G}^{(2)})$
is by definition a 
pair $(\widetilde{\mathcal H},\mathcal H)$ 
such that:
\begin{enumerate}
\item
$\widetilde{\mathcal H} : \widetilde{\frak G}^{(1)}
\to \widetilde{\frak G}^{(2)}$ and 
${\mathcal H} : {\frak G}^{(1)}
\to {\frak G}^{(2)}$ are morphisms such that the next diagram commutes.
\begin{equation}
\begin{CD}
\widetilde{\frak G}^{(1)} @ >{\widetilde{\mathcal H}}>> \widetilde{\frak G}^{(2)}
 \\
@ V{\mathscr F^{(1)}}VV @ VV{\mathscr F^{(2)}}V\\
{\frak G}^{(1)} @ >{\mathcal H}>> {\frak G}^{(2)}.
\end{CD}
\end{equation}
\item
The next diagram is a cartesian square.
\begin{equation}
\begin{CD}
\widetilde{\mathcal{OB}}^{(1)} @ >{\widetilde{\mathcal H}_{ob}}>> \widetilde{\mathcal{OB}}^{(2)}
 \\
@ V{\mathscr F^{(1)}_{ob}}VV @ VV{\mathscr F^{(2)}_{ob}}V\\
{\mathcal{OB}}^{(1)} @ >{\mathcal H_{ob}}>> {\mathcal{OB}}^{(2)}.
\end{CD}
\end{equation}
\end{enumerate}
Note Item (2) implies that for each $x \in {\mathcal{OB}}^{(1)}$, 
the restriction of $\widetilde{\mathcal H}_{ob}$ induces an isomorphism
$$
({\mathscr F^{(1)}_{ob}})^{-1}(x)
\cong
({\mathscr F^{(2)}_{ob}})^{-1}(\mathcal H_{ob}(x))
$$
\par
In case $(\widetilde{\frak G}^{(j)},\mathscr F^{(j)},\vec{\frak T}^{(j)})$
is a family of marked complex analytic spaces
parametrized by ${\frak G}^{(j)}$ for $j=1,2$,
a {\it morphism} between them is a pair $(\widetilde{\mathcal H},\mathcal H)$
satisfying (1)(2) and 
\begin{enumerate}
\item[(3)]
$
\widetilde{\mathcal H}_{ob}\circ \vec{\frak T}^{(1)}_i
=
\vec{\frak T}^{(2)}_i \circ {\mathcal H}_{ob}.
$
\end{enumerate}
\end{defn}
\begin{exm}
Let $(\widetilde{\frak G},\mathscr F)$
be a family of complex analytic spaces
parametrized by ${\frak G}$
and $\mathcal U$ an open set of 
$\mathcal{OB}$.
We put $\widetilde{\mathcal U}= \mathscr F_{ob}^{-1}(\mathcal U)
\subset \widetilde{\mathcal{OB}}
$.
We consider restrictions ${\frak G}\vert_{\mathcal U}$ of  
${\frak G}$ and $\widetilde{\frak G}\vert_{\widetilde{\mathcal U}}$
of   $\widetilde{\frak G}$.
\par
The restriction of $\mathscr F$ defines a morphism
$\mathscr F\vert_{\widetilde{\mathcal U}} :
\widetilde{\frak G}\vert_{\widetilde{\mathcal U}} \to 
\widetilde{\frak G}$.
\par
The pair 
$(\widetilde{\frak G}\vert_{\widetilde{\mathcal U}},
\mathscr F\vert_{\widetilde{\mathcal U}})$
becomes a family of complex analytic spaces
parametrized by ${\frak G}\vert_{\mathcal U}$.
We call it the {\it restriction} of $(\widetilde{\frak G},\mathscr F,\mathcal G)$
to $\mathcal U$.
\par
The obvious inclusion defines a morphism 
$(\widetilde{\frak G}\vert_{\widetilde{\mathcal U}},
\mathscr F\vert_{\widetilde{\mathcal U}},{\frak G}\vert_{\mathcal U})
\to 
(\widetilde{\frak G},
\mathscr F,{\frak G})$
of families of complex analytic spaces.
We call it an {\it open inclusion} of 
families of complex analytic varieties.
\par
The version with marking is similar.
\end{exm}
\begin{exm}
Let $\pi : \frak Y \to \frak X$ be a 
holomorphic map and $\frak X' \to \frak X$
a holomorphic map. We put
$\frak Y' = \frak Y \times_{\frak X} \frak X'$
and assume $\frak Y'$ is a complex manifold.
Suppose the assumptions in Construction \ref{const212}
is satisfied both for $\pi : \frak Y \to \frak X$
and $\pi' : \frak Y' \to \frak X'$.
\par
Then the morphism from the families induced by 
$\pi' : \frak Y' \to \frak X'$
to the families induced by 
$\pi : \frak Y \to \frak X$
is obtained in an obvious way.
\end{exm}
\begin{lem}\label{lem218}
Let $(\widetilde{\frak G}^{(j)},\mathscr F^{(j)})$
be a family of complex analytic spaces
parametrized by ${\frak G}^{(j)}$ for $j=1,2$,
and $(\widetilde{\mathcal H}^{(k)},\mathcal H^{(k)})$ 
a morphism from
$(\widetilde{\frak G}^{(1)},\mathscr F^{(1)},{\frak G}^{(1)})$
to
$(\widetilde{\frak G}^{(2)},\mathscr F^{(2)},{\frak G}^{(2)})$
for $k=1,2$.
\par
Suppose $\mathcal H^{(1)}$ is conjugate to $\mathcal H^{(2)}$.
Then $\widetilde{\mathcal H}^{(1)}$ is conjugate to $\widetilde{\mathcal H}^{(2)}$.
\end{lem}
\begin{proof}
Let $\mathcal T$ be a natural transformation from 
$\mathcal H^{(1)}$ to $\mathcal H^{(2)}$.
\par
Let $\tilde x \in \widetilde{\mathcal{OB}}$.
We put $x = \mathscr F_{ob}(\tilde x)$ and 
$y = {\rm Pr}_t(\mathcal T(x))$
Then $\mathcal T(x)$
induces
$$
\mathcal T(x) : \mathscr F_{ob}^{-1}(x) \to \mathscr F_{ob}^{-1}(y).
$$
by (\ref{form213}).
We put
$$
\tilde y = \mathcal T(x)(\tilde x).
$$
Using the cartesian square (\ref{diagram212})
there exists a unique $\widetilde{\mathcal T}(\tilde x)
\in \widetilde{\mathcal{MOR}}$ such that
$$
\mathscr F_{mor}(\widetilde{\mathcal T}(\tilde x))
= \mathcal T(x) \qquad
\widetilde{{\rm Pr}_s}(\widetilde{\mathcal T}(\tilde x)) = \tilde x, \qquad
\widetilde{{\rm Pr}_t}(\widetilde{\mathcal T}(\tilde x)) = \tilde y.
$$
Using the cartesian square (\ref{diagram212}) again 
it is easy to check that $\widetilde{\mathcal T}$ is a 
natural transformation from $\widetilde{\mathcal H}^{(1)}$ to 
$\widetilde{\mathcal H}^{(2)}$.
\end{proof}
\begin{defn}
We say $(\widetilde{\frak G}^{(1)},\mathscr F^{(1)})$
is conjugate to $(\widetilde{\frak G}^{(2)},\mathscr F^{(2)})$
if the assumption of Lemma \ref{lem218} is satisfied.
\end{defn}
\par
Our main interest in this paper is local theory.
We define the next notion for this purpose.
\begin{defn}\label{def218}
Let $(X,\vec z)$ be a pair of complex analytic space $X$ 
and an $\ell$-tuple of mutually distinct smooth points $\vec z = (z_1,\dots,z_{\ell})$.
A {\it deformation} \index{deformation} 
of $(X,\vec z)$ is by definition 
an object $(\widetilde{\frak G},
\mathscr F,{\frak G},\vec{\frak T},o,\iota)$
with the following properties.
\begin{enumerate}
\item
The triple $(\widetilde{\frak G},
\mathscr F,\vec{\frak T})$ is a marked family of 
complex variety 
parametrized by ${\frak G}$.
\item $o\in \mathcal{OB}$.
\item
$\iota : X \to \mathscr F_{ob}^{-1}(o)$ is a bi-holomorphic 
map.
\item
$\frak T_i(o) = \iota(z_i)$.
\end{enumerate}
\par
Let $\mathcal G^{(j)} = (\widetilde{\frak G}^{(j)},
\mathscr F^{(j)},{\frak G}^{(j)},\vec{\frak T}^{(j)},o^{(j)},\iota^{(j)})$
be a deformation of $(X,\vec z)$ for $j=1,2$.
A {\it strict morphism} from $\mathcal G^{(1)}$  to $\mathcal G^{(2)}$ 
is a morphism from  $(\widetilde{\mathcal H},\mathcal H)$  
from $\mathscr F^{(1)},{\frak G}^{(1)},\vec{\frak T}^{(1)}$
to 
$\mathscr F^{(2)},{\frak G}^{(2)},\vec{\frak T}^{(2)}$
such that:
\begin{enumerate}
\item[(i)]
$
{\frak T}^{(2)}_i \circ \mathcal H
= \widetilde{\mathcal H} \circ {\frak T}^{(1)}_i.
$
\item[(ii)]
$o^{(2)} = \mathcal H_{ob}(o^{(1)})$.
\item[(iii)]
$
\widetilde{\mathcal H}_{ob} \circ \iota^{(1)} = \iota^{(2)}.
$ 
\end{enumerate}
Let $\mathcal G = (\widetilde{\frak G},
\mathscr F,{\frak G},\vec{\frak T},o,\iota)$
be a deformation of $(X,\vec z)$
and $\mathcal U$ is an open neighboorhood
$o$ in $\mathcal{OB}$.
We define the restriction of $\mathcal G$ 
to $\mathcal U$ in an obvious way and 
denote it by $\mathcal G\vert_{\mathcal U}$.
\par
A {\it morphism} from $\mathcal G^{(1)}$
to $\mathcal G^{(2)}$ 
is a strict morphism from $\mathcal G^{(1)}\vert_{\mathcal U}$
to $\mathcal G^{(2)}$ for a certain open neighboorhood
$o$ in $\mathcal{OB}^{(1)}$.
\par
Two morphisms are said to {\it coincide as germs} 
if they coincides after further restricting to a smaller neighborhood 
of  $o$ in $\mathcal{OB}^{(1)}$.
\par
We can compose two strict morphisms or two morphisms in an obvious way.
There is an identity morphism from  $\mathcal G$ to itself.
\par
A morphism $\mathscr H = (\widetilde{\mathcal H},\mathcal H;\mathcal U)$ from $\mathcal G^{(1)}$
to $\mathcal G^{(2)}$ is said to be an {\it isomorphism},
if there exists a 
morphism  $\mathscr H' = (\widetilde{\mathcal H}',\mathcal H';\mathcal U')$
from $\mathcal G^{(1)}$ to $\mathcal G^{(2)}$ 
such that the compositions $\mathscr H' \circ \mathscr H$
and $\mathscr H \circ \mathscr H'$ coincides with the identity morphisms 
as germs.
\par
A {\it germ of deformation} of $(X,\vec z)$ is an isomorphism class 
with respect to the isomorphism defined above.
\end{defn}
\begin{defn}
Let $\mathcal G^{(j)} = (\widetilde{\frak G}^{(j)},
\mathscr F^{(j)},{\frak G}^{(j)},\vec{\frak T}^{(j)},o^{(j)},\iota^{(j)})$
be a deformation of $(X,\vec z)$ for $j=1,2$.
Two strict morphisms $(\widetilde{\mathcal H},\mathcal H)$,
$(\widetilde{\mathcal H}',\mathcal H')$
from $\mathcal G^{(1)}$ to $\mathcal G^{(2)}$ 
are said to be {\it conjugate} if there exists a pair of 
natural transformation $(\widetilde{\mathcal T},\mathcal T)$ from $(\widetilde{\mathcal H},\mathcal H)$
as in Lemma \ref{lem218} such that:
\begin{enumerate}
\item 
${\frak T}^{(2)}_i \circ \mathcal T = \widetilde{\mathcal T} \circ {\frak T}^{(2)}_i$.
\item
$\mathcal T_{ob}(o^{(1)}) = o^{(2)}$.
\item
$\widetilde{\mathcal T}_{ob} \circ \iota^{(1)} = \iota^{(2)}$.
\end{enumerate}
\par
Two morphisms from $\mathcal G^{(1)}$ to $\mathcal G^{(2)}$ 
are said to be conjugate if they are conjugate as strict morphisms 
after restricting to a certain open neighborhood of $o^{(1)}$.
\end{defn}

\section{Universal deformations of unstable marked curves}
\label{sec:universal}

In this section we specialize to the case of family of 
one dimensional complex varieties and show existence and 
uniqueness of a universal family for certain class of deformations.

\subsection{Universal deformation and its uniqueness}
\label{unideformation}

Let $\pi : \frak Y \to \frak  X$ be a holomorphic map
and $x \in \frak X$. 
We put $\Sigma_x = \pi^{-1}(x)$.

\begin{defn}
We say $\pi : \frak Y \to \frak  X$ is a {\it nodal family} and $\Sigma_x$ is a {\it nodal curve} \index{nodal curve} if
for each $y \in \Sigma_x$ one of the following holds.
\begin{enumerate}
\item
$D_y\pi : T_y\frak Y \to T_x\frak X$ is surjective.
$\dim_{\C}{\rm Ker}\,D_x\pi =1$.
\item
Let $\mathscr I_x$ be the ideal of germs of holomorphic functions
on $X$ at $x$ which vanish at $x$.
Then we have
$$
\frac{\mathcal O_y}{\pi^* \mathscr I_x} = \frac{\C\{z,w\}}{(zw)}.
$$ 
Here $\mathcal O_y$ is the ring of germs of holomorphic functions 
of $\frak Y$ at $y$. The ring $\C\{z,w\}$ is the convergent power
series ring of two variables. 
\end{enumerate}
We say $y$ is a {\it regular point} \index{regular point} if Item (1) happens 
and $y$ is a {\it nodal point} \index{nodal point} if Item (2) happens.
\end{defn}
\begin{defn}
Let $(\widetilde{\frak G},\mathscr F)$
be a family of complex analytic varieties
parametrized by ${\frak G}$.
We say that 
$(\widetilde{\frak G},\mathscr F,{\frak G})$
is a {\it family of nodal curves} \index{family of nodal curves}
if  $\mathscr F_{ob} : \widetilde{\mathcal{OB}}
\to \mathcal{OB}$ is a nodal family.
\par
A marked family 
$(\widetilde{\frak G},\mathscr F,\vec{\frak T})$
of complex analytic spaces parametrized by  $\frak G$ is 
said to be a {\it family of marked nodal curves}\index{family of marked nodal curves}
if $(\widetilde{\frak G},\mathscr F,{\frak G})$ is a family of nodal curves 
and the following holds.
\begin{enumerate}
\item
For any $x \in \mathcal{OB}$ the point $\frak T_i(x) 
\in \mathscr F_{ob}^{-1}(x)$ is a regular point
of $\mathscr F_{ob}^{-1}(x)$.
\item
If $i\ne j$ and $x \in \mathcal{OB}$,
then $\frak T_i(x)\ne \frak T_j(x)$.
\end{enumerate}
\end{defn}
\begin{defn}
Let $(\widetilde{\frak G},\mathscr F)$
be a family of complex analytic spaces
parametrized by ${\frak G}$.
We say that $(\widetilde{\frak G},\mathscr F,{\frak G})$
is {\it minimal at $o$} \index{minimal at $o$}
if the following holds.
\par
If $\varphi \in \mathcal{MOR}$ with ${\rm Pr}_s(\varphi) = o$
then ${\rm Pr}_t(\varphi) = o$.
\end{defn}
\begin{defn}\label{defn34}
Let $(\Sigma,\vec z)$ be a marked nodal curve
and $\mathcal G = (\widetilde{\frak G},
\mathscr F,{\frak G},\vec{\frak T},o,\iota)$
a deformation of $(\Sigma,\vec z)$.
We say that $\mathcal G$ is a 
{\it universal deformation} \index{universal deformation}
of $(\Sigma,\vec z)$
if the following holds.
\begin{enumerate}
\item
$(\widetilde{\frak G},
\mathscr F,{\frak G},\vec{\frak T})$
is a family of nodal curves and is minimal at $o$.
\item
For any deformation $\mathcal G' = (\widetilde{\frak G'},
\mathscr F',{\frak G}',\vec{\frak T}',o',\iota')$
of $(\Sigma,\vec z)$ such that $(\widetilde{\frak G}',
\mathscr F',{\frak G}',\vec{\frak T}')$
is a family of nodal curves,
there exists a morphism 
$(\widetilde{\mathcal H},\mathcal H)$ 
(Definition \ref{def218}) from 
$\mathcal G'$ to $\mathcal G$.
\item
In the situation of Item (2) if 
$(\widetilde{\mathcal H}',\mathcal H')$
is another morphism then $\mathcal H'$ is conjugate 
to $\mathcal H$.
\end{enumerate}
\end{defn}
The main result of this section is the following.
\begin{thm}\label{them35}
For any marked nodal curve $(\Sigma,\vec z)$
there exists its universal deformation 
$\mathcal G = (\widetilde{\frak G},
\mathscr F,{\frak G},\vec{\frak T},o,\iota)$.
\par
If  
$\mathcal G^{(j)} = (\widetilde{\frak G}^{(j)},
\mathscr F^{(j)},{\frak G}^{(j)},\vec{\frak T}^{(j)},o^{(j)},\iota^{(j)})$,
$j=1,2$ are both universal deformations of $(\Sigma,\vec z)$
then they are isomorphic as germs in the sense of Definition \ref{def218}.
\end{thm}
\begin{rem}
If $(\Sigma,\vec z)$ is marked stable curve, \index{stable curve} that is, the group of its automorphisms 
is a finite group, the universal deformation 
$\mathcal G = (\widetilde{\frak G},
\mathscr F,{\frak G},\vec{\frak T},o,\iota)$ is \'etale. Namely
${\rm Pr}_s : \mathcal{MOR} \to \mathcal{OB}$ is a local diffeomorphism.
Theorem \ref{them35} in this case follows from the classical result that the moduli space 
of marked stable curve is an orbifold.
(In some case this orbifold is not effective.)
Orbifold is a classical and well-established notion in differential geometry
\cite{satake}. The fact that orbifold  can be studied using the 
language of  \'etale groupoid is also classical \cite{he}.
\par
In the case when $(\Sigma,\vec z)$ is not stable, 
$\dim \mathcal{MOR} > \dim \mathcal{OB}$ and so ${\frak G}$ is not \'etale.
Therefore using the language of Lie groupoid is more important in this case
than the case of orbifold.
\par
It seems unlikely that there is a literature which proves a similar 
result as Theorem \ref{them35} by the method of differential geometry.
Something equivalent to Theorem \ref{them35} 
is known in algebraic geometry using the terminology of Artin Stack
(\cite{artin}).
See \cite[Chapter V 3.2.1 and 5.5.3]{manin}.
For our purpose of proving Theorem \ref{thm54}, differential geometric 
approach is important. So we provide a detailed proof of 
Theorem \ref{them35} below.
\end{rem}
\begin{proof}
In this subsection we prove the uniqueness.
The existence will be proved in the next subsection.
\par
Suppose 
$\mathcal G^{(j)} = (\widetilde{\frak G}^{(j)},
\mathscr F^{(j)},{\frak G}^{(j)},\vec{\frak T}^{(j)},o^{(j)},\iota^{(j)})$,
$j=1,2$ are both universal deformations of $(\Sigma,\vec z)$.
Then by Definition \ref{defn34} (2), there exists a morphism 
$(\widetilde{\mathcal H},\mathcal H)$ from 
$\mathcal G^{(1)}$ to $\mathcal G^{(2)}$
and also 
a morphism 
$(\widetilde{\mathcal H}',\mathcal H')$ from 
$\mathcal G^{(2)}$ to $\mathcal G^{(1)}$.
\par
The composition 
$(\widetilde{\mathcal H}',\mathcal H')
\circ (\widetilde{\mathcal H},\mathcal H)$
is a morphism from $\mathcal G^{(1)}$
to itself.
By Definition \ref{defn34} Item (3)
it is conjugate to the identity morphism.
\begin{lem}\label{lem37}
A morphism from $(\widetilde{\frak G},
\mathscr F,{\frak G})$
to itself which is conjugate to the identity 
morphism is necessarily an isomorphism in a neighborhood of $o$, 
if $(\widetilde{\frak G},
\mathscr F,{\frak G})$ is minimal at $o$.
\end{lem}
Postponing the proof of the lemma we continue the 
proof.
\par
By the lemma we replace 
$(\widetilde{\mathcal H}',\mathcal H')$ if 
necessary and may assume that
$(\widetilde{\mathcal H}',\mathcal H')
\circ (\widetilde{\mathcal H},\mathcal H)
= {\rm id}$.
\par
By the same argument 
the composition 
$(\widetilde{\mathcal H},\mathcal H)
\circ (\widetilde{\mathcal H}',\mathcal H')$
is an isomorphism.
We may replace 
$(\widetilde{\mathcal H}',\mathcal H')$
by 
$(\widetilde{\mathcal H}'',\mathcal H'')$
and find that
$(\widetilde{\mathcal H},\mathcal H)
\circ (\widetilde{\mathcal H}'',\mathcal H'')
= {\rm id}$.
\par
Then by a standard argument 
$(\widetilde{\mathcal H}',\mathcal H')
=(\widetilde{\mathcal H}'',\mathcal H'').$
\par
Thus to complete the proof of uniqueness 
it remains to prove the lemma.
\end{proof}
\begin{proof}[Proof of Lemma \ref{lem37}]
Let $(\widetilde{\mathcal H},\mathcal H)$ 
be a morphism from $(\widetilde{\frak G},
\mathscr F,{\frak G})$
to itself, which is conjugate to the identity.
\par
By Lemma \ref{lem28} there exists 
$\mathcal T : \mathcal{OB} \to \mathcal{MOR}$
such that $\mathcal H = \frak{conj}^{\mathcal T}$.
\begin{sublem}\label{sublem38}
The map $\mathcal H_{ob} : \mathcal{OB} \to \mathcal{OB}$
is a diffeomorphism on a neighborhood of $o$.
\end{sublem}
\begin{proof}
By minimality at $o$,
we find $\mathcal H_{ob}(o) = o$.
Let $\varphi_o = \mathcal T(o)$.
We have
\begin{equation}\label{form31}
\mathcal H_{ob}(x) = {\rm Pr}_t(\mathcal T(x)).
\end{equation}
Using implicit function theorem we may identify
a neighborhood of $\mathcal T(o)$ in $\mathcal{MOR}$
with $\mathcal U \times \mathcal V$ such that
$\mathcal V \subset {\rm Pr}_s^{-1}(o)$ is an 
open neighborhood of $\mathcal T(o)$,
$\mathcal U$ is an open neighborhood of $o$ 
in $\mathcal{OB}$
and 
$
{\rm Pr}_s : \mathcal U \times \mathcal V \to \mathcal{OB}
$
is the projection.
We remark that the derivative in the $\mathcal V$
direction of 
$
{\rm Pr}_t : \mathcal U \times \mathcal V \to \mathcal{OB}
$
is zero on $\{o\} \times \mathcal V$ by minimality.
On the other hand the derivative in the $\mathcal U$ direction 
of ${\rm Pr}_t$ at $(o,\mathcal T(o))$ is invertible.
This is because ${\rm Pr}_t$ is a submersion and 
the derivative in the $\mathcal V$ direction is zero.
\par
Thus derivative of (\ref{form31}) is invertible at $o$.
In fact
$$
\mathcal H_{ob}(x) = {\rm Pr}_t(x,\mathcal T'(x))
$$
for some $\mathcal T' : \mathcal U \to \mathcal V$.
Sublemma \ref{sublem38} now follows from inverse function theorem.
\end{proof}
Thus we proved that $\mathcal H_{ob}$ is invertible.
It is easy to see that it implies that $\mathcal H$ is invertible.
\par
By Lemma \ref{lem218}, $\widetilde{\mathcal H}$ is conjugate 
to the identity morphism. We can use it in the same way as 
above to show that $\widetilde{\mathcal H}$ is an 
isomorphism if we restrict to a smaller neighborhood of $o$.
The proof of Lemma \ref{lem37} is complete.
\end{proof}

\subsection{Existence of the universal deformation}
\label{unideformationexi}

In this section we prove the existence part of 
Theorem \ref{them35}.
We use the existence of universal deformation 
of {\it stable} marked nodal curve, 
which was well established long time ago 
and by now well-known, and use it to 
study unstable case.
\par
Let $(\Sigma,\vec z)$ be a marked nodal curve.
We decompose $\Sigma$ into irreducible components \index{00A3@$\mathcal A$}
\begin{equation}
\Sigma = \bigcup_{a \in \mathcal A} \Sigma_a.
\end{equation}
We regard the intersection $\vec z \cap \Sigma_a$ 
and all the nodal points on $\Sigma_a$ as marked points of 
$\Sigma_a$ and denote it by $\vec z_a$.
We put $\vec z_a = (z_{a,1},\dots,z_{a,\ell_a})$.
We recall that $(\Sigma_a,\vec z_a)$ is {\it stable}
\index{stable; irreducible component}\index{unstable; irreducible component} unless one of the following 
holds:
\begin{enumerate}
\item[(US.0)] The genus of $\Sigma_a$ is $0$ and $\# \vec z_a = 0$.
\item[(US.1)] The genus of $\Sigma_a$ is $0$ and $\# \vec z_a = 1$.
\item[(US.2)] The genus of $\Sigma_a$ is $0$ and $\# \vec z_a = 2$.
\item[(US.3)] The genus of $\Sigma_a$ is $1$ and $\# \vec z_a = 0$.
\end{enumerate}
Note in case (US.0), $\Sigma = \Sigma_a$ and it is easy to construct a universal deformation.
(In fact $\mathcal{OB}$ consists of one point, 
$\mathcal{MOR} = PSL(2;\C)$.
$\widetilde{\frak{G}}$ is defined by using $PSL(2;\C)$
action on $S^2$.)
In case (US.3), again $\Sigma = \Sigma_a$.
We can define universal deformation easily also.
($\mathcal{OB}$ is an open subset of the moduli space of elliptic curves.
Other objects can be obtained by applying 
Construction \ref{const2124} to the universal family of elliptic curves.)
\par
Therefore we consider the case when all the unstable 
components are either of type (US.1) or (US.2).
\par
Let $\mathcal A_s$ be the subset of $\mathcal A$
\index{00A3S@$\mathcal A_s$}\index{00A3U@$\mathcal A_u$} 
consisting of elements $a$ such that $(\Sigma_a,\vec z_a)$ is stable.
We put $\mathcal A_u = \mathcal A \setminus \mathcal A_s$.
\par
Suppose $(\Sigma_a,\vec z_a)$ is a stable curve.
Let $g_a$ be its genus and $\ell_a = \#\vec z_a$.
We consider the moduli space $\mathcal M_{g_a,\ell_a}$
\index{00M3Fwll@$\mathcal M_{g,\ell}$}
of stable curves with genus $g_a$ and with $\ell_a$ 
marked points.
$\mathcal M_{g_a,\ell_a}$ is an orbifold. (In some exceptional 
case it is not effective.)
Let 
$\mathcal V_a/\mathcal G_a$ be a neighborhood of $(\Sigma_a,\vec z_a)$
in $\mathcal M_{g_a,\ell_a}$. Here $\mathcal G_a$ is a finite group 
which is the group of automorphisms of $(\Sigma_a,\vec z_a)$.
Namely 
$$
\mathcal G_a = \{ v : \Sigma_a \to \Sigma_a \mid 
\text{$v$ is bi-holomorphic}, \,\, v(z_{a,i}) = z_{a,i}\}.
$$
$\mathcal V_a$ is a smooth complex manifold on which 
a finite group $\mathcal G_a$ acts. 
We have a universal family 
\begin{equation}
\pi_a : \mathcal C_a \to \mathcal V_a
\end{equation}
where $\mathcal C_a$ is a complex manifold and $\pi_a$ is a 
proper submersion.
The group $\mathcal G_a$ acts on $\mathcal C_a$ and $\pi_a$ is $\mathcal G_a$
equivariant.
We also have holomorphic maps
\begin{equation}
\frak t_{a,i} : \mathcal V_a \to \mathcal C_a
\end{equation}
for $i=1,\dots,\ell_a$, 
such that
$\pi_a \circ \frak t_{a,i} = {\rm id}$ and 
$\frak t_{a,i}$ is $\mathcal G_a$ equivariant.
Moreover $\frak t_{a,i}(x) \ne \frak t_{a,j}(x)$
for $x \in \mathcal V_a$, $i\ne j$.
Finally the marked Riemannn surface
$$
(\pi_a^{-1}(x),(\frak t_{a,1}(x),\dots,\frak t_{a,\ell_a}(x)))
$$
is a representative of the element $[x] \in \mathcal V_a/\mathcal G_a \subset \mathcal M_{g_a,\ell_a}$.
Existence of such 
$\mathcal G_a,\mathcal V_a,\mathcal C_a,\pi_a,\frak t_{a,i}$ is classical.
(See \cite{alcorGri} for example. This one dimensional and local version 
of deformation theory of complex structure 
had been known in 19th century.)
\par
Suppose $(\Sigma_a,\vec z_a)$ is unstable. 
We put $\mathcal V_a = {\rm point}$. The group of 
automorphisms $\mathcal G_a$ is $\C^* = \C \setminus \{0\}$ if 
(US.1) is satisfied. The group of 
automorphisms $\mathcal G_a$ consists of affine maps 
$z \mapsto az + b$ in case (US.2) is satisfied.
(Here we identify $(\Sigma_a,\vec z_a) = (\C\cup \{\infty\},\infty)$.
\par
We put
\begin{equation}\label{form3535}
\mathcal G = \{v : \Sigma \to \Sigma 
\mid \text{$v$ is bi-holomorphic}, \,\, v(z_{i}) = z_{i}\}.
\end{equation}
We then have an exact sequence of groups:
\begin{equation}
1 \to \prod_{a \in \mathcal A}\mathcal G_a 
\to \mathcal G \to \mathscr H \to 1.
\end{equation}
Here $\mathscr H$ is a finite group.
The group $\mathscr H$ is a subgroup of the automorphism group of 
the dual graph \index{dual graph} of $\Sigma$. (Here the dual graph is defined as follows.
We associate a vertex to each of the irreducible components of $\Sigma$.
We associate an edge to each of the nodal points. The vertices of 
an edge is one associated to the irreducible components 
containing that nodal points.
See Figure \ref{Figure2} below.)
\begin{figure}[h]
\centering
\includegraphics[scale=0.4]{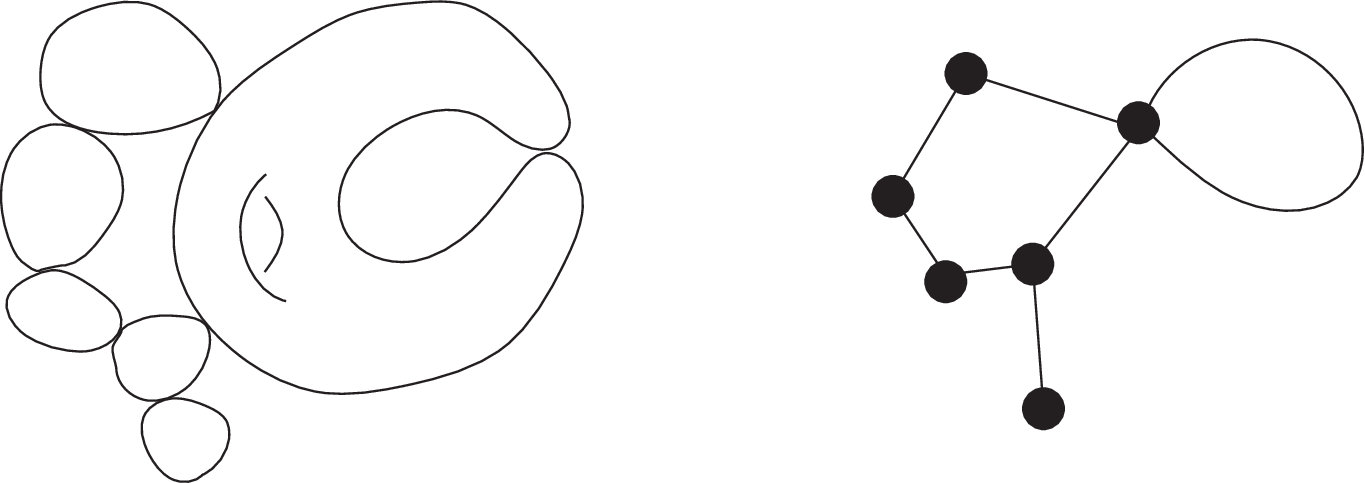}
\caption{Dual graph to nodal curve}
\label{Figure2}
\end{figure}
We put $\overline{\mathcal G} = \pi_0(\mathcal G)$.
Then we have an exact sequence 
\begin{equation}
1 \to \prod_{a \in \mathcal A_s}\mathcal G_a 
\to \overline{\mathcal G} \to \mathscr H \to 1.
\end{equation}
We put
\begin{equation}
\mathcal V_0 = \prod_{a \in \mathcal A_s} \mathcal V_a.
\end{equation}
Let $o \in \mathcal V_0$ be an element corresponding to 
$\Sigma$.\index{00V30delta@$\mathcal V_0$}
\par
The group $\overline{\mathcal G}$ acts on $\mathcal V_0$ in an obvious way.
For each $x= (x_a)_{a\in \mathcal A_s}$ we define 
$\Sigma(x)$ as follows.
We take $\Sigma(x_a) = \pi^{-1}(x_a)$ for $a \in \mathcal A_s$.
If $a \in \mathcal A_u$ we take $\Sigma_a$. We glue
$$
\coprod_{a \in \mathcal A_s} \Sigma(x_a)
\sqcup
\coprod_{a \in \mathcal A_u} \Sigma_a
$$
at their marked points in exactly the same way as $\Sigma$.
We then obtain a nodal curve $\Sigma(x)$.
We define
$$
\mathcal C_0 = \coprod_{x \in \mathcal V_0} \Sigma(x) \times \{x\}.
$$
We have an obvious projection
\begin{equation}\label{form39}
\pi : \mathcal C_0 \to \mathcal V_0.
\end{equation}
$\overline{\mathcal G}$ acts on $\mathcal C_0$ in an obvious way
and then (\ref{form39}) is a deformation of $\Sigma$ while 
keeping singularities.
Later in (\ref{defmathcalC}) we will 
embed $\mathcal C_0$ to a complex manifold 
so that $\mathcal C_0$ is a complex subvariety. The choice of 
complex structure 
of $\mathcal C_0$ then will become clear.
Using the map $\frak t_{a,i}$ which does not correspond to the 
nodal point of  $\Sigma(x)$ we obtain maps \index{00T4J@$\frak t_{j}$}
$$
\frak t_{j} : \mathcal V_0 \to \mathcal C_0
$$
for $j=1,\dots,\ell$ such that $\pi\circ \frak t_{j} = {\rm id}$
and that $\frak t_{j}$ is $\overline{\mathcal G}$ equivariant.
\par
We next include the parameter to smooth nodal points of $\Sigma(x)$.
We need to choose a coordinate at each nodal points, in the 
following sense.
Let $D^2(r)$ be the open ball of radius $r$ centered at $0$ in $\C$.
\begin{defn}{\rm (\cite[Definition 8.1]{foootech})}\label{cooddinateatinf}
An {\it  analytic family of coordinates} 
\index{analytic family of coordinates} at $\frak t_{a,i}$ is a 
holomorphic map 
$$
\varphi_{a,i} : \mathcal V_{a} \times D^2(2) \to \mathcal C_a
$$
such that:
\begin{enumerate}
\item
$\pi(\varphi_{a,i}(x,z)) = x$, for all $z \in D^2(2)$.
\item
$\varphi_{a,i}(x,0) = \frak t_{a,i}(x)$.
\item
For each $x$ the map $z \mapsto \varphi_{a,i}(x,z)$ 
is a bi-holomorphic map from $D^2(2)$ to a neighborhood of 
$\frak t_{a,i}(x)$ in $\pi_a^{-1}(x)$.
\end{enumerate}
We say that a system $(\varphi_{a,i})_{a\in \mathcal A_s,i=1,\dots,\ell_a}$
of analytic families of  coordinates are $\overline{\mathcal G}$ equivariant
if the following holds.
Let $\gamma \in \overline{\mathcal G}$ and $[\gamma] \in {\mathscr H}$.
We consider
$$
{\rm Node} = \{(a,i) \mid \text{$z_{a,i}$ corresponds to a nodal point 
of $\Sigma$ on $\Sigma_a$.}\}
$$
Since ${\mathscr H}$ acts on the dual graph of $\Sigma$ it acts on 
${\rm Node}$ also. Now we require:
\begin{enumerate}
\item[(*)] If $[\gamma](a,i) = (a',i')$ then \index{00P59hiAI@$\varphi_{a,i}(x,z)$}
$$
\gamma(\varphi_{a,i}(x,z)) = \exp(\theta_{\gamma,a,i} \sqrt{-1})  \varphi_{a',i'}(x,z).
$$
Here $\theta_{a,i} \in \R$.
\end{enumerate}
\end{defn}
\begin{lem}\label{lem399}
There exists a $\overline{\mathcal G}$ equivariant 
analytic families of  coordinates.
\end{lem}
See \cite[Lemma 8.4]{foooexp} for the proof of this lemma.
\begin{rem}\label{rem31215}
We use analytic family of coordinates at the marked 
points corresponding to the nodal points only.
\end{rem}
For each $(z,i) \in {\rm Node}$ we take a copy of $\C$ and denote it 
by $\C_{(z,i)}$.
We fix an orientation of the edges of the dual graph $\Gamma(\Sigma)$ 
\index{00G5ammaSigma@$\Gamma(\Sigma)$ } of 
$\Sigma$. For each edge ${\rm e}$ of $\Gamma(\Sigma)$, that corresponds 
to the nodal points, let $z_{-,{\rm e}}, z_{+,{\rm e}} \in {\rm Node}$
such that the orientation of ${\rm e}$ goes from the vertex 
corresponding to $z_{-,{\rm e}}$ to the vertex corresponding to $z_{+,{\rm e}}$.
\begin{defn}\label{defn313}
We put \index{00V31@$\mathcal V_1$}
$$
\mathcal V_1 = \bigoplus
\C_{-,{\rm e}} \otimes \C_{+,{\rm e}},
$$
where the direct sum is taken over all the edges ${\rm e}$ of $\Gamma(\Sigma)$.
The element $\gamma \in \overline{\mathscr G}$ acts on $\mathcal V_1$ 
by sending $w \in \C_{-,{\rm e}}^*$
(resp. $w \in \C^*_{+,{\rm e}}$)  to 
$\exp(\theta_{\gamma,a,i} \sqrt{-1})w$ \index{00T5hetagamma@$\theta_{\gamma,a,i}$}
if $z_{-,{\rm e}} = z_{a,i}$
(resp. $\exp(\theta_{\gamma,a',i'} \sqrt{-1})w$ 
if $z_{+,{\rm e}} = z_{a',i'}$).
\end{defn}
\begin{cons}\label{cost314}
We put $\mathcal V_+ = \mathcal V_0 \times \mathcal V_1$.
We are going to define 
a neighborhood $\mathcal V$ of $(0,0)$ in $\mathcal V_+$,
a complex manifold $\mathcal C$, and a map 
$\mathcal C \to \mathcal V$ as follows.
\par
For each 
$\vec x = (x_a)_{a\in \mathcal A_s} \in \mathcal V_0 = \prod \mathcal V_{a}$
we take
$$
\bigcup_{a \in \mathcal A_s}\Sigma(x_a) \cup \bigcup_{a \in \mathcal A_u}\Sigma_a.
$$
We remove the union of 
$\varphi_{a,i}(D^2)$ for all $\varphi_{a,i}$ corresponding to the nodal 
point.
We denote it as $\Sigma(\vec x)_0$.
Let
$$
\mathcal C_0 = \bigcup_{\vec x \in \mathcal V_0} 
\left(\Sigma(\vec x)_0 \times\{\vec x\}\right)
$$
and $
\mathcal C_0 \to \mathcal V_0
$
the obvious projection. $\mathcal C_0$ is a complex manifold 
and the projection is holomorphic.
We compactify the fibers of $
(\mathcal C_0 \times \mathcal V_1) \to \mathcal V
$ as follows.
Let $\vec\rho = (\rho_{\rm e})_{{\rm e} \in \Gamma(\Sigma)}
\in \mathcal V_1$. We put $z_{-,{\rm e}} = z_{a,i}$, 
$z_{+,{\rm e}} = z_{a',i'}$ and $r_{\rm e} = \vert \rho_{\rm e} \vert$.
We consider 
$$
\left(D^2(2) \setminus D^2(r_{\rm e})\right) \cup 
\left(D^2(2) \setminus D^2(r_{\rm e})\right)
$$
and identify $z$ in the first summand with $w$ in the second summand 
if $zw = \rho_{\rm e}$.
\par
We also identify $z$ with $\varphi_{i,a}(z)$ if $\vert z \vert >1$ 
and $w$ with $\varphi_{i',a'}(w)$ if $\vert w \vert >1$.
\par
Performing this gluing at all the nodal points we obtain 
$\Sigma(\vec x,\vec\rho)$.
We put
\begin{equation}\label{defmathcalC}
\mathcal C = \bigcup_{\vec x \in \mathcal V_0,\vec\rho \in 
\mathcal V_1} \Sigma(\vec x,\vec\rho) \times \{(\vec x,\vec\rho)\}.
\end{equation}
\index{00C3ca@$\mathcal C$}
The natural projection induces a map $\pi : \mathcal C \to \mathcal V$.
It is easy to see from the construction that $\mathcal C$ is a 
complex manifold and $\pi$ is holomorphic.
Moreover the fiber of $\pi$ are nodal curves.
$
\frak t_{j} : \mathcal V_0 \to \mathcal C_0
$ can be regarded as a map 
$
\frak t_{j} : \mathcal V \to \mathcal C
$
by which $(\pi^{-1}(x);\frak t_1(x),\dots,\frak t_{\ell}(x))$
becomes a marked nodal curve.
\end{cons}
\par
The most important part of the proof of Theorem \ref{them35} is 
the following:
\begin{prop}\label{prop312}
Let
$\mathcal{MOR}$ be the set of triples 
$(\varphi,x,y)$ where $x,y \in \mathcal V$ and
$\varphi : \pi^{-1}(x) \to \pi^{-1}(y)$ an isomorphism such that
$\varphi(\frak t_j(x)) = \frak t_j(y)$ for $j=1,\dots,\ell$.
Then
\begin{enumerate}
\item
$\mathcal{MOR}$ has a structure of a smooth complex manifold.
\item
The two projections $\mathcal{MOR} \to \mathcal V$,
$(\varphi,x,y) \mapsto x$, $(\varphi,x,y) \mapsto y$
are both submersions.
\end{enumerate}
\end{prop}
Note by Proposition \ref{prop312} and Constructions \ref{const212} 
and \ref{const2124}
we obtain a family of marked nodal curves.

\begin{proof}
We first define a topology (metric) on $\mathcal{MOR}$.
Note $\mathcal C$ and $\mathcal V$ are obviously metrizable. We take 
its metric. 

\begin{defn}
We say $d((\varphi,x,y),(\varphi',x',y')) \le \epsilon$ if
$$
d(x,x') \le \epsilon, \quad d(y,y') \le \epsilon
$$
and 
$$
\vert d(\varphi(a),\varphi'(b)) - d(a,b)\vert \le \epsilon
$$
for $a \in \pi^{-1}(x)$, $b \in \pi^{-1}(x')$.
\par
It is easy to see that $d$ defines a metric on $\mathcal{MOR}$.
\end{defn}
\begin{defn}
The {\it minimal stabilization} $\vec w_a$ \index{minimal stabilization} of an unstable component 
$(\Sigma_a,\vec z_a)$ is as follows.
\par
In case (US.1), $\vec w_a$ consists of (distinct) two points which do not intersect 
with $\vec z_a$.
\par
In case (US.2), $\vec w_a$ consists of one point which does not intersect 
with $\vec z_a$.
\end{defn}
Note $(\Sigma_a,\vec z_a\cup \vec w_a)$ becomes stable. In fact it is 
a sphere with three marked points and so there is no deformation and no 
automorphism.
The choice of minimal stabilization is unique up to isomorphism.
\par
We add minimal stabilization to each unstable components 
and obtain a stable marked curve $(\Sigma,\vec z\cup \vec w)$.
The next lemma is obvious.
\begin{lem}
$\overline{\mathcal G}$ acts on $\Sigma$ such that it preserves $\vec w$ as 
a set.
\end{lem}
We denote $\vec w = \{w_1,\dots w_k\}$.
By construction we have sections $\mathcal S_{j,0} : \mathcal V_0 \to \mathcal C_0$
such that $\mathcal S_{j,0}(x)$ is identified with $w_j$.
Using the description of $\Sigma(\vec x,\vec\rho)$ we gave above 
we obtain a marked point $w_j(\vec x,\vec\rho)$.
Thus we obtain holomorphic sections   $\mathcal S_j : \mathcal V \to \mathcal C$.
The next lemma is a consequence of a standard result of the deformation theory 
of stable nodal curve.
(See \cite{alcorGri}.)
Let $\overline{\mathcal G}_0$ be a subgroup of $\overline{\mathcal G}$
consisting of elements which fix each point $w_j$.
\begin{lem}\label{lem315}
$((\mathcal C,\mathcal V),\pi,(\frak T_j;j=1,\dots,\ell)
\cup (\mathcal S_j;j=1,\dots,k))$ \index{00S3J@$\mathcal S_j$}
divided by $\overline{\mathcal G}_0$ is a local universal family 
of genus $g$ stable curves with $k+\ell$ marked points.
\end{lem}
See for example \cite{alcorGri}, \cite{manin} for the definition of universal 
family of genus $g$ stable curves with $k+\ell$ marked points.
Actually it is a special case of Definition \ref{defn34} 
where ${\rm Pr}_s$ and ${\rm Pr}_t$ are local diffeomorphisms.
\par
We now start constructing a chart of $\mathcal{MOR}$.
We first consider  $(\varphi,o,o)$, that is 
the case when $\varphi : (\Sigma,\vec z) \to (\Sigma,\vec z)$ is an automorphism.
\par
Let $U$ be a neighborhood of $\varphi$ in the group of automorphisms of 
$(\Sigma,\vec z)$.
Let $\mathcal V'$ be a sufficiently small neighborhood of $o$ in $\mathcal V$.
We put $\mathcal C' = \pi^{-1}(\mathcal V') \subset \mathcal C$.
We will construct a bijection between $U \times \mathcal V'$ 
to a neighborhood of $\varphi$ in $\mathcal{MOR}$.
We consider
$$
\Pi : U \times \mathcal C \to U \times \mathcal V
$$
which is a direct product of $\pi : \mathcal C \to \mathcal V$ and the 
identity map. $\frak T_j$ induces its sections. 
\par
For $\psi \in U$ we consider $w_j(\psi) = \psi^{-1}(w_j)$. 
Using $w_j(\psi)$ instead of $w_j$ we can construct  
$\widetilde{\mathcal S_j}(\psi; \cdot) : \mathcal V \to \mathcal C$,
such that
$(\psi,x) \mapsto \widetilde{\mathcal S_j}(\psi; x)$,
$j=1,\dots,k$, are holomorphic sections and that 
\begin{equation}
\mathcal S_j(\psi,o) =\psi^{-1}(w_j).
\end{equation}  
We denote this section by $\mathcal S^{U}_j$.
\par
Then $((U \times \mathcal C,U \times \mathcal V),\Pi,\{\frak T_j\},\mathcal S^{U}_j)$ 
is a family of marked nodal curves of genus $g$ and with $k+\ell$ marked points.
Therefore by the universality in Lemma \ref{lem315},
there exist maps
$$
F : U \times \mathcal V' \to \mathcal V,
\quad
\tilde F : U \times \mathcal C' \to \mathcal C
$$
such that:
\begin{enumerate}
\item
$\pi \circ \tilde F = F\circ \Pi$ as maps
$U \times \mathcal C' \to \mathcal V$.
\item
For $(\psi,x) \in U \times \mathcal V'$ we have:
\begin{enumerate}
\item
$(\tilde F\circ \frak T_j) (\psi,x)= 
\frak T_j(F(\psi,x))$,
\item
$(\tilde F\circ \widetilde{\mathcal S_j}) (\psi,x)= \mathcal S_j(F(\psi,x))$.
\end{enumerate}
\end{enumerate}
Now we define 
$$
\Psi : U \times \mathcal V' \to \mathcal{MOR}
$$
as follows.
Let $(\psi,x) \in U \times \mathcal V'$.
We put $y = F(\psi,x)$. 
We restrict $\tilde F$ to $\{\psi\} \times \pi^{-1}(x)$. Then 
by Item (1) above it defines a holomorphic map 
$\pi^{-1}(x) \to \pi^{-1}(y)$ which we denote 
$\tilde \psi$.
Since $\tilde F$ is a part of the morphism of family of 
marked nodal curves, 
we can show that
$\tilde \psi$ is an isomorphism.
Item (2)(a) implies that $\tilde \psi$ preserves marked points $\frak T_j$.
We put
$$
\Psi(\psi,x) = (\tilde\psi,x,y).
$$
\begin{lem}\label{lam317}
The image of $\Psi$ contains a neighborhood of $(\varphi,o,o)$
in $\mathcal{MOR}$.
\end{lem}
\begin{proof}
Let $(\varphi_i,x_i,y_i)$ be a sequence of $\mathcal{MOR}$ converging to 
$(\varphi,o,o)$. 
Note
$$
\psi \mapsto (\psi^{-1}(w_j):j=1,\dots,k)
$$
is a diffeomorphism from $U$ onto an open subsets of 
$\Sigma^k$. Therefore by inverse function theorem, 
the map
$$
\psi \mapsto (\widetilde{\mathcal S_j}(\psi_i,x_i):j=1,\dots,k)
$$
is a diffeomorphism from an neighborhood of $\varphi$
onto an open subset of $\Sigma(x_i)^k$ for sufficiently large $i$. 
Since $\lim_{i\to \infty}\widetilde{\mathcal S_j}(\psi_i,x_i) = 
\widetilde{\mathcal S_j}(\psi,o) = \psi^{-1}(w_j)$
we may assume that this open subset contains $(\psi^{-1}(w_j):j=1,\dots,k)$ 
by taking $U$ small.
\par
On the other hand, 
$
\lim_{i\to \infty} S_j(y_i) = S_j(o) = w_j.
$
Hence
$
\lim_{i\to\infty} \varphi_i^{-1}(S_j(y_i)) = \psi^{-1}(w_j).
$
Therefore there exists unique $\psi_i \in U$
such that
\begin{equation}\label{form310}
\widetilde{\mathcal S_j}(\psi_i,x_i) = \varphi^{-1}_i(\mathcal S_j(y_i)).
\end{equation}
We next prove that 
%by replacing $\psi_i$ if necessary we may assume that 
$\Psi(\psi_i,x_i) = (\varphi_i,x_i,y_i)$ for 
sufficiently large $i$.
\par
We put $\Psi(\psi_i,x_i) = (\tilde\psi_i,x_i,y_i)$.
By definition 
$
\tilde\psi_i
$
is a restriction of $\tilde F$ to $\{\psi_i\} \times \pi^{-1}(x_i)$
and $y_i = F(\psi_i,x_i)$.
Therefore Item (2) (a),(b) implies
$$
\tilde\psi_i(\widetilde{\mathcal S_j}(\psi_i,x_i))
= \widetilde{\mathcal S_j}(y_i),  
\qquad \tilde\psi_i(\widetilde{\mathcal T_j}(\psi_i,x_i))
= \widetilde{\mathcal T_j}(y_i)
$$
On the other hand 
(\ref{form310}) implies 
$
\varphi_i(\widetilde{\mathcal S_j}(\psi_i,x_i))
= \widetilde{\mathcal S_j}(y_i)
$
Moreover
$
\varphi_i(\frak T_j(x_i)) = \frak T_j(y_i)
$
follows from definition.
\par
Since $\tilde\psi_i$, $\varphi_i$ 
are both contained in $U$ 
we have $\tilde\psi_i = \varphi_i$.
\par
The proof of Lemma \ref{lam317} is complete.
\end{proof}
We thus proved that $\mathcal{MOR}$ is a manifold 
and ${\rm Pr}_s$, ${\rm Pr}_t$ are submersions near the 
point of the form $(\varphi,o,o)$.
\par
We next consider the general case.
Let $(\varphi,x,y) \in \mathcal{MOR}$.
We consider the nodal curve $\Sigma_x = \pi^{-1}(x)$
(where $\pi : \mathcal C \to \mathcal V$)
together with marked points $\frak T_j(x)$, $j=1,\dots,\ell$.
We denote it by $(\Sigma_x,\vec z_x)$.
We start from $(\Sigma_x,\vec z_x)$ in place of 
$(\Sigma,\vec z)$ and obtain 
$\pi_x : \mathcal C_x \to \mathcal V_x$ 
and its sections 
$\frak T_{x,j}$, $j=1,\dots,\ell$.

\begin{sublem}\label{sublem321}
There exists an open neghborhood $W_x$ of $0$ in $\C^{d}$ for some $d$
and 
bi-holomorphic maps,
$$
\widetilde{\Phi}_x : W_x \times \mathcal C_x \to \mathcal C,
\qquad
{\Phi}_x : W_x \times \mathcal V_x \to \mathcal V,
$$
onto open subsets, with the following properties.
\begin{enumerate}
\item
The next diagram commutes.
\begin{equation}\label{diagram3150}
\begin{CD}
W_x \times \mathcal C_x @ >{\widetilde{\Phi}_x}>> \mathcal C \\
@ VV{{\rm id} \times \pi}V @ VV{\pi}V\\
W_x \times \mathcal V_x @ >{\Phi_x}>> \mathcal V.
\end{CD}
\end{equation}
\item
For $w \in W_x$, $\frak x \in \mathcal V_x$ we have
$$
\widetilde{\Phi}_x(w,\frak T_{x,j}(\frak x)) = \frak T_{j}({\Phi}_x(w,\frak x)).
$$
\item
$\Phi_x(0,o) = x$.
\end{enumerate}
\end{sublem}
\begin{proof}
We consider the sections $\mathcal S_j$, $j=1,\dots,k$.
We can take a subset $\frak J$ of $\{1,\dots,k\}$
such that $\{\mathcal S_j(x) \mid j \in \frak J\}$
is a minimal stabilization of $(\Sigma_x,\vec z_x)$.
We put $\vec w_x = \{\mathcal S_j(x) \mid j \in \frak J\}$ and
$k' = \#\frak J$. We can identify 
$\pi_x : \mathcal C_x \to \mathcal V_x$ 
with the universal family of deformation of
the stable curve
$(\Sigma_x,\vec z_x \cup \vec w_x)$.
\par
Therefore forgetful map of the marked points $\{\mathcal S_j(x) 
\mid j \notin \frak J\}$ defines maps 
$$
\tilde\Pi : U(\mathcal C) \to \mathcal C_x, \qquad
\Pi : U(\mathcal V) \to \mathcal V_x.
$$
Here $U(\mathcal V)$ is a neighborhood of $x$ in $\mathcal V$ 
and $U(\mathcal C) = \pi^{-1}(U(\mathcal V)) \subset \mathcal C$.
By construction we have
$$
\pi \circ \tilde\Pi = \Pi \circ \pi.
$$
Since $(\Sigma_x,\vec z_x \cup \vec w_x)$ is stable,
the forgetful map $\Pi$ is defined simply by forgetting 
marked points and does not involve the process of shrinking the 
irreducible components which become unstable.
Therefore the maps $\tilde\Pi$ and $\Pi$ are both submersions.
Therefore, by implicit function theorem, we can find an open set $W_x$ and $\widetilde{\Phi}_x$, 
${\Phi}_x$ such that Diagram (\ref{diagram3150}) commutes.
\par
We also remark that 
$$
\frak T_j \circ \Pi = \tilde\Pi \circ \frak T_{x,j}.
$$
We can use it to prove Item (2) easily.
\end{proof}
We apply the same sublemma to $y$ and obtain 
$W_y$ and $\widetilde{\Phi}_y$, 
${\Phi}_y$.
We remark that $(\Sigma_x,\vec z_x \cup \vec w_x)$
is isomorphic to $(\Sigma_y,\vec z_y \cup \vec w_y)$.
Therefore a neighborhood of 
$(\varphi,x,y)$ in $\mathcal{MOR}$
is identified with a neighborhood of 
$(\varphi',o,o)$ in $\mathcal{MOR}_x$
times
$
W_x \times W_y
$.
Here $\mathcal{MOR}_x$ is obtained from 
$\pi_x : \mathcal C_x \to \mathcal V_x$ 
in the same way as 
$\mathcal{MOR}$ is obtained from 
$\pi : \mathcal C \to \mathcal V$.
The morphism 
$\varphi'$ is an element of $\mathcal{MOR}_x$ with 
${\rm Pr}_s(\varphi') = {\rm Pr}_t(\varphi') = o$.
Therefore using the case of 
$(\varphi',o,o)$ which we already proved,
we have proved Proposition \ref{prop312} in the general case.
\end{proof}
\begin{rem}\label{rem321333}
The smooth and complex structure of the chart of $\mathcal{MOR}$ we gave 
here is characterized by the following properties.
\par
We consider the case of ${\rm Pr}_s(\varphi) = x$, ${\rm Pr}_t(\varphi) = y$.
Let $S(x)$ (resp. $S(y)$) be the set of nodal points of $\Sigma_x$ 
(resp. $\Sigma_y$.) 
We take an open neighborhood $U_x$ (resp. $U_y$) of $x$ (resp. $y$)
in $\mathcal{V}$ and compact subsets $K_x$ (resp. $K_y$) of 
$\Sigma_x \setminus S(x)$ 
(resp. $\Sigma_y \setminus S(y)$), which 
is a complement of a sufficiently small 
neighborhood of $S(x)$ in $\Sigma_x$ (resp. $S(y)$ in $\Sigma_y$). There exist 
complex structures of $K_x \times U_x$ and of $K_y \times U_y$ 
and open holomorphic embeddings $\Phi_1 : K_x \times U_x
\to \mathcal C$ and $\Phi_2 :  K_y \times U_y
\to \mathcal C$
such that the next diagram commutes:
\begin{equation}
\begin{CD}
K_x \times U_x @ >>> \mathcal C \\
@ VVV @ VVV\\
U_x @ >>> \mathcal{V}
\end{CD}
\quad
\begin{CD}
K_y \times U_y @ >>> \mathcal C \\
@ VVV @ VVV\\
U_y @ >>> \mathcal{V}
\end{CD}
\end{equation}
We also require that the restriction to $K_x \times\{x\}$
(resp. $K_y \times\{y\}$) is the canonical embedding 
$K_x \subset \Sigma_x \subset \mathcal C$ (resp. $K_y \subset \Sigma_y
\subset \mathcal C$).
\par
Let $\mathcal U$ be a small neighborhood of $\varphi$ in $\mathcal{MOR}$.
By shrinking $K_x$ a bit we have a map
$$
\Psi : \mathcal U \times K_x \to K_y
$$
as follows. Let $\varphi' \in \mathcal U$ and 
${\rm Pr}_s(\varphi') = x'$,  ${\rm Pr}_t(\varphi') = y'$.
Now $\Psi$ is defined by
$$
\Phi_2(\Psi(\varphi',z),y') = \varphi(\Phi_1(z,x')).
$$
We require that $\Psi$ is a smooth and holomorphic map with respect to the 
given smooth and complex structure of $\mathcal U$.
Moreover there exist a finite number of points 
$z_1,\dots,z_n \in K_x$ such that 
$\mathcal U \to K_x^n$ $\varphi \mapsto (\varphi(z_1),\dots,\varphi(z_n))$ 
is a smooth embedding.
\par
It is easy to see from the definition that the structures we gave 
satisfies this condition.
We can use this characterization to show that the coordinate 
changes are smooth and holomorphic. 
\end{rem}
The construction of the deformation  
$\mathscr G = (\widetilde{\frak G},
\mathscr F,{\frak G},\vec{\frak T},o,\iota)$
is complete.
We will prove that it is universal.
The minimality at $o$ is obvious from construction.
\par
Let 
$\mathscr G' = (\widetilde{\frak G}',
\mathscr F',{\frak G}',\vec{\frak T}',o',\iota')$
be another deformation.
We will construct a morphism $(\widetilde{\mathcal H},\mathcal H)$
from $\mathscr G'$ to $\mathscr G$.
\par
Note we took a minimal stabilization $\vec w$ 
of $(\Sigma,\vec z)$. 
Since $\mathcal G'$ is a deformation of  $(\Sigma,\vec z)$,
there exists 
$$
\mathcal S'_j : \mathcal{OB}' \to \widetilde{\mathcal{OB}}'
$$
for $j=1,\dots,k$,
after replacing $\mathcal{OB}'$ by a smaller neighborhood of $o'$
if necessary,
such that the following holds.
\begin{enumerate}
\item
$\mathcal S'_j$ is holomorphic, for $j=1,\dots,k$.
\item
$\pi' \circ \mathcal S'_j
={\rm id} : \mathcal{OB}' \to \mathcal{OB}'$, for $j=1,\dots,k$.
\item
At $o' \in \mathcal{OB}'$ we have
$$
\mathcal S'_j(o') = \iota'(w_j),
$$
for $j=1,\dots,k$.
\end{enumerate}

Thus we have an 
$\mathcal{OB}'$ parametrized family of 
stable marked curves of genus $g$ with $k+\ell$
marked points as 
$$
x' \mapsto ((\pi')^{-1}(x'),\{\frak T'_j(x')\}\cup \{\mathcal S'_j(x')\}).
$$
Therefore by the universality of the family of marked {\it stable} curves in 
Lemma \ref{lem315}
we have a map (by shrinking $\mathcal{OB}'$ if necessary)
$$
(\mathfrak F,\widetilde{\mathfrak F}) :
(\mathcal{OB}',\widetilde{\mathcal{OB}}') \to (\mathcal{OB},\widetilde{\mathcal{OB}}).
$$
such that\footnote{Note $\mathcal{OB} = \mathcal V$ 
and $\widetilde{\mathcal{OB}} = \mathcal C$ by the 
construction of our family $\mathcal G$.}
$
\mathfrak F : \mathcal{OB}' \to \mathcal{OB}
$
and 
$
\widetilde{\mathfrak F} : \widetilde{\mathcal{OB}}' \to \widetilde{\mathcal{OB}}
$
are holomorphic,
the next diagram commutes and is a cartesian square:
\begin{equation}\label{diagram313}
\begin{CD}
\widetilde{\mathcal{OB}}' @ >{\widetilde{\mathfrak F}}>> \widetilde{\mathcal{OB}} \\
@ VVV @ VVV\\
\mathcal{OB}' @ >{\mathfrak F}>> \mathcal{OB}.
\end{CD}
\end{equation}
Moreover
\begin{equation}
\frak T_j \circ \mathfrak F
= \widetilde{\mathfrak F}\circ \frak T'_j
\qquad
\mathcal S_j \circ \mathfrak F
= \widetilde{\mathfrak F}\circ \mathcal S'_j.
\end{equation}
\par
We define $\mathcal H$.  Its object part is 
$\mathfrak F$. We define 
its morphism part.
Let $\tilde\varphi \in \mathcal{MOR}'$.
Suppose ${\rm Pr}_s(\tilde\varphi) = x'$, 
${\rm Pr}_t(\tilde\varphi) = y'$.  
Using the fact that Diagram (\ref{diagram313}) is a cartesian square
there exists a unique bi-holomorphic map $\varphi$ such that the next 
diagram commutes:
\begin{equation}\label{diagram315}
\begin{CD}
({\pi'})^{-1}(x') @ >{\tilde\varphi}>> ({\pi'})^{-1}(y') \\
@ VV{\widetilde{\mathfrak F}\vert_{{({\pi'})^{-1}(x')}}}V @ VV{\widetilde{\mathfrak F}\vert_{{({\pi'})^{-1}(y')}}}V\\
({\pi})^{-1}(x) @ >{\varphi}>> ({\pi})^{-1}(y).
\end{CD}
\end{equation}
Here $x = \mathfrak F(x')$, $y = \mathfrak F(y')$.
In fact all the arrows (except $\varphi$) is defined and are isomorphisms.
We define the morphism part of $\mathcal H$ by 
$\tilde\varphi \mapsto \varphi$. 
It is easy to see that this map is holomorphic and 
has other required properties.
We thus defined 
$\mathcal H: {\frak G}' \to {\frak G}$.
\par
We next define 
$\widetilde{\mathcal H}: \widetilde{\frak G}' \to \widetilde{\frak G}$.
Its object part is $\widetilde{\mathfrak F}$.
The morphism part is defined 
from $\widetilde{\mathfrak F}$ and the morphism part of ${\mathcal H}$,
by using the fact
$$
\widetilde{\mathcal{MOR}}' = {\mathcal{MOR}}' \,\,{}_{{\rm Pr}_s}\times_{\mathscr F} \widetilde{\mathcal{OB}}',
\quad
\widetilde{\mathcal{MOR}} = {\mathcal{MOR}} \,\,{}_{{\rm Pr}_s}\times_{\mathscr F} \widetilde{\mathcal{OB}}.
$$
\par
We thus obtain $\widetilde{\mathcal H}$.
\par
It is straight forward to check that $(\widetilde{\mathcal H},{\mathcal H})$
has the required properties.
\par
We finally prove the uniqueness part of the universality property 
of our deformation.
Let
$\mathcal G' = (\widetilde{\frak G}',
\mathscr F',{\frak G}',\vec{\frak T}',o',\iota')$
be another deformation
and $(\widetilde{\mathcal H},\mathcal H)$,
$(\widetilde{\mathcal H}',\mathcal H')$ be two morphisms 
from $\mathcal G'$ to $\mathcal G$.
We will prove that $(\widetilde{\mathcal H},\mathcal H)$ is conjugate to
$(\widetilde{\mathcal H}',\mathcal H')$.
\par
Let $x' \in \mathcal{OB}'$. By definition 
there exists a biholomorphic map 
$$
\mathcal T(x') : \pi^{-1}(\mathcal H'_{ob}(x'))
\to \pi^{-1}(\mathcal H_{ob}(x'))
$$
such that the next diagram commutes.
\begin{equation}\label{diagram316}
\begin{CD}
({\pi'})^{-1}(x') @ >{\rm id}>> ({\pi'})^{-1}(x') \\
@ VV{\mathcal H'\vert_{({\pi'})^{-1}(x')}}V @ VV{\mathcal H\vert_{({\pi'})^{-1}(x')}}V\\
\pi^{-1}(\mathcal H'_{ob}(x')) @ >{\mathcal T(x')}>> \pi^{-1}(\mathcal H_{ob}(x')).
\end{CD}
\end{equation}
In fact two vertical arrows are isomorphisms.
Moreover 
$$
\mathcal T(x')(\frak T_j(\mathcal H'_{ob}(x')))
=
\mathcal T(x')(\mathcal H'(\frak T'_j(x')))
=
\mathcal H_{ob}(\frak T'_j(x'))
=
\frak T_j(\mathcal H_{ob}(x')).
$$
Namely $\mathcal T(x')$ preserves marked points.
Therefore by definition 
$\mathcal T(x') \in \mathcal{MOR}$.
It is easy to see that $x' \mapsto \mathcal T(x')$ is the 
required natural transformation.
\par
The proof of Theorem \ref{them35} is now complete.
\qed\par\medskip
For our application of Theorem  \ref{them35} we need the following 
additional properties of our universal family.
\begin{prop}\label{prop318}
Let $\mathcal G_c$ be a compact subgroup of the group
$\mathcal G$ in  (\ref{form3535}).
Then $\mathcal G_c$ acts on our universal family 
$\mathscr G = (\widetilde{\frak G},
\mathscr F,{\frak G},\vec{\frak T},o,\iota)$
in the following sense.
\begin{enumerate}
\item
$\mathcal G_c$ acts on the spaces of objects and of morphisms 
of $\widetilde{\frak G}$ and  ${\frak G}$.
The action is a smooth action.
\item
The action of each element of $\mathcal G_c$ in (1) is holomorphic.
\item
Maps appearing in $\mathscr G$ are all $\mathcal G_c$ equivariant.
In particular 
$\iota: \Sigma \to \widetilde{\mathcal{OB}}
$ 
is $\mathcal G_c$ equivariant.
\end{enumerate}
\end{prop}
\begin{proof}
While constructing our universal family we take 
 analytic families of  coordinates at the nodal 
points so that it is invariant under $\overline{\mathcal G}$
action.
(Lemma \ref{lem399}.)
\par
We slightly modify the notion of invariance of 
analytic family of  coordinates and 
may assume that it is invariant under the $\mathcal G_c$ action as follows.
\par
We first remark that there exists an exact sequence
\begin{equation}\label{form317}
1 \to \prod_{a \in \mathcal A}\mathcal G_{c,a} 
\to \mathcal G_c \to \mathscr H_c \to 1.
\end{equation}
Here $\mathscr H_c$ is a finite group 
and $\mathcal G_{c,a}$ is a compact subgroup of 
$\mathcal G_{a}$.
In case $\Sigma_a$ is unstable, 
we consider the case  (US.1).
Then $\mathcal G_{c,a}$ is a compact subgroup of 
the group of transformations of the form $z \mapsto az + b$.
(Here $a\in \C\setminus \{0\}$, $b \in \C$. We may take 
the coordinate of $S^2 = \C \cup \{\infty\}$ such that 
$\mathcal G_{c,a}$ consists of elements of the form $z \mapsto az$ with 
$\vert a \vert =1$. Then we take $w= 1/z$ as the coordinate at infinity
(= the node).
\par
In case (US.2), we may take $\vec z_a = \{0,\infty\}$.
So $\mathcal G_{c,a}$ consists of elements of the form $z \mapsto az$ with 
$\vert a \vert =1$. Then take $1/z$ as the coordinate at infinity
(= the node).
Thus in all the cases we may assume that $\gamma \in \mathcal G_{c,a}$
acts in the form Definition \ref{cooddinateatinf} (*).
\par
Now $\mathcal G_{c}$ acts on $\mathcal V_1$ so that 
$\mathcal G_{c,a}$ acts by using (*) and $\mathscr H_c$ 
acts by exchanging the factors.
$\mathcal G_{c}$ also acts on $\mathcal V_0$.
Therefore $\mathcal G_{c}$ also acts on $\mathcal V$.
It is easy to see from construction that this action 
lifts to an action to $\mathcal C$.
The proposition follows. 
\end{proof}
\begin{exm}
Let $\Sigma$ be obtained by gluing  two copies of $S^2 = \C \cup \{\infty\}$ 
at $\infty$. (We put no marked point on it.) The group $\mathcal G$ of 
automorphisms of $\Sigma$ has an exact sequence,
$$
1 \to {\rm Aut}(S^2,\infty) \times {\rm Aut}(S^2,\infty)
\to \mathcal G \to \Z_2 \to 1,
$$
where ${\rm Aut}(S^2,\infty)$ is the group consisting  
of the transformations $z \mapsto az + b$
on $\C$. We embed 
$S^1 \to {\rm Aut}(S^2,\infty) \times {\rm Aut}(S^2,\infty)$ 
by $\sigma \mapsto (\sigma^2,\sigma^3)$.
Where $\sigma \in \{z \mid \vert z\vert =1\}$ and
$\sigma^k$ acts on $\C$ by $z \mapsto \sigma^kz$.
\par
The space $\mathcal V$ we obtain in this case is $D^2$ 
which consists of gluing parameter.
The action of $S^1$ is by $\sigma \mapsto (\rho \mapsto \sigma^5\rho)$.
\par
Let $z_1,z_2$ be the coordinates of the first and second irreducible components 
of $\Sigma$,
respectively. When we glue those two components by the parameter $\rho$, 
we equate $z_1 z_2 = \rho$. 
So if we define $z'_1 = \sigma^2 z_1$, $z'_2 = \sigma^3 z_2$,
then the equation turn out to be $z'_1 z'_2 = \sigma^5\rho$.
\par
Suppose $u : \Sigma \to S^2$ is the map which is $z_1 \mapsto z_1^3$, 
$z_2 \mapsto z_2^2$. We define an $S^1$ action on $S^2 = \C \cup \{\infty\}$ 
by $(\sigma,w) \mapsto \sigma w$. Then the above group $S^1$ is the 
isotropy group of this $S^1$ action.
(Which we write $\widehat{\mathcal G}_c$, (\ref{defnGc}).)
\end{exm}
The next example shows that the 
(noncompact) group $\mathcal G$ may  act 
on our universal family.
\begin{exm}
We consider the case when $\Sigma = S^2_1 \cup S^2_2$
and $\vec z =$ 3 points.
We identify $S^2_1 = \C \cup \{\infty\}$ and 
$\vec z = (1,\sqrt{-1},-\sqrt{-1})$. 
$S^2_2 = \C \cup \{\infty\}$ also.
We use $z$ and $w$ as coordinates of $S^2_1$ and $S^2_2$.
They are glued at $0 \in S^2_1$ and $0 \in S^2_2$.
$\mathcal V_0$ is identified with the small neighborhood of $0$,
(that is, the coordinate of the node in $S^2_1$.)
We denote this coordinate of $\mathcal V_0$ by $v$.
$\rho$ is the parameter to glue $S^2_1$ and 
$S^2_2$. We use it to equate
$$
zw = \rho.
$$
We use $w' = 1/w$ as a parameter.
$\mathcal G$ is the group consisting of transformations 
of the form $w' \mapsto g_{a,b}(w') = aw' + b$.
\par
Now following the proof of Theorem \ref{them35}
we take two additional marked points on $S^2_2$, 
say, $w' =0,1$.
So after gluing we have 5 marked points,
$\vec z$ and $v, v + \rho$.
\par
When we first move $w' =0,1$ by $g_{a,b}$ and glue 
then the 5 marked points are
$\vec z$ and $v + \rho b$, $v + \rho(a+b)$.
(See Figure \ref{Figuresec3}.)
\par
Now $v, v + \rho$ may be identified with an element of $\mathcal V$.
The fiber ${\rm Pr}_s : \mathcal{MOR} \to \mathcal V 
= \mathcal{OB}$ is then identified with $\mathcal G$.
We consider $\varphi \in \mathcal{OB}$ corresponding to $((v, v + \rho),g_{a,b})$.
Then by the construction its target ${\rm Pr}_t(\varphi)$ is 
$v + \rho b$, $v + \rho(a+b)$.
Thus we can write
$$
(v, v + \rho)g_{a,b} = (v + \rho b,v + \rho(a+b)).
$$
See Figure \ref{Figuresec3}.
Note
$$
g_{a,b}g_{a',b'} = g_{aa',b+ab'}
$$
We can check
$$
((v, v + \rho)g_{a,b})g_{a',b'} = 
(v+ \rho(b + a b'), v + \rho(aa' + b + ab'))
= (v, v + \rho)(g_{a,b}g_{a',b'}).
$$
So this is a genuine action. However we can define this action 
only on the part where $v$ is small. In fact we use the coordinate 
$z \mapsto z +v$ around $v$ in the above construction.
We can not use this coordinate when $v$ gets closer to $\vec z$.
\end{exm}
\begin{figure}[h]
\centering
\includegraphics[scale=0.4]{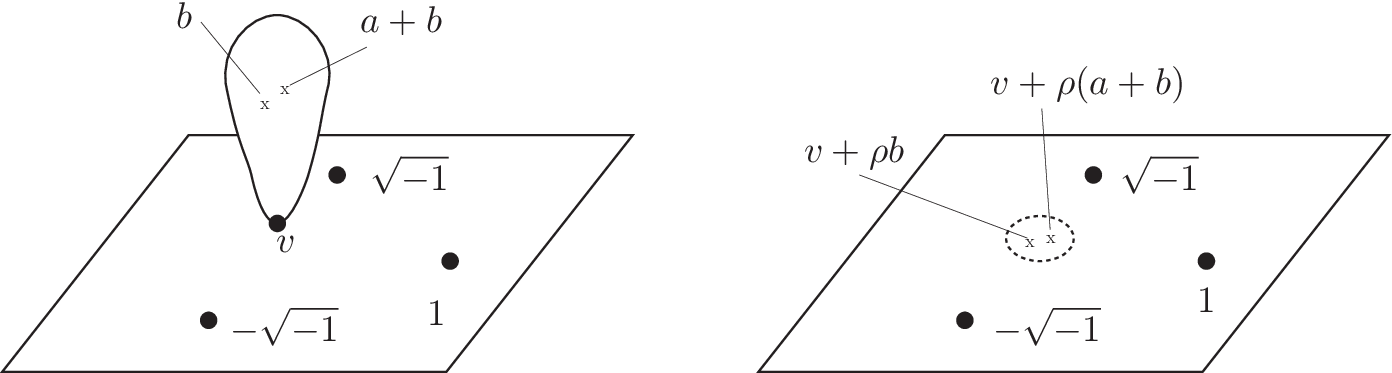}
\caption{Universal family of deformation of $S^2$ with 3 marked points 
and one sphere bubble}
\label{Figuresec3}
\end{figure}

\begin{rem}
In the situation of Theorem \ref{them35} we consider a 
neighborhood of the image of $\mathcal{ID} : \mathcal{OB} \to \mathcal{MOR}$.
Since ${\rm Pr}_t$ is a submersion we may identify this neighborhood with a 
direct product $\frak H \times \mathcal{OB}$.
We assign to $(\varphi,{\bf x}) \in \frak H \times \mathcal{OB}$ the element 
${\rm Pr}_s(\varphi)$. We thus obtain a map
\begin{equation}\label{form320}
\frak H \times \mathcal{OB} \to \mathcal{OB}.
\end{equation}
If $(\varphi,{\bf x}) \in \frak H \times \mathcal{OB}$ are sent to ${\bf y}$ 
then $\varphi$ induces an isomorphism between two marked nodal
curves represented by ${\bf x}$ and by ${\bf y}$.
The map (\ref{form320}) is nothing but the map $act$ appearing in 
\cite[page 990]{FO}.
Since the product decomposition $\frak H \times \mathcal{OB}$ of the neighborhood 
of the image of $\mathcal{ID}$ is not canonical, 
this is not really an action as we mentioned in \cite[page 990]{FO}.
\end{rem}

\section{$\epsilon$-closeness and obstruction bundle}
\label{sec:obstruction}

Let $((\Sigma,\vec z),u)$ be a stable map \index{stable map} of genus $g$ with $\ell$ marked 
points in a symplectic manifold $(X,\omega)$
on which $G$ acts preserving $\omega$. 
See for example \cite[Definition 7.4]{FO} for the definition of stable map. 
We take the universal family of deformations
$\mathscr G = (\widetilde{\frak G},
\mathscr F,{\frak G},\vec{\frak T},o,\iota)$
of  $(\Sigma,\vec z)$. We fix Riemannian metrics on 
the spaces of morphisms and objects of 
$\widetilde{\frak G}$, ${\frak G}$.\index{00G4frak@$\widetilde{\frak G}$}
We also choose a $G$-invariant Riemannian metric on $X$.
We put \index{00G3C@$\widehat{\mathcal G}_c$}
\begin{equation}\label{defnGc}
\widehat{\mathcal G}_c = \left\{(\gamma,g) \,\, \left\vert
\aligned
&  \gamma : (\Sigma,\vec z) \to 
(\Sigma,\vec z), 
\text{$\gamma$ is bi-holomorphic}, \\
&g \in G \quad u(\gamma x) = g u(x)
\endaligned
\right\}\right..
\end{equation}
\index{00G3C@$\widehat{\mathcal G}_c$}
We define its group structure by
\begin{equation}
(\gamma_1,g_1) \cdot (\gamma_2,g_2)
=
(\gamma_1\gamma_2,g_1g_2).
\end{equation}

We define a group homomorphism $\widehat{\mathcal G}_c \to 
\mathcal G$ by $(\gamma,g) \mapsto \gamma$
and denote by ${\mathcal G}_c$ the image.
Note $\mathcal G$ is defined by (\ref{form3535}).
This is a compact subgroup of $\mathcal G$. Using
Proposition \ref{prop318} we may assume that 
$\mathscr G$ has $\mathcal G_c$ \index{00G3C@$\mathcal G_c$}
action in the sense stated in 
Proposition \ref{prop318}.
\par
We will next fix a `trivialization' of the `bundle'
$\mathscr F_{ob} : \widetilde{\mathcal{OB}} \to \mathcal{OB}$.
Note this `bundle' coincides with $\pi : \mathcal C \to \mathcal V$ using the 
notation we used during the proof of Theorem \ref{them35}.
We first recall that we take universal families 
$\mathcal C_{a} \to \mathcal V_a$ 
of deformations of $(\Sigma_a,\vec z_a)$ for each stable irreducible component
$a \in \mathcal A_s$.
They are fiber bundles. Therefore we obtain their $C^{\infty}$
trivialization by choosing $\mathcal V_a$ small.
It gives a diffeomorphism
$$
\phi_a : \mathcal V_a \times \Sigma_a \to \mathcal C_{a}
$$
onto an open subset such that the next diagram commutes:
\begin{equation}\label{diagram51333}
\begin{CD}
\mathcal V_a \times \Sigma_a @ >{\phi_a}>> \mathcal C_{a} \\
@ VVV @ VV{\pi}V\\
\mathcal V_a @ >{\rm id}>> \mathcal V_a.
\end{CD}
\end{equation}

We require the following properties:
\begin{enumerate}
\item[(Tri.1)]
$\phi_a$ is $\mathcal G_a$ equivariant. \index{Tri}
\item[(Tri.2)]
$$
(\phi_a)^{-1}(\frak T_{a,j}(x)) = (x,z_{a,j}).
$$
Namely by this trivialization the sections $\frak T_{a,j}$ 
becomes a constant map to $z_{a,j}$ (that is, the $j$-th 
marked point of $(\Sigma_a,\vec z_a)$).
\item[(Tri.3)] Let $
\varphi_{a,i} : \mathcal V_{a} \times D^2(2) \to \mathcal C_a
$ be the analytic family of  coordinates  
as in Definition \ref{cooddinateatinf}.
Then we have
$$
(\phi_a)^{-1}(\varphi_{a,i}(x,z)) = (x,\varphi_{a,i}(0,z)).
$$
Here $0 \in \mathcal V_{a}$ corresponds to the point $\Sigma_a$.
\item[(Tri.4)]
Let $\mathscr H_c$ be as in (\ref{form317}). Then the next diagram commutes
for $\gamma \in \mathscr H_c$.
Note $\mathscr H_c$ acts on the dual graph of $\Sigma$. So 
for $a \in \mathcal A_s$ we obtain
$\gamma a \in \mathcal A_s$.
\begin{equation}\label{diagram516}
\begin{CD}
\mathcal V_a \times \Sigma_a @ >{\phi_a}>> \mathcal C_{a} \\
@ VV{\gamma}V @ VV{\gamma}V\\
\mathcal V_{\gamma a} \times \Sigma_{\gamma a} @ >{\Phi_{\gamma a}}>> \mathcal C_{\gamma a}.
\end{CD}
\end{equation}
\end{enumerate}
Existence of such trivialization in $C^{\infty}$ category 
is standard. (It is nothing but the local smooth triviality of 
fiber bundles, which is a consequence of local contractibility of the 
group of diffeomorphisms.)
\par
The above trivialization is defined on $\mathcal V_0 \subset \mathcal V$.
We extend it including the gluing parameter as follows.
\par
Let $\delta > 0$. We put \index{00V31delta@$\mathcal V_1(\delta)$}
\begin{equation}\label{form45555}
\mathcal V_1(\delta) 
= \{(\rho_{\rm e})_{{\rm e} \in \Gamma(\Sigma)}
\mid \forall {\rm e}, \vert\rho_{\rm e}\vert < \delta\}.
\end{equation}
Let ${\bf x} = ((x_a)_{a \in\mathcal A_s},(\rho_{\rm e})_{{\rm e} \in \Gamma(\Sigma)})
\in \mathcal V_0 
\times \mathcal V_1(\delta) \subset \mathcal{OB}$.
We put  \index{00S5igmax@$\Sigma({\bf x})$}
\begin{equation}\label{sigmaxxx}
\Sigma({\bf x}) = \mathscr F_{ob}^{-1}({\bf x})
\end{equation}
(= $\pi^{-1}(\bf x) \subset \mathcal C$) 
and $\vec z({\bf x}) = (\frak T_j({\bf x}))_{j=1}^{\ell}$.
\par
We also put \index{00S5igmadela@$\Sigma(\delta)$}
\begin{equation}\label{formula46}
\Sigma(\delta) =  \bigcup_{a \in \mathcal A} 
(\Sigma_a \setminus \bigcup \varphi_{a,j}(D^2(\delta))),
\end{equation}
where the union $\bigcup \varphi_{a,j}(D^2(\delta))$ is taken 
over all nodal points contained in $\Sigma_a$.
\begin{figure}[h]
\centering
\includegraphics[scale=0.3]{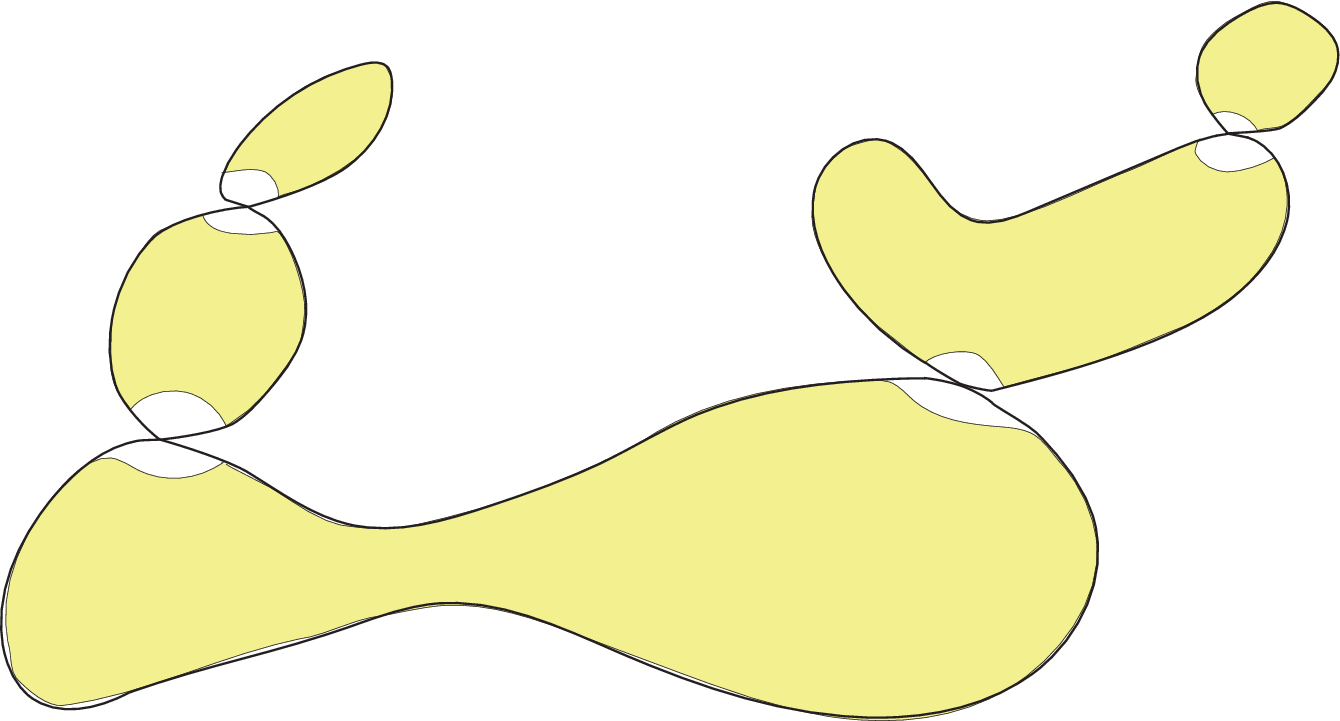}
\caption{$\Sigma(\delta)$}
\label{Figure3}
\end{figure}
We will construct a smooth embedding \index{00P5hixdelta@$\Phi_{{\bf x},\delta}$}
\begin{equation}\label{loctrimap}
\Phi_{{\bf x},\delta} : \Sigma(\delta) \to \Sigma({\bf x})
\end{equation}
below.
Let $\Sigma_a(x_a) = \pi^{-1}(x_a) \subset \mathcal C_a$.
We put
\begin{equation}
\Sigma(\delta;(x_a)_{a \in\mathcal A_s})) =  \bigcup_{a \in \mathcal A} 
(\Sigma_a(x_a) \setminus \bigcup \varphi_{a,j}(\{x_a\} \times D^2(\delta))),
\end{equation}
The maps $\phi_a$ for $a\in \mathcal A_s$ define a diffeomorphism 
$$
\Phi_{((x_a)_{a \in\mathcal A_s})} :
\Sigma(\delta) \to  \Sigma(\delta;(x_a)_{a \in\mathcal A_s}).
$$
(Note for an unstable component $\Sigma_a$ the corresponding component 
of $\Sigma(\delta;(x_a)_{a \in\mathcal A_s})$ is 
identified with $\Sigma_a$ itself. 
In this case $\Phi_{{\bf x},\delta}$ on this component is the identity map.)
\par
The $C^{\infty}$ embedding 
$$
\Sigma(\delta;(x_a)_{a \in\mathcal A_s})
\to \Sigma({\bf x})
$$
is obtained by construction.
(In fact $\Sigma({\bf x})$ is obtained by gluing
$\Sigma_a(x_a) \setminus \bigcup \varphi_{a,j}(\{x_a\} \times D^2(\vert\rho_a\vert))$.)
Thus we obtain an open embedding of $C^{\infty}$ class
\begin{equation}
\Phi_{\bf x,\delta}  : \Sigma(\delta) \to \Sigma({\bf x}) 
\end{equation}
by composing them.
\begin{defn}\label{defn410}
Let $F : A \to X$ be a continuous map from a topological space to a metric space.
We say {\it $F$ has diameter $< \epsilon$ on $A$}
\index{has diameter $< \epsilon$ on}
if for each connected component $A_a$ of $A$ the 
diameter of $F(A_a)$ is smaller than $\epsilon$.
\end{defn}
\begin{defn}\label{defn41}
We consider a triple $((\Sigma',\vec z^{\,\prime}),u')$ where 
$(\Sigma',\vec z^{\,\prime})$ is a nodal curve of genus $g$ with $\ell$ marked points, 
$u' : \Sigma' \to X$ is a smooth map.
\par
We say that $((\Sigma',\vec z^{\,\prime}),u')$ is {\it $G$-$\epsilon$-close} to
\index{00G2psiolonGclose@$G$-$\epsilon$-clos}
$((\Sigma,\vec z),u)$ if there exist $g \in G$, $\delta > 0$, 
${\bf x} = ((x_a)_{a \in\mathcal A_s},(\rho_{\rm e})_{{\rm e} \in \Gamma(\Sigma)})
\in \mathcal V_0 
\times \mathcal V_1(\delta) \subset \mathcal{OB}$,
and a bi-holomorphic map $\phi : (\Sigma({\bf x}),\vec z({\bf x}))
\cong (\Sigma',\vec z^{\,\prime})$ with 
the following properties.
\begin{enumerate}
\item The $C^{2}$ difference between
$u' \circ \phi \circ \Phi_{\bf x,\delta}$ and $g \circ u\vert_{\Sigma(\delta)}$
is smaller than $\epsilon$.
\item The distance between $\bf x$ and $o \in \mathcal{OB}$
is smaller than $\epsilon$. Moreover $\delta < \epsilon$.
\item
The map $u'\circ \phi$ has diameter $< \epsilon$ on $\Sigma({\bf x}) \setminus 
{\rm Image}(\Phi_{\bf x,\delta})$.
\end{enumerate}
\begin{figure}[h]
\centering
\includegraphics[scale=0.4]{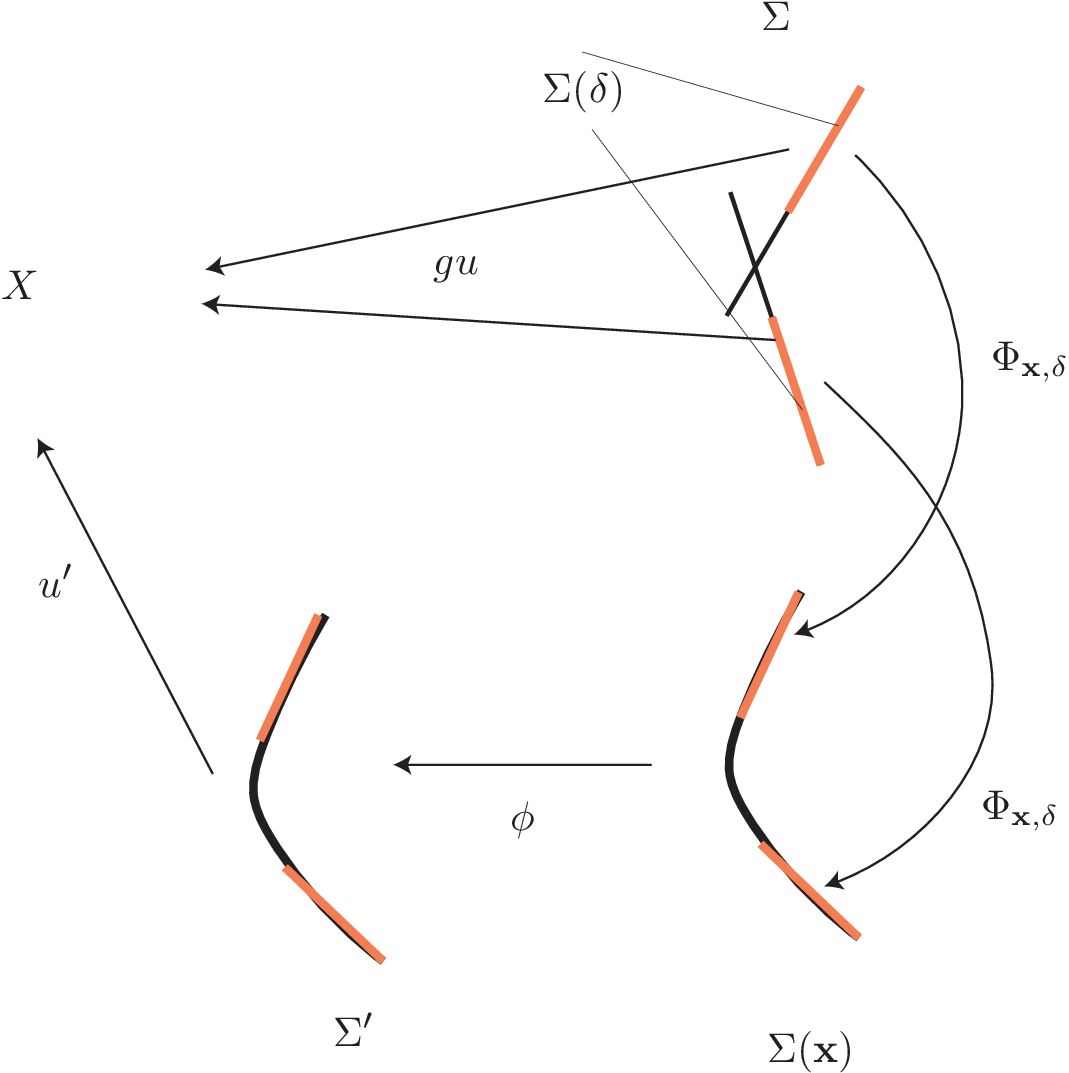}
\caption{$((\Sigma',\vec z^{\,\prime}),u')$ is $G$-$\epsilon$-close 
to
$((\Sigma,\vec z),u)$}
\label{Figure4-1}
\end{figure}
In case we need to specify $g$, $\bf x$, $\phi$ we say 
$((\Sigma',\vec z^{\,\prime}),u')$ is $G$-$\epsilon$-close to
$((\Sigma,\vec z),u)$
by $g$, $\bf x$, $\phi$.
\par
We say that $((\Sigma',\vec z^{\,\prime}),u')$ is {\it $\epsilon$-close} to
\index{00E5spiloncllse@$\epsilon$-close}
$((\Sigma,\vec z),u)$ if (2),(3) are satisfied and 
(1) is satisfied with $g=1$.
In case  we need to specify $\bf x$, $\phi$ we say 
$((\Sigma',\vec z^{\,\prime}),u')$ is $\epsilon$-close to
$((\Sigma,\vec z),u)$
by $\bf x$, $\phi$.
\end{defn}
The main part of the construction of our Kuranishi chart is 
to associate a finite dimensional subspace
$$
E((\Sigma',\vec z^{\,\prime}),u')
\subset
C^{\infty}(\Sigma({\bf x});(u')^*TX\otimes \Lambda^{01})
$$
to each $((\Sigma',\vec z^{\,\prime}),u')$ which is $G$-$\epsilon$-close 
to $((\Sigma,\vec z),u)$
such that 
$$
E((\Sigma',\vec z^{\,\prime}),gu') = g_*E((\Sigma',\vec z^{\,\prime}),u')
$$
holds for $g \in G$.
\par
The construction of such $E((\Sigma',\vec z^{\,\prime}),u')$ will 
be completed in Section \ref{sec:main}
using center of mass technique which we review in Section 
\ref{sec:centerofmass}.

\begin{defn}\label{defn42}
We say a subspace 
$$
E((\Sigma,\vec z),u)
\subset 
C^{\infty}(\Sigma({\bf x});u^*TX\otimes \Lambda^{01})
$$
an {\it obstruction space at origin}\index{obstruction space at origin} if the following is 
satisfied.
\begin{enumerate}
\item
$E((\Sigma,\vec z),u)$ is a finite dimensional 
linear subspace.
\item
The support of each element of 
$E((\Sigma,\vec z),u)$ is 
contained in the complement of the image of $\varphi_{a,i} : D^2(2) \to 
\Sigma_a$ for all $a$ and $i$ corresponding to a nodal 
point. 
\item
$E((\Sigma,\vec z),u)$ is invariant under the  
$\widehat{\mathcal G}_c$ action, which we explain below.
\item
$E((\Sigma,\vec z),u)$ satisfies the transversality 
condition, Condition \ref{conds45} below.
\end{enumerate}
\end{defn}
We define $\widehat{\mathcal G}_c$ action on 
$C^{\infty}(\Sigma({\bf x});u^*TX\otimes \Lambda^{01})$.
Let $(\gamma,g) \in \widehat{\mathcal G}_c$  be as in (\ref{defnGc})
and 
$v \in C^{\infty}(\Sigma({\bf x});u^*TX\otimes \Lambda^{01})$.
Using the differential of $g$ we have
$$
g_*v \in C^{\infty}(\Sigma({\bf x});(g\circ u)^*TX\otimes \Lambda^{01}).
$$
Since $g\circ u = u\circ \gamma$ we may regard
$$
g_*v \in C^{\infty}(\Sigma({\bf x});(u\circ \gamma)^*TX\otimes \Lambda^{01}).
$$
Since $\gamma : \Sigma \to \Sigma$ is bi-holomorphic we have
$$
(g,\gamma)_* v \in 
C^{\infty}(\Sigma({\bf x});u^*TX\otimes \Lambda^{01}).
$$
We thus defined $\widehat{\mathcal G}_c$ action on 
$C^{\infty}(\Sigma({\bf x});u^*TX\otimes \Lambda^{01})$.
Item (3) above requires that the subspace 
$E((\Sigma,\vec z),u)$ is invariant under this action.
\par
We next define transversality conditions in Item (4).
We decompose $\Sigma$ into irreducible components 
$\Sigma_a$ ($a \in \mathcal A$).
We consider
$$
L^2_{m+1}(\Sigma_a;u^*TX)
$$
the Hilbert space of sections of $u^*TX$ of $L^2_{m+1}$ class 
on $\Sigma_a$.
(We take $m$ sufficiently large and fix it.)
For each $z_{a,j}$ we have an evaluation map:
$$
{\rm Ev}_{z_{a,j}} : L^2_{m+1}(\Sigma_a;u^*TX) 
\to T_{u(z_{a,j})} X.
$$
\index{00E1Vzaj@${\rm Ev}_{z_{a,j}}$}
(Since $m$ is large elements of $L^2_{m+1}(\Sigma_a;u^*TX)$ 
are continuous and ${\rm Ev}_{z_{a,j}}$ is well-defined and continuous.)
\begin{defn}\label{defn43}
The Hilbert space 
$$
W^2_{m+1}(\Sigma;u^*TX)
$$
is the subspace of the direct sum
\begin{equation}
\bigoplus_{a \in \mathcal A} L^2_{m+1}(\Sigma_a;u^*TX)
\end{equation}
consisting of elements $(v_a)_{a\in \mathcal A}$
such that the following holds.
\par
For each edge ${\rm e}$ of $\Gamma(\Sigma)$, that corresponds 
to the nodal points, let $z_{-,{\rm e}}, z_{+,{\rm e}} \in {\rm Node}$
such that the orientation of ${\rm e}$ goes from the vertex 
corresponding to $z_{-,{\rm e}}$ to the vertex corresponding to $z_{+,{\rm e}}$.
Let $a({\rm e},-)$ and $a({\rm e},+)$ be the irreducible components 
containing $z_{-,{\rm e}}$, $z_{+,{\rm e}}$, respectively.
We then require
\begin{equation}
{\rm Ev}_{z_{-,{\rm e}}}(v_{a({\rm e},-)})
= {\rm Ev}_{z_{+,{\rm e}}}(v_{a({\rm e},+)}).
\end{equation}
\par
Note $\vec z$ (the set of marked points of $\Sigma$) is 
a subset of $\bigcup_{a\in \mathcal A} \vec z_a$.
Therefore we obtain an evaluation maps
\begin{equation}
{\rm Ev}_{z_i} : W^2_{m+1}(\Sigma;u^*TX) \to T_{u(z_i)}X.
\end{equation}
\par
We put
$$
L^2_{m}(\Sigma;u^*TX\otimes \Lambda^{01})
= 
\bigoplus_{a \in \mathcal A} L^2_{m}(\Sigma_a;u^*TX\otimes \Lambda^{01})
$$
The linearization of the equation $\overline{\partial} u =0$
defines a linear differential operator of first order:
\begin{equation}\label{operator410}
D_u\overline{\partial} 
:
W^2_{m+1}(\Sigma;u^*TX)
\to L^2_{m}(\Sigma;u^*TX\otimes \Lambda^{01}).
\end{equation}
It is well-known that (\ref{operator410}) is a 
Fredholm operator.
\par
In fact the operator
\begin{equation}\label{operator411}
D_{u_a}\overline{\partial} :
L^2_{m+1}(\Sigma_a;{u_a}^*TX)
\to L^2_{m}(\Sigma_a;{u_a}^*TX\otimes \Lambda^{01})
\end{equation}
is Fredholm by ellipticity. The source of (\ref{operator410}) 
is a space of finite codimension of the direct sum of the
sources of (\ref{operator411}).
\end{defn}
\begin{rem}
In Definition \ref{defn43} we considered the compact 
spaces (manifold) $\Sigma_a$. Instead we may take 
$\Sigma_a \setminus \vec z_a$ and put cylindrical metric
(which is isometric to $S^1 \times [0,\infty)$ at 
the neighborhood of each nodal points), and use 
appropriate weighted Sobolev-norm. (See \cite[Section 4]{foooexp}
for example.)
The resulting transversality conditions are equivalent 
to one in Condition \ref{conds45}.
\end{rem}
\begin{conds}\label{conds45}
We require the next two transversality conditions.
\begin{enumerate}
\item The sum of the image of the operator 
$D_u\overline{\partial}$ (\ref{operator410}) and the 
subspace $E((\Sigma,\vec z),u)$ is 
$L^2_{m}(\Sigma;u^*TX\otimes \Lambda^{01})$.
\item
We consider \index{00K1erDU@${\rm Ker}^+ D_u\overline{\partial}$}
$$
{\rm Ker}^+ D_u\overline{\partial}
=
\{
v \in W^2_{m+1}(\Sigma;u^*TX)
\mid  D_u\overline{\partial}(v)
\in E((\Sigma,\vec z),u).
\}
$$
Then the restriction of ${\rm Ev}_{z_i}$ defines a 
{\it surjective} map
\begin{equation}
\bigoplus_{i=1}^{\ell} {\rm Ev}_{z_i}
:
{\rm Ker}^+ D_u\overline{\partial}
\to \bigoplus_{i=1}^{\ell} T_{u(z_i)}X.
\end{equation}
\end{enumerate}
\end{conds}
\begin{rem}
In certain situation we relax the condition (2) 
and require surjectivity of one of ${\rm Ev}_{z_i}$ only.
(See \cite{Fu1}.)
\end{rem}
\begin{prop}
There exists an obstruction space at origin as in Definition 
\ref{defn42}.
\end{prop}
\begin{proof}
This is mostly obvious using 
Fredholm property of $D_u\overline{\partial}$ and the unique continuation.
See \cite[Lemma 4.3.5]{foootoric32} for example.
\end{proof}

\section{Definition of $G$-equivariant Kuranishi chart
and the statement of the main theorem}
\label{sec:eqkurachar}

We review the notion of $G$-equivariant Kuranishi chart.
In the case of finite group $G$ it is defined for example in 
\cite[Definition 7.5]{fooo:inv}.
The notion of $S^1$ equivariant Kuranishi structure is in 
\cite[Definition 28.1]{foootech}.
In fact we studied in \cite[Definition 28.1]{foootech}
the $S^1$ action on the moduli space induced by the $S^1$ action 
of the source curve. Such an $S^1$ action is {\it much easier} to 
handle than the target space action we are studying here.
(This $S^1$ action had been used in the study of periodic Hamiltonian 
system and  thorough detail of its construction 
and of its usage had been written in 
\cite[Part 5]{foootech}.)
\par
We first review the notion of group action on effective orbifolds.
For the definition of effective orbifolds and its 
morphisms etc. using coordinate we refer the reader to
\cite[Section 15]{foootech2}, \cite[Part 7]{fooospectr} or \cite[Section 23]{FOOOSpringer}.
\par
An orbifold\footnote{We assume that an orbifold is effective always in this 
paper} \index{orbifold} $M$ is a paracompact and 
Hausdorff topological space together with a system of local charts $(V,\Gamma,\phi)$, 
where $V$ is a manifold, $\Gamma$ is a finite group which acts on $V$ 
effectively and $\phi : V \to M$ is a smooth map which induces a
homeomorphism $V/\Gamma \to M$ onto an open neighborhood of $p$ in $M$.
When $M$ is covered by the images of several local charts  $(V_i,\Gamma_i,\phi_i)$
satisfying certain compatibility conditions (see 
\cite[Section 15]{foootech2}, \cite[Part 7]{fooospectr} or \cite[Section 23]{FOOOSpringer},
they give an orbifold structure of $M$.
An orbifold structure is the set of all charts 
$(V,\Gamma,\phi)$ which are compatible with the given charts.
\par
Let $M_1,M_2$ be orbifolds. A topological embedding 
$f : M_1 \to M_2$ is said to be an orbifold embedding if 
for each $p\in M_1$ we can take a chart $(V_1,\Gamma_1,\phi_1)$ 
of $p$ in $M_1$, $(V_2,\Gamma_2,\phi_2)$ 
of $f(p)$ in $M_2$ and $f_p : V_1 \to V_2$, 
$h_p : \Gamma_1 \to \Gamma_2$ such that:
\begin{enumerate}
\item
$f_p$ is a smooth embedding of manifolds.
\item
$h_p$ is an isomorphism of groups.
\item 
$f_p(g x) = h_p(g) f_p(x)$. 
\item 
$\phi_2\circ f_p = f \circ \phi_1$.
\end{enumerate}
Note two  orbifold embeddings are regarded as the same if 
they coincide set theoretically. 
(In other words, the existence of $f_p,h_p$ above is the condition for 
$f$ to be an orbifold embedding and is not a part of the data 
consisting of an orbifold embedding.)\footnote{If we 
include an orbifold, which  is not necessarily effective 
or consider a mapping between 
effective orbifolds which is not necessarily an embedding, then 
this point will be different. See \cite{ofdruan}.}
\par
A homeomorphism between orbifolds is said to be a diffeomorphism
if it is  an embedding of orbifold.\index{embedding of orbifold}
\par
The set of all diffeomorphisms of an orbifold $M$ becomes 
a group which we write ${\rm Diff}(M)$.
The group ${\rm Diff}(M)$ becomes a topological group 
by compact open topology.
\begin{defn}
Let $G$ be a Lie group. A {\it smooth action}
\index{smooth action on orbifold} of $G$ on $M$ 
is by definition a continuous group homomorphism 
$G \to {\rm Diff}(M)$ with the following properties.
Note $G \to {\rm Diff}(M)$ induces a continuous map 
$G \times M \to M$.
\par
For each $p \in M$ and $g \in G$ there exists 
a chart $(V_1,\Gamma_1,\phi_1)$ of $p$, 
a chart $(V_2,\Gamma_2,\phi_2)$ of $gp$, 
an open neighborhood $U$ of $g$, and 
maps $f_{p,g} : U \times V_1 \to V_2$,
$h_{p,g} : \Gamma_1 \to  \Gamma_2$
such that:
\begin{enumerate}
\item
$f_{p,g}$ is a smooth map.
\item
$h_{p,g}$ is a group isomorphism.
\item
$f_{p,g}$ is $h_{p,g}$ equivariant.
\item $\phi_2(gv) = g \phi_1(v)$. 
\end{enumerate}
\end{defn}
A (smooth) vector bundle $\mathcal E \to M$ on an orbifold is 
\index{vector bundle on an orbifold}
a pair of orbifolds $\mathcal E$, $M$ and a continuous map $\pi : 
\mathcal E\to M$ such that for each $\tilde p \in \mathcal E$ we can take a 
special choice of coordinates of $\tilde p$ and $\pi(p)$
as follows.
$(V,\Gamma,\phi)$ is a coordinate of $M$ at $p$.
$(V\times E,\Gamma,\tilde\phi)$ is a coordinate 
of $\mathcal E$ at $\tilde p$, where
$E$ is a vector space on which $G$ has a linear action.
Moreover the next diagram commutes,
\begin{equation}\label{diagram616}
\begin{CD}
V\times E @ >{\tilde\phi}>> \mathcal E \\
@ VVV @ VV{\pi}V\\
V @ >{\phi}>> M,
\end{CD}
\end{equation}
where the first vertical arrow is the obvious projection.
See \cite[Definition 15.7 (3)]{foootech2}, \cite[Definition 31.3]{fooospectr} 
or \cite[Definition 23.19]{FOOOSpringer} for the condition required to the coordinate change.
\par
Suppose $M$ has a $G$-action. A {\it $G$-action} on a vector 
\index{00G2actiohn@$G$-action on a vector bundle} bundle 
$\mathcal E \to M$ 
is by definition a $G$-action on $\mathcal E$ such that the projection 
$\mathcal E \to M$ is $G$-equivariant, (Here $G$-equivariance means 
that $\pi (g\tilde p) = g \pi(\tilde p)$, set theoretically.)
and that the local expression
$$
f_{p,g} : G \times (V_1 \times E_1) \to V_2 \times E_2
$$
of $G$ action
preserves the structure of vector space of $E_1$, $E_2$.
Namely for each $g \in G$, $v \in V_1$ the map
$$
V \mapsto \pi_{E_2}(f(g,v,V)), E_1 \to E_2
$$
is linear. (Here $\pi_{E_2} : V_2\times E_2 \to E_2$ is the projection.)
\par
If $\mathcal E \to M$ is a vector bundle on an orbifold, 
its section is by definition an orbifold embedding 
$s : M \to \mathcal E$ such that the composition $M \to \mathcal E \to M$ 
is the identity map (set theoretically).
If $s$ is a section then 
$
(gs)(p) = g(s(g^{-1}p)) 
$ 
defines a section $gs : M \to \mathcal E$.
We say $s$ is $G$-equivariant if $gs = s$.
If $s$ is a $G$-equivariant section then 
$$
s^{-1}(0) = \{x \in M \mid s(x) = 0\}
$$
is $G$-invariant subset of $M$.
(Here $0 \subset \mathcal E$ is the set such that 
by the coordinate $(V\times E,\Gamma,\tilde\phi)$
it corresponds to a point in $V\times \{0\}$.)
\par
Now we define the notion of $G$-equivariant Kuranishi chart as follows.
\begin{defn}
Let $X$ be a metrizable space on which a compact Lie group $G$ acts
and $p \in X$. 
A {\it $G$-equivariant Kuranishi chart} \index{00G2equiv@$G$-equivariant Kuranishi chart}
of $X$ at $p$ is an object $(U,\mathcal E,s,\psi)$ 
such that:
\begin{enumerate}
\item
We are given an orbifold $U$,  on which $G$ acts.
\item
We are given a $G$-equivariant vector bundle $\mathcal E$ on $U$.
\item
We are given  a $G$-equivariant smooth section $s$ of $\mathcal E$.
\item
We are given a $G$-equivariant homeomorphism $\psi : s^{-1}(0) \to X$ onto an open set.
\end{enumerate}
We call $U$ the {\it Kuranishi neighborhood}\index{Kuranishi neighborhood}, $\mathcal E$ the 
{\it obstruction bundle},\index{obstruction bundle}
$s$ the {\it Kuranishi map},\index{Kuranishi map} and $\psi$ the 
{\it parametrization}.\index{parametrization}
\end{defn}
Let $(X,\omega)$ be a compact symplectic manifold on which 
a compact Lie group $G$ acts preserving the symplectic structure $\omega$.
We define an equivalence relation on $\pi_2(X)$ by
$$
[v] \sim [v']  
\quad 
\Leftrightarrow
\qquad
\int v^* \omega = \int (v')^* \omega,
\quad v_*([S^2]) \cap c^1(X) = v'_*([S^2]) \cap c^1(X).
$$
We denote by $\Pi_2(X)$ the group of the equivalence classes of $\sim$.
Let $\alpha \in \Pi_2(X)$ and $g,\ell$ be nonnegative integers.
We take and fix a $G$-invariant compatible almost complex structure
$J$ on $X$.
Let $\mathcal M_{g,\ell}((X,J);\alpha)$ \index{00M3GELLXJ@$\mathcal M_{g,\ell}((X,J);\alpha)$} be the moduli space of 
$J$-holomorphic stable maps of genus $g$ with $\ell$ marked points and 
homology class is $\alpha$.
See for example \cite[Defnition 7.7]{FO} for its definition.
(The notion of stable map is introduced by Kontsevitch. 
Systematic study of the moduli space 
$\mathcal M_{g,\ell}((X,J);\alpha)$ 
in the semi-positive case was initiated by Ruan-Tian \cite{ruantian}
\cite{ruantian2}.
Studying $J$-holomorphic curve in symplectic geometry 
is a great invention by Gromov.
There is a nice account of genus zero case by McDuff-Salamon \cite{MSpaper}.)
The topology (stable map topology) on $\mathcal M_{g,\ell}((X,J);\alpha)$ was introduced 
by Fukaya-Ono (in the year 1996) in \cite[Defnition 10.3]{FO}
and 
they proved that $\mathcal M_{g,\ell}((X,J);\alpha)$ 
is compact (\cite[Theorem 11.1]{FO}) and 
Hausdorff (\cite[Lemma 10.4]{FO}), in this particular 
topology.
There exist evaluation maps
${\rm ev} : \mathcal M_{g,\ell}((X,J);\alpha) 
\to X^{\ell}$. (See  \cite[page 936, line 3]{FO}.)
\par
Since $J$ is $G$-equivariant it is easy to see that the group 
$G$ acts on the topological space $\mathcal M_{g,\ell}((X,J);\alpha)$.
\par
Now the main result of this paper is the following:
\begin{thm}\label{thm54}
For each $p \in \mathcal M_{g,\ell}((X,J);\alpha)$,
there exists a $G$-equivariant Kuranishi chart of 
$\mathcal M_{g,\ell}((X,J);\alpha)$ 
at $p$.
\par
The evaluation map ${\rm ev} : \mathcal M_{g,\ell}((X,J);\alpha) 
\to X^{\ell}$ is an underlying continuous map 
of a weakly submersive map.\footnote{See \cite[
Definition 3.38 (5)]{foootech2} or
\cite[Definition 32.1 (4)]{fooospectr} for the definition 
of this notion.}\index{weakly submersive map}
\end{thm}
\begin{rem}
Note since the parametrization $\psi$ is assumed to be 
$G$-equivariant its image necessary contains the $G$-orbit 
of $p$ in $\mathcal M_{g,\ell}((X,J);\alpha)$.
Therefore a $G$-equivariant Kuranishi chart cannot be 
completely local in $\mathcal M_{g,\ell}((X,J);\alpha)$.
\end{rem}
\begin{rem}
Once we proved Theorem \ref{thm54} we can 
construct a $G$-equivariant Kuranishi structure on 
$\mathcal M_{g,\ell}((X,J);\alpha)$ 
in the same way as the case without $G$ action.
(See for example \cite{foooconstr}, \cite{foooconstr2}.)
In this paper we focus on proving Theorem \ref{thm54} 
since this is the novel part in our $G$-equivariant situation. 
\end{rem}

\section{Proof of the main theorem}
\label{sec:main}

In this section we prove Theorem \ref{thm54} except a few points 
postponed to later sections.
Let $((\Sigma',\vec z^{\,\prime}),u')$ be an object which is 
G-$\epsilon_1$-close \index{00E5psilon@$\epsilon_1$} to $((\Sigma,\vec z),u)$. (We determine
the positive constant $\epsilon_1$ later.)
We fix ${\bf x}_0 \in \mathcal{OB}$ such that 
$(\Sigma',\vec z^{\,\prime})$ is bi-holomorphic to 
$(\Sigma({\bf x}_0),\vec z({\bf x}_0))$.
We also fix a bi-holomorphic map 
$\phi_0 : (\Sigma({\bf x}_0),\vec z({\bf x}_0)) \cong (\Sigma',\vec z^{\,\prime})$.
\begin{defn}\label{defn63}
We define  
$\mathcal W(\epsilon_1;{\bf x}_0,\phi_0;((\Sigma',\vec z^{\,\prime}),u'))$ 
as the set of pairs \index{00W3epsilon1@$\mathcal W(\epsilon_1;{\bf x}_0,\phi_0;((\Sigma',\vec z^{\,\prime}),u'))$}
$(\varphi,g)$ such that:
\begin{enumerate}
\item $\varphi \in \mathcal{MOR}$,
$g \in G$.
\item
${\rm Pr}_t(\varphi) = {\bf x}_0$.
\item
We put ${\bf x}' = {\rm Pr}_s(\varphi)$.
The morphism $\varphi$ defines a bi-holomorphic map 
$$
\varphi : (\Sigma({\bf x}'),\vec z({\bf x}')) 
\cong (\Sigma({\bf x}_0),\vec z({\bf x}_0)).
$$
We consider 
$\phi_0 \circ \varphi : (\Sigma({\bf x}'),\vec z({\bf x}')) 
\to (\Sigma',\vec z^{\,\prime})$.
Then 
$((\Sigma',\vec z^{\,\prime}),u')$ 
is 
$2\epsilon_1$-
close to $((\Sigma,\vec z),gu)$
by ${\bf x}'$, $\phi_0 \circ \varphi$.
\end{enumerate}
\end{defn}
\begin{lem}\label{lem6262}
The space $\mathcal W(\epsilon_1;{\bf x}_0,\phi_0;((\Sigma',\vec z^{\,\prime}),u'))$ 
has a structure of smooth manifold.
\end{lem}
\begin{proof}
The set of $(\varphi,g)$ satisfying Items (1)(2) has 
a structure of smooth manifold since $\mathcal{MOR}$ is a 
smooth manifold and ${\rm Pr}_t$ is a submersion.
Since the condition (3) is an open condition 
the space $\mathcal W(\epsilon_1;{\bf x}_0,\phi_0;((\Sigma',\vec z^{\,\prime}),u'))$ 
is an open set of a smooth manifold and so has a structure of 
smooth manifold.
\end{proof}
We used ${\bf x}_0,\phi_0$ to define 
$\mathcal W(\epsilon_1;{\bf x}_0,\phi_0;((\Sigma',\vec z^{\,\prime}),u'))$.
However this manifold is independent of the choice of
such ${\bf x}_0,\phi_0$ as the next lemma shows.
\begin{lem}\label{lem64}
Let ${\bf x}_1 \in \mathcal{OB}$ and
$\phi_1 : (\Sigma({\bf x}_1),\vec z({\bf x}_1)) \cong (\Sigma',\vec z^{\,\prime})$
be a bi-holomorphic map.
\par
The composition 
$\phi_1^{-1} \circ \phi_0$ determines an element 
$\psi \in \mathcal{MOR}$ such that
${\rm Pr}_s(\psi) =  {\rm Pr}_t(\psi) = {\bf x}_1$.
\par
Then the next two conditions are equivalent.
\begin{enumerate}
\item
$(\varphi,g) \in \mathcal W(\epsilon_1;{\bf x}_0,\phi_0;((\Sigma',\vec z^{\,\prime}),u'))$.
\item 
$(\psi\circ\varphi,g)
\in \mathcal W(\epsilon_1;{\bf x}_1,\phi_1;((\Sigma',\vec z^{\,\prime}),u'))$.
\end{enumerate}
\end{lem}
The proof of Lemma \ref{lem64}
are obvious from definition.
\begin{defn}\label{defn64}
Let $\sigma$ be a positive number which we will fix during the proof of 
Lemma \ref{lem67}.
We take a smooth function $\chi : \Sigma_0(\sigma) \to [0,1]$
such that:
\begin{enumerate}
\item 
$\chi \equiv 1$ on $\Sigma_0(2\sigma)$.
\item
$\chi$ has compact support.
\item
$\chi$ is $\mathcal G_c$ invariant. Here $\mathcal G_c$ is defined 
by (\ref{defnGc}). 
\end{enumerate}
\end{defn}
\begin{defn}\label{defn6666}
We define a function ${\rm meandist} \index{00M1@${\rm meandist}$}
: \mathcal W(\epsilon_1;{\bf x}_0,\phi_0;((\Sigma',\vec z^{\,\prime}),u'))
\to \R$ as follows. \index{00M1@${\rm meandist}$}
\begin{equation}\label{form611}
{\rm meandist}(\varphi,g)
= 
\int_{z \in \Sigma(\sigma)}
\chi(z)\,\, d_X^2((u'\circ \phi_0 \circ \varphi
\circ \Phi_{{\bf x}',\sigma})(z),gu(z)\,\, \Omega_{\Sigma}.
\end{equation}
Here $\Omega_{\Sigma}$ is the volume element of $\Sigma$ 
and $d_X$ is the Riemannian distance function on $X$.
We assume $\Omega_{\Sigma}$ \index{00O5megasigma@$\Omega_{\Sigma}$} is invariant under 
${\mathcal G}_c$ action. 
\end{defn}
The main properties of this function is given below.
\begin{lem}\label{lem67}
The function ${\rm meandist}$ has the following 
properties if $\epsilon_1$ is sufficiently small.
\begin{enumerate}
\item
${\rm meandist}$ is a convex function.
\item
If $\mathcal W(\epsilon_1;{\bf x}_0,\phi_0;((\Sigma',\vec z^{\,\prime}),u'))
\cong \mathcal W(\epsilon_1;{\bf x}_1,\phi_1;((\Sigma',\vec z^{\,\prime}),u'))$
is the isomorphism given in Lemma \ref{lem64} 
then ${\rm meandist}$  is compatible with this isomorphism.
\end{enumerate}
\end{lem}
\begin{proof}
(2) is obvious from construction.
The convexity of ${\rm meandist}$  follows from 
the convexity of distance function.
(We omit the detail of the proof of convexity here since 
we will prove a stronger result in Proposition \ref{prop69}.)
\end{proof}
The function ${\rm meandist}$ is not in general 
strictly convex. To obtain strictly convex function 
we need to take the quotient by the
$\widehat{\mathcal G}_c$ action as follows.
For each $\upsilon = (\gamma,h) \in \hat{\mathcal G}_c$
and ${\bf x} \in \mathcal{OB}$ we have
$\gamma {\bf x} \in \mathcal{OB}$
and a bi-holomorphic map
$\gamma_* : (\Sigma({\bf x}),\vec z({\bf x}))
\to (\Sigma(\gamma{\bf x}),\vec z(\gamma{\bf x}))$.
This is a consequence of Proposition \ref{prop318}.
(We write $\gamma {\bf x}$ and $\gamma_*$ since it is 
independent of $h$.)
By definition
\begin{equation}
u \circ \gamma_* = h \circ u
\end{equation}
where we consider the case ${\bf x} = 0$, that is, 
$\gamma_* : (\Sigma,\vec z)
\to (\Sigma,\vec z)$.
\begin{defn}\label{defn686861}
We define a right $\widehat{\mathcal G}_c$ action on 
$\mathcal W(\epsilon_1;{\bf x}_0,\phi_0;((\Sigma',\vec z^{\,\prime}),u'))$
as follows. 
Let
$(\varphi,g) \in \mathcal W(\epsilon_1;{\bf x}_0,\phi_0;((\Sigma',\vec z^{\,\prime}),u'))$.
Let $\upsilon = (\gamma,h) \in \widehat{\mathcal G}_c$.
We have ${\rm Pr}_s(\varphi) = {\bf x}$. Set
$\gamma^{-1} {\bf x} = {\bf y}$
and $\gamma_* : (\Sigma({\bf y}),\vec z({\bf y})) 
\cong (\Sigma({\bf x}),\vec z({\bf x}))$.
We may thus regard $\gamma_* \in \mathcal{MOR}$
with ${\rm Pr}_s(\gamma_*) = {\bf y}$
and ${\rm Pr}_t(\gamma_*) = {\bf x}$.
\par
We now put 
\begin{equation}\label{newform63}
\upsilon(\varphi,g)
= (\varphi \circ \gamma_*,gh).
\end{equation}
\par
It is easy to see that (\ref{newform63}) defines a {\it right}
$\widehat{\mathcal G}_c$ action on 
$\mathcal W(\epsilon_1;{\bf x}_0,\phi_0;((\Sigma',\vec z^{\,\prime}),u'))$.
We also observe that this action is free. 
In fact, if $(\gamma,e) \in \hat{\mathcal G}_c$ is not the unit
then $\varphi \circ \gamma_* \ne \varphi$.
(Here $e$ is the unit of $G$.)
\end{defn}

\begin{prop}\label{prop69}
For $\upsilon \in \hat{\mathcal G}_c$ we have
$$
{\rm meandist}(\upsilon(\varphi,g))
= {\rm meandist}(\varphi,g).
$$
Moreover the induced function \index{00M1over@$\overline{\rm meandist}$}
$$
\overline{\rm meandist} : 
\mathcal W(\epsilon_1;{\bf x}_0,\phi_0;((\Sigma',\vec z^{\,\prime}),u'))/\widehat{\mathcal G}_c 
\to \R
$$
is strictly convex if $\epsilon_1$ is sufficiently small.
\end{prop}
\begin{proof}
The first half follows from
$$
d_X((u'\circ\phi_0\circ \varphi \circ \gamma_*)(z),ghu(z))
= 
d_X((u'\circ\phi_0\circ \varphi)(w),gu(w))
$$
where $(\gamma_*)(z) = w$. (Note $\Phi_{{\bf x},\sigma}^{-1} \circ \gamma_*\circ \Phi_{{\bf y},\sigma} : 
\Sigma(\sigma) \to 
\Sigma(\sigma)$ preserves $\Omega_{\Sigma}$ since it coincides with the 
action of $\gamma \in {\mathcal G}_c$.)
\par
We next prove the strict convexity.
Since $C^{2}$ difference between
$u' \circ \phi_0 \circ \varphi  \circ \Phi_{\bf x,\delta}$ and $g \circ u\vert_{\Sigma(\delta)}$
is smaller than $\epsilon$ and the strict convexity is preserved by a small 
$C^2$ perturbation, it suffices to show the case of
$((\Sigma',\vec z^{\,\prime}),u') = ((\Sigma,\vec z),u)$.
\par
Let $t \mapsto (\varphi^t,h^t)$ be a geodesic of unit speed
in the manifold 
$\mathcal W(\epsilon_1;o,{\rm id};((\Sigma,\vec z),u))$, 
which is perpendicular to the $\widehat{\mathcal G}_c$ orbits
at $t=0$.
We will prove
\begin{equation}\label{eq640}
\frac{d^2}{dt^2} {\rm meandist}(\varphi^t,h^t) \ge \tau > 0.
\end{equation}
Note that $dh_t/dt\vert_{t=0}$ can be regarded as a vector field 
on $X$, which we denote by $V_h$.
We consider the following 3 cases separately.
\par\medskip
\noindent(Case 1)  
We first assume that there exists $z \in \Sigma$ such that
$$
V_h(u(z)) \notin D_zu(T_z{\Sigma}).
$$
\par
Note that the set of such points $z$ is open. 
Therefore we may choose $\sigma$ in Definition \ref{defn64} 
so that we may assume $z \in \Sigma(\sigma) \cap \Sigma_a$.
Then
\begin{equation}\label{eq64}
\int_{\Sigma_a \cap \Sigma(\sigma)}
\chi(z)\,\, \left\Vert\frac{d}{dt}u(\varphi^t_a(z)) - \frac{dh^t}{dt} u(z)
\right\Vert^2\Omega_{\Sigma}
\ge \rho.
\end{equation}
at $t=0$,
where $\rho > 0$ is independent of $\epsilon_1$. 
Therefore Proposition \ref{prop98} implies
(\ref{eq640}) at $t=0$.
Since the third derivative of ${\rm meandist}$ is uniformly 
bounded we can choose $\epsilon_1$ small to conclude the 
required strict convexity.
\par\medskip
\noindent(Case 2)
We next assume 
$
V_h(u(z)) \in D_zu(T_z{\Sigma})
$
for all $z \in \Sigma$.
We also assume that $V_h(u(z)) = 0$ for all the nodal point $z$.
\par
If $D_zu = 0$ this implies that $V_h(u(z))=0$.
If $D_zu \ne 0$ then $u$ is an immersion at $z$. 
Therefore there exists a unique $\tilde V_h(z) \in T_z\Sigma$
such that 
$V_h(u(z)) = D_zu(\tilde V_h(z))$.
Putting $\tilde V_h(z) = 0$ when $D_zu = 0$,
we obtain a vector field $\tilde V_h$ of $C^0$ class on each $\Sigma_a$.
The vector field $\tilde V_h$ is smooth on the open subset where $u$ is an immersion.
Note that $G$ action preserves almost complex structure and $u$ 
is pseudo-holomorphic. We use this fact to show 
that $\tilde V_h$ is a holomorphic vector field (and in particular 
is of $C^{\infty}$ class) as follows.
Let $K$ be a compact subset of the set of $z$ with $D_zu \ne 0$.
It is easy to show that we can integrate $\tilde V_h$ to obtain a family of
embeddings $\psi_t$ of a neighborhood of $K$ to $\Sigma$ for small $t$
such that
$$
u \circ \psi_t = {\rm Exp}(V_h) \circ u.
$$
Here we regard $V_h$ as an element of the Lie algebra of $G$. Since ${\rm Exp}(V_h) \in G$ preserves almost complex structure
$\psi_t$ is a holomorphic embedding.
Therefore $\tilde V_h$ is a holomorphic vector field on $K$.
Since $K$ is an arbitrary compact subset of the set of $z$ with $D_zu \ne 0$,
$\tilde V_h$ is a holomorphic vector field outside its zero set.
Therefore by Riemann's removable singularity theorem 
$\tilde V_h$ is a holomorphic vector field on each $\Sigma_a$.
\par
By assumption $\tilde V_h$ vanishes at each nodal points.
\par
Thus $(\tilde V_h,V_h)$ is an element of the Lie algebra of $\widehat{\mathcal G}_c$.
On the other hand, by assumption $\frac{d}{dt}(\varphi^t,h^t)\vert_{t=0}$ is perpendicular to 
a $\widehat{\mathcal G}_c$-orbit. 
Therefore there exists $\Sigma_a$ such that 
\begin{equation}\label{form6565}
\int_{\Sigma_a \cap \Sigma(\sigma)}
\chi(z)\,\,\left\Vert\frac{d\varphi^t}{dt}(0) - \tilde V_h(z)
\right\Vert^2\Omega_{\Sigma}
\ge \rho.
\end{equation}
Since $V_h(u(z)) = D_zu(\tilde V_h(z))$ 
it implies (\ref{eq64}) at $t=0$.
The rest of the proof is the same as (Case 1).
\par\medskip
\noindent(Case 3) We finally consider the case when
$V_h(u(z))$ is non-zero at certain nodal point $z$.
Since strict convexity is an open property, we use (Case 2) and can 
assume that $\Vert V_h(u(z))\Vert \ge c_0$ for some positive constant 
$c_0$.
\par
Note $\frac{d}{dt}u(\varphi^t_a(z))$ is zero at the nodal 
point $z$. Therefore we can choose $\sigma$ small such that
(\ref{eq64}) holds at $t=0$.
The rest of the proof is the same as (Case 1).
\par
The proof of Proposition \ref{prop69} is complete. 
\end{proof}
\begin{lem}\label{lem610}
If $\epsilon_1$ is enough small then 
$\overline{\rm meandist}$ attains its local 
minimum at a unique point of 
$\mathcal W(\epsilon_1;{\bf x}_0,\phi_0;((\Sigma',\vec z^{\,\prime}),u'))/\widehat{\mathcal G}_c$.
\end{lem}
\begin{proof}
In case $((\Sigma',\vec z^{\,\prime}),u') =  ((\Sigma,\vec z),gu)$ 
the local 
minimum is attained only at the point $({\rm id},g)$.
In general $u'$ is $C^2$ close to $gu$ by reparametrization.
We can find a $C^2$-small homotopy between $gu$ and $u'$.
Strict convexity implies that uniqueness of minima does not change
during this homotopy.
\end{proof}
Now let 
$(\varphi,g) 
\in \mathcal W(\epsilon_1;{\bf x}_0,\phi_0;((\Sigma',\vec z^{\,\prime}),u'))$ 
be a representative of unique minimum of $\overline{\rm meandist}$.
We put
${\bf y} = {\rm Pr}_s(\varphi)$
then
\begin{equation}
d_X((u' \circ \phi_0 \circ \varphi)(\Phi_{{\bf y},\delta}(z)),
gu(z))
\le 2\epsilon_1,
\end{equation}
by Definitions \ref{defn63} and \ref{defn41}.
We define
$
\Psi : K' \to \Sigma
$
by
\begin{equation}\label{psiinSec6}
\Psi(w) =\Phi_{{\bf y},\delta}^{-1}(\varphi^{-1}(\phi_0^{-1}(w))).
\end{equation}
Here $K'\subset \Sigma' \setminus {\rm nodes}$ is a compact subset such that
$\varphi^{-1}(\phi_0^{-1}(K')) \subset {\rm Im}\,\Phi_{{\bf y},\delta}$.
We remark that
$$
d_X(gu(\Psi(w)),u'(w)) \le \epsilon_1.
$$
We define
\begin{equation}\label{formula65}
{\rm Pal} : C_0^{\infty}({\rm Int}\,K';(u')^*TX) \to 
C^{\infty}(\Sigma;(gu)^*TX)
\end{equation}
by the parallel transportation along the unique minimal 
geodesic joining $u'(w)$ and $gu(\Psi(w))$. 
(Here $C_0^{\infty}$ stands for the set of smooth sections with compact support.)
We take a ($G$-equivariant) unital connection of $TM$ to define 
the parallel transportation so that ${\rm Pal}$ is 
complex linear.\footnote{In various literature
people use Levi-Civita connection in a similar situations.
There is no particular reason to take Levi-Civita connection.}
\par
Note $\Psi$ is in general not holomorphic since 
$\Phi_{\bf y,\delta}$ is not holomorphic.
We decompose
$$
D\Psi : T_w \Sigma' \to T_{\Psi(w)} \Sigma
$$
into complex linear part and complex anti-linear part. 
Let $D^h\Psi : T_w \Sigma'(\delta) \to T_{\Psi(w)} \Sigma$
be the complex linear part. It induces 
\begin{equation}\label{formula66}
d^h\Psi : \Lambda_{\Psi(w)}^{01} \Sigma \to \Lambda_w^{01}\Sigma'.
\end{equation}
We use (\ref{formula65}) and (\ref{formula66})
to obtain \index{00I2X0phi@$I_{{\bf x}_0,\phi_0;((\Sigma',\vec z^{\,\prime}),u')}$}
\begin{equation}\label{formula663333}
I_{{\bf x}_0,\phi_0;((\Sigma',\vec z^{\,\prime}),u')}
: C^{\infty}(K;(gu)^*TX \otimes \Lambda^{01})
\to C^{\infty}(K';(u')^*TX \otimes \Lambda^{01})
\end{equation}
for a compact subset $K \subset \Sigma$. 
contained in the image of $
\Psi : K' \to \Sigma
$.
We may choose $K'$ and $K$ so that $K$ contains the support of 
elements of the obstruction space at origin $E((\Sigma,\vec z),u)$.
\begin{defn}
We define a finite dimensional linear subspace
$$
E({\bf x}_0,\phi_0;((\Sigma',\vec z^{\,\prime}),u'))
\subset C^{\infty}(\Sigma'(\delta);(u')^*TX \otimes \Lambda^{01})
$$
as the image of 
$$
g_*E((\Sigma,\vec z),u) \subset C^{\infty}(K;(gu)^*TX \otimes \Lambda^{01})
$$
by the map $I_{{\bf x}_0,\phi_0;((\Sigma',\vec z^{\,\prime}),u')}$.
\end{defn}
\begin{lem}
$E({\bf x}_0,\phi_0;((\Sigma',\vec z^{\,\prime}),u'))$ 
depends only on $((\Sigma',\vec z^{\,\prime}),u')$.
Namely:
\begin{enumerate}
\item 
It does not change when we replace $(\varphi,g)$
by an alternative representative of 
$\mathcal W(\epsilon_1;{\bf x}_0,\phi_0;((\Sigma',\vec z^{\,\prime}),u'))/\widehat{\mathcal G}_c$
in $\mathcal W(\epsilon_1;{\bf x}_0,\phi_0;((\Sigma',\vec z^{\,\prime}),u'))$.
\item
It does not change when we replace ${\bf x}_0,\phi_0$ by other choices.
\end{enumerate}
\end{lem}
\begin{proof}
(1) is a consequence of $\widehat{\mathcal G}_c$ invariance of 
$E((\Sigma,\vec z),u)$. (Definition \ref{defn42} (3).)
(2) is a consequence of  Lemmata \ref{lem64}, \ref{lem67}.
\end{proof}
Hereafter we write 
$E((\Sigma',\vec z^{\,\prime}),u')$ \index{00E2Sigmaz@$E((\Sigma',\vec z^{\,\prime}),u')$}
in place of $E({\bf x}_0,\phi_0;((\Sigma',\vec z^{\,\prime}),u'))$.
We call $E((\Sigma',\vec z^{\,\prime}),u')$ {\it the obstruction space 
at $((\Sigma',\vec z^{\,\prime}),u')$}.
\begin{lem}\label{lem613233}
If $h \in G$ then
$$
E((\Sigma',\vec z^{\,\prime}),hu') = h_*E((\Sigma',\vec z^{\,\prime}),u').
$$
\end{lem}
\begin{proof}
Let 
$(\varphi,g) 
\in \mathcal W(\epsilon_1;{\bf x}_0,\phi_0;((\Sigma',\vec z^{\,\prime}),u'))$ 
be a representative of the unique minimum of $\overline{\rm meandist}$.
Then 
$(\varphi,hg)\in \mathcal W(\epsilon_1;{\bf x}_0,\phi_0;((\Sigma',\vec z^{\,\prime}),hu')) 
$ is a representative of the unique minimum of $\overline{\rm meandist}$.
The lemma follows immediately.
\end{proof}
\begin{defn}
Let $((\Sigma',\vec z^{\,\prime}),u')$, $((\Sigma'',\vec z^{\,\prime\prime}),u'')$
be two objects which are $G$-$\epsilon_1$ close to 
$((\Sigma,\vec z),u)$. We say that 
$((\Sigma',\vec z^{\,\prime}),u')$ is {\it isomorphic}\index{isomorphic} to $((\Sigma'',\vec z^{\,\prime\prime}),u'')$
if there exists a bi-holomorphic map 
$\varphi : (\Sigma',\vec z^{\,\prime}) \to (\Sigma'',\vec z^{\,\prime\prime})$
such that $u'' \circ \varphi = u'$.
\end{defn}
\begin{defn}\label{defn61616}
We denote by $U(((\Sigma,\vec z),u);\epsilon_2)$  \index{00U2sigma@$U(((\Sigma,\vec z),u);\epsilon_2)$}
the set of all isomorphism classes of $((\Sigma',\vec z^{\,\prime}),u')$
which are $G$-$\epsilon_2$-close to $((\Sigma,\vec z),u)$ and
\begin{equation}\label{eqmain615}
\overline\partial u' \in E((\Sigma',\vec z^{\,\prime}),u'). \index{00E5psilon2@$\epsilon_2$}
\end{equation}
It is easy to see that if $((\Sigma',\vec z^{\,\prime}),u')$ is 
equivalent to $((\Sigma'',\vec z^{\,\prime\prime}),u'')$
then $((\Sigma',\vec z^{\,\prime}),u')$ satisfies (\ref{eqmain615}) 
if and only if $((\Sigma'',\vec z^{\,\prime\prime}),u'')$ satisfies (\ref{eqmain615}).
\par
Because of Lemma \ref{lem613233} 
there exists a $G$-action on $U(((\Sigma,\vec z),u);\epsilon_2)$
defined by $h((\Sigma',\vec z^{\,\prime}),u') = ((\Sigma',\vec z^{\,\prime}),hu')$.
\end{defn}
\begin{prop}\label{prop616}
If $\epsilon_2$ is small we have the following.
\begin{enumerate}
\item $U(((\Sigma,\vec z),u);\epsilon_2)$ has a structure of 
effective orbifold. 
The $G$-action defined above becomes a smooth action.
\item
There exists a smooth vector bundle $E(((\Sigma,\vec z),u);\epsilon_2)$
on \index{00E2sigmavec@$E(((\Sigma,\vec z),u);\epsilon_2)$} $U(((\Sigma,\vec z),u);\epsilon_2)$ whose fiber at 
$[((\Sigma',\vec z^{\,\prime}),u')]$ is identified with 
$E((\Sigma',\vec z^{\,\prime}),u')$.
The vector bundle
$E(((\Sigma,\vec z),u);\epsilon_2)$ has a smooth $G$-action.
\item
The Kuranishi map $s$ which assigns 
$\overline\partial u' \in E((\Sigma',\vec z^{\,\prime}),u')$
to $[((\Sigma',\vec z^{\,\prime}),u')]$  becomes a 
smooth section of $E((\Sigma',\vec z^{\,\prime}),u')$ and
is $G$-equivariant.
\item
The set
$$
s^{-1}(0) = \{[((\Sigma',\vec z^{\,\prime}),u')] \in U(((\Sigma,\vec z),u);\epsilon_2)
\mid s([(\Sigma',\vec z^{\,\prime}),u')]) = 0\}
$$
is homeomorphic (by an obvious map) to an open neighborhood of 
$[(\Sigma,\vec z),u]$ in $\mathcal M_{g,\ell}((X,J);\alpha)$,
$G$-equivariantly.
\item
The map which sends $[(\Sigma',\vec z^{\,\prime}),u')]$
to $((u'(z'_1),\dots,u'(z'_{\ell})),[\Sigma',\vec z^{\,\prime}])$ defines a 
$G$-equivariant smooth
submersion $U(((\Sigma,\vec z),u);\epsilon_2) \to X^{\ell}
\times \mathcal M_{g,\ell}$.
\end{enumerate}
\end{prop}
Theorem \ref{thm54} follows immediately from Proposition \ref{prop616}.
The remaining part of the proof of Proposition \ref{prop616}
is gluing analysis. 
Actually gluing analysis is mostly the same as one we described in 
detail in \cite{foooexp}.
The new point we need to check is the behavior of the 
(family of) obstruction spaces $E((\Sigma',\vec z^{\,\prime}),u')$
while we move $((\Sigma',\vec z^{\,\prime}),u')$, especially 
while $\Sigma'$ becomes nodal in the limit.
We will describe this point in the next section 
(Subsection \ref{mainprop}).
We also provide detail of the way how
to use gluing analysis to prove 
Proposition \ref{prop616},
though this part is mostly the same 
as \cite[Part 4]{foootech} and \cite{foooconstr}.

\section{Gluing and smooth charts}
\label{sec:decay}

In this section, we show that the gluing analysis we detailed in \cite{foooexp}
can be applied to prove Proposition \ref{prop616}.
We remark that to work out gluing analysis we need to 
`stabilize' the domain curve. This is because we need to 
specify the coordinate of the source curve for gluing analysis.
We can use  the frame work of this paper, 
the universal family parametrized by a Lie groupoid,  
for this purpose also. In fact if we use Lemma \ref{lem610}
we can specify the coordinate of the source 
curve $\Sigma'$ (depending on the map $u'$.)
However here we do not take this way to prove our main theorem. 
We use another method to `stabilize' the domain curve,
that is, to add extra marked points and eliminate 
the extra parameter (of moving added marked points) 
by using transversal codimension 2 submanifolds.
This is the way taken in \cite[Appendix]{FO}.
The main reason why we use this method is the consistency with the 
existing literature. 
For example this method was used in \cite{foootech,foooconstr} to specify the 
coordinate of the source curve.
We remark that this way to stabilize the domain 
breaks the symmetry of $G$-action. 
This fact however does not affect the proof of 
Proposition \ref{prop616}. 
In fact the family of obstruction spaces 
$E((\Sigma',\vec z^{\,\prime}),u')$ and 
the solution set (the thickened moduli space) 
$U(((\Sigma,\vec z),u);\epsilon_2)$ are already 
defined and are $G$-equivariant.
The gluing analysis we describe below 
is used to establish certain {\it properties} of them
and is not used to {\it define} them.
By this reason we can break the $G$-equivariance of 
the construction here.
(See Subsection \ref{smoothchart2}, 
especially (the proof of) Lemma \ref{lem72727}, for more explanation 
on this point.)

\subsection{Construction of the smooth chart 1: The way how we adapt the result 
of \cite{foooexp}}
\label{smoothchart1}

For the purpose of proving Proposition \ref{prop616}
we construct 
a chart of $U(((\Sigma,\vec z),u);\epsilon_2)$ centered at each point 
$[((\Sigma_1,\vec z_1),u_1)]$ of $U(((\Sigma,\vec z),u);\epsilon_2)$.
Here $(\Sigma_1,\vec z_1)$ is a marked nodal curve of genus $g$ and 
with $\ell$ marked points and $u_1 : \Sigma_1 \to X$ is a map such that 
$(\Sigma_1,\vec z_1)$ is $G$-$\epsilon_2$ close to 
$((\Sigma,\vec z),u)$. We require the map $u_1$ to satisfy the equation
\begin{equation}
\overline{\partial} u_1 \in E(((\Sigma_1,\vec z_1),u_1)).
\end{equation}
Let \index{00G311@$\mathcal G_1$}
$$
\mathcal G_1 =
\mathcal G((\Sigma_1,\vec z_1),u_1) 
=
\{ v : (\Sigma_1,\vec z_1) 
\to (\Sigma_1,\vec z_1) 
\mid \text{$v$ is bi-holomorphic and $u_1 \circ v = u_1$}\}.
$$
Since $((\Sigma,\vec z),u)$ is a stable map 
$\mathcal G((\Sigma,\vec z),u)$ is a finite group.
We may choose $\epsilon_2$ small so that $\mathcal G_1$ 
is a subgroup of $\mathcal G((\Sigma,\vec z),u)$.
Therefore $\mathcal G_1$ is a finite group.

\begin{defn}\label{defn71}{\rm (See \cite[Definition 17.5]{foootech})}
{\it Stabilization data} \index{stabilization data} of the source curve of $((\Sigma_1,\vec z_1),u_1)$
are choices of $\vec w_1$ and $\vec{\mathcal N} = \{\mathcal N_j\}$ with the following 
properties.\index{00W2i@$\vec w_1$}\index{00N3Nj@$\mathcal N_j$}\index{00N3vect@$\vec{\mathcal N}$}
\begin{enumerate}
\item
$\vec w_1$ consists of finitely many ordered points $(w_{1,1},\dots,w_{1,k})$
of $\Sigma_1$. None of those points are nodal. 
$\vec w_1 \cap \vec z_1 = \emptyset$ and $w_{1,i} \ne w_{1,j}$
for $i \ne j$.
\item
The marked nodal curve
$(\Sigma_1,\vec z_1\cup \vec w_1)$ is stable.
Moreover its automorphism group is trivial.
\item
The map $u_1$ is an immersion at each added marked points $w_{1,i}$.
\item
For each $v \in \mathcal G_1$, 
there exists a permutation $\sigma_v : 
\{1,\dots,k\} \to \{1,\dots,k\}$
such that
$$
v(w_{1,i}) = w_{1,\sigma_v(i)}.
$$
\item
$\mathcal N_j$ is a codimension $2$ submanifold of $X$. 
\item
There exists a neighborhood $U_{1,j}$ of $w_{1,j}$
such that 
$$
u_1^{-1}(\mathcal N_j) \cap U_{1,j} = \{ w_{1,j}\}
$$
and $u_1(U_{1,j})$ intersects with $\mathcal N_j$ 
transversality at $u_1(w_{1,j})$.
\item
If $v \in \mathcal G_1$
then
$$
\mathcal N_{\sigma_v(i)} = \mathcal N_{i}.
$$
(Note $u_1(w_{1,\sigma_v(i)}) = u_1(w_{1,i})$
and $u_1 \circ v = u_1$ on a neighborhood of $w_{1,i}$,
by Item (4).)
\end{enumerate}
We also assume the following extra condition.
(The condition below implies that $w_{1,j}$ is  away from the neck region.)
\begin{enumerate}
\item[(8)]  
We decompose $\Sigma_1$ into irreducible components as 
(\ref{decomp7272}). 
\begin{enumerate}
\item Suppose the Euler number of $\Sigma_{1,a} \setminus (\Sigma_{1,a}\cap \vec z_1) \setminus$ nodal points of $\Sigma_1$ 
is negative.
We put complete Riemannian metric of 
constant negative curveture $-1$ and with finite volume 
on this space. Then the injectivity radius at 
$w_{1,j}$ is not smaller than some positive universal constant 
$\epsilon_0$. (In fact we may take $\epsilon_0$ to be the Margulis constant.
For example the number ${\rm arcsinh}(1)$ appearing in \cite[Chapter IV 4]{hu} 
is the Margulis constant.)
\item
Suppose the Euler number of $\Sigma_{1,a} \setminus (\Sigma_{1,a}\cap \vec z_1) \setminus$ nodal points of $\Sigma_1$ 
is non-negative.
By stability, the map $u_1$ is non-constant on $\Sigma_{1,a}$.
We require that 
$$
d(u_1(w_{1,j},u(z')) \ge \epsilon_X
$$
for any nodal or marked point $z'$ of $\Sigma$.
Here $\epsilon_X$ is a positive number depending on $X$ and is 
sufficiently small so that the fact $u_1$ is non-constant implies 
the existence of such $w_{1,j}$.
\end{enumerate}
\end{enumerate}
\par
Choices of $\vec w_i$ satisfying (1)(2)(3)(4)(8)
are called {\it weak stabilization data}.
\index{weak stabilization data}
$\vec{\mathcal N}$ (resp. $\mathcal N_j$) is called local transversals (a local 
transversal).
\end{defn}
It is easy to see that stabilization data exist.
We consider a neighborhood of 
$(\Sigma_1,\vec z_1 \cup \vec w_1)$
in the Deligne-Mumford compactification 
$\mathcal M_{g,\ell+k}$ consisting of stable 
curves of genus $g$ with $\ell+k$ marked points.
We consider a $\mathcal G_1$ action on 
$\mathcal M_{g,\ell+k}$ as follows.
An element of $\mathcal M_{g,\ell+k}$ 
is represented by 
$(\Sigma',\vec z^{\,\prime} \cup \vec w^{\,\prime})$
where $\Sigma'$ is a genus $g$ nodal curve 
and $\vec z^{\,\prime}$ (resp. $\vec w^{\,\prime}$)
are $\ell$ (resp. $k$) marked points on it.
Using $\sigma_v$ in item (4) we define:
$$
v \cdot (\Sigma',\vec z^{\,\prime} \cup \vec w^{\,\prime})
= 
(\Sigma',\vec z^{\,\prime} \cup (w'_{\sigma^{-1}_v(1)},\dots,
w'_{\sigma_v^{-1}(k)})).
$$
Namely the action is defined by permutation of the 
marked points $\vec w^{\,\prime}$ by $\sigma_v$.
This is a left action.
\par
Note $[\Sigma_1,\vec z_1\cup \vec w_1] \in \mathcal M_{g,\ell+k}$
is a fixed point of this $\mathcal G_1$-action.
We also remark that Definition \ref{defn71}
(2) implies that $[\Sigma_1,\vec z_1\cup \vec w_1]$
is a smooth point of the orbifold $\mathcal M_{g,\ell+k}$.
\par
In a way similar to the map (\ref{loctrimap}) we take a local `trivialization'
of the universal family in a neighborhood of 
$[\Sigma_1,\vec z_1\cup \vec w_1]$.
For this purpose, we need to fix two types of data, that is, 
local trivialization (Definition \ref{defn7272})
and analytic families of  coordinates
(Definition \ref{defn7372}).
\par
We decompose $\Sigma_1$ into irreducible components:
\begin{equation}\label{decomp7272}
\Sigma_1 = \bigcup_{a \in \mathcal A_1} \Sigma_{1,a},
\end{equation}
where $\mathcal A_1$ is a certain index set. \index{00S5igma1a@$\Sigma_{1,a}$}
The smooth Riemann surface $\Sigma_{1,a}$ together with the marked or nodal 
points of $\Sigma_1$ on $\Sigma_{1,a}$ defines an element 
$$
[\Sigma_{1,a},\vec z_{1,a}] \in \mathcal M_{g_{1,a},\ell_{1,a}+k_{1,a}}.
$$
Here marked points are by definition elements of 
$\vec z_1\cup \vec w_1$.
$k_{1,a} = \# (\vec w_1\cap\Sigma_{1,a})$ and 
$\ell_{1,a}$ is $\# (\vec z_1\cap\Sigma_{1,a})$ plus the number of nodal 
points on $\Sigma_a$.
\begin{defn}\label{defn7272}
A {\it local trivialization} \index{local trivializationt}
at $(\Sigma_1,\vec z_1\cup \vec w_1)$
consists of $\mathcal V_{1,a}$ and $\phi_{1,a}$
with the following properties.
\begin{enumerate}
\item
$\mathcal V_{1,a}$ is a neighborhood of 
$(\Sigma_{1,a},\vec z_{1,a})$ in 
$\mathcal M_{g_{1,a},\ell_{1,a}+k_{1,a}}$.
\item
Let $\pi : \mathcal C_{g_{1,a},\ell_{1,a}+k_{1,a}} \to \mathcal M_{g_{1,a},\ell_{1,a}+k_{1,a}}$  
\index{00C3Gell@$\mathcal C_{g,\ell}$}
be the universal family. \index{00P4hi1a@$\phi_{1,a}$}
$\phi_{1,a}$ is a diffeomorphism
$\phi_{1,a} : \mathcal V_{1,a} \times \Sigma_{1,a} \to \mathcal C_{g_1,\ell_a+k_{1,a}}$
onto the open subset $\pi^{-1}(\mathcal V_{1,a})
\subset \mathcal C_{g_1,\ell_a+k_{1,a}}$
such that the next diagram commutes.
\begin{equation}\label{diagram51622}
\begin{CD}
\mathcal V_{1,a} \times \Sigma_{1,a} @ >{\phi_{1,a}}>> \mathcal C_{g_{1,a},\ell_{1,a}+k_{1,a}} \\
@ VVV @ VV{\pi}V\\
\mathcal V_{1,a} @ >{{\rm id}}>> \mathcal M_{g_{1,a},\ell_{1,a}+k_{1,a}}.
\end{CD}
\end{equation}
Here the left vertical arrow is the projection to the first factor.
\item
Let $v \in \mathcal G_1$. 
We define $v(a)$ by $v(\Sigma_{1,a}) = \Sigma_{1,v(a)}$.
We can identify $[\Sigma_{1,a},\vec z_{1,a}]$
and $[\Sigma_{1,v(a)},\vec z_{1,v(a)}]$
using bi-holomorphic map $v$. Then 
$\mathcal V_{1,a} = \mathcal V_{1,v(a)}$
and the next diagram commutes.
\begin{equation}\label{diagram51699}
\begin{CD}
\mathcal V_{1,a} \times \Sigma_{1,a} @ >{\phi_{1,a}}>> \mathcal C_{g_{1,a},
\ell_{1,a}+k_{1,a}} \\
@ VV{v}V @ VV{v}V\\
\mathcal V_{1,v(a)} \times \Sigma_{1,v(a)} @ >{\phi_{1,v(a)}}>> \mathcal C_{g_{1,a},
\ell_{1,a}+k_{1,a}}.
\end{CD}
\end{equation}
Here the left vertical arrow is defined by the identification \index{00V31a@$\mathcal V_{1,a}$}
$\mathcal V_{1,a} = \mathcal V_{1,v(a)}$ via 
the map $v : \Sigma_{1,a} \to \Sigma_{1,v(a)}$.
The right vertical arrow is defined by identifying the 
marked points on $\Sigma_{1,a}$ and ones on $\Sigma_{1,v(a)}$
by using the map $v$.\footnote{Note $a = v(a)$ may occur. In that case 
the map $\mathcal V_{1,a} \to \mathcal V_{1,v(a)} = \mathcal V_{1,a}$ 
is defined by the permutation of the enumeration of the marked points of 
$\Sigma_{1,a}$.}
\item
Let $\frak t_{1,a,j}$ $(j=1,\dots,\ell_{1,a}+k_{1,a})$ be the 
sections of $\pi : \mathcal C_{g_{1,a},\ell_{1,a}} \to \mathcal M_{g_{1,a},\ell_{1,a}
+k_{1,a}}$
assigning the $j$-th marked point.
Suppose $\frak t_{1,a,j}({\bf x})$ corresponds to a nodal point of $\Sigma_1({\bf x}) = \pi^{-1}({\bf x})$.
Then
\begin{equation}\label{form7575}
\phi_{1,a}^{-1}(\frak t_{1,a,j}({\bf x})) = ({\bf x},z_{1,a,j})
\end{equation}
for ${\bf x} \in \mathcal V_{1,a}$.
(In other words, the $\Sigma_{1,a}$ factor of left hand side does not 
move when we move ${\bf x}$.)
\end{enumerate}
\end{defn}

Note Conditions (2)(3)(4) are similar to the commutativity of Diagram (\ref{diagram51333}),
(Tri.1)+(Tri.4), (Tri.2) respectively.
\begin{rem}
We assume (\ref{form7575}) only for the marked points 
corresponding to the nodal point.
See Remark \ref{rem756}.
\end{rem}
\begin{defn}\label{defn7372}
{\it Analytic families of  coordinates} 
\index{analytic families of coordinates}
on $\prod_{a \in \mathcal A_1} \mathcal V_{1,a}$
assigns $\varphi_{1,a,j} : \mathcal V_{1,a} \times 
D^2(2) \to \mathcal C_{g_{1,a},\ell_{1,a}+k_{1,a}}$
for all $a$ and some $j$ with the following properties.
\begin{enumerate}
\item
The map $\varphi_{1,a,j}$ is defined if $z_{1,a,j}$ 
is a nodal point contained in $\Sigma_{1,a}$.
\item
The map $\varphi_{1,a,j}$ defines an analytic family of  coordinates 
at $\frak t_{1,a,j}$ in the sense of Definition \ref{cooddinateatinf}.
Here $\frak t_{1,a,j}$ is the holomorphic section of 
$\pi : \mathcal C_{g_{1,a},\ell_{1,a}+k_{1,a}} \to \mathcal M_{g_{1,a},\ell_{1,a}+k_{1,a}}$
assigning the $j$-th marked point.
\item
The  analytic family of  coordinates  $\varphi_{1,a,j}$  is 
compatible with the trivialization data. Namely the equality
$$
(\phi_{1,a})^{-1}(\varphi_{1.a,i}({\bf x},z)) = ({\bf x},\varphi_{1,a,i}(o,z)).
$$
holds, were $o \in \mathcal V_{1,a}$ corresponds to the point $\Sigma_{1,a}$.
\item
Let $v \in \mathcal G_1$ and $v(z_{1,a,i})
= z_{1,a',a'}$. Then
$$
v(\varphi_{1,a,i}({\bf x},z)) = \exp(\theta_{v,a,i} \sqrt{-1})  \varphi_{1,a',i'}({\bf x},z).
$$
Here $\theta_{v,a,i} \in \R$.
\end{enumerate}
\end{defn}
Note Conditions (3),(4) above are similar to (Tri.3) 
and (*) right above Lemma \ref{lem399}, respectively.
(As we mentioned in Remark \ref{rem31215} we only use analytic families of 
coordinates at the nodal points.)
\begin{defn}\label{defn720}
Let $((\Sigma_1,\vec z_1),u_1) \in U(((\Sigma,\vec z),u);\epsilon_2)$.
{\it Strong stabilization data}\index{strong stabilization data} ${\frak W}^{(1)}$ 
\index{00W4frak@${\frak W}$} at $((\Sigma_1,\vec z_1),u_1)$
are by definition the  choices of the following data. \index{00W4.i@(${\frak W}$.i)}
\begin{enumerate}
\item[(${\frak W}$.1)]
Stabilization data $\vec w_1$, $\vec{\mathcal N}_{1} = \{\mathcal N^{(1)}_{j}\mid j=1,\dots,k_i\}$
of $((\Sigma_1,\vec z_1),u_1)$.
(Definition \ref{defn71}.)
\item[(${\frak W}$.2)]
A local trivialization (Definition \ref{defn7272})
at $(\Sigma_1,\vec z_1\cup \vec w_1)$.
\item[(${\frak W}$.3)]
Analytic families of coordinates of $((\Sigma_1,\vec z_1),u_1)$.
(Definition \ref{defn7372}).
\end{enumerate}
We denote the totality of those data by $({\frak W}^{(1)},\vec{\mathcal N}_1)$. \index{00W4N00@$({\frak W},\vec{\mathcal N})$}
\par
{\it Stabilization and trivialization data}\index{stabilization and trivialization data}
\index{00W400@${\frak W}$} ${\frak W}^{(1)}$ at $((\Sigma_1,\vec z_1),u_1)$ 
are weak stabilization data $\vec w_1$
together with (${\frak W}$.2) and (${\frak W}$.3).
\end{defn}
Suppose stabilization and trivialization data ${\frak W}$ are given. 
\begin{defn}
We put \index{00V310@$\mathcal V_{1,0}$}
$$
\mathcal V_{1,1} = \bigoplus_{{\rm e}\in \Gamma(\Sigma_1)}
\C_{-,{\rm e}} \otimes \C_{+,{\rm e}}.
$$
as in Definition \ref{defn313} 
and
$$
\mathcal V_{(1)} = \mathcal V_{1,0} \times \mathcal V_{1,1}
$$
with 
$$
\mathcal V_{1,0} = 
\prod_{a \in \mathcal A_1}\mathcal V_{1,a}.
$$
\index{00V311@$\mathcal V_{1,1}$}\index{00V31@$\mathcal V_{(1)}$}
We carry out the same construction as Construction \ref{cost314}
and obtain
\begin{equation}\label{defmathcalC2}
\mathcal C_{(1)} = \bigcup_{\vec x \in \mathcal V_{1,0},\vec\rho \in 
\mathcal V_{1,1}} \Sigma_1(\vec x,\vec\rho) \times \{(\vec x,\vec\rho)\}.
\end{equation}
We thus obtain \index{00C31ca@$\mathcal C_{(1)}$} a family of nodal curves:
\begin{equation}\label{familyon1}
\mathcal C_{(1)} \to \mathcal V_{(1)}
\end{equation}
together with sections $\frak t_j$ ($j=1,\dots,k+\ell$).
They consist a local universal family over 
$\mathcal V_{(1)}$, which is an open neighborhood of 
$[\Sigma',\vec z^{\,\prime} \cup \vec w^{\,\prime}]
\in \mathcal M_{g,k+\ell}$.
\par
Hereafter we write $\vec z_{1}({\bf x})
= (\frak t_1({\bf x}),\dots,\frak t_{\ell}({\bf x}))$
\index{00Z21x@$\vec z_{1}({\bf x})$} \index{00W21x@$\vec w_{1}({\bf x})$}
and $\vec w_{1}({\bf x})
= (\frak t_{\ell+1}({\bf x}),\dots,\frak t_{k+\ell}({\bf x}))$.
\par
Moreover (\ref{familyon1}) is acted by 
$\mathcal G_1$
such that
$$
z_{1,j}(v {\bf x}) = v(z_{1,j}({\bf x})),
\qquad
w_{1,\sigma_v(j)}(v {\bf x}) = v(w_{1,j}({\bf x})).
$$
\par
We define $\Sigma_1(\delta) \subset \Sigma_1$ \index{00S5igma1delta@$\Sigma_1(\delta)$}
in the same way as  (\ref{formula46}).
We define $\mathcal V_{1,1}(\delta)$
in the same way as (\ref{form45555}).
For ${\bf x} \in \mathcal V_{1,0} \times \mathcal V_{1,1}(\delta)$
we define $\Sigma_1({\bf x})$ in the same way as (\ref{sigmaxxx}).
We also define \index{00P5HI@$\Phi_{1,{\bf x},\delta}$} an open embedding
\begin{equation}\label{loctrimap3}
\Phi_{1,{\bf x},\delta} : \Sigma_1(\delta) \to \Sigma_1({\bf x})
\end{equation}
in the same way as (\ref{loctrimap}).
\end{defn}
\begin{rem}
Since $\mathcal V_{(1)}$, $\mathcal C_{(1)}$, 
$\Phi_{1,{\bf x},\delta}$ and ${\frak W}^{(1)}, \vec{\mathcal N}_1$ are objects related to 
$\Sigma_1$ we put suffix $(1)$ or $1$ to them.
In case when $\Sigma_2$ etc. appears 
(in Subsections \ref{indsmoothchart}, \ref{subsec:Cinfinity}) we write
$\mathcal V_{(2)}$, $\mathcal C_{(2)}$, 
$\Phi_{2,{\bf x},\delta}$, etc..
We also write its strong stabilization data (resp. stabilization and trivialization data)
by $({\frak W}^{(2)},\vec{\mathcal N}_{2})$ (resp. ${\frak W}^{(2)}$).
\end{rem}
We remark that we use Definition \ref{defn7272} (4) 
(which is assumed for the nodal point) and 
analytic family of 
coordinates at the nodal 
points to define (\ref{loctrimap3}).
(A similar data for marked points are not used.)
\par
Now in a similar way as Definition \ref{defn41}
we define as follows.

\begin{defn}\label{defn412}
{\rm (See \cite[Definition 17.12]{foootech}.)}
We consider a triple $((\Sigma',\vec z^{\,\prime}),u')$ where 
$(\Sigma',\vec z^{\,\prime})$ is a nodal curve of genus $g$ with $\ell$ marked points, 
$u' : \Sigma' \to X$ is a smooth map.
(Namely $u'$ is a continuous map which is smooth on each stratum.)
\par
We say that $((\Sigma',\vec z^{\,\prime}),u')$ is {\it $\epsilon$-close} to
$((\Sigma_1,\vec z_1),u_1)$ 
with respect to the given stabilization and trivialization data ${\frak W}$,
if the following holds.
\par
There exist $\vec w^{\,\prime}$ and
$\delta > 0$, 
${\bf x} = ((\rho_{\rm e})_{{\rm e} \in \Gamma(\Sigma)}),(x_a)_{a \in\mathcal A_s})
\in \mathcal V_{1,0}(\delta) 
\times \mathcal V_{1,1}$,
and a bi-holomorphic map $\phi : (\Sigma_1({\bf x}),\vec z_1({\bf x})
\cup \vec w_1({\bf x}))
\cong (\Sigma',\vec z^{\,\prime}\cup \vec w^{\,\prime})$ with 
the following properties.
\begin{enumerate}
\item The $C^{2}$ norm of the difference between
$u' \circ \phi \circ \Phi_{1,\bf x,\delta}$ and $u_1\vert_{\Sigma_1(\delta)}$
is smaller than $\epsilon$.
\item The distance between $\bf x$ and $[\Sigma_1,\vec z_1 \cup \vec w_1]$
in $\mathcal M_{g,k+\ell}$
is smaller than $\epsilon$. Moreover $\delta < \epsilon$.
\item
The map $u' \circ \phi$ has diameter $< \epsilon$ on $\Sigma_1({\bf x}) \setminus 
{\rm Im}(\Phi_{1,\bf x,\delta})$.
\footnote{See Definition \ref{defn410} for this terminology.}
\end{enumerate}
See Figure \ref{Figure5-1}.
\par
In case we specify $\delta$ we say 
$((\Sigma',\vec z^{\,\prime}),u')$ is {\it $\epsilon$-close} to
$((\Sigma_1,\vec z_1),u_1)$ with respect to $\delta$.
\begin{figure}[h]
\centering
\includegraphics[scale=0.4]{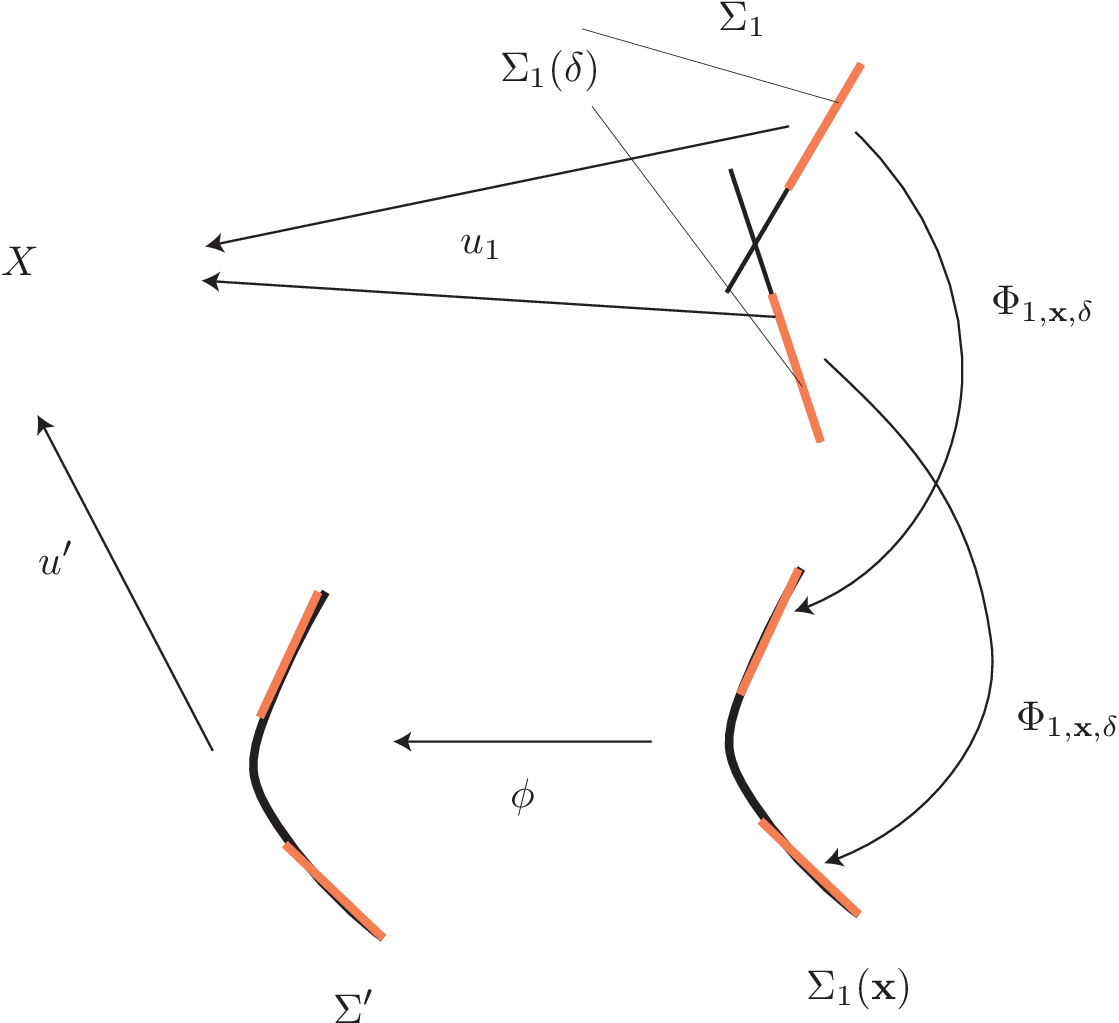}
\caption{$((\Sigma',\vec z^{\,\prime}),u')$ is $\epsilon$-close 
to
$((\Sigma_1,\vec z_1),u_1)$}
\label{Figure5-1}
\end{figure}
\end{defn}
This is the definition we used in \cite{foootech,foooconstr}.
(We remark that this definition and Definition \ref{defn41}
are similar to the definition of stable map topology
introduced in \cite{FO}.)
\par
The next lemma is sometimes useful to check Definition \ref{defn412}
Condition  (3).
\begin{lem}\label{lem7777}
Suppose $\delta > \delta'>0$ and 
$\phi : (\Sigma_1({\bf x}),\vec z_1({\bf x})
\cup \vec w_1({\bf x}))
\cong (\Sigma',\vec z^{\,\prime}\cup \vec w^{\,\prime})$
is an isomorphism with 
${\bf x} = ((\rho_{\rm e})_{{\rm e} \in \Gamma(\Sigma)}),(x_a)_{a \in\mathcal A_s})
\in (\mathcal V_{1,0}(\delta'/2))
\times \mathcal V_{1,1}$.
\par
We assume Conditions (1)(2)(3) are satisfied for $\delta$ and $\epsilon$
and the map $u' \circ \phi$ is holomorphic outside the image of $\Phi_{1,\bf x,\delta}$.
\par
Then Conditions (1)(2)(3) are satisfied for $\delta'$ and $o(\epsilon)$.
\end{lem}
\begin{rem}
Here and hereafter $o(\epsilon)$ \index{00O2epsu@$o(\epsilon)$}
is a positive number depending on $\epsilon$ and 
such that $\lim_{\epsilon\to 0}o(\epsilon) = 0$.
\end{rem}
\begin{proof}
The conditions (2)(3) for $\delta'$ are obvious.
The $C^{2}$ norm of the difference between
$u' \circ \phi \circ \Phi_{1,\bf x,\delta}$ and $u_1\vert_{\Sigma_1(\delta)}$
is smaller than $\epsilon$ by assumption.
By (1)(3) for $\delta$, 
the map $u_1$ has diameter $< o(\epsilon)$ on $\partial\Sigma_1(\delta)$.
Since $u'\circ \phi$ is holomorphic on $\Sigma_1 \setminus \Sigma_1(\delta)$
it implies that
the map $u_1$ has diameter $< o(\epsilon)$ on
$\Sigma_1 \setminus \Sigma_1(\delta)$.
Therefore by (1)(3) again, 
$C^0$ distance between 
$u' \circ \phi \circ \Phi_{1,{\bf x},\delta'/2}$ and $u_1\vert_{\Sigma_1(\delta'/2)}$ 
is $o(\epsilon)$.
Since $u' \circ \phi \circ \Phi_{1,{\bf x},\delta'/2}$ and $u_1$ are 
both holomorphic on $\Sigma_1(\delta'/2) \setminus \Sigma_1(\delta)$
we can estimate $C^2$ distance between them on $\Sigma_1(\delta') \setminus \Sigma_1(\delta)$
by the $C^0$  distance between them on 
$u' \circ \phi \circ \Phi_{1,{\bf x},\delta'/2}$ and $u_1\vert_{\Sigma_1(\delta'/2)}$.
\end{proof}
\begin{defn}\label{defnUUUU}
Let $((\Sigma_1,\vec z_1),u_1) \in U(((\Sigma,\vec z),u);\epsilon_2)$.
Suppose we are given 
its  stabilization and trivialization data ${\frak W}$.
\par
We denote by $\mathcal U(\epsilon;(\Sigma_1,\vec z_1),u_1,{\frak W})$
\index{00U3epsigma1@$\mathcal U(\epsilon;(\Sigma_1,\vec z_1),u_1,{\frak W})$}
the set of the isomorphism classes of elements 
of $U(((\Sigma,\vec z),u);\epsilon_2)$ which is 
$\epsilon$-close to $((\Sigma_1,\vec z_1),u_1)$ with respect to ${\frak W}$.
\end{defn}
We will show that the set  $\mathcal U(\epsilon;(\Sigma_1,\vec z_1),u_1,{\frak W})$ has a structure of  
a smooth orbifold.
(We actually show that this is a quotient 
of a smooth manifold by the action of the group 
$\mathcal G_1$.)
\par
The proof is by gluing analysis.
To carry out gluing analysis, we 
study how the obstruction bundle 
$E((\Sigma',\vec z^{\,\prime}),u')$
behaves when we move 
$((\Sigma',\vec z^{\,\prime}),u')$.
We first take an appropriate 
parametrization of the set of the triples 
$((\Sigma',\vec z^{\,\prime}),u')$
which is $\epsilon$-close to 
$((\Sigma_1,\vec z_1),u_1)$.
\par
We first observe the following.
\begin{lem}\label{lem76}
The vector space $E((\Sigma',\vec z^{\,\prime}),u')$
depends only on 
$(\Sigma',\vec z^{\,\prime})$ and the restriction of 
$u'$ to the image of $\phi\circ \Phi_{1,\bf x,\delta} : 
\Sigma_1(\delta) \to \Sigma_1({\bf x}) \to \Sigma'$, 
if $\delta$ is small.
\end{lem}
\begin{proof}
We remark that the support of elements of $E((\Sigma',\vec z^{\,\prime}),u')$
is in the image of $\phi\circ \Phi_{1,\bf x,\delta}$
by construction.
\par
Moreover
for $(\varphi,g) \in \mathcal W(\epsilon_1;{\bf x}_0,\phi_0;((\Sigma',\vec z^{\,\prime}),u'))$ the value 
${\rm meandist}(\varphi,g)$ of the function 
${\rm meandist}$ does not change when we change $u'$
outside the image of $\phi\circ \Phi_{1,\bf x,\delta}$. 
This is an immediate consequence of 
its definition (\ref{form611}).
\par
The lemma follows from these two facts.
\end{proof}
Suppose $((\Sigma_1,\vec z_1),u_1)$ 
is $G$-$\epsilon_2$-close to 
$((\Sigma,\vec z),u)$ by
$g_1$, ${\bf x}_1$, $\phi_1$.
For simplicity of notation we identify 
$\Sigma({\bf x}_1)$ with $\Sigma_1$ by $\phi_1$ 
and regard 
$\Sigma_1 = \Sigma({\bf x}_1)$.
\par
Let ${\bf x} = (\vec x,\vec \rho) \in 
\mathcal V_{(1)} = \mathcal V_{1,0} \times \mathcal V_{1,1}(\delta)$
where 
$(x_a)_{a \in\mathcal A_s}
\in \mathcal V_{1,0}$
and $\vec \rho = (\rho_{\rm e})_{{\rm e} \in \Gamma(\Sigma_1)}
\in \mathcal V_{1,1}(\delta)$.
\par
We consider 
$(\Sigma_1({\bf x}),\vec z_1({\bf x})\cup \vec w_1({\bf x}))$.
\par
Hereafter we denote by $\Sigma_1({\bf x})(\delta)$ 
\index{00S5igma1xdelat@$\Sigma_1({\bf x})(\delta)$} the image of 
the map $\phi\circ \Phi_{1,\bf x,\delta} : 
\Sigma_1(\delta) \to \Sigma_1({\bf x}) \to \Sigma'$.
Let $\hat u' : \Sigma_1(\delta) \to X$ be an $L^{2}_{m+1}$
map which is close to $u_1$ in $L^{2}_{m+1}$ norm.
We consider
\begin{equation}\label{defnofuprime}
u' = \hat u' 
\circ \Phi_{1,{\bf x},\delta}^{-1} : \Sigma_1({\bf x})(\delta) \to X.
\end{equation}
By Lemma \ref{lem76} the subspace
$$
E((\Sigma_1({\bf x}),\vec z_1({\bf x})),u')
\subset
L^2_{m+1}(\Sigma_1({\bf x})(\delta);(u')^*TX \otimes \Lambda^{01})
$$
is well-defined.
(Here we use the fact that $u' \circ \Phi_{1,{\bf x},\delta}$ 
is $C^2$ close to $u_1$. Note $m$ in $L^2_{m+1}$ is chosen sufficiently 
large. So $L^2_{m+1} \subset C^2$ in particular.)
\par
By assumption
$$
d_X(u'(\Phi_{1,{\bf x},\delta}(z)),u_1(z))
$$
is small. Therefore we can use 
parallel transportation (with respect to a $G$-invariant hermitian connection of $TX$) 
along the minimal geodesic 
joining 
$u'(\Phi_{1,{\bf x},\delta}(z)))$ to $u_1(z)$
to obtain
\begin{equation}\label{form77}
{\rm Pal} :
L^2_{m+1}(\Sigma_1({\bf x})(\delta);(u')^*TX)
\to 
L^2_{m+1}(\Sigma_1(\delta);u_1^*TX).
\end{equation}
Moreover using the diffeomorphism 
$\Phi_{1,{\bf x},\delta}$ we obtain a map
\begin{equation}\label{formula66rev}
d^h\Phi_{1,{\bf x},\delta} : \Lambda_{\Phi_{1,{\bf x},\delta}(z)}^{01} 
\Sigma_1({\bf x}) \to \Lambda_z^{01}\Sigma_1.
\end{equation}
in the same way as (\ref{formula66}).
\par
Using (\ref{form77}) and (\ref{formula66rev}) we obtain:
\index{00I2hatuprime@$I_{\hat u',{\bf x}}$}
\begin{equation}\label{form711}
I_{\hat u',{\bf x}} :
L^2_{m+1}(\Sigma_1({\bf x})(\delta);(u')^*TX\otimes \Lambda^{01})
\to 
L^2_{m+1}(\Sigma_1(\delta);u_1^*TX\otimes \Lambda^{01}).
\end{equation}
\begin{defn}\label{defn}
We define \index{00E2hatu@$E(\hat u',{\bf x})$}
$$
E(\hat u',{\bf x})
=
I_{\hat u',{\bf x}}(E((\Sigma_1({\bf x}),\vec z_1({\bf x})),u'))
\subset
L^2_{m+1}(\Sigma_1(\delta);u_1^*TX\otimes \Lambda^{01}).
$$
\end{defn}
Note $E(\hat u',{\bf x})$ is a finite dimensional subspace of 
the Hilbert space 
$L^2_{m+1}(\Sigma_1(\delta);u_1^*TX\otimes \Lambda^{01})$,
which is independent of $(\hat u',{\bf x})$.
So we can discuss $(\hat u',{\bf x})$ dependence of 
$E(\hat u',{\bf x})$.
\par
Now the main new point we need to check to work out the 
gluing analysis is the following.
We put 
$$
d = \dim E((\Sigma',\vec z^{\,\prime}),u').
$$
\begin{prop}\label{prop78}
Let $U'(\epsilon)$ \index{00U2primeep@$U'(\epsilon)$} be an $\epsilon$ neighborhood of 
$u_1$ in $L^{2}_{m+1}$ norm 
and $\mathcal V_{(1)}(\epsilon)$ \index{00V31epsu@$\mathcal V_{(1)}(\epsilon)$} an $\epsilon$
neighborhood of $0$ in $\mathcal V_{(1)}$.
\par
There exists $d$ smooth maps $e_i(\hat u',{\bf x})$
from $U'(\epsilon) \times \mathcal V_{(1)}(\epsilon)$
to 
$L^2_{m+1}(\Sigma_1(\delta);u_1^*TX\otimes \Lambda^{01})$ 
such that
for each $\hat u',{\bf x}$ 
$$
(e_1(\hat u',{\bf x}),\dots,e_d(\hat u',{\bf x}))
$$
is a basis of $E(\hat u',{\bf x})$.
\par
Moreover the $C^n$ norm of the map $e_i$ is uniformly bounded on 
$U'(\epsilon) \times \mathcal V_{(1)}(\epsilon)$
for any $n$.
\end{prop}
See Figure \ref{Figure737} in Subsection \ref{mainprop}.
\par
We prove Proposition \ref{prop78} in Subsection \ref{mainprop}.
\par
To clarify the fact that Proposition \ref{prop78}
gives the control of the behavior of $E(\hat u',{\bf x})$
needed for the gluing analysis detailed in \cite{foooexp} to work, 
we change variables and restate Proposition \ref{prop78}
below.
We took ${\bf x} = (\vec x,\vec \rho) \in 
\mathcal V_{(1)} = \mathcal V_{1,0} \times \mathcal V_{1,1}$
where
$\vec x = (x_a)_{a \in\mathcal A_s}
\in \mathcal V_{1,0}$ and 
$\vec \rho = (\rho_{\rm e})_{{\rm e} \in \Gamma(\Sigma_1)}
\in \mathcal V_{1,1}(\delta)$.
We define $T_{\rm e}$ and $\theta_{\rm e}$
for each ${\rm e} \in \Gamma(\Sigma_1)$ 
by the following formula:
\begin{equation}\label{71371111}
\aligned
\exp(-10 \pi T_{\rm e}) &= \vert \rho_{\rm e}\vert \\
\exp(2\pi \theta_{\rm e}\sqrt{-1}) &= \rho_{\rm e}/\vert \rho_{\rm e}\vert.
\endaligned
\end{equation}
Compare \cite[(8.18)]{foooexp}.
We thus identify
$$
\mathcal V_{0,1}(\delta)
\cong
\prod_{{\rm e} \in \Gamma(\Sigma_1)}
(-\log \delta/10\pi,\infty] \times \R/\Z,
$$
using $T_{\rm e}$, $\theta_{\rm e}$ as coordinates.
Now we rewrite the smoothness of the map 
$e_i(\hat u',{\bf x})$
as follows.
\begin{cor}\label{cor79}
We have an inequality
\begin{equation}\label{estimate715}
\aligned
&\left\Vert
\frac{\partial}{\partial T_{{\rm e}_1}}
\dots \frac{\partial}{\partial T_{{\rm e}_{n_1}}}
\frac{\partial}{\partial \theta_{{\rm e}'_1}}
\dots \frac{\partial}{\partial \theta_{{\rm e}'_{n_2}}}
e_i
\right\Vert_{C^n} \\
&\le
C_{n,n_1,n_2} \exp\left(-\delta_{n,n_1,n_2}
\left(\sum_{i=1}^{n_1} T_{{\rm e}_i} + \sum_{i=1}^{n_2} T_{{\rm e}'_i}\right)\right).
\endaligned
\end{equation}
Here $C^n$ in the left hand side is the $C^n$ norm as 
a map from $U'(\epsilon) \times \mathcal V_{1,1}$
to $L^2_{m+1}(\Sigma_1(\delta);u_1^*TX\otimes \Lambda^{01})$.
In other words, we fix $\vec{\rho}$ (or $T_{\rm e}, \theta_{\rm e}$) 
and regard $\hat u'$ and $\vec x$ as variables to define $C^n$ norm.
\end{cor}
It is easy to see that the exponential factor 
in the right hand side appears by the 
change of variables from 
$\rho_{\rm e}$ to $(T_{\rm e},\theta_{\rm e})$.
So Corollary \ref{cor79} is an immediate consequence of 
Proposition \ref{prop78}.
\par
Corollary \ref{cor79} corresponds to \cite[Proposition 8.19]{foooexp}.
This is all the properties we need for the proof of 
\cite{foooexp} to work in the case obstruction bundle is  given as
$E(\hat u',{\bf x})$.
Thus by \cite{foooexp} we obtain the next two Propositions 
\ref{prop710} and \ref{prop711}.
We need to introduce some notations to state them.

We define a linear differential operator
\begin{equation}\label{operator410rev}
D_{u_1}\overline{\partial} 
:
W^2_{m+1}(\Sigma_1;u_1^*TX)
\to L^2_{m}(\Sigma_1;u_1^*TX\otimes \Lambda^{01}),
\end{equation}
in the same way as (\ref{operator410}).
\par
Condition \ref{conds45} (1)  implies that
we can choose $\epsilon_2$ small such that 
$$
{\rm Im}(D_{u_1}\overline{\partial})
+ E((\Sigma_1,\vec z_1),u_1)
=
L^2_{m}(\Sigma_1;u_1^*TX\otimes \Lambda^{01}),
$$
if $((\Sigma_1,\vec z_1),u_1)$ is $\epsilon_2$-close to $((\Sigma,\vec z),u)$.
\par
In the same way as we did in Condition \ref{conds45} (2),
we put
\begin{equation}\label{form715}
{\rm Ker}^+ D_{u_1}\overline{\partial}
=
\{
v \in W^2_{m+1}(\Sigma;u_1^*TX)
\mid  D_{u_1}\overline{\partial}(v)
\in E((\Sigma_1,\vec z_1),u_1).
\}
\end{equation}
This is a finite dimensional space consisting of 
smooth sections. \index{00K1erD@${\rm Ker}^+ D_{u_1}\overline{\partial}$}
This space is $\mathcal G_1$ invariant.
\par
Let  $\mathcal V_{\rm map}(\epsilon)$ be the \index{00V3mapepsilon@$\mathcal V_{\rm map}(\epsilon)$}
$\epsilon$ neighborhood of $0$ in 
${\rm Ker}^+ D_{u_1}\overline{\partial}$
and $\mathcal V_{(1)}(\epsilon)$ the $\epsilon$ 
neighborhood of $o$ in $\mathcal V_{(1)}$.
\begin{prop}\label{prop710}
For sufficiently small $\epsilon$ there exists a 
family of maps \index{00U2vx@$u_{{\bf v},{\bf x}}$}
$$
u_{{\bf v},{\bf x}} : \Sigma_1({\bf x}) \to X
$$
parametrized by
$$
({\bf v},{\bf x}) \in \mathcal V_{\rm map}(\epsilon)
\times  \mathcal V_{(1)}(\epsilon)
$$
with the following properties.
\begin{enumerate}
\item The equation
$$
\overline{\partial} u_{{\bf v},{\bf x}} 
\in E((\Sigma_1({\bf x}),\vec z_1({\bf x})),u_{{\bf v},{\bf x}})
$$
is satisfied. Moreover for each connected component of $\Sigma_1({\bf x}) \setminus 
{\rm Im}(\Phi_{1,\bf x,\delta})$ 
the diameter of its image by $u_{{\bf v},{\bf x}}$ is smaller than $\epsilon$.
\item
There exists $\epsilon' > 0$ such that if $((\Sigma',\vec z^{\,\prime} \cup \vec w^{\,\prime}),u')$ 
satisfies the next four conditions 
(a)(b)(c)(d) then 
then there exists ${\bf v} \in \mathcal V_{\rm map}(\epsilon)$
such that
$$
u' \circ \phi = u_{{\bf v},{\bf x}}.
$$
\begin{enumerate}
\item
$
\overline{\partial} u' 
\in E((\Sigma',\vec z^{\,\prime}),u')
$:
\item
$[\Sigma',\vec z^{\,\prime} \cup \vec w^{\,\prime}] \in  \mathcal V_{(1)}(\epsilon)$:
\item
Let $(\Sigma',\vec z^{\,\prime} \cup \vec w^{\,\prime})
\cong (\Sigma_1({\bf x}),\vec z_1({\bf x}) \cup \vec w_1({\bf x}))$
and $\phi$ is the isomorphism. We assume
that the $C^2$ norm between
$u' \circ \phi \circ \Phi_{1,\bf x,\delta}$ and $u_1\vert_{\Sigma_1(\delta)}$
is smaller than $\epsilon'$:
\item
The map $u'$ has diameter $< \epsilon$ on $\Sigma' \setminus 
{\rm Im}(\phi \circ\Phi_{1,\bf x,\delta})$.
\end{enumerate}
\item
If ${\bf v} \ne {\bf v}'$ then 
$u_{{\bf v},{\bf x}} \ne u_{{\bf v}',{\bf x}}$ for any ${\bf x} 
\in \mathcal V_{(1)}(\epsilon)$.
\item If ${\bf v} = {\bf 0}$ and ${\bf x} = o$ (the point corresponding to 
$\Sigma_1$) then
$u_{{\bf 0},o} = u_1$.
\end{enumerate}
The map $({\bf v},{\bf x}) \mapsto u_{{\bf v},{\bf x}}$
is $\mathcal G_1$-equivariant.
\end{prop}
\begin{proof}
The construction of the family of maps $u_{{\bf v},{\bf x}}$ 
satisfying Items (1), (4) above is by alternating method 
we detailed in \cite[Sections 4 and 5]{foooexp}.
(We use the estimate Corollary \ref{cor79} for the proof
in \cite{foooexp}.)
\par
(2)(3) are surjectivity and injectivity
of the gluing map, respectively, which are proved in 
\cite[Section 7]{foooexp}.
\end{proof}
To state the next proposition we need a notation.
We consider the family of maps $u_{{\bf v},{\bf x}}$
in Proposition \ref{prop710}.
We consider the smooth open embedding
$$
\Phi_{1,{\bf x},\delta} : \Sigma_1(\delta) \to \Sigma_1({\bf x})
$$
defined in (\ref{loctrimap3}).
We denote the composition by \index{00R1es@${\rm Res}$}
\begin{equation}
{\rm Res}(u_{{\bf v},{\bf x}})
= u_{{\bf v},{\bf x}} \circ \Phi_{1,{\bf x},\delta}
:
\Sigma_1(\delta) \to X.
\end{equation}
We remark that the domain and the target of the map 
${\rm Res}(u_{{\bf v},{\bf x}})$ is independent of 
${\bf v},{\bf x}$. So we regard 
${\bf v},{\bf x} \mapsto {\rm Res}(u_{{\bf v},{\bf x}})$ as a map:
$$
\mathcal V_{\rm map}(\epsilon)
\times   \mathcal V_{(1)}(\epsilon)
\to 
L^2_{m+1}(\Sigma_1(\delta),X).
$$
Here $L^2_{m+1}(\Sigma_1(\delta),X)$ is the Hilbert manifold of 
the maps of $L^2_{m+1}$ classes.
\begin{prop}\label{prop711}
For each $n$, $m>n+10$ the map
$$
{\bf v},{\bf x} \mapsto {\rm Res}(u_{{\bf v},{\bf x}})
$$
is of $C^n$ class as a map
$$
\mathcal V_{\rm map}(\epsilon)
\times  \mathcal V_{(1)}(\epsilon)
\to 
L^2_{m+1-n}(\Sigma_1(\delta),X).
$$
Moreover for $n_1 + n_2 \le n$, $n' \le n$, we have the next estimate
\begin{equation}\label{718777}
\aligned
&\left\Vert
\frac{\partial}{\partial T_{{\rm e}_1}}
\dots \frac{\partial}{\partial T_{{\rm e}_{n_1}}}
\frac{\partial}{\partial \theta_{{\rm e}'_1}}
\dots \frac{\partial}{\partial \theta_{{\rm e}'_{n_2}}}
u_{{\bf v},{\bf x}}
\right\Vert_{C^{n'}} \\
&\le
C_{n',n_1,n_2} \exp\left(-\delta'_{n',n_1,n_2}
\left(\sum_{i=1}^{n_1} T_{{\rm e}_i} + \sum_{i=1}^{n_2} T_{{\rm e}'_i}\right)\right).
\endaligned
\end{equation}
Here the $C^{n'}$ norm in the left hand side is defined as follows.
We regard 
the $T$ and $\theta$ differential of $u_{{\bf v},{\bf x}}$
as a map 
$$
\mathcal V_{\rm map}(\epsilon) \times \mathcal V_{(1)}(\epsilon)
\to
L^2_{m+1-n}(\Sigma_1(\delta),X).
$$
Then $\Vert \Vert_{C^{n'}}$ is the $C^{n'}$ norm the map 
in the left hand side for a fixed $T_{\rm e}$ and $\theta_{\rm e}$.
\end{prop}
\begin{proof}
This is \cite[Theorem 6.4]{foooexp}.
\end{proof}
\begin{rem}
We remark that the number $\delta'_{n',n_1,n_2}$ appearing 
in Proposition \ref{prop711} is different from 
$\delta_{n',n_1,n_2}$ in (\ref{estimate715}).
Actually (\ref{718777}) does {\it not} imply 
the smooth-ness of $u_{{\bf v},{\bf x}}$ with respect to 
$\rho_{\rm e}$ in (\ref{71371111}).
See \cite[Remark A1.63]{fooobook2}.
This is the reason why we will change the smooth structure 
in Definition \ref{defn712}.
\end{rem}

\subsection{Construction of the smooth chart 2: Construction of 
smooth chart at one point of $U(((\Sigma,\vec z),u);\epsilon_2)$}
\label{smoothchart15}

We now use Propositions \ref{prop710} and \ref{prop711}
to construct a smooth structure at each point of 
$U(((\Sigma,\vec z),u);\epsilon_2)$.
Let $((\Sigma_1,\vec z_1),u_1) \in U(((\Sigma,\vec z),u);\epsilon_2)$.
Let ${\frak W}$  be its  stabilization and trivialization data.
\par
We obtain a map
$$
{\bf v},{\bf x} \mapsto {\rm Res}(u_{{\bf v},{\bf x}})\,\,:\,
\mathcal V_{\rm map}(\epsilon)
\times  \mathcal V_{(1)}(\epsilon)
\to 
L^2_{m+1}(\Sigma_1(\delta),X).
$$
We define a smooth structure on 
$\mathcal V_{(1)}(\epsilon)$ as follows.
Note $(T_{\rm e},\theta_{\rm e})_{{\rm e} \in \Gamma(\Sigma_1)}$
is the coordinate of $\mathcal V_{(1)}(\epsilon)$,
where $T_{\rm e} \in (\log \delta/10,\infty] \times \R/\Z$.
We put 
\begin{equation}
\frak s_{\rm e} = \frac{e^{2\pi \theta_{\rm e} \sqrt{-1}}}{T_{\rm e}}
\in \C.
\end{equation}
\begin{defn}\label{defn712}
We define a $C^{\infty}$ structure on $\mathcal V_{1,1}$
such that $(\frak s_{\rm e})_{{\rm e} \in \Gamma(\Sigma_1)}$ is a smooth coordiante.
\end{defn}
We put a standard $C^{\infty}$ structure on $\mathcal V_{\rm map}(\epsilon)$ 
and $\mathcal V_{1,0}$.
Note $\mathcal V_{\rm map}(\epsilon)$ is an open subset of a finite dimensional 
vector space and  $\mathcal V_{1,0}$ is a product of open neighborhoods of smooth 
points of the moduli spaces of marked curves (without node). So they have 
canonical smooth structure. Thus the smooth structure 
of $\mathcal V_{(1)}$ and of its open subset $\mathcal V_{(1)}(\epsilon)$ is defined. 
\begin{defn}
We define the {\it evaluation map}\index{evaluation map}\index{00E1Vwij@${\rm EV}_{w_{1,j}}$}
$$
({\rm EV}_{w_{1,j}})_{j=1,\dots,k} : 
\mathcal V_{\rm map}(\epsilon)
\times \mathcal V_{(1)}(\epsilon)
\to 
X^{k}.
$$
by
$$
{\rm EV}_{w_{1,j}}({\bf v},{\bf x})
= u_{{\bf v},{\bf x}}(w_{1,j}({\bf x})).
$$
(We remark that $(w_{1,j}({\bf x}) = \frak t_{\ell+j}({\bf x})
\in \Sigma_1({\bf x})$.)
\end{defn}
\begin{lem}\label{lem713}
If $\epsilon$ is sufficiently small then 
$({\rm EV}_{w_{1,j}})_{j=1,\dots,k}$
is transversal to 
$\prod_{j}\mathcal N_j$.
\end{lem}
\begin{proof}
Proposition \ref{prop711} implies that the map
$({\rm EV}_{w_{1,j}})_{j=1,\dots,k}$ is of $C^{n}$ class
for any fixed $n$ with respect to the smooth structure 
in Definition \ref{defn712}, if $\epsilon$ is sufficiently 
small.
(We work using $L^2_{m+1}$ spaces with $m$ sufficiently large 
compared to $n$.) 
In fact
$$
{\rm EV}_{w_{1,j}}({\bf v},{\bf x})
= {\rm Res}(u_{{\bf v},{\bf x}})(\Phi^{-1}_{1,{\bf x},\delta}(w_{1,j}({\bf x})))
$$
and ${\bf x} \mapsto \Phi^{-1}_{1,{\bf x},\delta}(w_{1,j}({\bf x}))$
is a smooth map $: \mathcal V_{(1)}(\epsilon) \to 
\Sigma_1(\delta)$.
\par
Therefore it suffices to show the lemma at origin
(which corresponds to $((\Sigma_1,\vec z_1\cup \vec w_1),u_1)$).
We consider the submanifold $Y$ of 
$\mathcal V_{(1)}(\epsilon)$ which consists of  
elements $(\Sigma_1,\vec z_1\cup \vec w^{\,\prime}_1)$
where $(\Sigma_1,\vec z_1)$ is the nodal curve with   
marked points which we take at the beginning of this section, 
and $w'_{1,j}$ is in a neighborhood of $w_{1,j}$.
This is a $2k$ dimensional submanifold.
(Here $2k$ is the number of parameters to move $k$ points $w'_{1,j}$
on Riemann surface.)
The restriction of 
$({\rm EV}_{w_{1,j}})_{j=1,\dots,k}$
to $\{{\bf 0}\} \times \{{\bf 0}\} \times Y$
can be identified with the map
\begin{equation}\label{form721ni}
\vec w^{\,\prime}_1 \mapsto (u_1(w'_{1,j}))_{j=1,\dots,k}.
\end{equation}
Note $E(\Sigma_1({\bf x}),\vec z_1({\bf x}))$ is independent of 
$w'_{1,j}$.\footnote{The situation here is different from 
one in \cite[Lemma 9.11]{foooconstr}.} Therefore, by construction, 
(that is, the proof of Proposition \ref{prop710} by Newton's iteration) for all $w'_{1,j}$
$$
u_{{\bf 0},[\Sigma_1,\vec z_1\cup w'_{1,j}]}
= u_1.
$$
(Here we regard $[\Sigma_1,\vec z_1\cup w'_{1,j}]$ as an element of 
$\{{\bf 0}\} \times Y$.)
\par
By Definition \ref{defn71} (6) the map (\ref{form721ni}) 
is transversal to $\prod_{j}\mathcal N_j$.
\end{proof}
\begin{defn}\label{defn71474}
Let $({\frak W},\vec{\mathcal N})$ be strong stabilization data at $((\Sigma_1,\vec z_1),u_1)$.
We put \index{00V2sigmaazi@$V(((\Sigma_1,\vec z_1),u_1);\epsilon,({\frak W},\vec{\mathcal N}))$}
$$
\aligned
&V(((\Sigma_1,\vec z_1),u_1);\epsilon,({\frak W},\vec{\mathcal N})) \\
&=
\{({\bf v},{\bf x}) \in 
\mathcal V_{\rm map}(\epsilon)
\times \mathcal V_{(1)}(\epsilon)
\mid
{\rm EV}_{w_{1,j}}({\bf v},{\bf x}) \in \mathcal N_j,
\,\,\, j=1,\dots,k
\}.
\endaligned
$$
\end{defn}
We take $\epsilon$ sufficiently small so that 
conclusion of Lemma \ref{lem713} holds.
\begin{lem}\label{lem71533}
For each sufficiently small $\epsilon_3$ \index{00E5psilon3@$\epsilon_3$}
there exists $\epsilon$ with the following properties.
If 
$$
[(\Sigma',\vec z^{\,\prime}),u'] \in
\mathcal U(\epsilon;(\Sigma_1,\vec z_1),u_1,{\frak W})
$$
then there exists an additional marked points $\vec w^{\,\prime}$ 
and an element 
$({\bf v},{\bf x})$
of the space $V(((\Sigma_1,\vec z_1),u_1);\epsilon_3,({\frak W},\vec{\mathcal N}))
$
such that
$$
((\Sigma_1({\bf x}),\vec z_1({\bf x})\cup \vec w_1({\bf x})),u_{{\bf v},{\bf x}})
\cong
((\Sigma',\vec z^{\,\prime}\cup \vec w^{\,\prime}),u').
$$
Namely there exists a bi-holomorphic map $\phi : 
(\Sigma_1({\bf x}),\vec z_1({\bf x})\cup \vec w_1({\bf x})) \cong
(\Sigma',\vec z^{\,\prime}\cup \vec w^{\,\prime})$
with $u' \circ \phi = u_{{\bf v},{\bf x}}$.
\par
Moreover such $({\bf v},{\bf x}) \in V(((\Sigma_1,\vec z_1),u_1);\epsilon_3,({\frak W},\vec{\mathcal N}))$
is unique up to $\mathcal G_1$ action.
\end{lem}
\begin{proof}
We first prove the existence.
Let $\vec w^{\,\prime\prime}$ and $\phi' : (\Sigma_1({\bf x}'),\vec z_1({\bf x}')
\cup \vec w_1({\bf x}'))
\cong (\Sigma',\vec z^{\,\prime}\cup \vec w^{\,\prime\prime})$
be the isomorphism as in Definition \ref{defn412}.
(We write ${\bf x}'$ and $\vec w^{\,\prime\prime}$ here
instead of ${\bf x}$ and $\vec w^{\,\prime}$  in Definition \ref{defn412}.)
\par
Note $u'(w^{\,\prime\prime}_j)$\footnote{Here $w^{\,\prime\prime}_j$ 
is the $j$-th member of $\vec w^{\,\prime\prime}$.} is close to $u_1(w_{1,j})$ and 
$u_1(w_{1,j}) \in \mathcal N_j$.
Moreover $u'$ is $C^1$ close to $u_1\vert_{\Sigma_1(\delta)}$.
Therefore by Definition \ref{defn71} (6)
we can find $w^{\,\prime}_j$ which is close to $w^{\,\prime\prime}_j$
such that
$u'(w^{\,\prime}_j) \in \mathcal N_j$.
\par
Then there exists ${\bf x}$ which is close to ${\bf x}'$ and 
a bi-holomorphic map $\phi : (\Sigma_1({\bf x}),\vec z_1({\bf x})
\cup \vec w_1({\bf x}))
\cong (\Sigma',\vec z^{\,\prime}\cup \vec w^{\,\prime})$.
\par
Using Proposition \ref{prop710} (2), 
there exists ${\bf v}$ such that
$$
u' \circ \phi = u_{{\bf v},{\bf x}}. 
$$
Since $u'(w^{\,\prime}_j) \in \mathcal N_j$ we have
$$
u_{{\bf v},{\bf x}}(w^{\,\prime}_j({\bf x})) \in \mathcal N_j.
$$
Therefore $({\bf v},{\bf x}) \in V(((\Sigma_1,\vec z_1),u_1);({\frak W},\vec{\mathcal N}))$
as required.
\par
We next prove the uniqueness.
Let $({\bf v}^{(i)},{\bf x}^{(i)}) \in V(((\Sigma_1,\vec z_1),u_1);\epsilon_3,({\frak W},\vec{\mathcal N}))$
($i=1,2$) and $\vec w^{\,\prime}_{(i)}$ ($i=1,2$) both  
have the required properties. 
\par
We observe
$$
(\Sigma_1({\bf x}^{(1)}),\vec z_1({\bf x}^{(1)}))
\overset{\phi^{(1)}}\cong
(\Sigma',\vec z^{\,\prime})
\overset{(\phi^{(2)})^{-1}}\cong
(\Sigma_1({\bf x}^{(2)}),\vec z_1({\bf x}^{(2)}))
$$
and
\begin{equation}
u' \circ \phi^{(i)} = u_{{\bf v}^{(i)},{\bf x}^{(i)}}.
\end{equation} 
Moreover the $C^2$ distance between
$
u' \circ \phi^{(i)}\circ \Phi_{1,{\bf x}^{(i)},\delta}
$
and $u_1$ on $\Sigma_1(\delta)$
is smaller than $o(\epsilon_3)$.

By taking $\epsilon_3$, $\delta'$ small 
the composition
$\phi = 
\Phi_{1,{\bf x}^{(2)},\delta'}^{-1} \circ (\phi^{(2)})^{-1}\circ \phi^{(1)} \circ \Phi_{1,{\bf x}^{(1)},\delta}$
is defined on $\Sigma_1(\delta)$.
Then the  $C^2$ distance between
$
u_1 \circ \phi 
$ 
and $u_1$ as maps on $\Sigma_1(\delta)$ is smaller than $o(\epsilon_3)$.
(Note using Lemma \ref{lem7777}
the $C^2$ distance between
$
u' \circ \phi^{(i)}\circ \Phi_{1,{\bf x}^{(i)},\delta}
$
and $u_1$ on $\Sigma_1(\delta')$ is still smaller than $o(\epsilon_3)$.
We use this fact.)
\begin{sublem}
If $\epsilon$ and $\delta$ are sufficiently small then 
there exists $v \in \mathcal G_1$ such that the $C^2$ 
distance between $\phi$ and $v$ is smaller than $o(\epsilon_3) + o(\delta)$.
\end{sublem}
\begin{proof}
Suppose the sublemma is false. Then there exist
${\bf x}_c^{(i)}$, ${\bf v}_c^{(i)}$ (for $i=1,2$) and
$u'_c : \Sigma'_c \to X$ such that:
\begin{enumerate}
\item $\lim_{c\to \infty} {\bf x}_c^{(i)} =o
= [\Sigma_1,\vec z_1]$, 
$\lim_{c\to \infty} {\bf v}_c^{(i)} = {\bf 0}$.
\item
$u'_c \circ \phi_c^{(i)} = u_{{\bf v}_c^{(i)},{\bf x}_c^{(i)}}$.
\item
We consider 
$
\phi_c^{(i)}
:
(\Sigma_1({\bf x}_c^{(i)}),\vec z_1({\bf x}_c^{(i)}))
\cong
(\Sigma'_c,\vec z_c^{\,\prime})
$
and the composition
$$
\phi_c = (\phi_c^{(1)})^{-1} \circ \phi_c^{(2)}.
$$  
Then the limit $\lim_{c\to \infty}\phi_c$ is not an element of $\mathcal G_1$ even 
after taking a subsequence.
\end{enumerate}
\par
We regard $\phi_c$ as a map
$$
\phi_c : \Sigma_1({\bf x}^{(2)}_c) \to \mathcal C_{(1)}
$$
where $\mathcal C_{(1)}$ is the total space of the universal 
deformation of $(\Sigma_1,\vec z_1)$.
The energy of this map $\phi_c$ is uniformly bounded.
Therefore we can use Gromov's compactness theorem
\cite[Theorem 11.1]{FO} to find its limit
(with respect to the stable map topology), which is a stable map
$$
\phi_{\infty} : (\widehat{\Sigma}_1,\widehat{\vec z}_1) \to 
(\Sigma_1,\vec z_1)
$$
such that $\widehat{\Sigma}_1$ is ${\Sigma}_1$ plus bubbles,
namely $\widehat{\Sigma}_1 \to {\Sigma}_1$ exists.
Suppose $\widehat{\Sigma}_1 \ne {\Sigma}_1$.
Then there exists a sphere component $S^2_a$ of $\widehat{\Sigma}_1$
which is unstable. The map $\phi_{\infty}$ is non-constant on $S^2_a$.
Let $S^{2,\prime}_{a} = \phi_{\infty}(S^2_a) \subset {\Sigma}_1$
be the image. Since $S^{2,\prime}_{a}$ is unstable the map
$u_{{\bf v}^{(2)},{\bf x}_c^{(2)}}$
is non-constant there.
Therefore the diameter of the 
image of $u'_c$ on $\Phi_{1,{\bf x}^{(2)}_c,\delta}(S^2_a)$
is uniformly away from $0$. Since
$S^2_a$ shrink to a point in $\Sigma_1$ the image of 
$\Phi_{1,{\bf x}^{(2)}_c,\delta}(S^2_a)$ has diameter $\to 0$
as $c \to \infty$.
This is impossible since 
\begin{equation}\label{eq72222}
\lim_{c\to \infty}u'_c \circ \phi^{(i)}_c\circ \Phi_{1,{\bf x}_c^{(i)},\delta} = u_1.
\end{equation}
on $\Sigma_1(\delta)$.
\par
Therefore $\widehat{\Sigma}_1 = \Sigma_1$
and 
$
\phi_{\infty} : {\Sigma}_1 \to \Sigma_1
$
is an isomorphism.
We use (\ref{eq72222}) to obtain
$
u_1 \circ \phi_{\infty} = u_1
$. Namely $\phi_{\infty} \in \mathcal G_1$.
This is a contradiction.
\end{proof}
Using $\mathcal G_1$ equivariance of the map $({\bf v},{\bf x}) \mapsto u_{{\bf v},{\bf x}}$
in Proposition \ref{prop710} 
we may assume that $v = 1$, by replacing ${\bf x}^{(2)}$ 
etc. if necessary.
In other words, we may assume $\phi^{(2)}\circ \Phi_{1,{\bf x}^{(2)},\delta}$ 
is $C^2$ close to $\phi^{(1)} \circ \Phi_{1,{\bf x}^{(2)},\delta}$.
\par
By assumption 
$$
\vec w^{\,\prime}_{(i)} = \phi^{(i)}(w_{1,j}({\bf x}^{(i)})).
$$
On the other hand
\begin{equation}\label{formform23}
d(w_{1,j}({\bf x}^{(i)}),\Phi_{1,{\bf x}^{(i)},\delta}(w_{1,j}))
< o(\epsilon_3).
\end{equation}
Here $d$ is a metric on $\Sigma({\bf x}^{(i)})$
which is the restriction of a metric 
of the total space of the universal family of 
deformation of $(\Sigma_i,\vec z_i \cup \vec w_i)$.
\par
(\ref{formform23}) follows from the fact that
$\Phi_{1,{\bf x}^{(i)},\delta}$ converges to the 
identity map as ${\bf x}^{(i)}$ converges to ${\bf o}_i = [\Sigma_i,\vec z_i \cup 
\vec w_i]$ and $w_{1,j}({\bf x}^{(i)})$ converges to 
$w_{i,j}$ as ${\bf x}^{(i)}$ converges to $o_i$.
\begin{rem}
Note $w_{1,j}({\bf x}^{(i)}) \ne \Phi_{1,{\bf x}^{(i)},\delta}(w_{1,j})$
in general since we do {\it not} assume 
Definition \ref{defn7272} (4) for marked points of $(\Sigma_1)_a$ 
other than nodal points of $\Sigma_1$.
\end{rem}
Therefore $\vec w^{\,\prime}_{(1)}$ is close to $\vec w^{\,\prime}_{(2)}$
in $\Sigma'$.
Furthermore we have
$$
u'(\vec w^{\,\prime}_{(1)}), u'(\vec w^{\,\prime}_{(2)}) \in \mathcal N_j.
$$
Using also Definition \ref{defn71} (6) it implies that
$$
\vec w^{\,\prime}_{(1)} = \vec w^{\,\prime}_{(2)}.
$$
Therefore
$$
\aligned
(\Sigma_1({\bf x}^{(1)}),\vec z_1({\bf x}^{(1)})\cup \vec w_1({\bf x}^{(1)}))
&\overset{\phi^{(1)}} 
\cong
(\Sigma',\vec z^{\,\prime} \cup \vec w^{\,\prime}_{(1)})
=
(\Sigma',\vec z^{\,\prime} \cup \vec w^{\,\prime}_{(2)})\\
&\overset{(\phi^{(2)})^{-1}}
\cong
(\Sigma_1({\bf x}^{(2)}),\vec z_1({\bf x}^{(2)})\cup \vec w_1({\bf x}^{(2)})).
\endaligned
$$
Thus ${\bf x}^{(1)} = {\bf x}^{(2)}$. 
Now Proposition \ref{prop710} (3) implies ${\bf v}^{(1)} = {\bf v}^{(2)}$.
The proof of the uniqueness is complete.
\end{proof}
Lemma \ref{lem71533} implies that the set 
$\mathcal U(\epsilon;(\Sigma_1,\vec z_1),u_1,{\frak W})$ 
is identified with a neighborhood of the origin of the quotient space
\begin{equation}\label{form79}
V(((\Sigma_1,\vec z_1),u_1);\epsilon_3,({\frak W,\vec{\mathcal N})})/\mathcal G_1.
\end{equation}
Thus $\mathcal U(\epsilon;(\Sigma_1,\vec z_1),u_1,{\frak W})$ has an orbifold chart.
By Proposition \ref{prop78}
there exists a smooth vector bundle (orbibundle)
$E(((\Sigma_1,\vec z_1),u_1);\epsilon_3,{\frak W})$  on (\ref{form79}) such that the fiber of 
$((\Sigma',\vec z^{\,\prime}),u')$
is identified with
$E((\Sigma',\vec z^{\,\prime}),u')$.
Moreover it implies that
the map which associate to $((\Sigma',\vec z^{\,\prime}),u')$
the element 
$$
s((\Sigma',\vec z^{\,\prime}),u')= \overline{\partial}u' \in E((\Sigma',\vec z^{\,\prime}),u')
$$
is a smooth section of $E((\Sigma',\vec z^{\,\prime}),u')$.
\par
We define
$$
\psi((\Sigma',\vec z^{\,\prime}),u')
= [(\Sigma',\vec z^{\,\prime}),u'] 
\in \mathcal M_{g,\ell}((X,J);\alpha) 
$$
if $((\Sigma',\vec z^{\,\prime}),u')$ 
is an element of (\ref{form79})
with $s((\Sigma',\vec z^{\,\prime}),u') = 0$.
This defines a parametrization map
$$
\psi : s^{-1}(0)/\mathcal G_1 \to \mathcal M_{g,\ell}((X,J);\alpha) .
$$
Now we sum up the conclusion of this subsection as follows.
\begin{prop}\label{prop716}
For each $n$ there exists $\epsilon_{(n)}$ such that 
$$
(V(((\Sigma_1,\vec z_1),u_1);\epsilon_{(n)},
({\frak W},\vec{\mathcal N}))/\mathcal G_1,E(((\Sigma_1,\vec z_1),u_1);\epsilon_{(n)},({\frak W},\vec{\mathcal N})),s,\psi)
$$ is a Kuranishi neighborhood of $C^n$ class 
at $[((\Sigma_1,\vec z_1),u_1)]$ of $\mathcal M_{g,\ell}((X,J);\alpha)$. 
\end{prop}
In the next subsection we use it to define a 
$G$-equivariant Kuranishi chart containing the $G$ orbit of $[(\Sigma,\vec z),u]$.

\subsection{Construction of the smooth chart 3: Proof of Proposition \ref{prop616}}
\label{smoothchart2}
We first define a topology of the set
$U(((\Sigma,\vec z),u);\epsilon_2)$.
We use the sets $\mathcal U(\epsilon;(\Sigma_1,\vec z_1),u_1,{\frak W})$
defined in Definition \ref{defnUUUU}
for this purpose.

\begin{lem}\label{lem71818}
Suppose $((\Sigma_2,\vec z_2),u_2) \in 
\mathcal U(\epsilon;(\Sigma_1,\vec z_1),u_1,{\frak W}^{(1)})$.
\par
Then there exists
$\epsilon'>0$ such that
\begin{equation}\label{form72020}
\mathcal U(\epsilon';(\Sigma_2,\vec z_2),u_2,{\frak W}^{(2)})
\subset
\mathcal U(\epsilon;(\Sigma_1,\vec z_1),u_1,{\frak W}^{(1)}).
\end{equation}
\end{lem}
\begin{proof}
We prove this lemma in Subsection \ref{indsmoothchart}.
\end{proof}
\begin{prop}\label{prop719}
There exists a topology of $U(((\Sigma,\vec z),u);\epsilon_2)$
such that the family of its subsets,
\begin{equation}\label{form725}
\{\mathcal U(\epsilon;(\Sigma_1,\vec z_1),u_1,{\frak W}^{(1)})
\mid 
\epsilon>0, ((\Sigma_1,\vec z_1),u_1), {\frak W}^{(1)}\}
\end{equation}
is a basis of the topology.
\par
This topology is Hausdorff.
\end{prop}
\begin{proof}
The existence of the topology for which 
(\ref{form725})
is a basis of neighborhood system
is a consequence of Lemma \ref{lem71818}.
(See for example \cite[Theorem 11, p47]{kelly}.)
\par
We also remark that Lemma \ref{lem71818}
implies that for any ${\frak W}^{(1)}$, 
the set $U$ containing $[(\Sigma_1,\vec z_1),u_1]$ is a neighborhood of 
$[(\Sigma_1,\vec z_1),u_1]$ if and only if 
$\mathcal U(\epsilon;(\Sigma_1,\vec z_1),u_1,{\frak W}^{(1)}) \subset U$
for sufficiently small $\epsilon$.
\par
We next prove that this topology is Hausdorff.
Let $[(\Sigma_i,\vec z_i),u_i] \in U(((\Sigma,\vec z),u);\epsilon_2)$
and ${\frak W}^{(i)}$  stabilization and trivialization data, for $i=1,2$.
We assume 
$[(\Sigma_1,\vec z_1),u_1] \ne [(\Sigma_2,\vec z_2),u_2]$.
It suffices to show that
$$
\mathcal U(\epsilon;(\Sigma_1,\vec z_1),u_1,{\frak W}^{(1)})
\cap 
\mathcal U(\epsilon;(\Sigma_2,\vec z_2),u_2,{\frak W}^{(2)})
= 
\emptyset
$$
for sufficiently small $\epsilon$.
Suppose this does not hold.

We consider the universal family of deformation 
of $(\Sigma,\vec z)$ produced by Theorem \ref{them35}.
Then there exist $o(c) \to 0$, 
${\bf x}_c  \in \mathcal{OB}$
and $u_c : \Sigma({\bf x}_c) \to X$, such that
$[(\Sigma({\bf x}_c),\vec z({\bf x}_c)),u_c]$
lies in 
$$
\mathcal U(\epsilon(c);(\Sigma_1,\vec z_1),u_1,{\frak W}^{(1)})
\cap 
\mathcal U(\epsilon(c);(\Sigma_2,\vec z_2),u_2,{\frak W}^{(2)}).
$$
We may take ${\bf x}_c \in \mathcal{OB}$ 
and $u_c : \Sigma({\bf x}_c) \to X$ such that 
${\rm meandist}$ attains its minimum at 
$((\Sigma({\bf x}_c),\vec z({\bf x}_c)),u_c)$.
\begin{rem}
More precisely 
`${\rm meandist}$ attains its minimum at 
$((\Sigma({\bf x}_c),\vec z({\bf x}_c)),u_c)$'
means the following.
We consider
${\rm id} : (\Sigma({\bf x}_c),\vec z({\bf x}_c)) \cong (\Sigma({\bf x}_c),
\vec z({\bf x}_c))$
in Definition \ref{defn63}. In other words we consider the case
$((\Sigma',\vec z^{\,\prime}),u')
= ((\Sigma({\bf x}_c),\vec z({\bf x}_c),u_c)$
and ${\bf x}_0 = {\bf x}_c$, $\phi_0 = {\rm id}$.
We then obtain
$$
{\rm meandist} :
\mathcal W(\epsilon_1;{\bf x}_c,{\rm id};(\Sigma({\bf x}_c),\vec z({\bf x}_c)),u_c))
\to \R
$$
by Definition \ref{defn6666}.
We require that at 
$$
(\varphi,g) = (\mathcal{ID}({\bf x}_c),g)
\in \mathcal W(\epsilon_1;{\bf x}_c,{\rm id};(\Sigma({\bf x}_c),
\vec z({\bf x}_c),u_c))
$$ 
the function
${\rm meandist}$ attains its minimum, for some $g$.
\end{rem}
By Definition \ref{defn412},
there exists $\vec w_i(c) \subset \Sigma({\bf x}_c)$,
${\bf y}_{i,c}
\in \mathcal V^{(i)}_{1,0} 
\times \mathcal V^{(i)}_{1,1}(\delta_i(c))$,
(Here we put superscript $(i)$ to indicate that the right hand 
side is associated with ${\frak W}^{(i)}$.)
and a bi-holomorphic map 
\begin{equation}\label{formform725}
\phi_{i,c} : 
(\Sigma_i({\bf y}_{i,c}),\vec z_i({\bf y}_{i,c}) \cup 
\vec w_i({\bf y}_{i,c}))
\cong (\Sigma({\bf x}_c),\vec z({\bf x}_c)
\cup \vec w_i(c))
\end{equation}
with 
the following properties.
\begin{enumerate}
\item The $C^{2}$ norm of the difference between
$u_c \circ \phi_{i,c} \circ \Phi_{i,{\bf y}_{i,c},\delta_i(c)}$ and $u_i\vert_{\Sigma_i(\delta_i(c))}$
is smaller than $o(c)$.
Here
$$
\Phi_{i,{\bf y}_{i,c},\delta_i(c)} :
\Sigma_i(\delta_i(c)) \to \Sigma_i({\bf y}_{i,c})
$$
is obtained from ${\frak W}^{(i)}$.
\item The distance between ${\bf y}_{i,c}$ and $[\Sigma_i,\vec z_i \cup \vec w_i]$
in $\mathcal M_{g,k_i+\ell}$
is smaller than $o(c)$. Moreover $\delta_i(c) < o(c)$.
(Here $k_i = \# \vec w_i$.)
\item
The map $u_c \circ \phi_{i,c}$ has diameter $<o(c)$
on $\Sigma_i({\bf y}_{i,c}) \setminus 
{\rm Im}(\Phi_{i,{\bf y}_{i,c},\delta_i(c)})$.
\end{enumerate}
\begin{rem}\label{rem730}
Here and hereafter, the positive numbers $o(c)$ \index{00O2C@$o(c)$} depend on $c$ and satisfies
$
\lim_{c \to \infty}o(c) =0 
$.
\end{rem}
Using Lemma \ref{lem7777} we may assume $\lim_{c\to \infty}\delta_i(c) = 0$.
\par
By Definition \ref{defn71} (8), 
the point $w_{i,j}(c)$ (which is the $j$-th member of 
$\vec w_i(c)$) is contained in the image of 
$\Phi_{{\bf x}_{c},\delta_i(c)} : \Sigma(\delta_i(c)) 
\to \Sigma({\bf x}_c)$.
We take $\tilde w_{i,j}(c) \in \Sigma(\delta_i(c))$
such that
$$
\Phi_{{\bf x}_{c},\delta_i(c)}(\tilde w_{i,j}(c)) 
= w_{i,j}(c).
$$
By taking a subsequence if necessary we may assume that
the limit
\begin{equation}\label{form722}
\lim_{c\to\infty}{\bf x}_c = {\bf x}_{\infty} \in \mathcal{OB}
\end{equation}
exists.
Moreover we may assume
\begin{equation}\label{form723}
\lim_{c\to\infty}\Phi_{{\bf x}_{c},\delta_i(c)}(\tilde w_{i,j}(c))
= w_{i,j}(\infty) \in \Sigma({\bf x}_{\infty})
\end{equation}
converges by taking a subsequence if necessary.
Here (\ref{form723}) is the convergence in the total space of the 
universal family of deformation of $(\Sigma,\vec z)$.
\begin{sublem}\label{sublem722}
$w_{i,j}(\infty) \ne w_{i,j'}(\infty)$
if $j \ne j'$.
\end{sublem}
\begin{proof}
We can prove the sublemma by using minimality of ${\rm meandist}$ as follows.
Suppose $w_{i,j}(\infty) = w_{i,j'}(\infty)$
with $j \ne j'$.
We may assume $w_{i,j}$ and $w_{i,j'}$ are in the 
same irreducible component of $\Sigma_i$, by replacing $j,j'$ if necessary.
(Here $w_{i,j}$ is the $j$-th member of $\vec w_i \subset \Sigma_i$.)
In fact suppose $w_{i,j}$ and $w_{i,j'}$ are in the 
different irreducible components.
If there exists $j'' \ne j$ (resp. $j'' \ne j'$) 
such that $w_{i,j}$ and $w_{i,j''}$ (resp. 
$w_{i,j'}$ and $w_{i,j''}$) are in the same irreducible
component, then we may replace $w_{i,j'}$ by $w_{i,j''}$
(resp. $w_{i,j}$ and $w_{i,j''}$).
\par
If there exists no such $j''$ then one of the following holds 
because $(\Sigma_i,\vec z_i \cup \vec w_i)$ is stable.
Let $\Sigma_{ij}$ (resp. $\Sigma_{ij'}$) be the 
irreducible component containing $w_{i,j}$ (resp. $w_{i,j'}$).
\begin{enumerate}
\item[(j,I)] There exist a marked point $z_{i,k_j}$ on $\Sigma_{ij}$.
\item[(j,II)] The genus of $\Sigma_{ij}$ is positive.
\end{enumerate}
We may assume ${\rm (j',I)}$ or ${\rm (j',II)}$ also.
\par
If ${\rm (j,I)}$ and ${\rm (j',I)}$ hold then 
since $z_{i,k_j} \ne z_{i,k_{j'}}$ then 
$w_{i,j}(\infty) \ne w_{i,j'}(\infty)$
and we are done.
In the other 3 cases we can show 
$w_{i,j}(\infty) \ne w_{i,j'}(\infty)$
in a similar way.
\par
We thus may assume that $w_{i,j}$ and $w_{i,j'}$ are in the 
same irreducible component of $\Sigma_i$.
\par
Moreover the map $u_i$ is non-constant on the irreducible component containing $w_{i,j}$ by Definition \ref{defn71} (3).
Therefore by Item (1)
the map  $u_c$ has some nontrivial energy  in a small neighborhood of 
$\{ w_{i,j}(c), w_{i,j'}(c) \}$. (The energy there 
can be estimated uniformly from below because it is larger than 
the half of the energy of nontrivial holomorphic sphere 
for example.)
\par
This implies that the total energy of $u_c$ outside a small neighborhood 
of the set $\{ w_{i,j}(c), w_{i,j'}(c) \}$ is uniformly strictly smaller than 
the energy of $u$. Therefore
${\rm meandist}$ is greater than some number independent of $\epsilon_1$.
If $\epsilon_1$ is small then ${\rm meandist}$ does not attain 
its minimum at $((\Sigma({\bf x}_c),\vec z({\bf x}_c)),u_c)$.
This contradicts to our choice.
\end{proof}
\begin{sublem}\label{sublem726}
$ (\Sigma({\bf x}_{\infty}),\vec z({\bf x}_{\infty})
\cup \vec w_i({\infty}))$ 
is stable.
\end{sublem}
\begin{proof}
Suppose there is an unstable component 
$\Sigma({\bf x}_{\infty})_a$.
Then there exists an unstable component 
$\Sigma_{\tilde a}$ of $\Sigma$ such that
$$
\Phi_{{\bf x}_{\infty},\delta}(\Sigma_{\tilde a}
\cap \Sigma(\delta)) 
\subset \Sigma({\bf x}_{\infty})_a.
$$
\begin{subsublemma}
There exists $v_-(c), v_+(c) \in \Sigma({\bf x}_c)$ with the following 
properties.
\begin{enumerate}
\item[(a)] $\lim_{c\to\infty} v_-(c)$, $\lim_{c\to\infty} v_+(c)$ converges to points 
of $\Sigma({\bf x}_{\infty})_a$.
\item[(b)]
$d_X(u_c(v_-(c)),u_c(v_+(c)))$ is uniformly 
bounded away from $0$ as $c \to \infty$. 
\item[(c)]
$v_-(c), v_+(c)$ are uniformly away from the 
nodes or marked points.
\end{enumerate}
\end{subsublemma}
\begin{proof}
(See Figure \ref{732Figure}.)
By stability of $((\Sigma,\vec z),u)$ the map $u$ is nontrivial on $\Sigma_{\tilde a}$.
Therefore there exist $v_-, v_+ \in \Sigma_{\tilde a}$
such that for each $z_- \in B_{c_0}(u(v_-))$, $z_+ \in B_{c_0}(u(v_+))$
the inequality 
$$d(u(z_-),u(z_+)) > c_0$$ 
holds. Here $c_0 > 0$ depends only on $X$. The notation $B_{c_0}(\cdot)$
stand for the metric ball of radius $c_0$.
\par
Using the fact that the ${\rm meandist}$ attains 
its minimum at $((\Sigma({\bf x}_c),\vec z({\bf x}_c)),u_c)$ and 
that we may assume the meandist $\le o(\epsilon_1)$ is small compared with $c_0$,
we can find $v_-(c)$, $v_+(c)$ satisfying (b)(c) above
in a small neighborhood of $v_-$, $v_+$ (in the universal family 
$\widetilde{\mathcal{OB}} \to \mathcal{OB}$).
By taking a subsequence if necessary we may assume (a) also.
\end{proof}
\par
Since $\lim_c v_-(c)$, $\lim_c v_+(c)$ both 
converge to points on an unstable component $\Sigma({\bf x}_{\infty})_a$
and since $(\Sigma_i,\vec z_i\cup \vec w_i)$
is stable we find that
$$
\lim_{c\to \infty}
d((\Phi_{i,{\bf y}_{i,c},\delta_i(c)}^{-1} \circ \phi_{i,c}^{-1})(v_-(c)),
(\Phi_{i,{\bf y}_{i,c},\delta_i(c)}^{-1} \circ \phi_{i,c}^{-1})(v_+(c)))
= 0.
$$
Note the distance here is the Riemannian distance in $\Sigma_i$.
In fact Sublemma \ref{sublem722} implies that there exists a holomorphic map 
$\Sigma({\bf x}_{\infty}) \to \Sigma_i$ shrinking the unstable components
of $(\Sigma({\bf x}_{\infty}),\vec z({\bf x}_{\infty}) \cup \vec w_1({\bf x}_{\infty}))$.
\par
On the other hand Item (1) and the fact 
$d_X(u_c(v_-(c)),u_c(v_+(c)))$ is uniformly 
bounded away from $0$ implies that
$$
d_X(u_i((\Phi_{i,{\bf y}_{i,c},\delta_i(c)}^{-1} \circ \phi_{i,c}^{-1})(v_-(c))),
u_i(\Phi_{i,{\bf y}_{i,c},\delta_i(c)}^{-1} \circ \phi_{i,c}^{-1})(v_+(c))))
$$
is uniformly bounded away from $0$.
Since $\{u_i\}$ is equicontinuous away from the nodes this is a contradiction.
\end{proof}
\begin{figure}[h]
\centering
\includegraphics[scale=0.3]{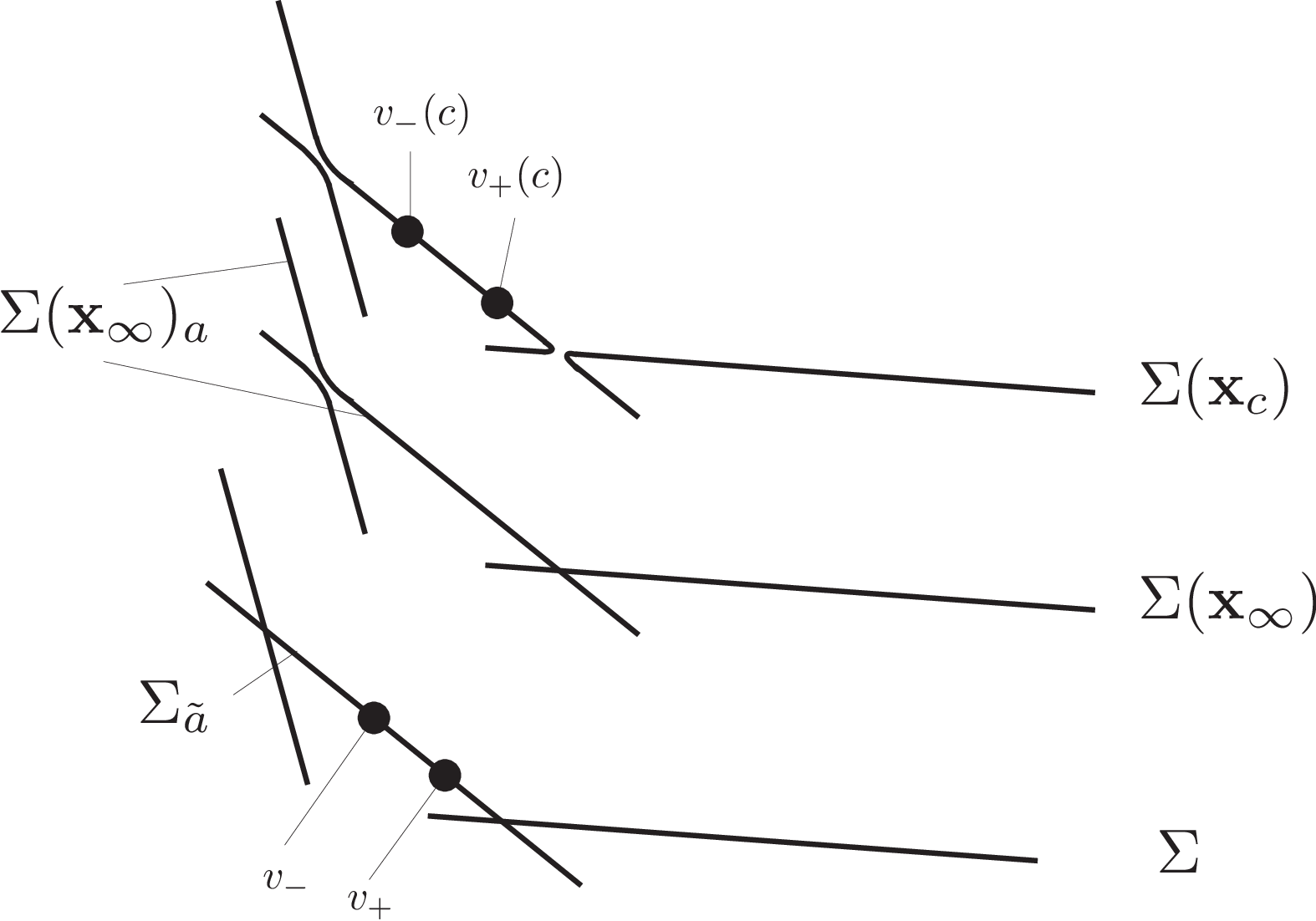}
\caption{$v_-(c)$ and $v_+(c)$}
\label{732Figure}
\end{figure}
Next we will prove that we can take a subsequence such that  there exists 
$$
u_{\infty} : \Sigma({\bf x}_{\infty}) \to X
$$
satisfying
\begin{equation}\label{form72400}
\lim_{c \to \infty} u_{c} = u_{\infty},
\end{equation}
in the following sense.
\par
The spaces $\Sigma({\bf x}_c)$ are submanifolds 
of the metric space $\widetilde{\mathcal{OB}}$, the total 
space of our universal family.
This sequence of submanifolds $\Sigma({\bf x}_c)$ converges to $\Sigma({\bf x}_{\infty})$ by Hausdorff distance
(of subsets of $\widetilde{\mathcal{OB}}$). Let $\rho_c$ be the Hausdorff distance 
between them. (Note $\lim_{c\to\infty}\rho_c = 0$.)
\par
Now (\ref{form72400}) means that
\begin{equation}\label{defform725}
\lim_{c\to\infty} 
\sup \{d_X(u_c(x),u_{\infty}(y)) \mid (x,y) 
\in \Sigma({\bf x}_c) \times \Sigma({\bf x}_{\infty}),
\,\,
d(x,y) \le 2\rho_c\} = 0.
\end{equation}
(See Figure \ref{Figure725}.)
\begin{figure}[h]
\centering
\includegraphics[scale=0.3]{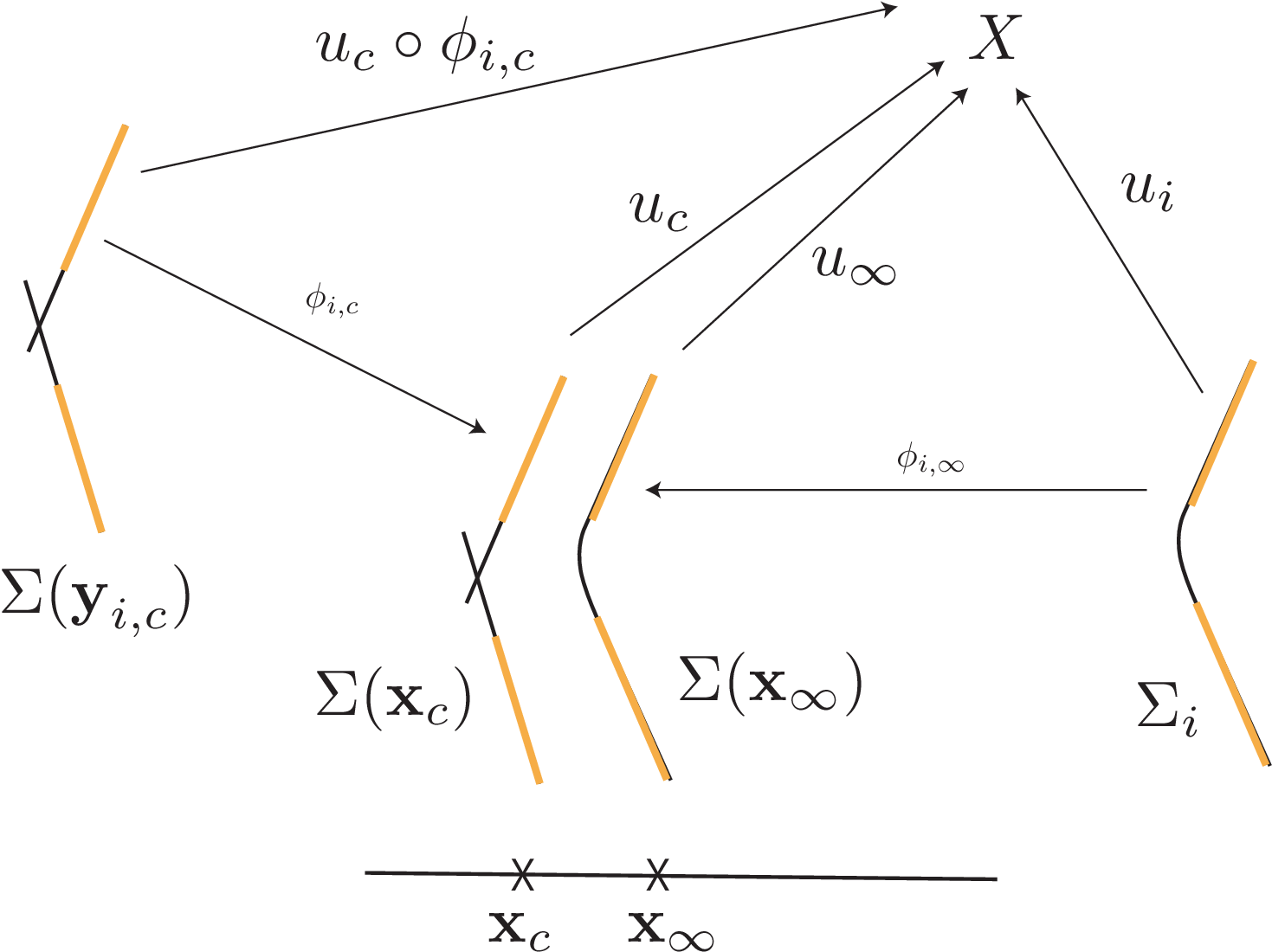}
\caption{$\lim_{c \to \infty} u_{c} = u_{\infty}$}
\label{Figure725}
\end{figure}
\par\smallskip
We now prove the existence of the limit 
$u_{\infty}$.
\par
Item (2) above, Sublemmata \ref{sublem722}, \ref{sublem726} and (\ref{form722}),(\ref{form723})
imply that, for $i=1,2$, there exists a unique isomorphism
\begin{equation}\label{formula730}
\phi_{i,\infty} : 
(\Sigma_i,\vec z_i \cup 
\vec w_i)
\cong (\Sigma({\bf x}_{\infty}),\vec z({\bf x}_{\infty})
\cup \vec w_i({\infty})).
\end{equation}
We consider 
$$
\phi_{i,c} \circ \Phi_{i,{\bf y}_{i,c},\delta_i(c)} : 
\Sigma_i(\delta_i(c)) \to \Sigma({\bf x}_c)
$$
and regard it as a map to $\widetilde{\mathcal{OB}}$.

\begin{sublem}\label{sublem7262}
There exists $\phi_{i,\infty}$ such that
$$
\lim_{c\to \infty}\phi_{i,c} \circ \Phi_{i,{\bf y}_{i,c},\delta_i(c)}
= \phi_{i,\infty}.
$$
\end{sublem}
\begin{proof}
We consider the total space of the universal deformation 
of the stable marked nodal curve $(\Sigma_i,\vec z_i \cup 
\vec w_i)$.
In this space the sequence $\vec z({\bf y}_{i,c}) \cup 
\vec w({\bf y}_{i,c})$
converges to 
$\vec z_i \cup 
\vec w_i$.
On the other hand, in the total space 
$\widetilde{\mathcal{OB}}$, 
the marked points $\vec z({\bf x}_c)
\cup \vec w_i(c)$ 
converges to 
$\vec z({\bf x}_{\infty})
\cup \vec w_i({\infty})$.
The sublemma is then an immediate
consequence of 
(\ref{formform725}), (\ref{formula730}),
the stability of $(\Sigma_i,\vec z_i \cup 
\vec w_i)$
and the fact that 
$\Phi_{i,{\bf y}_{i,c},\delta(c)}$
converges to the identity map.
\end{proof}
Item (1) above implies that
$$
\lim_{c\to\infty}\sup\{d_X((u_c \circ \phi_{i,c} \circ \Phi_{i,{\bf y}_{i,c},\delta(c)})(z),u_i(z))
\mid z \in   \Sigma_i(\delta(c)) \} = 0,
$$
and that $u_c \circ \phi_{i,c} \circ \Phi_{i,{\bf y}_{i,c},\delta(c)}$ is equicontinuous.
\par
Item (3) above implies that the diameter of the image by $u_c \circ \phi_{i,c}$
of each connected component of $\Sigma_i({\bf y}_{i,c}) \setminus 
{\rm Im}(\Phi_{i,{\bf y}_{i,c},\delta_i(c)})$ 
is smaller than $o(c)$.
\par
Therefore we can take
$$
u_{\infty} = u_i \circ (\phi_{i,\infty})^{-1}.
$$
Note $u_\infty$ is independent of $i$ since
it satisfies Formula (\ref{defform725}).
\par\smallskip
We are now in the position to complete the proof of 
Proposition \ref{prop719}.
By (\ref{defform725}) we have 
$$
((\Sigma_1,\vec z_1),u_1)
\overset{\phi_{\infty,1}}\cong 
((\Sigma({\bf x}_{\infty}),\vec z({\bf x}_{\infty})),u_{\infty})
\overset{\phi_{\infty,2}^{-1}}\cong 
((\Sigma_2,\vec z_2),u_2).
$$
This contradicts to 
$[(\Sigma_1,\vec z_1),u_1] \ne [(\Sigma_2,\vec z_2),u_2]$.
\end{proof}
\begin{rem}
In the proof of Proposition \ref{prop719} we proved 
Hausdorff-ness directly. Alternatively we can proceed as follows.
(See \cite[Section 3]{FOOOSpringer}, \cite[Section 3]{foootech2}, \cite[Part 7]{fooospectr}
and etc. 
for the definition of Kuranishi structure and good coordiante system.)
By Proposition \ref{prop716} we find a Kuranishi chart at each 
point of the $G$-orbit of $[(\Sigma,\vec z),u] \in \mathcal M_{g,\ell}((X,J);\alpha)$.
Using (the proof of) Proposition \ref{prop720} 
we can show the existence of the coordinate change and obtain a Kuranishi structure on 
the $G$-orbit of $[(\Sigma,\vec z),u] \subset \mathcal M_{g,\ell}((X,J);\alpha)$.
We take a good coordinate system compatible with it.
(See \cite[Section 11]{FOOOSpringer} etc. for the existence of such a good coordinate system.)
Then by \cite[Theorem 2.9]{foooshrink} we 
can shrink this good coordinate system
so that we obtain a Hausdorff space by gluing the Kuranishi charts which are 
members of the good coordinate system obtained by the 
above shrinking.
This will become the Kuranishi neighborhood we look for.
\footnote{Note this argument does not prove Hausdorffness of 
$U(((\Sigma,\vec z),u);\epsilon_2)$ itself. Instead it produces certain 
Hausdorff orbifold which becomes a Kuranishi neighborhood.}
\par
The proof we gave here is self-contained and does not use 
the results of \cite{foooshrink} or the existence theorem of
compatible good coordinate system.
\end{rem}
Let $({\frak W}^{(i)},\vec{\mathcal N}_{(i)})$ be strong stabilization data
at $(((\Sigma_i,\vec z_i),u_i);\epsilon,{\frak W}^{(i)})$,
for $i=1,2$.
We defined 
$V(((\Sigma_i,\vec z_i),u_i);\epsilon,({\frak W}^{(i)},\vec{\mathcal N}^{(i)}))$ in
Definition \ref{defn71474}.
\begin{prop}\label{prop720}
If $\epsilon$ is smaller than a positive number 
depending on $n$ and $\epsilon'$ is smaller 
than a positive number depending on $n$ and $\epsilon$, then
the embedding (\ref{form72020}) becomes a $C^n$ embedding
with respect to the smooth structure as 
open subsets of 
$V(((\Sigma_1,\vec z_1),u_1);\epsilon,({\frak W}^{(1)},\vec{\mathcal N}^{(1)}))/\mathcal G_1$
and
$V(((\Sigma_2,\vec z_2),u_2);\epsilon',({\frak W}^{(2)},\vec{\mathcal N}^{(2)}))/\mathcal G_2$. 
Here $\mathcal G_2= \mathcal G((\Sigma_2,\vec z_2),u_2)$,
the group of automorphisms of $((\Sigma_2,\vec z_2),u_2)$.
\end{prop}
This is  proved in \cite[Part 4]{foootech} and \cite[Section 10]{foooconstr}.
We repeat the proof in Subsection \ref{indsmoothchart} for the 
sake of completeness.
\par
We remark that by definition the restriction of the bundle 
$E(((\Sigma_1,\vec z_1),u_1);\epsilon,({\frak W}^{(1)},\vec{\mathcal N}_1))$
to 
$\mathcal U(\epsilon';(\Sigma_2,\vec z_2),u_2,{\frak W}^{(2)})$
is canonically isomorphic to the restriction of the bundle
$E(((\Sigma_2,\vec z_2),u_2);\epsilon',({\frak W}^{(2)},\vec{\mathcal N}_2))$.
\begin{lem}\label{lemprop721}
This canonical isomorphism preserves the $C^n$ structure of 
vector bundles.
\end{lem}
The proof is also in Subsection \ref{indsmoothchart}.
\par
By Proposition \ref{prop720} and Lemma \ref{lemprop721}
we can glue $C^n$ structures to obtain a 
$C^n$ structure on $U(((\Sigma,\vec z),u);\epsilon_2)$
and on the vector bundle 
$E(((\Sigma,\vec z),u);\epsilon_2)$. \index{00E2sigmavec@$E(((\Sigma,\vec z),u);\epsilon_2)$}
(The later is obtained by gluing 
$E((\Sigma_1,\vec z_1),u_1);\epsilon,({\frak W}^{(1)},\vec{\mathcal N}_1))$.)
We can then glue the Kuranishi map $s$ and parametrization 
map $\psi$ defined for various $((\Sigma_1,\vec z_1),u_1);\epsilon,({\frak W}^{(1)},\vec{\mathcal N}^{(1)}))$
in Proposition \ref{prop716}.
\par
Thus we obtain a Kuranishi chart
\begin{equation}\label{form721}
(U(((\Sigma,\vec z),u);\epsilon_2),E(((\Sigma,\vec z),u);\epsilon_2),s,\psi),
\end{equation}
of $C^n$ class for any $n$.
\begin{lem}\label{lem735735}
The Kuranishi chart (\ref{form721}) is of $C^{\infty}$ class.
\end{lem}
This is proved in \cite[Section 26]{foootech}, \cite[Section 12]{foooconstr}.
We repeat the proof in Subsection \ref{subsec:Cinfinity} for the 
sake of completeness.
\par
We finally prove:
\begin{lem}\label{lem72727}
The Kuranishi chart (\ref{form721}) is $G$-equivariant.
\end{lem}
\begin{proof}
Let $((\Sigma_1,\vec z_1),u_1) \in U(((\Sigma,\vec z),u);\epsilon_2)$.
We take 
its strong stabilization data $({\frak W},\vec{\mathcal N})$ as in Definition \ref{defn720}.
Note all the data in ${\frak W}$ are independent of $u_1$.
Therefore we can define $g{\frak W}$ for $((\Sigma_1,\vec z_1),gu_1)$
so that it is the same as ${\frak W}$. We replace $\mathcal N_i$
by $g\mathcal N_i$ to define $g\vec{\mathcal N}$
\par
Then there exists an isomorphism
$$
V(((\Sigma_1,\vec z_1),u_1);({\frak W},\vec{\mathcal N})) \cong V(((\Sigma_1,\vec z_1),gu_1);(g{\frak W},g\vec{\mathcal N}))
$$
sending $({\bf v},{\bf x})$ to $(g_*{\bf v},{\bf x})$.
Note ${\bf v}$ is an element of  ${\rm Ker}^+ D_{u_1}\overline{\partial}$
defined in (\ref{form715}).
Therefore $g_*{\bf v}$ is an element of  ${\rm Ker}^+ D_{gu_1}\overline{\partial}$.
This is because
$$
g_*E((\Sigma_1,\vec z_1),u_1) = E((\Sigma_1,\vec z_1),gu_1).
$$
(Lemma \ref{lem613233}.)
\par
Furthermore, the gluing construction of $u_{{\bf v},{\bf x}}$
is invariant of $G$ action. Nameley: 
$$
gu_{{\bf v},{\bf x}} = u_{g{\bf v},{\bf x}}.
$$
Therefore 
$((\Sigma',\vec z^{\,\prime}),u') \mapsto ((\Sigma',\vec z^{\,\prime}),gu')$
defines a smooth map 
from a neighborhood of $[(\Sigma_1,\vec z_1),u_1]$ in 
$U((\Sigma,\vec z),u);\epsilon_2)$
to a neighborhood of $[(\Sigma_1,\vec z_1),gu_1]$
in $U(((\Sigma,\vec z),u);\epsilon_2)$.
Thus $G$ action is a smooth action on 
$U(((\Sigma,\vec z),u);\epsilon_2)$.
\par
Smoothness of the $G$-action on the obstruction bundle
can be proved in the same way.
$G$-equivariance of $s$ and $\psi$ is obvious from the definition.
\end{proof}
The proof of Proposition \ref{prop616} except the parts 
deferred to Subsections \ref{mainprop}, \ref{indsmoothchart} 
and \ref{subsec:Cinfinity} is now complete.
\qed

\subsection{Exponential decay estimate of obstruction bundle}
\label{mainprop}

In this section we prove Proposition \ref{prop78}.
\par
We consider the set $\mathcal V_{(1)} = \mathcal V_{1,0} \times \mathcal V_{1,1}$.
It is a set of $(\vec x,\vec{\rho})$ where $\vec{\rho}\in \mathcal V_{1,1}$ is the parameter 
to smooth the node of $(\Sigma_1,\vec z_1 \cup \vec w_1)$ and 
$\vec x = (x_{a})_{a\in \mathcal A_s} \in \mathcal V_{1,0}$ is the parameter to deform 
the complex structure of each irreducible component of 
$(\Sigma_1,\vec z_1 \cup \vec w_1)$.
It comes with the universal family 
$\pi_{(1)} : \mathcal C_{(1)} \to \mathcal V_{(1)}$.
(See (\ref{familyon1}).)
For each ${\bf x} \in \mathcal V_{(1)}$ its fiber together with 
marked points is written as $(\Sigma_1({\bf x}),\vec z_1({\bf x}) 
\cup \vec w_1({\bf x}))$.
\par
We assumed that $((\Sigma_1,\vec z_1),u_1)$ is $G$-$\epsilon_2$-close to 
$((\Sigma,\vec z),u)$. We consider the universal family of deformation 
of $(\Sigma,\vec z)$. Suppose that the universal family 
is obtained from $\pi : \mathcal C \to \mathcal V$, which is a holomorphic map 
between complex manifolds and has nodal curves as fibers, 
by Construction \ref{const2124}.
(See the proof of Theorem \ref{them35} in Subsection \ref{unideformationexi}.)
\begin{lem}
There exist holomorphic maps
$$
\tilde\psi : \mathcal C_{(1)} \to \mathcal C, \qquad
\psi : \mathcal V_{(1)} \to \mathcal V
$$
with the following properties.
\begin{enumerate}
\item
The next diagram commutes
\begin{equation}\label{diagram728888}
\begin{CD}
\mathcal C_{(1)} @ >{\tilde\psi}>> \mathcal C,\\
@ VV{\pi_{(1)}}V @ VV{\pi}V\\
\mathcal V_{(1)} @ >{\psi}>> \mathcal V,
\end{CD}
\end{equation}
and is cartesian.
\item
The next diagram commutes for $j=1,\dots,\ell$.
\begin{equation}\label{diagram7288882}
\begin{CD}
\mathcal C_{(1)} @ >{\tilde\psi}>> \mathcal C,\\
@ AA{\frak T_j}A @ AA{\frak t_j}A\\
\mathcal V_{(1)} @ >{\psi}>> \mathcal V,
\end{CD}
\end{equation}
Here $\frak T_j$ and $\frak t_j$ are sections 
which assign the marked points.
\item
$\tilde \psi$ and $\psi$ are $\mathcal G_1$ equivariant.
\end{enumerate}
\end{lem}
\begin{proof}
By forgetting $\frak t_j$ for $j=\ell+1,\dots,\ell+k$ 
(namely the marked points $\vec{w}({\bf x})$),
the family $\mathcal C_{(1)} \to \mathcal V_{(1)}$ becomes a
deformation of $(\Sigma_1,\vec z_1)$.
Therefore in case $(\Sigma_1,\vec z_1) = (\Sigma,\vec z)$
the lemma is a consequence of the universality of 
$\pi : \mathcal C \to \mathcal V$.
\par
The general case can be reduced to the case $(\Sigma_1,\vec z_1) = (\Sigma,\vec z)$
by using Sublemma \ref{sublem321}.
\end{proof}
\begin{rem}\label{rem741}
We use the universality in Theorem \ref{them35} here.
We remark that here the universality we use is one for the complex analytic family.
A similar universality for $C^{\infty}$ family also holds. This is a part of classical 
theory by Kodaira-Spencer in the case of non-singular curves. 
In the general case where curves are nodal, one can prove it, for example, by working out the 
study of a neighborhood of a point of Deligne-Mumford moduli space that corresponds 
to a nodal curve by an analytic method of gluing. (See for example \cite[Section 25]{fooo:techI}
and compare \cite[Remark 8.36]{foooexp}.)
\end{rem}
We are given a map $u_1 : \Sigma_1 \to X$. 
Note that $\Sigma_1 = \Sigma({\bf x}_1)$. 
We regard it as a subset $\pi^{-1}({\bf x}_1)$ of the total space $\mathcal C$ of the universal family
and extends it to a smooth map $F$ from a neighborhood of 
$\pi^{-1}({\bf x}_1)$ in  $\mathcal C$. Let $F_1 = F\circ \tilde \psi$ be its composition 
with $\tilde \psi$.
For ${\bf x} \in \mathcal V_{(1)}$ such that $\psi({\bf x})$ is close to ${\bf x}_1$, 
we denote the restriction of 
$F_1$ to $\pi^{-1}({\bf x})$ by $u_{{\bf x}} : \Sigma_1({\bf x}) \to 
X$.

\begin{defn}\label{defn7422}
For $\epsilon_4 > 0$, 
we define 
$\mathcal W(\epsilon_4)$ as follows:\index{00W3EPSILONi@$\mathcal W(\epsilon_4)$}
\begin{equation}
\mathcal W(\epsilon_4)
=
\bigcup_{{\bf x} \in \mathcal V_{(1)}}
\mathcal W(\epsilon_4;\psi({\bf x}),(\tilde{\psi}\vert_{
\Sigma_1({\bf x})})^{-1};((\Sigma_1({\bf x}),\vec z_1({\bf x})),u_{{\bf x}}))
\times \{{\bf x}\}.
\end{equation}
See Definition \ref{defn63} for the notation appearing in the right hand side.
Note that $(\tilde{\psi}\vert_{
\Sigma_1({\bf x})})^{-1} : \Sigma(\psi({\bf x})) \to \Sigma_1({\bf x})$ is an 
isomorphism which plays the role of $\phi_0$ in Definition \ref{defn63}.
\par
We define 
${\rm Pro} : \mathcal W(\epsilon_4) \to \mathcal V_{(1)}$ by 
assigning $\bf x$ to all the elements of the subset
$\mathcal W(\epsilon_4;\psi({\bf x}),(\tilde{\psi}\vert_{
\Sigma_1({\bf x})})^{-1};((\Sigma_1({\bf x}),\vec z_1({\bf x})),u_{{\bf x}}))
\times \{{\bf x}\}$.
\par
In a similar way as Definition \ref{defn686861} we define a 
right $\widehat{\mathcal G}_c$
action on $\mathcal W(\epsilon_4)$ as follows. 
Let $(\varphi,g,{\bf x}) \in \mathcal W(\epsilon_4)$ and $\upsilon = (\gamma,h) \in \widehat{\mathcal G}_c$.
We have ${\rm Pr}_t(\varphi) = \psi({\bf x})$.
We put $\overline{\bf y} = {\rm Pr}_s(\varphi)$.
Then $\gamma \in \mathcal{MOR}$
induces an isomomrphism $\gamma_* : (\Sigma(\overline{\bf y}),\vec z(\overline{\bf y})) 
\cong (\Sigma(\psi({\bf x})),\vec z(\psi({\bf x}))$.
We put $\vec w' = \gamma_*^{-1}(\tilde{\psi}\vert_{
\Sigma_1({\bf x})}(\vec w(\bf x)))$.
There exists a unique ${\bf y}$ such that
$$
(\Sigma(\overline{\bf y}),\vec z(\overline{\bf y}),\vec w')
\cong (\Sigma_1({\bf y}),\vec z_1({\bf y}),\vec w_1({\bf y})).
$$
In particular $\psi({\bf y}) = \overline{\bf y}$. 
We now put 
\begin{equation}\nonumber
\upsilon(\varphi,g,{\bf x})
= (\varphi \circ \gamma_*,gh,{\bf y}).
\end{equation}
The $\widehat{\mathcal G}_c$ action is free and smooth.
We denote 
$$
\overline{\mathcal W}(\epsilon_4) = \mathcal W(\epsilon_4)/\widehat{\mathcal G}_c.
$$
${\rm Pro}$ induces a map 
$\overline{\rm Pro} : \overline{\mathcal W}(\epsilon_4) \to \mathcal V_{(1)}$.
\index{00P1roover@$\overline{\rm Pro}$}
\end{defn}
\begin{lem}
$\mathcal W(\epsilon_4)$ has a structure of complex manifold 
and $\overline{\rm Pro}$ is a submersion.
\end{lem}
\begin{proof}
The proof is the same as the proof of Lemma \ref{lem6262}, 
using the fact that ${\rm Pr}_t$ is a submersion.
\end{proof}
Let $U'(\epsilon)$ be the $\epsilon$ neighborhood of $u_1\vert_{\Sigma_1(\delta)}$ 
in $L^2_{m+1}$ norm (as in Proposition \ref{prop78}).
(By taking $\epsilon$ small we may regard $U'(\epsilon)$
as an open subset of an appropriate Hilbert space (that is, $L^2_{m+1}$ space).)
\begin{defn}\label{defn744}
We define a function \index{00M1@${\rm meandist}$}
$$
{\rm meandist} : \mathcal W(\epsilon_4) \times U'(\epsilon) \to \R
$$
as follows. 
Let 
$(\varphi,g,{\bf x},\hat u')  \in \mathcal W(\epsilon_4;\psi({\bf x}),(\tilde{\psi}\vert_{
\Sigma_1({\bf x})})^{-1};((\Sigma_1({\bf x}),\vec z_1({\bf x})),u_{{\bf x}}))
\times \{{\bf x}\} \times U'(\epsilon)$.
We slightly modify (\ref{form611}) to set:
\begin{equation}\label{form611mod}
\aligned
&{\rm meandist}(\varphi,g,{\bf x},\hat u') \\
&= 
\int_{z \in \Sigma(\sigma)}
\chi(z)\,\, d_X^2((\hat u' \circ \Phi_{1,{\bf x},\delta}^{-1}\circ (\tilde{\psi}\vert_{
\Sigma_1({\bf x})})^{-1} \circ \varphi 
\circ \Phi_{\overline{\bf x}',\sigma})(z),gu(z))\,\, \Omega_{\Sigma},
\endaligned
\end{equation}
where $\overline{\bf x}' = {\rm Pr}_s(\varphi)$.
Note that the map  
$\hat u' \circ \Phi_{1,{\bf x},\delta}^{-1}\circ (\tilde{\psi}\vert_{
\Sigma_1({\bf x})})^{-1} \circ \varphi 
\circ \Phi_{\overline{\bf x}',\sigma}$ is a composition
$$
\Sigma(\sigma) \longrightarrow  \Sigma(\overline{\bf x}')(\sigma) 
\longrightarrow \Sigma(\psi({\bf x}))(\sigma)
\longrightarrow \Sigma_1({\bf x})(\delta)
\longrightarrow \Sigma_1(\delta)
\longrightarrow X,
$$
which is defined if $\delta$ is small compared with $\sigma$.
\par
It is $\widehat{\mathcal G}_c$ invariant and induces \index{00M1over@$\overline{\rm meandist}$}
$$
\overline{\rm meandist} : \overline{\mathcal W}(\epsilon_4) \times U'(\epsilon) \to \R.
$$
\end{defn}
\begin{lem}\label{lem732}
We assume $\epsilon_4$ and $\epsilon$ are sufficiently small.
Then the functions ${\rm meandist}$ and $\overline{\rm meandist}$ are smooth 
functions. The restriction of  $\overline{\rm meandist}$
to the fibers of $\overline{\rm Pro} : \overline{\mathcal W}(\epsilon_4) \to \mathcal V_{(1)}$
are strictly convex.
Moreover the restriction of $\overline{\rm meandist}$ to the 
fibers of $\overline{\rm Pro}$ attains its minimum at a unique point.
\end{lem}
\begin{proof}
The smoothness of ${\rm meandist}$ and $\overline{\rm meandist}$
is immediate from (\ref{form611}).
To show strict convexity 
it suffices to consider the case 
$\psi({\bf x}) = {\bf x}_1$, that is, 
$\Sigma_1({\bf x}) =  \Sigma_1$. 
(This is because we can then shrink $\mathcal V_{(1)}$ and use the 
fact that strict convexity is an open property.) This case is proved as Proposition \ref{prop69}.
The uniqueness of minimum is Lemma \ref{lem610}.
\end{proof}
We remark that 
${\rm Pro}^{-1}({\bf x}) \subset {\mathcal W}(\epsilon_4)$ 
is identified with an open subset 
$\mathcal{MOR} \times G$, which is nothing but 
$$
\mathcal W(\epsilon_4;\psi({\bf x}),(\tilde{\psi}\vert_{
\Sigma_1({\bf x})})^{-1};((\Sigma_1({\bf x}),\vec z_1({\bf x})),u_{{\bf x}})).
$$
\begin{lem}\label{lem733}
There exists a smooth map
$$
\overline\Phi : U'(\epsilon) \times \mathcal V_{(1)} \to (\mathcal{MOR} \times G)/\widehat{\mathcal G}_c
$$
such that for each $\hat u' \in U'(\epsilon)$, ${\bf x} \in \mathcal V_{(1)}$,
the element $\overline \Phi(\hat u',{\bf x})$ is contained in the 
subset 
$\overline{\mathcal W}(\epsilon_4)$
of $(\mathcal{MOR} \times G)/\widehat{\mathcal G}_c$
and $\overline \Phi(\hat u',{\bf x})$ is the unique point on 
$\overline{\rm Pro}^{-1}({\bf x})$ where 
$\overline{\rm meandist}$ attains its minimum.
\end{lem}
\begin{proof}
This is a consequence of Lemma \ref{lem732} and 
Lemma \ref{lem999}.
\end{proof}
\begin{rem}
We remark that elements of $U'(\epsilon)$ are $L^2_{m+1}$ maps.
Nevertheless our map $\overline\Phi$ is smooth. 
This is because the function ${\rm meansdist}$ is defined by 
using integration and so is a smooth function on the space of $L^2_{m+1}$ maps.
\end{rem}
The way we define $\overline{\rm meandist}$ in this subsection is 
slightly different from the way we defined it in Section \ref{sec:main}. However 
the local minimum is assumed at the same point.
More precisely we can prove the next lemma.
\par
To state the lemma we need notations.
Suppose $((\Sigma',\vec z'),u')$ is $G$-$\epsilon_1$-close to $((\Sigma,\vec z),u)$.
We take $\overline{\bf x} \in \mathcal V$ such that there exists an isomorphism
$$
\phi_0 : (\Sigma(\overline{\bf x}),\vec z(\overline{\bf x})) \cong (\Sigma',\vec z').
$$
Then Definition \ref{defn64} defines
$$
\overline{\rm meandist}^{\S\ref{sec:main}}:
\mathcal W(\epsilon_4;\overline{\bf x},\phi_0;((\Sigma',\vec z'),u'))/\widehat{\mathcal G}_c
\to \R
$$
\par
We assume that $\overline{{\bf x}}$ is in a small neighborhood of ${\bf x}_1$ and 
take ${\bf x} \in \mathcal V_{(1)}$ such that $\psi({\bf x}) = \overline{{\bf x}}$.
We then put
\begin{equation}\label{eq737}
\hat u' = u' \circ \phi_0 \circ \tilde{\psi}\vert_{
\Sigma_1({\bf x})} \circ \Phi_{1,{\bf x},\delta}
: 
\Sigma_1(\delta) \to X.
\end{equation}
Note that this is a composition 
$$
\Sigma_1(\delta) \longrightarrow \Sigma_1({\bf x})(\delta) \longrightarrow \Sigma(\psi({\bf x}))
\longrightarrow \Sigma' \longrightarrow X.
$$
We assume $\hat u' \in U'(\epsilon)$.
Then, by Definition \ref{defn744}, we obtain a function:
$$
\overline{\rm meandist}^{\S\ref{sec:decay}}:
\mathcal W(\epsilon_4;\psi({\bf x}),(\tilde{\psi}\vert_{
\Sigma_1({\bf x})})^{-1};((\Sigma_1({\bf x}),\vec z_1({\bf x})),u_{{\bf x}}))/\widehat{\mathcal G}_c
\to \R
$$
which is nothing but
$$
[\varphi,g] \mapsto \overline{\rm meandist}(\varphi,g,{\bf x},\hat u')
$$
where $\overline{\rm meandist}$ in the right hand side is one defined in Definition \ref{defn744}.
\begin{lem}\label{lem748}
We can shrink $\mathcal V_{(1)}$ and replace $\epsilon$ by a smaller positive number 
so that the next two conditions are equicalent.
\begin{enumerate}
\item
$[\varphi,g] \in \mathcal W(\epsilon_4;\overline{\bf x},\phi;((\Sigma',\vec z'),u'))/\widehat{\mathcal G}_c$
and $\overline{\rm meandist}^{\S\ref{sec:main}}$ takes the minimum there.
\item
$[\varphi,g] \in \mathcal W(\epsilon_4;\psi({\bf x}),(\tilde{\psi}\vert_{
\Sigma_1({\bf x})})^{-1};((\Sigma_1({\bf x}),\vec z_1({\bf x})),u_{{\bf x}}))/\widehat{\mathcal G}_c$
and $\overline{\rm meandist}^{\S\ref{sec:decay}}$ takes the minimum there.
\end{enumerate}
\end{lem}
\begin{proof}
Comparing formulas (\ref{eq737}), (\ref{form611mod}) and (\ref{form611})
it is easy to see that 
$\overline{\rm meandist}^{\S\ref{sec:main}} = \overline{\rm meandist}^{\S\ref{sec:decay}}$
when both are defined.
(We remark that the domains of the both functions are quotient of 
open subsets of $\mathcal{MOR} \times G$ by $\widehat{\mathcal G}_c$.
So the intersection of the domains makes sense.) In fact
$$
\aligned
&\hat u' \circ \Phi_{1,{\bf x},\delta}^{-1} \circ (\tilde \psi\vert_{\Sigma_{1}({\bf x})})^{-1}
\circ \varphi \circ \Phi_{{\bf x}',\sigma} \\
&=
u' \circ \phi_0 \circ \tilde \psi\vert_{\Sigma_{1}({\bf x})}  
 \circ \Phi_{1,{\bf x},\delta}^{-1} \circ  \Phi_{1,{\bf x},\delta}  \circ
(\tilde \psi\vert_{\Sigma_{1}({\bf x})})^{-1}
\circ \varphi \circ \Phi_{{\bf x}',\sigma} \\
&= u' \circ \phi_0 \circ  \varphi \circ \Phi_{{\bf x}',\sigma}.
\endaligned
$$
\par
In the case when $\psi({\bf x}) = {\bf x}_1$ and $u' = u_1$, the domain also coincides. So the lemma holds in that case.
We can shrink  $\mathcal V_{(1)}$ and replace $\epsilon$ by a smaller positive number
and use the strict covexity of $\overline{\rm meandist}^{\S\ref{sec:main}}$ and of $\overline{\rm meandist}^{\S\ref{sec:decay}}$ to show the general case.
\end{proof}
\par
Proposition \ref{prop78} is a consequence of Lemmas \ref{lem733} and \ref{lem748}
and a straight forward computation based on the definition.\footnote{
Experts of geometric analysis certainly will find that Proposition \ref{prop78} 
follows from Lemma \ref{lem733} and the definition 
immediately by inspection.} 
For completeness' sake\footnote{In fact this proof is related 
to the part we were asked to provide the detail by several people 
in the easier case when $G$ is trivial.}
we provide the detail of its proof below.
\par
Let $\mathcal V_{(1)}(\epsilon)$ be the $\epsilon$ neighborhood of 
$0$ in $\mathcal V_{(1)}$ as in Proposition \ref{prop78}.
If $\epsilon$ is small we can find a lift
$$
\Phi : U'(\epsilon) \times \mathcal V_{(1)}(\epsilon) \to  \mathcal{MOR} \times G
$$
of $\overline\Phi$. We put
\begin{equation}\label{form73666}
\Phi(\hat u',{\bf x}) = (\varphi(\hat u',{\bf x}),g(\hat u',{\bf x})).
\end{equation}
Thus \index{00G2hatuu@$g(\hat u',{\bf x})$}
$$
\varphi(\cdot) : U'(\epsilon) \times \mathcal V_{(1)}(\epsilon) \to \mathcal{MOR}, 
\quad
g(\cdot) : U'(\epsilon) \times \mathcal V_{(1)}(\epsilon) \to G,
$$
are smooth maps (from a Hilbert space to finite dimensional manifolds).
\par
We may take our lift $g(\cdot) : U'(\epsilon) \times \mathcal V_{(1)}(\epsilon) \to G$
so that its image lies in a neighborhood of $g_1$.
We denote by $U_G(g_1)$ this neighborhood.
\par
We calculate the finite dimensional subspace 
$E(\hat u',{\bf x})$ using local coordinate.
We cover $\Sigma(\delta)$ by a sufficiently small 
open sets $W_{\sigma}$.
$$
\Sigma(\delta) \subset \bigcup_{\sigma \in \mathcal S} W_{\sigma}.
$$
We will specify how small $W_{\sigma}$ is later.
We fix a complex coordinate of $W_{\sigma}$ and denote it by $z_{\sigma}$.

We first assume the following:
\begin{assump}
For each $\sigma$ there exists a convex open subset
$\Omega_{\sigma}$ of $X$ in one chart such that
the following holds for $\hat u' \in U'(\epsilon)$, ${\bf x}_1 \in V_{(1)}$.
\begin{enumerate}
\item
Note $(\Sigma_1,\vec z_1) = (\Sigma({\bf x}_1),\vec z({\bf x}_1))$.
We have
$
\Phi_{{\bf x}_1,\delta} : \Sigma(\delta) \to \Sigma_1.
$
We require
$$
u_1(\Phi_{{\bf x}_1,\delta}(W_{\sigma})) \subset \Omega_{\sigma}.
$$
\item
We also require
$$
\hat u'(\Phi_{{\bf x}_1,\delta}(W_{\sigma})) \subset \Omega_{\sigma}.
$$
\item
We also require
$$
g
u(W_{\sigma}) \subset \Omega_{\sigma},
$$ 
for $g \in U_{G}(g_1)$. 
\end{enumerate}
\end{assump}
Note that if the diameter of $W_{\sigma}$ is small and 
$\epsilon$, $\epsilon_4$ are small then 
the the diameter of the union of 
$u_1(\Phi_{{\bf x}_1,\delta}(W_{\sigma}))$, 
$\hat u'(\Phi_{{\bf x}_1,\delta}(W_{\sigma}))$ and 
$g_1u(W_{\sigma})$ is small.
In fact $\hat u'$ is close to $u_1$  and $u_1\circ \Phi_{{\bf x}_1,\delta}$ is close to 
$g_1 u$.
Therefore we may assume the existence of $\Omega_{\sigma}$.
\par
Let $\partial_{\sigma}^1,\dots,\partial_{\sigma}^d$ be a 
local (complex) frame of the complex tangent bundle $TX$ on $\Omega_{\sigma}$.
\begin{defn}
We define a (complex) matrix valued smooth function 
$
({\rm Pal}^i_j(p,q))_{i,j=1}^{\dim X}
$ 
on $\Omega_{\sigma}^2$ with the following properties.
Let $p,q \in \Omega_{\sigma}^2$.
We take the shortest geodesic $\gamma$ joining $p$ to $q$.
Using local frames $\partial_{\sigma}^i$ at $p,q$ 
and the parallel transportation
(with respect to an appropriate hermitian connection)
$$
{\rm Pal}_p^q : T_pX \to T_qX
$$
along $\gamma$, we define
\begin{equation}\label{form730}
{\rm Pal}_p^q(\partial_{\sigma}^j)
= 
\sum_i  {\rm Pal}^i_j(p,q)\partial_{\sigma}^i.
\end{equation}
\end{defn}
Other than parallel transportation,
the differentials of $\Psi$ (See (\ref{psiinSec6}).)
and of $\Phi_{1,{\bf x},\delta}$  (See (\ref{formula66rev}).)
appear in the definition of $E(\hat u',{\bf x})$.
We write them by local coordinate below.
\par
Let $\hat u' \in U'(\epsilon)$, ${\bf x} \in \mathcal V_{(1)}(\epsilon)$.
We put
$$
\overline{\bf y} = \overline{\bf y}(\hat u',{\bf x})
= {\rm Pr}_s(\varphi(\hat u',{\bf x})),
$$
where $\varphi$ is defined by (\ref{form73666}).
(\ref{psiinSec6}) in this case becomes:
\begin{equation}\label{newform732}
\Psi_{\hat u',{\bf x}}
= (\Phi_{\overline{\bf y}(\hat u',{\bf x}),\delta})^{-1}
\circ (\varphi(\hat u',{\bf x}))^{-1}
\circ \tilde{\psi}\vert_{
\Sigma_1({\bf x}))}
: 
\Sigma_1({\bf x})(\delta') \to \Sigma(\delta).
\end{equation} 
Note that $\phi_0 = ( \tilde{\psi}\vert_{
\Sigma_1({\bf x}))})^{-1}$ in our case as we mentioned 
in Definition \ref{defn7422}.
(\ref{newform732}) is a composition
$$
\Sigma_1({\bf x})(\delta') \longrightarrow 
\Sigma(\psi({\bf x}))(\delta)
\longrightarrow  \Sigma(\overline{\bf y})(\delta)
\longrightarrow  \Sigma(\delta).
$$

We compose it with $\Phi_{1,{\bf x},\delta'}$
to obtain
\begin{equation}\label{form731}
\Psi_{\hat u',{\bf x}} \circ  \Phi_{1,{\bf x},\delta'}
:\Sigma_1(\delta') \to \Sigma(\delta) .
\end{equation}
Note the source and the target is independent of 
$(\hat u',{\bf x})$.
This family of maps depends smoothly on $\hat u',{\bf x}$.
\begin{assump}
There exists a coordinate chart $W_{1,\sigma}$ of 
$\Sigma_1$ independent of $\hat u',{\bf x}$ such that
$$
(\Psi_{\hat u',{\bf x}}\circ  \Phi_{1,{\bf x},\delta})
(W_{1,\sigma}) \supset W_{\sigma}.
$$
Moreover 
$$
\Psi_{\hat u',{\bf x}}(\Phi_{1,{\bf x},\delta}(W_{1,\sigma}))
$$
is contained in a coordinate chart $W^+_{\sigma}$
containing $W_{\sigma}$, to which the coordinate $z_{\sigma}$
extends.
\end{assump}
By choosing $\epsilon$ small and $W_{\sigma}$ small 
we can assume that such $W_{1,\sigma}$ exists.
We fix a complex coordinate $z_{1,\sigma}$ of $W_{1,\sigma}$.
\par
Using complex linear part of the differential of 
$\Psi_{\hat u',{\bf x}}$ we obtain a bundle map
$$
d^h\Psi_{\hat u',{\bf x}}  :  \Lambda^{01}\Sigma(\delta)
\to  \Lambda^{01}\Sigma_1({\bf x}).
$$
We also have
$$
d^h\Phi_{1,{\bf x},\delta} : \Lambda^{01}\Sigma_1({\bf x})
\to \Lambda^{01}\Sigma_1.
$$
We denote the composition of them by
$$
d^h\Phi_{1,{\bf x},\delta}\circ d^h\Psi_{\hat u',{\bf x}}
: \Lambda^{01}\Sigma(\delta)
\to \Lambda^{01}\Sigma_1.
$$
This is a bundle map which covers the (local) inverse of (\ref{form731}).
\begin{lem}
There exists a smooth function 
$$
\frak f : U'(\epsilon) \times \mathcal V_{(1)}(\epsilon)
\times W_{1,\delta} \to \C
$$
such that
\begin{equation}\label{form732}
(d^h\Phi_{1,{\bf x},\delta} \circ d^h\Psi_{\hat u',{\bf x}})
(d\overline z_{\sigma}) (w)
= \frak f(\hat u',{\bf x},w)d\overline z_{1,\sigma}(w),
\end{equation}
where $w \in  W_{1,\delta}$.
\end{lem}
\begin{proof}
This is immediate from smooth dependence of 
$\Psi_{\hat u',{\bf x}}$ and 
$\Phi_{1,{\bf x},\delta}$ on $\hat u',{\bf x}$.
\end{proof}
We next write $G$ action by local coordinate.
We recall that $U_G(g_1)$ is a neighborhood of $g_1$ in $G$ such that 
the image of the map $g(\cdot)$ is contained in $U_G(g_1)$.
\begin{assump}
We take $\Omega_{\sigma}$ so that 
there exist coordinate neighborhoods $\Omega^0_{\sigma} \subset \Omega^+_{\sigma}$
such that
$$
\Omega^0_{\sigma}
\subset g^{-1} \Omega_{\sigma} \subset \Omega^+_{\sigma},
\qquad 
u(W_{\sigma}) \subset \Omega^0_{\sigma}.
$$
for any $g \in U_G(g_1)$.
\end{assump}
We can find such $\Omega^0_{\sigma}$, $\Omega^+_{\sigma}$
by taking $\epsilon$ sufficiently small.
(Note that $g_1 u$ is close to $\hat u' \circ \Phi_{{\bf x}_1,\delta}$.)
\par
Let $\partial_{\sigma,0}^i$, $i=1,\dots,d$, be a local frame of the complex tangent bundle 
$TX$ on $\Omega^+_{\sigma}$.
\begin{lem}
There exists a matrix valued smooth function 
$(G(\hat u',{\bf x},p)_i^j)_{i,j=1}^{\dim X}$
on $U'(\epsilon) \times \mathcal V_{(1)}(\epsilon) \times \Omega_{\sigma}$
such that
\begin{equation}\label{form733}
\left(d g(\hat u',{\bf x}) 
( \partial_{\sigma,0}^i)\right)(p)
=
\sum_j G(\hat u',{\bf x},p)_i^j \partial_{\sigma}^j
(p).
\end{equation}
The map $d g(\hat u',{\bf x})$ is the differential of the 
map defined by $g(\hat u',{\bf x}) \in G$ action.
\par
Note $G(\cdot)_i^j : U'(\epsilon) \times \mathcal V_{(1)}(\epsilon)
\times  \Omega_{\sigma} \to \C$
is a map from the product of Hilbert space and a finite dimensional 
manifold to the complex plane.
\end{lem}
\begin{proof}
The proof is immediate from the 
smoothness of $(\hat u',{\bf x}) \mapsto g(\hat u',{\bf x})$, that is 
Lemma \ref{lem733} and (\ref{form73666}).
\end{proof}
We now write the map 
$$
E((\Sigma,\vec z),u) \to C^{\infty}(\Sigma_1;u_1^*TX\otimes \Lambda^{01})
$$
which we use to define
$$
E(\hat u',{\bf x}) \subset C^{\infty}(\Sigma_1;u_1^*TX\otimes \Lambda^{01})
$$
explicitly using smooth functions appearing in 
(\ref{form730}), (\ref{form732}), (\ref{form733}) etc..
\par
Let 
$$
e \in C^{\infty}(W_{\sigma};u^*TX\otimes \Lambda^{01})
$$
has compact support. We write
\begin{equation}
e = \sum_i e^i \partial_{\sigma,0}^i \otimes d\overline z_{\sigma}.
\end{equation}
Here $e^i$ is a smooth function on $W_{\sigma}$.
By (\ref{form733}) we have
\begin{equation}\label{form753}
(g(\hat u',{\bf x})_* (e))(w)
= 
\sum_{i,j} G(\hat u',{\bf x},w)_i^j 
e^i(w) \partial_{\sigma}^j,
\end{equation}
for $w \in W_{\sigma}$.
Now we apply  the maps
$$
I_{{\bf x}_0,\phi_0;((\Sigma',\vec z^{\,\prime}),u')}
: C^{\infty}(K;(gu)^*TX \otimes \Lambda^{01})
\to L^2_{m+1}(\Sigma'(\delta);(u')^*TX \otimes \Lambda^{01})
$$
and
$$
I_{\hat u',{\bf x}} :
L^2_{m+1}(\Sigma_1({\bf x})(\delta);(u')^*TX\otimes \Lambda^{01})
\to 
L^2_{m+1}(\Sigma_1(\delta);u_1^*TX\otimes \Lambda^{01})
$$
to the right hand side of (\ref{form753}).
(Note they are the maps
(\ref{formula663333}) and
(\ref{form711}), respectively.)
\footnote{We extend (\ref{formula663333}) to the case when $u'$ is 
in the Sobolev space of $L^{2}_{m+1}$ maps. So the target of 
$I_{{\bf x}_0,\phi_0;((\Sigma',\vec z^{\,\prime}),u')}$ here is 
$L^2_{m+1}$ space.}
\par
Note that we take 
$({\bf x}_0,\phi_0) = (\psi({\bf x}),(\tilde{\psi}\vert_{
\Sigma(\psi({\bf x}))})^{-1})$, 
$(\Sigma',\vec z^{\,\prime}) = (\Sigma_1({\bf x}),\vec z_1({\bf x}))$
and
$$
u' = \hat u' 
\circ \Phi_{1,{\bf x},\delta}^{-1} : \Sigma_1({\bf x})(\delta) \to X.
$$
by (\ref{defnofuprime}) and $g = g(\hat u',{\bf x})$.
\par
(\ref{form730}), (\ref{form732}) and the definition implies
that for $w\in W_{1,\sigma}$
\begin{equation}\label{form737}
\aligned
I_{\hat u',{\bf x}}
&(I_{{\bf x}_0,\phi_0;((\Sigma',\vec z^{\,\prime}),u')}
(g(\hat u',{\bf x})_* (e)))(w) \\
=\sum_{j,j_1,j_2,i}&{\rm Pal}^{j}_{j_1}(
\hat u'(w),u_1(w))
\\
&\times
{\rm Pal}^{j_1}_{j_2}(
g(\hat u',{\bf x})((u\circ \Psi_{\hat u',{\bf x}}
\circ \Phi_{1,{\bf x},\delta})(w)),\hat u'(w))
\\
&\times
G(\hat u',{\bf x},(\Psi_{\hat u',{\bf x}}
\circ \Phi_{1,{\bf x},\delta})(w))_i^{j_2}
e^i((\Psi_{\hat u',{\bf x}}
\circ \Phi_{1,{\bf x},\delta})(w))
\\
&\times
\frak f(\hat u',{\bf x},w)
\partial_{\sigma}^j\otimes d\overline z_{1,\sigma}.
\endaligned
\end{equation}
Here $\Psi_{\hat u',{\bf x}}$, is as in (\ref{newform732}).
See Figure \ref{Figure737}.
\begin{figure}[h]
\centering
\includegraphics[scale=0.3]{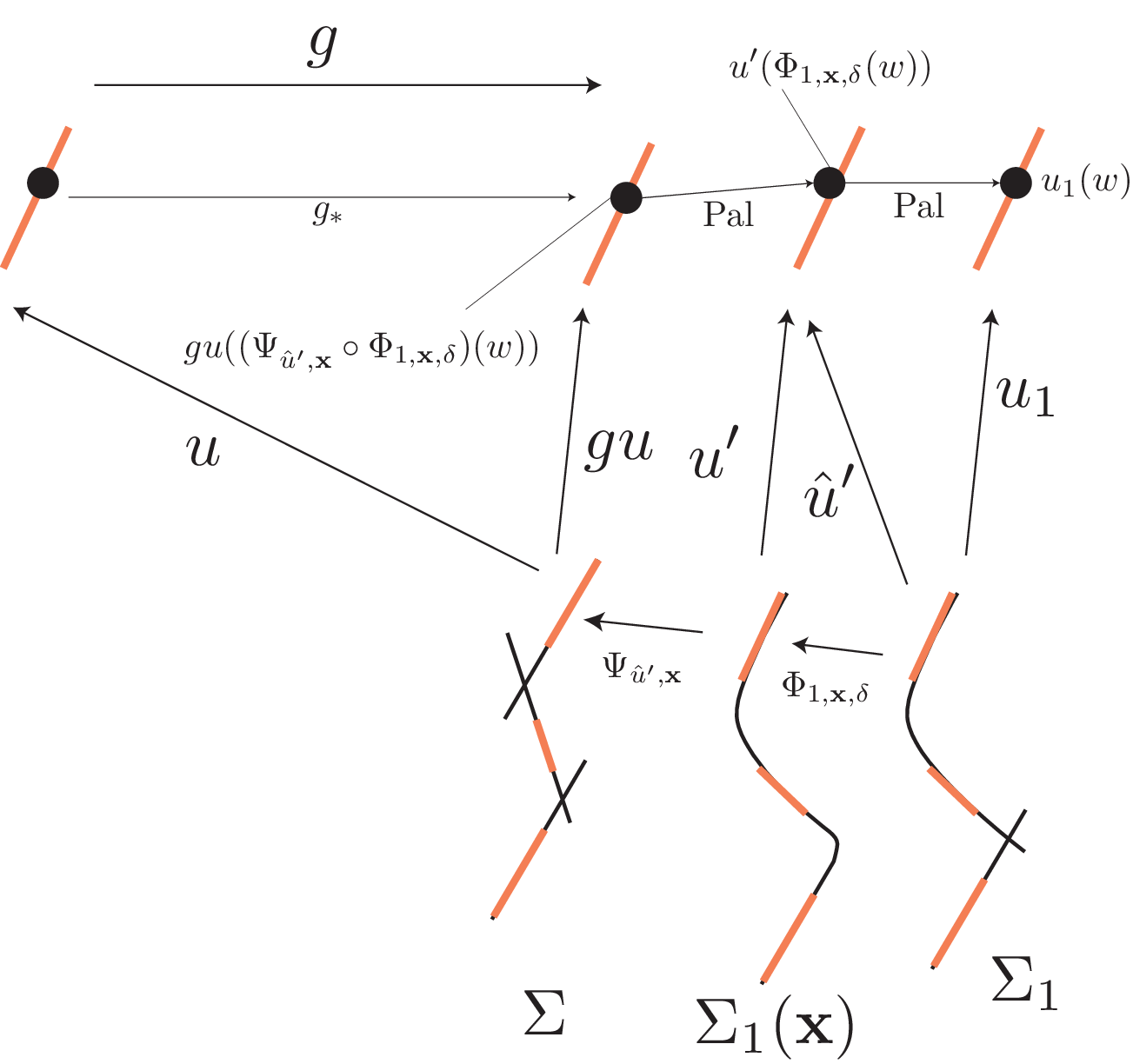}
\caption{$I_{\hat u',{\bf x}}
\circ I_{{\bf x}_0,\phi_0;((\Sigma',\vec z^{\,\prime}),u')}$}
\label{Figure737}
\end{figure}
\begin{lem}\label{lem740}
We fix $e$ and regard (\ref{form737}) as a map
$$
U'(\epsilon) \times \mathcal V_{(1)}(\epsilon)
\to L^2_{m+1}(\Sigma_1(\delta);u_1^*TX\otimes \Lambda^{01}).
$$
Then it is a smooth map between Hilbert spaces.
\end{lem}
\begin{proof}
Note $e^i$ and $u$ are fixed smooth maps. 
Moreover 
$$
(\hat u',{\bf x}) \mapsto \Psi_{\hat u',{\bf x}},
\qquad
(\hat u',{\bf x}) \mapsto \Phi_{1,{\bf x},\delta}
$$
are smooth families of smooth maps.
Note even though $\hat u'$ is only of $L^2_{m+1}$ 
class and is not smooth, the family  
$\Psi_{\hat u',{\bf x}}$ is a smooth family of smooth maps. In fact
$\hat u'$ are involved here only through  
$g(\hat u',{\bf x})$
and $\varphi(\hat u',{\bf x})$. By Lemma \ref{lem733} both
$g(\hat u',{\bf x})$
and $\varphi(\hat u',{\bf x})$ are smooth with respect to $\hat u' \in L^2_{m+1}$.
$\Phi_{1,{\bf x},\delta}$ is independent of $\hat u'$ and 
depend smoothly on ${\bf x}$.
\par
Therefore
$$
(\hat u',{\bf x}) \mapsto
e^i\circ \Psi_{\hat u',{\bf x}}
\circ \Phi_{1,{\bf x},\delta},
\quad
(\hat u',{\bf x}) \mapsto
u\circ \Psi_{\hat u',{\bf x}}
\circ \Phi_{1,{\bf x},\delta}
$$
are smooth maps.
\par
The lemma then follows immediately from 
the smoothness of $g(\cdot,\cdot)$, 
${\rm Pal}^{i}_{j}(\cdot)$, $G(\cdot)_j^i$, $\frak f(\cdot,\cdot,\cdot)$.
(We use also the fact that 
$v \mapsto F \circ v$ is a smooth map between
$L^2_{m+1}$ spaces if $F$ is a smooth map and $m$ is sufficiently 
large.)
\end{proof}
Now we are in the position to complete the proof of 
Proposition \ref{prop78}.
We take a partition of unity $\chi_{\sigma}$ subordinate 
to the covering $W_{\sigma}$.
Let $e_1,\dots,e_d$ be a basis of 
$E((\Sigma,\vec z),u)$.
We put
\begin{equation}\label{from74300}
e_i(\hat u',{\bf x})
= I_{\hat u',{\bf x}}
(I_{{\bf x}_0,\phi_0;((\Sigma',\vec z^{\,\prime}),u')}
(g(\hat u',{\bf x})_* (e_i)))
\end{equation}
as in the right hand side of (\ref{form737}).
By definition 
$(e_i(\hat u',{\bf x}))_{i=1}^d$ is a basis of 
$E(\hat u',{\bf x})$.
\par
On the other hand since
$$
e_i(\hat u',{\bf x})
= 
\sum_{\sigma }
I_{\hat u',{\bf x}}
(I_{{\bf x}_0,\phi_0;((\Sigma',\vec z^{\,\prime}),u')}
(g(\hat u',{\bf x})_* (\chi_{\sigma}e_i)))
$$
Lemma \ref{lem740}
implies that
$(\hat u',{\bf x}) \mapsto e_i(\hat u',{\bf x})$ is smooth.
The proof of Proposition \ref{prop78} is complete.
\qed

\subsection{Independence of the local smooth structure of the choices}
\label{indsmoothchart}

In this subsection we prove Proposition \ref{prop720}.
Let $\frak p_i = ((\Sigma_i,\vec z_i),u_i) \in U(((\Sigma,\vec z),u);\epsilon_2)$
for $i=1,2$ and we take 
strong stabilization data $({\frak W}^{(i)},\vec{\mathcal N}_i)$ (Definition \ref{defn720})
at $\frak p_i$ for $i=1,2$.
\par
We obtained a map \index{00I6iepsilon@$\mathscr I_{i,\epsilon}$} 
$$
\mathscr I_{i,\epsilon_{(i)}} : V(\frak p_i;\epsilon_{(i)},({\frak W}^{(i)},\vec{\mathcal N}_i))
\to U(((\Sigma,\vec z),u);\epsilon_2)
$$
which is $\mathcal G_i$-equivariant for sufficiently small
$\epsilon_2$.
(Note $\mathcal G_i$ is the group of automorphisms of 
$\frak p_i = ((\Sigma_i,\vec z_i),u_i)$ and is a finite group.
$\epsilon_{(1)} = \epsilon$, $\epsilon_{(2)} = \epsilon'$.)
\par
In fact
\begin{equation}
\mathscr I_{i,\epsilon_{(i)}}({\bf v},{\bf x})
=
[(\Sigma_i({\bf x}),\vec z_i({\bf x})),u^i_{{\bf v},{\bf x}}].
\end{equation}
See Proposition \ref{prop710} and Definition \ref{defn71474}.
Note $u^i_{{\bf v},{\bf x}}$ is $u_{{\bf v},{\bf x}}$ 
in Proposition \ref{prop710}. Since this map depends on 
$\frak p_i$ and ${\frak W}^{(i)}$ we put superscript $i$ 
and write $u^i_{{\bf v},{\bf x}}$.
\par
Suppose $\frak p_2 = ((\Sigma_2,\vec z_2),u_2)$
is $\epsilon$-close to 
$\frak p_1 = ((\Sigma_1,\vec z_1),u_1)$
for some $\epsilon$ depending on $\frak p_1$.
Then $\mathcal G_2 \subset \mathcal G_1$.
\par
To prove Proposition \ref{prop720} 
it suffices to find a $\mathcal G_2$-equivariant $C^n$ open embedding
$$
\mathscr J_{12;\epsilon,\epsilon'} : V(\frak p_2;\epsilon',({\frak W}^{(2)},\vec{\mathcal N}_2)) \to 
V(\frak p_1;\epsilon,({\frak W}^{(1)},\vec{\mathcal N}_1))
$$
such that \index{00J612epsilon@$\mathscr J_{12;\epsilon,\epsilon'}$}
\begin{equation}\label{formform740}
\mathscr I_{1,\epsilon} \circ 
\mathscr J_{12;\epsilon,\epsilon'}
= 
\mathscr I_{2,\epsilon'},
\end{equation}
for sufficiently small $\epsilon'$.
\par
Existence of such a map 
$\mathscr J_{12;\epsilon,\epsilon'}$ 
(set theoretically) 
is 
a consequence of Proposition \ref{prop710}.
We will prove that it is a $C^n$ map using the exponential 
decay estimate, Proposition \ref{prop711}.
We will prove Lemma \ref{lem71818} at the same time.
The detail follows.
\par
Our proof is divided into various cases.
In the first four cases we assume $\frak p_1 = \frak p_2$.
\par\medskip
\noindent(Case 1) 
We assume $\frak p_1 = \frak p_2 = ((\Sigma_1,\vec z_1),u_1)$.
We also require $({\frak W}^{(1)},\vec{\mathcal N}_1) \subseteq ({\frak W}^{(2)},\vec{\mathcal N}_2)$
in the following sense.
\begin{enumerate}
\item[(1-1)]
Let $\vec w_1^{\,(i)} = (w_{1,1}^{(i)},\dots,w_{1,k_i}^{(i)})$.
We assume $k_1 \le k_2$ and
$
w_{1,j}^{(1)} = w_{1,j}^{(2)}
$
for $j=1,\dots,k_1$.
\item[(1-2)]
We require
$
\mathcal N_j^{(1)} = \mathcal N_j^{(2)}
$
for $j=1,\dots,k_1$.
\end{enumerate}
We consider an open neighborhood $\mathcal V_{(i)} 
\subset \mathcal M_{g,\ell+k_i}$ of $(\Sigma_1,\vec z_1
\cup \vec w_1^{\,(i)})$ and the universal 
family of deformation $\pi_{(i)} : \mathcal C_{(i)} \to 
\mathcal V_{(i)}$ on it.
It comes with sections $\frak t^{(i)}_j :  \mathcal V_{(i)} \to \mathcal C_{(i)}
$, $j=1,\dots,\ell + k_i$,
which assigns the $j$-th marked point.
\begin{lem}
There exists holomorphic maps $\tilde\psi : \mathcal C_{(2)} \to \mathcal C_{(1)}$
and $\psi : \mathcal V_{(2)} \to \mathcal V_{(1)}$ 
such that the following holds.
\begin{enumerate}
\item
The next diagram commutes and is cartesian.
\begin{equation}\label{diagram741}
\begin{CD}
\mathcal C_{(2)} @ >{\tilde\psi}>> \mathcal C_{(1)},\\
@ VV{\pi_{(2)}}V @ VV{\pi_{(1)}}V\\
\mathcal V_{(2)} @ >{\psi}>> \mathcal V_{(1)},
\end{CD}
\end{equation}
\item
The next diagram commutes for $j=1,\dots,\ell + k_1$.
\begin{equation}\label{diagram7412}
\begin{CD}
\mathcal C_{(2)} @ >{\tilde\psi}>> \mathcal C_{(1)},\\
@ AA{\frak t_j^{(2)}}A @ AA{\frak t_j^{(1)}}A\\
\mathcal V_{(2)} @ >{\psi}>> \mathcal V_{(1)},
\end{CD}
\end{equation}
\item
$\tilde\psi$ and $\psi$ are $\mathcal G_2$ equivariant.
\item
$\psi$ is a submersion and the complex dimension of its fibers are 
$k_2 - k_1$.
\end{enumerate}
\end{lem}
\begin{proof}
By forgetting $k_1+1,\dots,k_2$-th marked points 
$\pi_{(2)} : \mathcal C_{(2)} \to 
\mathcal V_{(2)}$ becomes a deformation of 
$(\Sigma_1,\vec z_1
\cup \vec w_1^{\,(1)})$.
Therefore we obtain desired maps $\tilde \psi$ and $\psi$ by the 
universality of $\pi_{(1)} : \mathcal C_{(1)} \to 
\mathcal V_{(1)}$ (together with $\frak t_j^{(1)}$'s.)\footnote{Remark \ref{rem741} also applies here also.}
\end{proof}
\begin{proof}
[The proof of Lemma \ref{lem71818} in Case 1]
\par
For the proof of Lemma \ref{lem71818} 
we consider the situation when we are given  
stabilization and trivialization data  ${\frak W}^{(i)}$.
We assume Item (1-1) only. ((1-2) does not make sense.)
We assume
$
[(\Sigma_2,\vec z_2),u_2] = 
[(\Sigma_1,\vec z_1),u_1]
$
and 
$$
[(\Sigma(c),\vec z_c),u_c] \in \mathcal U(\epsilon(c);(\Sigma_2,\vec z_2),u_2,{\frak W}^{(2)})$$
with $\lim_{c\to\infty}\epsilon(c) \to 0$.
It suffices to show 
$[(\Sigma(c),\vec z_c),u_c] \in
\mathcal U(\epsilon;(\Sigma_1,\vec z_1),u_1,{\frak W}^{(1)})$
for sufficiently large $c$.
By assumption there exists ${\bf x}_c \in \mathcal V_{(2)}$
converging to the origin $o$, the $k_2$ extra marked points $\vec w_c \subset 
\Sigma(c)$  and isomorphisms
$$
\phi_c : (\Sigma^{(2)}({\bf x}_c),\vec z^{\,(2)}({\bf x}_c)
\cup \vec w^{\,(2)}({\bf x}_c))
\to (\Sigma(c),\vec z_c \cup \vec w_c).
$$
(Here $(\Sigma^{(2)}({\bf x}_c),\vec z^{\,(2)}({\bf x}_c)
\cup \vec w^{\,(2)}({\bf x}_c))$ is  a
marked stable curve of genus $g$ and $\ell+k_2$ marked 
points representing ${\bf x}_c$. We identify $\Sigma(c)$
with the fiber $\pi_{(2)}^{-1}({\bf x}_c)$.)
Moverover there exists $\delta_c < \epsilon(c)$ such that:
\begin{enumerate}
\item
The $C^2$ norm of the difference between $u_c \circ 
\phi_c \circ  \Phi_{2,{\bf x}_c,\delta_c}$
and $u_2$ is smaller than $o(c)$.
\footnote{See Remark \ref{rem730} for the definition of $o(c)$.}
\item
The map $u_c \circ 
\phi_c$ has diameter $< o(c)$ on 
$\Sigma^{(2)}({\bf x}_c) \setminus {\rm Im}(\Phi_{2,{\bf x}_c,\delta_c})$.
\end{enumerate}
We put ${\bf x}'_c = \psi({\bf x}_c)$.
We define $\vec w^{\, \prime}_c$ by 
forgetting the last $k_2 - k_1$ marked points of   $\vec w_c$.
We have an isomorphism
$$
\phi'_c : (\Sigma^{(1)}({\bf x}'_c),\vec z^{\,(1)}({\bf x}'_c)
\cup \vec w^{\,(1)}({\bf x}'_c))
\to (\Sigma(c),\vec z_c \cup \vec w^{\, \prime}_c).
$$
We have
$$
\phi'_c \circ \tilde{\psi}\vert_{\Sigma^{(2)}({\bf x}_c)}
= \phi_c.
$$
\par
Note $\Phi_{2,{\bf x}_c,\delta_c}$ and 
$\Phi_{1,{\bf x}'_c,\delta_c}$
both converge to the identity map as maps
$\Sigma_i(\delta_c) \to \mathcal C_{(i)}$,
in $C^2$ topology.
Therefore the $C^2$ difference between
$$
\tilde{\psi}\vert_{\Sigma^{(2)}({\bf x}_c)}
\circ \Phi_{2,{\bf x}_c,\delta_c}
\quad
\text{and}
\quad
\Phi_{1,{\bf x}'_c,\delta_c}
$$
goes to $0$ as $c\to \infty$.
Therefore the $C^2$ difference between
$$
u_c \circ 
\phi_c \circ  \Phi_{2,{\bf x}_c,\delta_c}
\quad
\text{and}
\quad
u_c \circ 
\phi'_c \circ  \Phi_{1,{\bf x}'_c,\delta_c}
$$
goes to $0$ as $c\to \infty$.
Therefore by (1) the 
$C^2$ difference between
$$
u_1
\quad
\text{and}
\quad
u_c \circ 
\phi'_c \circ  \Phi_{1,{\bf x}'_c,\delta_c}
$$
goes to $0$ as $c\to \infty$. (Note $u_1 = u_2$.)
\begin{sublem}\label{sublem751}
The map $u_c \circ \phi'_c$ has diameter $<o(c)$ 
on $\Sigma^{(1)}({\bf x}'_c) \setminus 
{\rm Im}(\Phi_{1,{\bf x}'_c,\delta_c})$.
\end{sublem}
\begin{proof}
There exists $\delta^+_c \to 0$ such that 
$\delta^+_c > \delta_c$ and 
$$
{\rm Im}(\Phi_{1,{\bf x}'_c,\delta_c})
\supset
\tilde\psi ({\rm Im}(\Phi_{2,{\bf x}_c,\delta^+_c})).
$$
Let $W$ be a connected component of 
$\Sigma^{(1)}({\bf x}_{c'}) \setminus 
{\rm Im}(\Phi_{1,{\bf x}'_c,\delta_c})$.
There exists a connected component $W_+$ of 
$\Sigma^{(1)}({\bf x}_{c'}) \setminus 
\tilde\psi ({\rm Im}(\Phi_{2,{\bf x}_c,\delta^+_c}))$
which contains it.
It suffices to show
\begin{equation}\label{form748}
\lim_{c\to\infty}{\rm Diam} (u_c\circ \phi'_c)(W_+) = 0.
\end{equation}
Note 
$$
\partial W_+\ = \tilde{\psi}
(\partial({\rm Im}(\Phi_{2,{\bf x}_c,\delta^+_c}))
=
(\tilde{\psi}\circ\Phi_{2,{\bf x}_c,\delta^+_c})
(\partial \Sigma_1(\delta_c^+)).
$$
On $\partial \Sigma_1(\delta_c^+)$, the map
$u_c \circ 
\phi'_c \circ \tilde{\psi}\circ \Phi_{2,{\bf x}_c,\delta^+_c}
= u_c \circ 
\phi_c \circ \Phi_{2,{\bf x}_c,\delta^+_c}
$ is $C^2$ close to $u_1$. 
Since $\delta_c^+ \to 0$, 
$$
\lim_{c\to\infty} {\rm Diam} (u_c \circ 
\phi_c \circ \Phi_{2,{\bf x}_c,\delta^+_c})(\partial \Sigma_1(\delta_c^+)) = 0.
$$
Therefore
$$
\lim_{c\to\infty}{\rm Diam}(u_c\circ \phi'_c)(\partial W_+) = 0.
$$
Since $u_c\circ \phi'_c$ is 
holomorphic on $W_+$ 
(this is because it satisfies the equation (\ref{eqmain615})
in Definition \ref{defn61616}
and the supports of the elements of the obstruction 
spaces are away from $W_+$), the formula (\ref{form748}) follows.
\end{proof}

Therefore 
$[(\Sigma(c),\vec z_c),u_c] \in
\mathcal U(\epsilon;(\Sigma_1,\vec z_1),u_1,{\frak W}^{(1)})$
for sufficiently large $c$.
\end{proof}
\begin{proof}[Proof of Proposition \ref{prop720} in Case 1]
For ${\bf x} \in \mathcal V_{(i)}$
we denote by $(\Sigma_1^{(i)}({\bf x}),\vec z_1^{\,(i)}({\bf x}) \cup 
\vec w_1^{\,(i)}({\bf x}))$ the fiber $\pi_{(i)}^{-1}({\bf x})$
together with marked points.
\par
For ${\bf x} \in \mathcal V_{(2)}$,
(\ref{loctrimap3}) defines an open embedding
\begin{equation}\label{form7431}
\Phi^{(2)}_{1,{\bf x},\delta'} : \Sigma_1(\delta') \to \Sigma_1^{(2)}({\bf x})
\end{equation}
which is canonically determined by the data ${\frak W}^{(2)}$.
The restriction of $\tilde \psi$ to the fiber 
$\Sigma_1^{(2)}({\bf x})$ defines a map (isomorphism)
\begin{equation}\label{form7442}
\tilde \psi_{\bf x} : \Sigma_1^{(2)}({\bf x}) \to \Sigma_1^{(1)}(\psi({\bf x})).
\end{equation}
If $\delta'$ is sufficiently small compared to $\delta$, 
we compose the maps 
the inverse of (\ref{form7431}), (\ref{form7442}) and 
$\Phi^{(1)}_{1,\psi({\bf x}),\delta}$ (which is defined also by (\ref{loctrimap3})) to obtain
\begin{equation}\label{newform750}
\Psi_{{\bf x}} = (\Phi^{(2)}_{1,{\bf x},\delta'})^{-1} \circ \tilde \psi_{\bf x}^{-1} \circ \Phi^{(1)}_{1,\psi({\bf x}),\delta}
: \Sigma_1(\delta) \to \Sigma_1(\delta').
\end{equation}
The next lemma is obvious.
\begin{lem}\label{lem74222}
The map $\hat\Psi: \mathcal V_{(2)} \to C^{\infty}(\Sigma_1(\delta),\Sigma_1(\delta'))$
which assigns $\Psi_{{\bf x}}$ to ${\bf x}$ 
is a $C^{\infty}$ map.
\end{lem}
We next recall the following standard fact.
\begin{lem}\label{lem743111}
The map 
$$
{\rm comp} : L^2_{m+n+1}(\Sigma_1(\delta'),X) \times C^{\infty}(\Sigma_1(\delta),\Sigma_1(\delta'))
\to L^2_{m+1}(\Sigma_1(\delta),X)
$$
defined by
$$
{\rm comp}(F,\phi) = F \circ \phi
$$
is a $C^n$ map in a neighborhood of $(F_0,\phi_0)$
if $m > 10$ and $\phi_0$ is an open embedding.
\end{lem}
We take sufficiently large $m$ and put $m_1 = m+n$, $m_2 = m+2n$.
Note
$\mathcal V_{\rm map}(\epsilon)$ is the 
$\epsilon$ neighborhood of $0$ in 
${\rm Ker}^+ D_{u_1}\overline{\partial}$.
So this space is the same for ${\frak W}^{(1)}$ and 
${\frak W}^{(2)}$.
\par
We next define a map
$$
\mathcal R_{(i)} :
\mathcal V_{\rm map}(\epsilon)
\times \mathcal V_{(i)}(\epsilon)
\to 
L^2_{m_i+1-n}(\Sigma_1(\delta^{(i)}),X) 
\times \mathcal V_{(i)}(\epsilon).
$$
Here $\delta^{(1)} = \delta$, $\delta^{(2)} = \delta'$ 
and $\mathcal V_{(i)}(\epsilon)$ is defined as follows.
Recall for ${\bf x} \in \mathcal V_{(1)}(\epsilon)$ and 
${\bf v} \in \mathcal V_{\rm map}(\epsilon)$ 
the map $u^i_{{\bf v},{\bf x}}$
is defined. (See Proposition \ref{prop710}.)
$\mathcal V_{(1)}(\epsilon)$ is an 
open neighborhood of $(\Sigma_1,\vec z_1
\cup \vec w_i)$ in $\mathcal M_{g,\ell+k_i}$. 
Note $\mathcal V_{(1)}(\epsilon)$ are actually ${\frak W}^{(i)}$
and $i$ dependent.
We  define: \index{00R3o@$\mathcal R_{(i)}$}
\begin{equation}\label{753753}
\mathcal R_{(i)}({\bf v},{\bf x})
= (u^i_{{\bf v},{\bf x}}\circ \Phi^{(i)}_{1,{\bf x},\delta^{(i)}},{\bf x})
\end{equation}
\begin{lem}\label{lemma74444}
We put the smooth structure on $\mathcal V_{(i)}(\epsilon)$ as in 
Definition \ref{defn712}.
Then $\mathcal R_{(i)}$ is a $C^n$ embedding for sufficiently 
small $\epsilon$.
\end{lem}
\begin{proof}
The fact that $\mathcal R_{(i)}$ is a $C^n$ map is 
a consequence of Proposition \ref{prop711}.
The derivative of the first factor of $\mathcal R_{(i)}$ at $(0,o_i)$
restricts to an embedding
$$
T_0V_{\rm map}(\epsilon)
\to L^2_{m_i+1-n}(\Sigma_1(\delta^{(i)}),u_i^*TX) .
$$
Here $\mathcal R_{(i)}(0,o_i) = (u_i,o_i)$ and $o_i = [\Sigma_1,\vec z_1 \cup \vec w_i]$.
The injectivity of this map is a consequence of unique continuation.
Note $\mathcal V_{(i)}(\epsilon)$ factor of $\mathcal R_{(i)}$
is $({\bf v},{\bf x}) \mapsto \bf x$. 
Therefore the derivative of $\mathcal R_{(i)}$
is injective at $(0,o_i)$.
The lemma now follows from the inverse function theorem.
\end{proof}
We define a map
$$
\Phi : L^2_{m_2+1-n}(\Sigma_1(\delta^{(2)}),X) 
\times \mathcal V_{(2)}(\epsilon)
\to 
L^2_{m_1+1-n}(\Sigma_1(\delta^{(1)}),X) 
\times \mathcal V_{(1)}(\epsilon)
$$
by 
\begin{equation}\label{defnofPhi}
\Phi(F,{\bf x}) = (F\circ \Psi_{{\bf x}},\psi({\bf x}))
\end{equation}
\begin{lem}\label{lemma753}
For small $\epsilon$ there exist positive numbers $\epsilon'$,$\delta'$  and 
a $C^n$-map 
$$
\tilde{\mathscr J} : \mathcal V_{\rm map}(\epsilon') \times \mathcal V_{(2)}(\epsilon') \to \mathcal V_{\rm map}(\epsilon) \times \mathcal V_{(1)}(\epsilon)
$$
such that the next diagram commutes.
\begin{equation}\label{diagram7460}
\begin{CD}
\mathcal V_{\rm map}(\epsilon') \times \mathcal V_{(2)}(\epsilon') @ >{\mathcal R_{(2)}}>> L^2_{m_2+1-n}(\Sigma_1(\delta'),X) 
\times \mathcal V_{(2)}(\epsilon')\\
@ VV{\tilde{\mathscr J}}V @ VV{\Phi}V\\
\mathcal V_{\rm map}(\epsilon) \times \mathcal V_{(1)}(\epsilon) @ >{\mathcal R_{(1)}}>> L^2_{m_1+1-n}(\Sigma_1(\delta),X) 
\times \mathcal V_{(1)}(\epsilon).
\end{CD}
\end{equation}
\end{lem}
\begin{proof}
The existence of a map $\tilde{\mathscr J}$ such that the diagram (\ref{diagram7460})
commutes is a consequence of Proposition \ref{prop710}
as follows.
\par
Let ${\bf v} \in \mathcal V_{\rm map}(\epsilon)$, 
${\bf x} \in \mathcal V_{(i)}(\epsilon)$.
Then
$$
\Phi (\mathcal R_{(2)}({\bf v},{\bf x}))
=
(u^2_{{\bf v},{\bf x}}\circ \Phi^{(2)}_{1,{\bf x},\delta'}\circ \Psi_{\bf x},\psi({\bf x}))
$$
Note $\Psi_{\bf x} = (\Phi^{(2)}_{1,{\bf x},\delta'})^{-1}
\circ \tilde\psi_{\bf x}^{-1} 
\circ \Phi^{(1)}_{1,\psi({\bf x}),\delta}$.
\par
We put
$$
u' = u^2_{{\bf v},{\bf x}} \circ \tilde{\psi}_{\bf x}^{-1} 
: \Sigma_1(\psi({\bf x})) \to X.
$$
Since 
$
\tilde{\psi}_{\bf x}
 : 
(\Sigma_2({\bf x}),\vec z_2({\bf x}))
\cong (\Sigma_1(\psi({\bf x})), \vec z_1(\psi({\bf x}))$ 
is a bi-holomorphic map
$$
\overline{\partial} u^2_{{\bf v},{\bf x}}
\in 
E((\Sigma_2({\bf x}),\vec z_2({\bf x})),u^2_{{\bf v},{\bf x}})
$$
implies
\begin{equation}
\overline{\partial} u' 
\in E((\Sigma_1(\psi({\bf x})),\vec z_1(\psi({\bf x})),u').
\end{equation}
Moreover the $C^2$ distance between
$$
u' \circ \Phi^{(1)}_{1,\psi({\bf x}),\delta}
\qquad 
\text{and}
\qquad
u^2_{{\bf v},{\bf x}}
\circ \Phi^{(2)}_{1,{\bf x},\delta}
$$
goes to $0$  as $\epsilon \to 0$ by Lemma \ref{prop710} (4).
By assumption the $C^2$ distance between
$$
u^2_{{\bf v},{\bf x}}
\circ \Phi^{(2)}_{1,{\bf x},\delta}
\qquad \text{and}
\qquad
u_2 = u_1
$$
is smaller than $\epsilon'$.
Therefore the $C^2$ distance between
$$
u' \circ \Phi^{(1)}_{1,\psi({\bf x}),\delta}
\qquad 
\text{and}
\qquad
u_2 = u_1
$$
is small. 
We can show that 
the map $u'$ has
diameter $<\epsilon$ on the complement of the image of 
$\Phi^{(1)}_{1,\psi({\bf x}),\delta}$
if $\epsilon'$ is sufficiently small,
using the fact
that $\frak p_2$ is $\epsilon$ close to $\frak p_1$ with 
respect to ${\frak W}^{(1)}$.
(We use Lemma \ref{lem7777} here.)
Moreover $d(o,\psi({\bf x}))$ goes to $0$ as 
$d(o,{\bf x})$ goes to zero.
\par
Therefore by Proposition \ref{prop710} (2)
there exists ${\bf v}'$ such that
\begin{equation}\label{neweq578}
u' = u^1_{{\bf v}',\psi({\bf x})}.
\end{equation}
Then
\begin{equation}\label{neweq578+1}
u^1_{{\bf v}',\psi({\bf x})} \circ \Phi^{(1)}_{1,\psi({\bf x}),\delta}
=
u^2_{{\bf v},{\bf x}} \circ \tilde{\psi}_{\bf x}^{-1} 
\circ \Phi^{(1)}_{1,\psi({\bf x}),\delta}
=
u^2_{{\bf v},{\bf x}} \circ \Phi^{(2)}_{1,{\bf x},\delta}
\circ \Psi_{\bf x}.
\end{equation}
By putting
$$
\tilde{\mathscr J}({\bf v},{\bf x})
= ({\bf v}',\psi({\bf x}))
$$
Diagram \ref{diagram7460}
commutes.
\par
Lemmata \ref{lem74222} and \ref{lem743111} then imply that $\Phi$ 
is a $C^n$ map.
Lemma \ref{lemma74444} implies that $\mathcal R_{(i)}$ are $C^n$ embedding.
Therefore the commutativity of Diagram \ref{diagram7460} implies that 
$\tilde{\mathscr J}$ is a $C^n$ map.
\end{proof}
By definition
$V(\frak p_i;\epsilon_{(i)},({\frak W}^{(i)},\vec{\mathcal N}_i))$ 
(See Definition \ref{defn71474}) is a submanifold of 
$\mathcal V_{\rm map}(\epsilon_{(i)}) \times \mathcal V_{(i)}(\epsilon_{(i)})$.
(Here $\epsilon_{(1)} = \epsilon$, $\epsilon_{(2)} = \epsilon'$.) 
\begin{lem}\label{lem753753}
There exists a map
$
\mathscr J_{12;\epsilon,\epsilon'} : V(\frak p_2;\epsilon',({\frak W}^{(2)},\vec{\mathcal N}_2)) \to 
V(\frak p_1;\epsilon,({\frak W}^{(1)},\vec{\mathcal N}_1))
$
such that the next diagram commutes.
\begin{equation}\label{diagram7477}
\begin{CD}
V(\frak p_2;\epsilon',({\frak W}^{(2)},\vec{\mathcal N}_2)) @ >{}>> \mathcal V_{\rm map}(\epsilon') \times \mathcal V_{(2)}(\epsilon')\\
@ V{\mathscr J_{12;\epsilon,\epsilon'}}VV @ V{\tilde{\mathscr J}}VV\\
V(\frak p_1;\epsilon,({\frak W}^{(1)},\vec{\mathcal N}_1))@ >{}>> \mathcal V_{\rm map}(\epsilon) \times \mathcal V_{(1)}(\epsilon),
\end{CD}
\end{equation}
where horizontal arrows are canonical inclusions.
Moreover $\mathscr J_{12;\epsilon,\epsilon'}$ 
is of $C^n$ class.
\end{lem}
\begin{proof}
Let $({\bf v},{\bf x}) \in V(\frak p_2;\epsilon',({\frak W}^{(2)},\vec{\mathcal N}_2))$.
By Definition \ref{defn71474} we have
$$
u^2_{{\bf v},{\bf x}}(w_{2,j}({\bf x})) \in \mathcal N_j^{(2)},
$$
for $j=1,\dots,k_2$. We remark $\mathcal N_j^{(1)} = \mathcal N_j^{(2)}$ 
for $j=1,\dots,k_1$ by our choice.
Let $({\bf v}',{\psi}({\bf x})) = \tilde{\mathscr J}({\bf v},{\bf x})$.
Then by the commutativity of Diagram (\ref{diagram7460})
and (\ref{neweq578}) (\ref{neweq578+1}), 
we have
$$
u^1_{{\bf v}',{\psi}({\bf x})}
= u^2_{{\bf v},{\bf x}} \circ \tilde{\psi}_{\bf x}^{-1}.
$$
By the commutativity of Diagram (\ref{diagram741})
$$
u^1_{{\bf v}',{\psi}({\bf x})}(w_{1,j}({\bf x}))
= u^1_{{\bf v}',{\psi}({\bf x})}(\tilde{\psi}_{\bf x}(w_{2,j}({\bf x}))
= u^2_{{\bf v},{\bf x}}(w_{2,j}({\bf x})) \in \mathcal N_j^{(2)},
$$
for $j=1,\dots,k_1$.
Therefore by Definition \ref{defn71474}
$$
\tilde{\mathscr J}({\bf v},{\bf x}) = ({\bf v}',{\psi}({\bf x}))
\in V(\frak p_1;\epsilon,({\frak W}^{(1)},\vec{\mathcal N}_1)).
$$
We thus find the map 
$\mathscr J_{12;\epsilon,\epsilon'}$ such that Diagram 
(\ref{diagram7477}) commutes.  Since the horizontal arrows are $C^n$ embeddings 
and right vertical
arrow is a $C^n$ map, the map $\mathscr J_{12;\epsilon,\epsilon'}$ 
is of $C^n$ class as required.
\end{proof}
Commutativity of Diagrams (\ref{diagram7460}) and (\ref{diagram7477})
implies that $\mathscr J_{12;\epsilon,\epsilon'}$ 
is $\mathcal G_1$ equivariant and (\ref{formform740})
commutes.
\par
We finally 
show that $\mathscr J_{12;\epsilon,\epsilon'}$ 
is an open embedding. We consider two sub-cases.
\par\smallskip
\noindent(Case 1-1) 
$k_1 = k_2$.
\par
In this case $\vec w_1 = \vec w_2$, $\vec{\mathcal N}_1 = \vec{\mathcal N}_2$. 
The difference of ${\frak W}^{(1)}$ and ${\frak W}^{(2)}$ is the
trivialization data and the analytic families of  coordinates.
\begin{lem}\label{loem757}
In Case 1-1, the map $\tilde{\mathscr J}$ in Diagram (\ref{diagram7460}) is an open embedding.
\end{lem}
\begin{proof}
Since $\vec w_1 = \vec w_2$ we can exchange the role of ${\frak W}^{(1)}$ and ${\frak W}^{(2)}$.
Then by definition $\Phi_{{\bf x}}$ will become $\Phi_{\psi({\bf x})}^{-1}$.
Thus $\mathscr J_{21;\epsilon',\epsilon}$ obtained by  exchanging the role of ${\frak W}^{(1)}$ and ${\frak W}^{(2)}$
is the inverse of $\mathscr J_{12;\epsilon,\epsilon'}$.
\end{proof}
Since $\mathcal N_j^{(1)} = \mathcal N_j^{(2)}$ and $\# \vec w_1 = \# \vec w_2$, 
the equations to cut down 
$V(\frak p_i;\epsilon,({\frak W}^{(i)},\vec N_i))$ from 
$\mathcal V_{\rm map}(\epsilon) \times \mathcal V_{(i)}(\epsilon)$ coincide
each other for $i=1,2$.
Therefore 
$\mathscr J_{12;\epsilon,\epsilon'}$ 
is an open embedding in Case 1-1.
We thus proved Proposition \ref{prop720}
in Case 1-1.
\qed
\par\smallskip
\noindent(Case 1-2)
We show that, in case $\frak p_1=\frak p_2$  and Case 1, we can 
change the trivialization data and analytic families 
of  coordinates of ${\frak W}^{(2)}$ to obtain 
${\frak W}^{(3)}$ so that $\mathscr J_{13;\epsilon,\epsilon''}$ 
is an open embedding.
\par
We consider  irreducible component $\Sigma_{1,a}$ of $\Sigma_1$ and 
corresponding irreducible component $\Sigma_{3,a}$ of $\Sigma_3 =\Sigma_1$.
Forgetful map of the marked points determines the following commutative diagram.
\begin{equation}\label{diagram741irr}
\begin{CD}
\mathcal C_{g_a,\ell_a + k_{2,a}} @ >{\tilde\psi}>> \mathcal C_{g_a,\ell_a + k_{1,a}},\\
@ VV{\pi}V @ VV{\pi}V\\
\mathcal M_{g_a,\ell_a + k_{2,a}} @ >{\psi}>> \mathcal M_{g_a,\ell_a + k_{1,a}},
\end{CD}
\end{equation}
Here $\ell_a + k_{1,a}$ (resp. $\ell+k_{2,a} = \ell+k_{3,a}$) is the number of marked or nodal points on $\Sigma_{1,a}$
(resp. $\Sigma_{3,a}$). 
(Note $k_a = \#(\vec w_1 \cap \Sigma_{1,a})$.)
The vertical arrows are projections of the universal families
of deformations of $\Sigma_{i,a}$ together with marked points.
\begin{lem}\label{lem7555}
We may take the trivialization data of ${\frak W}^{(3)}$ so that 
the next diagram commutes.
\begin{equation}\label{diagram74555}
\begin{CD}
\mathcal V_{3,a} \times \Sigma_{3,a} @ >{\phi_a^{(3)}}>> \mathcal C_{g_a,\ell_a + k_{2,a}}
\\
@ V{\psi \times {\rm id}}VV @ V{\tilde{\psi}}VV\\
\mathcal V_{1,a} \times \Sigma_{1,a} @ >{\phi_a^{(1)}}>> \mathcal C_{g_a,\ell_a + k_{1,a}}.
\end{CD}
\end{equation}
Here the maps $\phi_a^{(3)}$, $\phi_a^{(1)}$ have the properties of 
the maps $\phi_a$ in 
Definition \ref{defn7272} (2).
\end{lem}
\begin{proof}
We can define $\phi_a^{(3)}$ by Diagram  (\ref{diagram74555}) itself.
\end{proof}
\begin{rem}\label{rem756}
We remark that when we make the choice as in Lemma \ref{lem7555}
the $\Sigma_{3,a}$ factor of $(\phi_a^{(3)})^{-1}(w_{3,j}({\bf x}))$ cannot 
be independent of ${\bf x}$ for $j > k_1$.
Namely (\ref{form7575}) does not hold for those marked points of $\Sigma_3$.
This is {\it the} reason why we do {\it not} assume (\ref{form7575}) 
for marked points of $\Sigma_i$ but only for nodal points.
\end{rem}
We next choose the analytic  family of  coordinates
$\varphi_{3,a,j} : \mathcal V_{3,a} \times D^2(2) \to \mathcal C_{g_a,\ell_a + k_{2,a}}$
for marked points on $\Sigma_3$ corresponding to nodal points of $\Sigma_3$ as follows.
Let $\varphi_{1,a,j} : \mathcal V_{1,a} \times D^2(2) \to \mathcal C_{g_a,\ell_a + k_{1,a}}$
be the  analytic  family of  coordinates associated to ${\frak W}^{(1)}$ for the 
corresponding nodal of $\Sigma_1$. (Note $\Sigma_1 = \Sigma_3$.)
We require
\begin{equation}\label{form757575}
\tilde\psi(\varphi_{3,a,j}({\bf x},z)) = \varphi_{1,a,j}(\psi({\bf x}),z).
\end{equation}
It is obvious that there is such choice of $\varphi_{3,a,j}$.
By construction, the commutativity of Diagram (\ref{diagram74555}) and 
 (\ref{form757575}) imply the next formula.
 \begin{equation}\label{form758585}
 \Phi_{1,\psi({\bf x}),\delta} =  \tilde{\psi} \circ  \Phi_{3,{\bf x},\delta}.
\end{equation}
(\ref{form758585}) and (\ref{newform750}) imply that 
the map $\Phi$ defined in (\ref{defnofPhi}) is :
\begin{equation}\label{form760}
\Phi(F,{\bf x}) = (F,\psi({\bf x})).
\end{equation}
We remark that the obstruction bundle $E((\Sigma',\vec z^{\,\prime}),u')$ is independent of the 
extra marked points $\vec w^{\,\prime}$. Moreover the 
commutativity of Diagram  (\ref{diagram74555}) implies that the identification of 
the source curve with $\Sigma_1 = \Sigma_3$ we use during the gluing process is 
independent of the $\ell + k_1 +1$-th,\dots,$\ell + k_2$ marked points.
Therefore we have the next formula:
\begin{equation}\label{form761}
u^1_{{\bf v},\psi({\bf x})} \circ  \Phi_{1,\psi({\bf x}),\delta}
=
u^3_{{\bf v},{\bf x}} \circ  \Phi_{3,{\bf x},\delta}.
\end{equation}
Formulae (\ref{form760}) and (\ref{form761}) imply
\begin{equation}\label{form762}
\tilde{\mathscr J}({\bf v},{\bf x}) = ({\bf v},\psi({\bf x}))
\end{equation}
Using (\ref{form761}), (\ref{form762}) we can easily prove that 
$\mathscr J_{13;\epsilon,\epsilon''}$
is an open embedding in the same way as Lemma \ref{lem713}.
We thus proved Proposition \ref{prop720}
in Case 1-2.
\par
Now we consider the general case of Case 1.
Suppose $\frak p_1$ $\frak p_2$ and $({\frak W}^{(1)},\vec{\mathcal N}_1)$, $({\frak W}^{(2)},\vec{\mathcal N}_2)$
are as in Case 1.
Then we can take $\frak p_3$ and  $({\frak W}^{(3)},\vec{\mathcal N}_3)$
such that $\frak p_1$, $\frak p_3$ and $({\frak W}^{(1)},\vec{\mathcal N}_1)$, $({\frak W}^{(3)},\vec{\mathcal N}_3)$
are as in Case 1-2.
Moreover $\frak p_2$, $\frak p_3$ and $({\frak W}^{(2)},\vec{\mathcal N}_2)$, $({\frak W}^{(3)},\vec{\mathcal N}_3)$
are as in Case 1-1.
(Note $\frak p_1=\frak p_2=\frak p_3$ in this case.)
Therefore we obtain required 
$\mathscr J_{12;\epsilon,\epsilon'}$ 
by composing 
$\mathscr J_{13;\epsilon,\epsilon''}$
and an inverse of 
$\mathscr J_{23;\epsilon',\epsilon''}$.
The proof of Proposition \ref{prop720}
in Case 1 is complete.
\end{proof}
\par\medskip
\noindent(Case 2) 
We assume $\frak p_1 = \frak p_2 = ((\Sigma_1,\vec z_1),u_1)$.
We also require $({\frak W}^{(1)},\vec{\mathcal N}_1) \supseteq ({\frak W}^{(2)},\vec{\mathcal N}_2)$
\par
The proof of this case is entirely similar to Case 1 and so is 
omitted.
\par\medskip
\noindent(Case 3) 
We assume $\frak p_1 = \frak p_2 = ((\Sigma_1,\vec z_1),u_1)$.
We also require $\vec w^{\,(1)} \cap \vec w^{(2)} = \emptyset$.
\par
We define $({\frak W}^{(3)},\vec{\mathcal N}_3)$ as follows.
$\vec w^{\,(3)} = \vec w^{\,(1)} \cup \vec w^{(2)}$.
$\mathcal N^{(3)}_j$ is $\mathcal N^{(1)}_{j'}$ (resp.
$\mathcal N^{(2)}_{j'}$)
if $w^{(3)}_{1,j} = w^{(1)}_{1,j'}$ (resp. 
$w^{(3)}_{1,j} = w^{(2)}_{1,j'}$).
We take any choice of the trivialization data and 
of analytic families of  coordinates.
\par
Then the triples $\frak p_1,({\frak W}^{(1)},\vec{\mathcal N}_1)$, and $\frak p_3,({\frak W}^{(3)},\vec{\mathcal N}_3)$
(resp. $\frak p_2,({\frak W}^{(2)},\vec{\mathcal N}_2)$ and $\frak p_3,({\frak W}^{(3)},\vec{\mathcal N}_3)$)
satisfy the conditions for (Case 1) or (Case 2).
Therefore we obtain required 
$\mathscr J_{12;\epsilon,\epsilon'}$ 
by composing 
$\mathscr J_{23;\epsilon',\epsilon''}$
and an inverse of 
$\mathscr J_{13;\epsilon,\epsilon''}$.
\par
We can prove Lemma \ref{lem71818} also in the same way.
\par\medskip
\noindent(Case 4) 
We assume $\frak p_1 = \frak p_2 = ((\Sigma_1,\vec z_1),u_1)$ only.
\par
We can find $({\frak W}^{(3)},\vec{\mathcal N}_3)$ with $\frak p_3 = \frak p_1$
such that 
$\vec w^{\,(1)} \cap \vec w^{(3)} = \emptyset
= \vec w^{\,(2)} \cap \vec w^{(3)}$.
We then apply (Case 3) twice and compose the resulting maps 
to obtain the required 
$\mathscr J_{12;\epsilon,\epsilon'}$.
\par\medskip
We can prove Lemma \ref{lem71818} also in the same way.
We thus completed the case $\frak p_1 =  \frak p_2$.
\par\medskip
\noindent(Case 5)
We consider the general case where 
$\frak p_1 \ne \frak p_2$.
\par
\begin{proof}[The proof of Lemma \ref{lem71818} in Case 5]
Using (Case 4) it suffices to show the following.
For given stabilization and trivialization data ${\frak W}^{(1)}$ at $\frak p_1$ 
we can find stabilization  and trivialization data ${\frak W}^{(2)}$
at $\frak p_2$ such that 
Lemma \ref{lem71818} holds.
We will prove this statement below.
\par
We fixed ${\frak W}^{(1)}$, in particular we fixed 
$\vec w_1$. We take the universal family of deformation 
of $(\Sigma_1,\vec z_1 \cup \vec w_1)$ and denote it by $\pi : \mathcal C_{(1)} \to \mathcal V_{(1)}$.
(It comes with sections assigning marked points.)
\par
Since $\frak p_2 = [\Sigma_2,\vec z_2]$ is $\epsilon$-close to 
$\frak p_1$ with respect to ${\frak W}^{(1)}$
there exists ${\bf x}_2$ and $\vec w_2$ such that
$$
\phi : (\Sigma_1({\bf x}_2),\vec z_1({\bf x}_2) \cup \vec w_1({\bf x}_2)) \cong  (\Sigma_2, \vec z_2 \cup \vec w_2).
$$
\par
We take this $\vec w_2$ as a part of the data consisting ${\frak W}^{(2)}$.
Let $\pi : \mathcal C_{(2)} \to \mathcal V_{(2)}$ be the universal 
family of deformation of $(\Sigma_2, \vec z_2 \cup \vec w_2)$.
Then we have an open embeddings 
$\psi : \mathcal V_{(2)} \to \mathcal V_{(1)}$ and 
$\tilde\psi : \mathcal C_{(2)} \to \mathcal C_{(1)}$ 
such that Diagrams (\ref{diagram741}) and (\ref{diagram7412}) 
commute.
\begin{lem}\label{lemlem762}
There exist  stabilization and trivialization data ${\frak W}^{(2)}$ so that
\begin{equation}\label{form763}
\tilde{\psi}_{\bf x} \circ \Phi_{2,{\bf x},\delta'} = \Phi_{1,\psi({\bf x}),\delta} \circ \Phi_{1,{\bf x}_2,\delta}^{-1} \circ \phi^{-1}
\end{equation}
holds on 
$
\phi(\Phi_{1,{\bf x}_2,\delta}(\Sigma_1(\delta)))
$ if $\delta'$ is small.
Here $\tilde{\psi}_{\bf x}$ is the restriction of $\tilde\psi$ to $\Sigma_2({\bf x})$.
\end{lem}
\begin{equation}\label{diagram769}
\begin{CD}
(\Sigma_1({\bf x}_2)(\delta) \,\,\overset{\phi}\hookrightarrow\,\, \Sigma_2(\delta')) @ >{\Phi_{2,{\bf x},\delta'}}>> \Sigma_2({\bf x})\\
@ AA{\Phi_{1,{\bf x}_2,\delta}}A @ VV{\tilde\psi_{\bf x}}V\\
\Sigma_1(\delta) @ >{\Phi_{1,\psi({\bf x}),\delta}}>> \Sigma_1(\psi({\bf x})),
\end{CD}
\end{equation}
\begin{proof}
$\Sigma_1({\bf x}_2) \cong \Sigma_2$ is obtained 
by deforming the complex structure of irreducible components of $\Sigma_1$ and gluing several 
irreducible components along (some of) marked points. 
\par
Note the weak stabilization data $\vec w_2$ was defined by sending 
weak stabilization data $\vec w_1$ of ${\frak W}^{(1)}$ by $\Phi_{1,{\bf x}_2,\delta}$.
\par
We define families of coordinates at the nodes which are parts of ${\frak W}^{(2)}$ 
as follows.
Suppose $\Sigma_2({\bf x}_0)$ has the same number of nodes as $\Sigma_2$.
Note each node of $\Sigma_2({\bf x}_0)$ corresponds to an node of $\Sigma_1(\psi({\bf x}_0))$.
Therefore we pull back families of coordinates at the nodes of $\Sigma_1(\psi({\bf x}_0))$
by $\tilde\psi_{\bf x}$ to obtain the families of coordinates at the nodes we look for. 
We remark that according to the definition the families of coordinates at the nodes of $\Sigma_1({\bf x}_0')$ 
is given only in case $\Sigma_1({\bf x}_0')$ has the same number of 
nodes as $\Sigma_1$. (See Definition \ref{defn7372}.)
However they  canonically induce one of $\Sigma_1({\bf x}')$ for any ${\bf x}'$.
In fact,  $\Sigma_1({\bf x}')$ is obtained from some $\Sigma_1({\bf x}'_0)$
(with the same number of nodes as $\Sigma_1$) by performing the gluing 
at several nodes. The nodes of $\Sigma_1({\bf x}')$ correspond to the nodes of 
$\Sigma_1({\bf x}'_0)$ where such gluing construction were not performed. So neighborhoods of the nodes of 
$\Sigma_1({\bf x}')$ are canonically identified with neighborhoods of certain nodes of $\Sigma_1({\bf x}'_0)$.
\par
We finally define local trivializations. 
Suppose $\Sigma_2({\bf x}_0)$ has the same number of nodes as $\Sigma_2$.
We use 
$
\phi \circ \Phi_{1,{\bf x}_2,\delta} \circ \Phi_{1,\psi({\bf x}_0),\delta}^{-1} 
\circ \tilde\psi_{{\bf x}_0}
$
to define a diffeomorphism 
$$
\phi   \circ \Phi_{1,{\bf x}_2,\delta} \circ \Phi_{1,\psi({\bf x}_0),\delta}^{-1} 
\circ \tilde\psi_{{\bf x}_0} : 
\tilde\psi^{-1}_{{\bf x}_0}(\Phi_{1,\psi({\bf x}_0),\delta}(\Sigma_1(\delta))) \to 
\Sigma_1(\delta)  \to
\Sigma_1({\bf x}_2) 
\to \Sigma_2.
$$
Here $\tilde\psi^{-1}_{{\bf x}_0}(\Phi_{1,\psi({\bf x}_0),\delta}(\Sigma_1(\delta)))
\subset \Sigma_2({\bf x}_0)$. 
The complement
$$
\Sigma_2({\bf x}_0)
\setminus \tilde\psi^{-1}_{{\bf x}_0}(\Phi_{1,\psi({\bf x}_0),\delta}(\Sigma_1(\delta)))
$$
is the union of the following two types of 
connected components. (See Figure \ref{FigureOctGroupoind}.)
\begin{enumerate}
\item[(I)]
A neighborhood of the nodal points of $\Sigma_2({\bf x}_0)$.
\item[(II)]
A neck region. It is a part which we obtain by performing the gluing at certain 
nodes of $\Sigma_1(\psi({\bf x}_0))$.
\end{enumerate}
We extend $\phi_0$ to those connected components as follows.
\par
To the part (I) we extend so that it is compatible with the 
families of coordinates at the nodes we produced above.
There is a unique way to do so.
\par
To the part (II) we extend it in an arbitrary way. 
By taking $\mathcal V_2$ small we can show the existence of extension.
\par
Thus we defined ${\frak W}^{(2)}$.
The commutativity of (\ref{diagram769}) is an immediate consequence of the 
construction if $\Sigma_2({\bf x})$ has the same number of nodes as $\Sigma_2$.
\par
The commutativity of (\ref{diagram769}) in the general case then follows from the 
fact that the families of coordinates at the nodes 
we use to perform the gluing for $\Sigma_2$ and 
one at the nodes of $\Sigma_1$ which remain to be nodal in $\Sigma_2({\bf x}_0)$ are 
compatible by the construction.
\begin{figure}[h]
\centering
\includegraphics[scale=0.7]{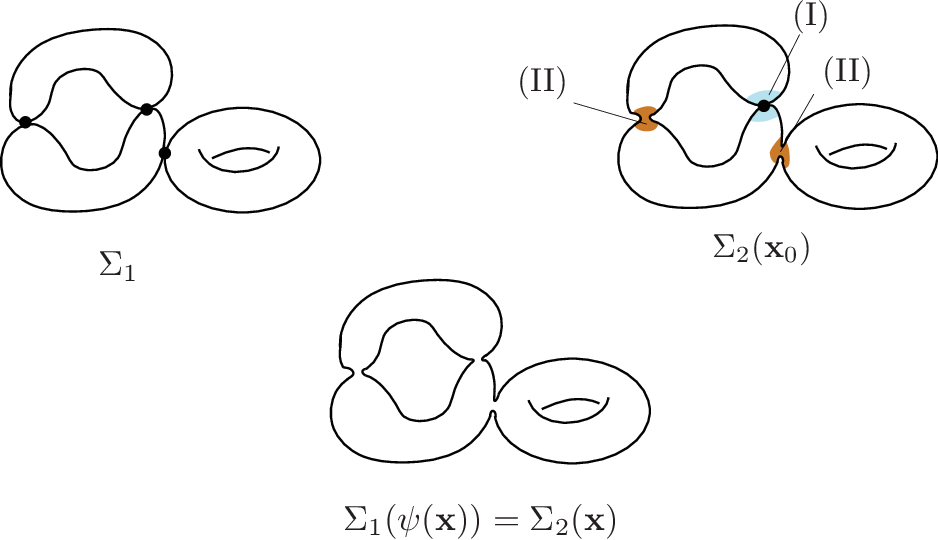}
\caption{Part (I) and (II)}
\label{FigureOctGroupoind}
\end{figure}
\end{proof}
We take the choice of ${\frak W}^{(2)}$ as in Lemma \ref{lemlem762}.
Let
$$
[(\Sigma_c,\vec z_c),u_c] \in \mathcal U(\epsilon(c);(\Sigma_2,\vec z_2),u_2,
{\frak W}^{(2)}).
$$
By definition there exists $\vec w_c \subset \Sigma_c$,
${\bf x}_c \in \mathcal V_{(2)}(\epsilon_c)$, 
a bi-holomorphic map
$$
\phi_c : (\Sigma_2({\bf x}_c),\vec z_2({\bf x}_c)\cup \vec w_2({\bf x}_c)) \to
(\Sigma_c,\vec z_c \cup \vec w_c),
$$
and $\delta_c$,
such that the following holds.
\begin{enumerate}
\item
The $C^2$ distance between $u_2$ and $u_c \circ \phi_c 
\circ \Phi_{2,{\bf x}_c,\delta_c}$ is smaller than $o(c)$.
\item $d({\bf x}_c,o)$ goes to zero as $c$ goes to infinity.
\item
The map $u_c \circ \phi_c$ has diameter $<o(c)$ on 
$\Sigma_2({\bf x}_c) \setminus {\rm Im}(\Phi_{2,{\bf x}_c,\delta_c})$.
\end{enumerate}
By Lemma \ref{lem7777} we may assume that $\delta_c \to 0$.
\par
By Item (2) we have $\psi({\bf x}_c) \in \mathcal V^{(1)}$ for large $c$ 
and an isomorphism
$$
\tilde{\psi}_{{\bf x}_c} :
(\Sigma_2({\bf x}_c),\vec z_2({\bf x}_c)\cup \vec w_2({\bf x}_c))
\to 
(\Sigma_1(\psi({\bf x}_c)),\vec z_1(\psi({\bf x}_c))\cup \vec w_1(\psi({\bf x}_c))).
$$
We put 
$$
\phi'_c = \phi_c \circ (\tilde{\psi}_{{\bf x}_c})^{-1}.
$$
By our choice of ${\frak W}^{(2)}$, the next diagram commutes.
\begin{equation}\label{diagram769}
\xymatrix{
&\Sigma_1({\bf x}_2)(\delta)  \ar[r]^{\phi} &\Sigma_2(\delta_c)\ar[r]^{\Phi_{2,{\bf x}_c,\delta_c}} &\Sigma_2({\bf x}_c)\ar[rd]^{\phi_c}\ar[dd]_{\tilde{\psi}_{{\bf x}_c}}  \\
&&&&
\Sigma_c \ar[r]^{u_c} &X
\\
&\Sigma_1(\delta)\ar[uu]^{\Phi_{1,{\bf x}_2,\delta}}\ar[rr]_{\Phi_{1,\psi({\bf x}_c),\delta}} 
&& 
\Sigma_1(\psi({\bf x}_c))\ar[ru]_{\phi'_c}
}
\end{equation}
Moreover the $C^2$ distance between 
$$
u_2 \circ \phi \circ \Phi_{1,{\bf x}_2,\delta}
\qquad
\text{and}
\qquad
u_1
$$
is smaller than $\epsilon$.
Therefore (using Item (1) above also)  for sufficiently large $c$ 
the $C^2$ distance between 
$$
u_c \circ \phi'_c \circ \Phi_{1,\psi({\bf x}_c),\delta}
\qquad
\text{and}
\qquad
u_1
$$
is smaller than $\epsilon$.
Thus we checked Definition \ref{defn412} (1). Definition \ref{defn412} (2) is easy to check.
\par
We will check Definition \ref{defn412} (3).
Let $C_c$ be a connected component of 
$\Sigma_1(\psi({\bf x}_c)) \setminus  {\rm Im}(\Phi_{1,\psi({\bf x}_c),\delta})$.
There is a unique nodal point $p_{C_c}$ of $\Sigma_1$ which corresponds to $C_c$.
(See Figure \ref{Figurep71}.)
By our choice of ${\frak W}^{(2)}$
one of the following holds.
\begin{enumerate}
\item[(a)]
$$
C_c = \tilde\psi_{{\bf x}_c}(C_c')
$$
for some 
connected component $C'_c$ of 
$\Sigma_2({\bf x}_c) \setminus  \Phi_{2,{\bf x}_c,\delta}$.
\item[(b)]
There is no nodal point corresponding to $p_{C_c}$ in $\Sigma_2$. 
\end{enumerate}
Suppose we are in case (a).
By the commutativity of Diagram (\ref{diagram769})
$$
(u_c \circ \phi_c)(C_c) = (u_c \circ \phi'_c)(C_c').
$$
Let 
$$
C''_c = C'_c \setminus {\rm Im}(\Phi_{2,{\bf x}_c,\delta_c}).
$$
$C''_c$ is connected and hence  the diameter of 
$(u_c \circ \phi'_c)(\tilde{\psi}_{{\bf x}_c}(C''_c))$
is smaller than $o(c)$.
\par
On the other hand, we have 
$$
C'_c \setminus C''_c \subseteq {\rm Im}
(\Phi_{2,{\bf x}_c,\delta_c} \setminus \Sigma_2(\delta)).
$$
We define $D_c$ by:
$$
C'_c \setminus C''_c = 
(\Phi_{2,{\bf x}_c,\delta_c})(D_c).
$$
On $D_c$ the $C^2$ distance between
$$
u_2 
\qquad
\text{and}
\qquad
u_c \circ \phi_c \circ \Phi_{2,{\bf x}_c,\delta_c}
$$
is smaller than $o(c)$.
On the other hand since $D_c$ is contained in a connected component 
of $\Sigma_2 \setminus \Sigma_2(\delta)$ 
the diameter of $u_2(D_c)$ is smaller than $\epsilon$.
\par
Therefore the diameter of $(u_c \circ \phi'_c)(C_c)$
is smaller than $\epsilon$ for sufficiently large $c$
in case (a).
\par
Suppose we are in case (b).
Note 
$$
u_c \circ \phi_c \circ \tilde{\psi}^{-1}_{{\bf x}_c}
 = u_c \circ \phi'_c
$$
holds on $C_c$. (In fact they both are defined there.)
Therefore
$$
\lim_{c \to \infty} {\rm Diam}(u_c \circ \phi'_c)(C_c)
= \lim_{c \to \infty} {\rm Diam}(u_c \circ \phi_c)(\tilde C_c)
$$
where $\tilde{\psi}_{{\bf x}_c}(\tilde C_c) = C_c$.
Since we are in case (b), $\tilde C_c = 
\Phi_{2,{\bf x}_c,\delta}(\hat C_c)$ for some $\hat C_c \subset \Sigma_2(\delta_c)$.
We may assume that $\hat C_c$ lies in the $o(c)$ neighborhood of 
$\phi(C_{0,c})$, where $C_{0,c}$ is a component of $\Sigma_1({\bf x}_2) \setminus {\rm Im}(\Phi_{1,{\bf x}_2,\delta})$.
Moreover on a neighborhood of $C_{0,c}$ 
the maps $u_c \circ \phi_c \circ \Phi_{2,{\bf x}_c,\delta_c}$ converges to $u_2$.
Thus
$$
{\rm Diam}(u_c \circ \phi_c)(\tilde C_c)
\le 
{\rm Diam}(u_2(C_{0,c})) + o(c).
$$
By assumption ${\rm Diam}(u_2(C_{0,c})) < \epsilon$.
So we conclude 
${\rm Diam}(u_c \circ \phi_c)(\tilde C_c) < \epsilon$
for sufficiently large $c$.
\begin{figure}[h]
\centering
\includegraphics[scale=0.5]{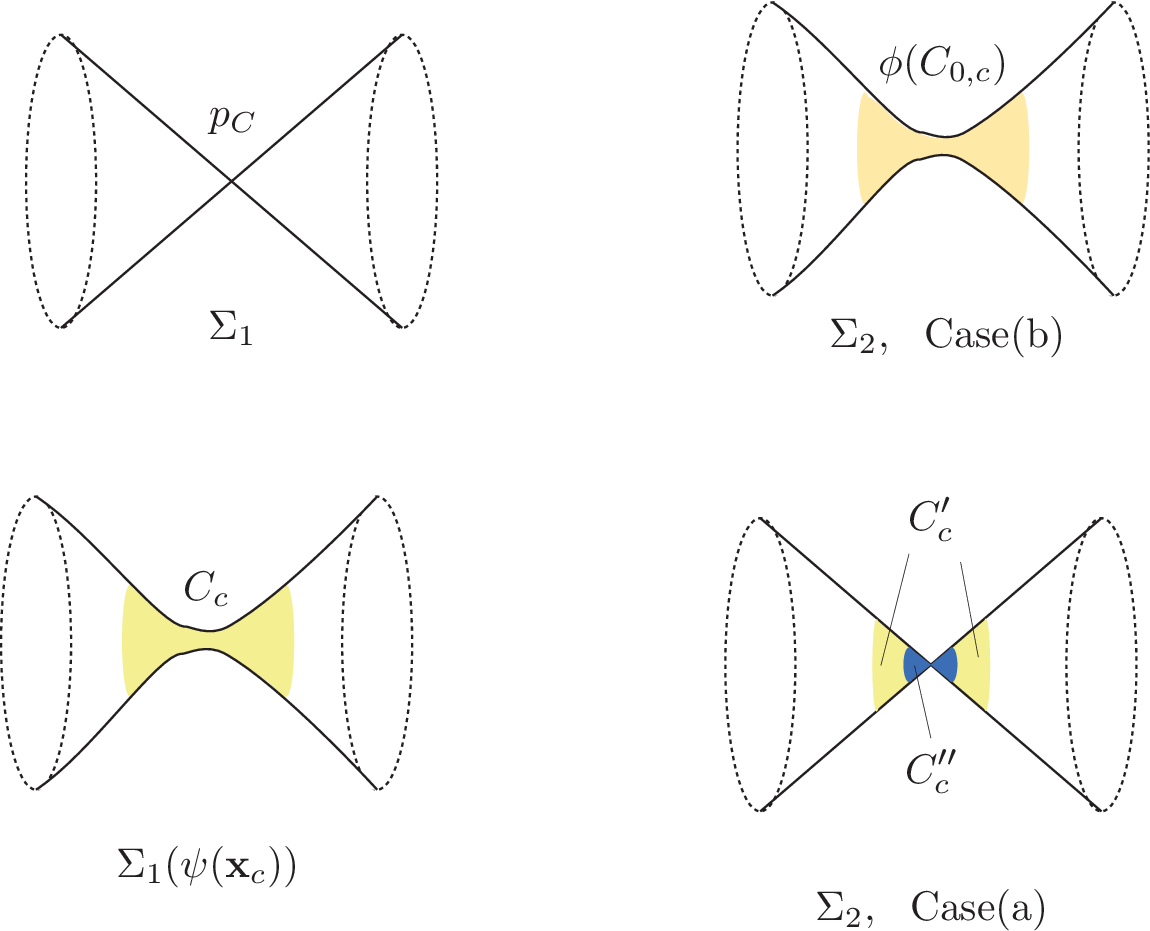}
\caption{Estimate of the diameter of 
$u_c \circ \phi'_c(C_c)$}
\label{Figurep71}
\end{figure}
Thus, we have proved:
$$
[(\Sigma_c,\vec z_c),u_c] \in 
\mathcal U(\epsilon(c);(\Sigma_1,\vec z_1),u_1,{\frak W}^{(1)}).
$$
for sufficiently large $c$, as required.
\end{proof}
\begin{proof}[The proof of Proposition \ref{prop720} in Case 5]
Using (Case 4) it suffices to show the following.
For given strong stabilization data $({\frak W}^{(1)},\vec{\mathcal N}_1)$ at $\frak p_1$ we can 
find strong stabilization data $({\frak W}^{(2)},\vec{\mathcal N}_2)$
at $\frak p_2$ such that 
$\mathscr J_{12;\epsilon,\epsilon'}$
exists.
We will prove this statement below.
\par
Since we fixed ${\frak W}^{(1)}$, in particular we fixed 
$\vec w_1$. We take the universal family of deformation 
of $(\Sigma_1,\vec z_1 \cup \vec w_1)$ and denote it by $\pi : \mathcal C_{(1)} \to \mathcal V_{(1)}$.
(It comes with sections assigning marked points.)
\par
Since $\frak p_2 = [\Sigma_2,\vec z_2]$ is $\epsilon$ close to 
$\frak p_1$ with respect to ${\frak W}^{(1)}$
there exists ${\bf x}'_2$ and $\vec w^{\,\prime}_2$ such that
$$
\phi' : (\Sigma_1({\bf x}'_2),\vec z_1({\bf x}'_2) \cup \vec w_1({\bf x}'_2)) \cong  (\Sigma_2, \vec z_2 \cup \vec w^{\,\prime}_2).
$$
and it satisfies other conditions related to the maps $u_1$ and $u_2$.
(See Definition \ref{defn412}.)
It implies that
$u_2(w'_{2,j})$ is close to $u_1(w_{1,j}) \in \mathcal N_j^{(1)}$.
Since $u_1$ intersects transversaly with $\mathcal N_j^{(1)}$
and $u_2 \circ \phi' \circ \Phi_{1,{\bf x}'_2,\delta'}$
is $C^2$ close to $u_1$, 
we can take $w_{2,j}$ such that
$u_2(w_{2,j}) \in \mathcal N_j^{(2)}$ and $d(w_{2,j},w'_{2,j}) < o(\epsilon)$.
We take $(w_{2,1},\dots,w_{2,k})$ as our $\vec w_2$.
Moreover we take $\mathcal N_j^{(2)} = \mathcal N_j^{(1)}$.
There exists ${\bf x}_2$ and $\vec w_2$ such that
$$
\phi : (\Sigma_1({\bf x}_2),\vec z_1({\bf x}_2) \cup \vec w_1({\bf x}_2)) \cong  (\Sigma_2, \vec z_2 \cup \vec w_2).
$$
In the same way as Lemma \ref{lemlem762}
we can find a system of analytic families of  coordinates and 
trivialization data consisting ${\frak W}^{(2)}$ so that
\begin{equation}\label{form76322}
\tilde{\psi}_{\bf x} \circ \Phi_{2,{\bf x},\delta'} = \Phi_{1,\psi({\bf x}),\delta} \circ \Phi_{1,{\bf x}_2,\delta}^{-1} \circ \phi^{-1}
\end{equation}
holds on 
$
\phi(\Phi_{1,{\bf x}_2,\delta}(\Sigma_1(\delta))).
$
\par
We remark that the $C^2$ distance between 
$$
u_2 \circ \phi\circ \Phi_{1,{\bf x}_2,\delta}
\qquad
\text{and}
\qquad
u_1
$$
is smaller than $o(\epsilon)$.
Note this $C^2$ distance may not be smaller than $\epsilon$,
since we changed $\vec w^{\,\prime}_2$ to $\vec w_2$.
However the $C^2$ distance  can certainly be estimated by $o(\epsilon)$.
\par
Therefore in the same way as the proof of Lemma \ref{lem71818}
we have
\begin{equation}\label{form72020222}
\mathcal U(\epsilon';(\Sigma_2,\vec z_2),u_2,{\frak W}^{(2)})
\subset
\mathcal U(o(\epsilon);(\Sigma_1,\vec z_1),u_1,{\frak W}^{(1)})
\end{equation}
if $\epsilon'$ is sufficiently small.
\par
We define $\Psi : \Sigma_1(\delta) \to \Sigma_2$ by
\begin{equation}\label{form773}
\Psi = (\Phi_{2,{\bf x},\delta'})^{-1} \circ (\tilde{\psi}_{\bf x})^{-1}
\circ \Phi_{1,{\psi}({\bf x}),\delta}
=
\phi \circ \Phi_{1,{\bf x}_2,\delta}.
\end{equation}
Here the right equality follows from (\ref{form763}). 
Compare (\ref{newform750}). Here however contrary to (\ref{newform750})
$\Psi$ is independent of ${\bf x}$.
\par
We define
$$
\Phi : L^2_{m+1}(\Sigma_2(\delta'),X) 
\times \mathcal V_{2}(\epsilon')
\to 
L^2_{m+1}(\Sigma_1(\delta),X) 
\times \mathcal V_{1}(\epsilon)
$$
by
\begin{equation}\label{form774}
\Phi(F,{\bf x}) = (F\circ \Psi,\psi({\bf x})).
\end{equation}
Compare (\ref{defnofPhi}).
Since $\Psi$ is independent of ${\bf x}$ and $\psi$ is smooth,
the map $\Phi$ is smooth.
\par
We define 
\begin{equation}
\aligned
&\mathcal R_{(1)} :
\mathcal V^{(1)}_{\rm map}(\epsilon) \times \mathcal V_{(1)}(\epsilon)
\to 
L^2_{m+1}(\Sigma_1(\delta),X)  \times \mathcal V_{(1)}(\epsilon)\\
&\mathcal R_{(2)} :
\mathcal V^{(2)}_{\rm map}(\epsilon') \times \mathcal V_{(2)}(\epsilon')
\to 
L^2_{m+1}(\Sigma_2(\delta'),X)  \times \mathcal V_{(2)}(\epsilon')\\
\endaligned
\end{equation}
by \index{00R3o@$\mathcal R_{(i)}$}
\begin{equation}\label{75375322}
\mathcal R_{(i)}({\bf v},{\bf x})
= (u^i_{{\bf v},{\bf x}}\circ \Phi^{(i)}_{1,{\bf x},\delta^{(i)}},{\bf x}).
\end{equation}
Here $\delta^{(1)} = \delta$, $\delta^{(2)} = \delta'$.
Compare (\ref{753753}).
\par 
In the same way as the proof of 
Lemma \ref{lemma753} we can show that there exists $\tilde{\mathscr J}$ such that 
the next diagram commutes.
\begin{equation}\label{diagram746}
\begin{CD}
\mathcal V^{(2)}_{\rm map}(\epsilon') \times \mathcal V_{(2)}(\epsilon') @ >{\mathcal R_{(2)}}>> L^2_{m+1}(\Sigma_2(\delta'),X) 
\times \mathcal V_{(2)}(\epsilon')\\
@ VV{\tilde{\mathscr J}}V @ VV{\Phi}V\\
\mathcal V^{(1)}_{\rm map}(\epsilon) \times \mathcal V_{(1)}(\epsilon) @ >{\mathcal R_{(1)}}>> L^2_{m+1}(\Sigma_1(\delta),X) 
\times \mathcal V_{(1)}(\epsilon).
\end{CD}
\end{equation}
The commutativity of Diagram (\ref{diagram746}) implies that 
$\tilde{\mathscr J}$ is of $C^n$ class.\footnote{Actually it is smooth.}
$\tilde{\mathscr J}$ induce a map $\mathscr J_{12;\epsilon,\epsilon'}$ 
in the same way as Lemma \ref{lem753753}.
Using the fact that $\#\vec w_1 = \# \vec w_2$ we can show that 
$\mathscr J_{12;\epsilon,\epsilon'}$  
is a $C^n$ embedding in the same way as the proof of Lemma \ref{loem757}, 
that is,  Case 1-1.
The proof of Proposition \ref{prop720}  now complete.
\end{proof}
The proof of Proposition \ref{prop720} and Lemma \ref{lem71818} 
is complete.
We turn to the proof of Lemma \ref{lemprop721}.

\begin{proof}[Proof of Lemma \ref{lemprop721}]
Let ${\mathscr L}^2_{m+1}(\Sigma_1(\delta),X)$ be an $\epsilon_0$ neighborhood of $u_1$ in 
$L^2_{m+1}(\Sigma_1(\delta),X)$.
We consider the direct product
\begin{equation}\label{form770}
{\mathscr L}^2_{m+1}(\Sigma_1(\delta),X) \times \mathcal V_{(1)}
\end{equation}
and a bundle 
$\frak E_1$ on it such that its total space is
\begin{equation}
\frak E_1
= 
\{ ((\hat u',{\bf x},V) \mid
(\hat u',{\bf x}) \in (\ref{form770}), 
\,
V \in L^2_m(\Sigma_1(\delta);(\hat u')^*TX \otimes \Phi_{1,{\bf x},\delta}^*\Lambda^{01})
\}
\end{equation}
with obvious projection. $\frak E_1
\to {\mathscr L}^2_{m+1}(\Sigma_1(\delta),X) \times \mathcal V_{(1)}$.
\begin{lem}
If $m$ is larger than 10, then
$\frak E_1$ has a structure of smooth vector bundle
and $\mathcal G_1$ acts on it.
\end{lem}
\begin{proof}
Let $(\hat u',{\bf x}) \in (\ref{form770})$.
We put $\Phi_{1,{\bf x},\delta}(\Sigma_1(\delta))
= \Sigma_1({\bf x})(\delta)$.
There exists a canonical identification 
$$
L^2_m(\Sigma_1({\bf x})(\delta);(\hat u')^*TX \otimes \Phi_{1,{\bf x},\delta}^*\Lambda^{01})
\cong
L^2_m(\Sigma_1(\delta);(u')^*TX \otimes \Lambda^{01})
$$
where $u' = \hat u' \circ \Phi_{1,{\bf x},\delta}^{-1}$.
We defined 
$$
I_{\hat u',{\bf x}} :
L^2_{m}(\Sigma_1({\bf x})(\delta);(u')^*TX\otimes \Lambda^{01})
\to 
L^2_{m}(\Sigma_1(\delta);u_1^*TX\otimes \Lambda^{01}).
$$
in (\ref{form711}). Combining them we obtain a \index{00E41@$\frak E_1$}
bijection
\begin{equation}\label{form772222}
\frak E_1 \cong 
{\mathscr L}^2_{m+1}(\Sigma_1(\delta),X) \times \mathcal V_{(1)}
\times L^2_{m}(\Sigma_1(\delta);u_1^*TX\otimes \Lambda^{01}).
\end{equation}
We define $C^{\infty}$ structure of $\frak E_1$ by this isomorphism.
$\mathcal G_1$ invariance of this trivialization is immediate 
from definition.
\end{proof}
The vector space $E(\hat u',{\bf x})$ is identified with a finite dimensional 
linear subspace of 
$L^2_{m+1}(\Sigma_1(\delta);u_1^*TX\otimes \Lambda^{01})$ by (\ref{loctrimap3}).
\begin{lem}\label{lem762762}
\begin{equation}\label{formulaforlem762}
\bigcup_{(\hat u',{\bf x}) \in (\ref{form770})}
\{(\hat u',{\bf x})\} \times E(\hat u',{\bf x})
\end{equation}
become a smooth subbundle of the right hand side of 
(\ref{form772222}) by this identification.
\end{lem}
\begin{proof}
This is nothing but Proposition \ref{prop78}.
\end{proof}
We pull back the bundle in Lemma \ref{lem762762}
by the $C^n$ embedding
$$
V(\frak p_1;\epsilon;({\frak W}^{(1)},\vec{\mathcal N}_1))
\to \mathcal V_{\rm map} \times \mathcal V_{(1)}(\epsilon)
\overset{\mathcal R_{(1)}}\to
{\mathscr L}^2_{m+1}(\Sigma_1(\delta),X) \times \mathcal V_{(1)}
$$
to obtain a (finite dimensional) vector bundle 
on $V(\frak p_1;\epsilon;({\frak W}^{(1)},\vec{\mathcal N}_1))$
of $C^n$ class,
which we write $\frak E(\frak p_1;\epsilon;({\frak W}^{(1)},\vec{\mathcal N}_1))$. 
$\mathcal G_1$ acts on it and the action is of $C^n$ class.
\par
To complete the proof of Lemma \ref{lemprop721}
it suffices to show that 
$\frak E(\frak p_1;\epsilon;({\frak W}^{(1)},\vec{\mathcal N}_1))$
can be glued to give a vector bundle of $C^n$ class 
on $U(((\Sigma,\vec z),u);\epsilon_2)$.
\par
Suppose $\frak p_2$ is $\epsilon$ close to $\frak p_1$
with respect to ${\frak W}^{(1)}$.
We choose strong stabilization data $({\frak W}^{(2)},\vec{\mathcal N}_2)$ at $\frak p_2$  
(Definition \ref{defn720}).
\par
Let 
$
\mathscr J_{12;\epsilon,\epsilon'} : V(\frak p_2;\epsilon',({\frak W}^{(2)},\vec{\mathcal N}_2)) \to 
V(\frak p_1;\epsilon,({\frak W}^{(1)},\vec{\mathcal N}_1))
$
be the map we produced during the proof of 
Proposition \ref{prop720}. We show that 
there exists a canonical lift of this map
to the fiber-wise linear map
\begin{equation}\label{maptildeJ}
\widetilde{
\mathscr J}_{12;\epsilon,\epsilon'}
:
\frak E(\frak p_2;\epsilon';({\frak W}^{(2)},\vec{\mathcal N}_2))
\to 
\frak E(\frak p_1;\epsilon;({\frak W}^{(1)},\vec{\mathcal N}_1)).
\end{equation}
\par
We take $e_1,\dots,e_d$ a basis of $E((\Sigma,\vec z),u)$.
Then a frame  
of $\frak E(\frak p_1;\epsilon;{\frak W}^{(1)})$
is given by 
\begin{equation}\label{form740000}
\aligned
e^1_j({\bf v},{\bf x})
&= I_{\hat u',{\bf x}}
(I^{(1)}_{{\bf x}^{(1)}_0,\phi^{(1)}_0;((\Sigma',\vec z^{\,\prime}),u')}
(g^{(1)}(\hat u',{\bf x})_* (e_j))) \\
&\in 
L^2_{m+1}(\Sigma_1(\delta);u_1^*TX\otimes \Lambda^{01})
\endaligned
\end{equation}
in (\ref{from74300}).
Here we write 
$I^{(1)}_{{\bf x}^{(1)}_0,\phi^{(1)}_0;((\Sigma',\vec z^{\,\prime}),u')}$
and 
$g^{(1)}(\hat u',{\bf x})_* (e_j))$ 
to indicate that they are associated to $({\frak W}^{(1)},\vec{\mathcal N}_1)$.
\footnote{Here $I^{(1)}_{{\bf x}^{(1)}_0,\phi^{(1)}_0;((\Sigma',\vec z^{\,\prime}),u')}$
is defined in (\ref{formula663333}) and $g^{(1)}(\hat u',{\bf x})_* (e_j))$  
is defined in (\ref{form73666}).}
Note
$({\bf x}_0,\phi_0) = (\psi({\bf x}),\tilde{\psi}\vert_{
\Sigma(\psi({\bf x}))})$, 
$(\Sigma',\vec z^{\,\prime}) = (\Sigma_1({\bf x}),\vec z_1({\bf x}))$
and
$$
u' = u^1_{{\bf v},{\bf x}} : \Sigma_1({\bf x})(\delta) \to X,
\qquad 
\hat u' =
u' \circ \Phi_{1,{\bf x},\delta} 
= u^1_{{\bf v},{\bf x}} \circ \Phi_{1,{\bf x},\delta} : \Sigma_1(\delta) \to X.
$$

\par
$e^1_j({\bf v},{\bf x})$ is a $C^{n}$ frame of $\frak E(\frak p_1;\epsilon;({\frak W}^{(1)},\vec{\mathcal N}_1))$
\index{00E4p1wepsilon@$\frak E(\frak p_1;\epsilon;{\frak W}^{(1)})$}
since it is a pull back of $C^{\infty}$ frame 
of the bundle (\ref{lem762762}).
Note (\ref{form740000}) depends on $\frak p_1$ and $({\frak W}^{(1)},\vec{\mathcal N}_1)$ 
but is independent of the various choices we made in 
Subsection \ref{mainprop}. (Those choices are used to 
{\it prove} the smoothness of right hand side.)
\par
When we use $\frak p_2$ and $({\frak W}^{(2)},\vec{\mathcal N}_2)$,
and $({\bf v}',{\bf x}') \in \mathcal V^{(2)}_{\rm map} \times \mathcal V_{(2)}(\epsilon')$
we put
$$
u'' = u^2_{{\bf v}',{\bf x}'} : \Sigma_2({\bf x}')(\delta') \to X,
\qquad 
\hat u'' =
u'' \circ \Phi_{2,{\bf x}',\delta'} 
= u^2_{{\bf v}',{\bf x}'} \circ \Phi_{2,{\bf x}',\delta} : \Sigma_2(\delta') \to X.
$$
Suppose 
$$
({\bf v}',{\bf x}') = {
\mathscr J}_{12;\epsilon,\epsilon'}({\bf v},{\bf x}).
$$
Then $\psi_1({\bf x}) = \psi_2({\bf x}')$. Note
$\psi_1$ in the left hand side is the map $\psi$ obtained from the deformation 
theory of $\Sigma_1$ and $\psi_2$ in the right hand side is the map  $\psi$ obtained from the deformation 
theory of $\Sigma_2$. It induces an isomorphism
$$
(\tilde{\psi}_2\vert_{
\Sigma(\psi_2({\bf x}'))})^{-1} \circ \tilde{\psi}_1\vert_{
\Sigma(\psi_1({\bf x}))}
: (\Sigma_1({\bf x}),\vec z_1({\bf x})) \cong (\Sigma_2({\bf x}'),\vec z_2({\bf x}'))
$$
such that
$$
u'' \circ (\tilde{\psi}\vert_{
\Sigma(\psi_2({\bf x}'))})^{-1} \circ \tilde{\psi}\vert_{
\Sigma(\psi_1({\bf x}))}= u'.
$$
Therefore the frame $e^2_j({\bf v}',{\bf x}')$ is given by
\index{00J6ytildep1p2@$\widetilde{
\mathscr J}_{\frak p_1,\frak p_2;{\frak W}^{(1)},{\frak W}^{(2)};\epsilon',\epsilon}$}
$$
\left(\widetilde{
\mathscr J}_{12;\epsilon,\epsilon'}
(e^2_j)\right)({\bf v},{\bf x})
=(I_{\hat u',{\bf x}} \circ (I_{\hat u'',{\bf x}'})^{-1})(e^1_j({\bf v},{\bf x}))
$$
In the same way as 
Subsection \ref{mainprop}
we can write the right hand side 
using local coordinates and prove that
$\widetilde{
\mathscr J}_{12;\epsilon,\epsilon'}$ 
is of $C^n$ class.
\end{proof}
\begin{proof}[Proof of Lemma \ref{lem735735}]
It remains to prove that
the Kuranishi map $s$ is of $C^n$ class.
We use the trivialization (\ref{form772222}) 
and regard $\frak E(\frak p_1;\epsilon;{\frak W}^{(1)})$ as a subbundle 
of the trivial bundle (the right hand side of (\ref{form772222})).
\par
For ${\bf x} \in \mathcal V_{(1)}$
we take the complex structure of $\Sigma_1({\bf x})$
and pull it back to $\Sigma_1(\delta)$ by 
$\Phi_{1,{\bf x},\delta}$. 
We thus obtain a $(\hat u',{\bf x})$ 
parametrized family of complex structures on $\Sigma_1(\delta)$,
which we denote by $j_{(\hat u',{\bf x})}$.
This is a family of complex structures 
depending smoothly on $(\hat u',{\bf x})$.
By definition
\begin{equation}\label{shikifors}
\aligned
&s(\hat u',(\hat u',{\bf x}))\\
&= \overline{\partial}_{j_{(\hat u',{\bf x})}}(\hat u')
\\
&=
\left(A^i_{\sigma}(\hat u',{\bf x})(\hat u'(z_{\sigma}))\frac{\partial \hat u'_i}{\partial x_{\sigma}}
+ 
B^i_{\sigma}(\hat u',{\bf x})(\hat u'(z_{\sigma}))\frac{\partial \hat u'_i}{\partial y_{\sigma}} \right)
\partial^i_{\sigma} \otimes d\overline z_{\sigma}.
\endaligned
\end{equation}
on a (sufficiently small) 
coordinate chart $W_{\sigma}$ on $\Sigma_1(\delta)$ and 
$\Omega_{\sigma}$ of $X$ containing a neighborhood of $u_1(W_{\sigma})$.
Here $z_{\sigma} = x_{\sigma} + \sqrt{-1} y_{\sigma}$ is a complex 
coordinate of $W_{\sigma}$ and $\partial^i_{\sigma}$ is a frame of the complex tangent bundle 
$TX$ on 
$\Omega_{\sigma}$.
$A^i_{\sigma}$ and $B^i_{\sigma}$ are smooth functions 
$$
\mathcal V_{\rm map} \times \mathcal V_{(1)}(\epsilon) \times \Omega_{\sigma}
\to \C.
$$
Therefore the Kuranshi map $s$ is of $C^{\infty}$ class in terms 
of this trivialization.
\end{proof}

The proof of Proposition \ref{prop616} is complete.
In fact, Proposition \ref{prop616} (5) holds at 
$[(\Sigma,\vec z),u]$ by Condition \ref{conds45} (2) and hence 
holds everywhere by taking $\epsilon_1$ small.
\qed

\subsection{From $C^n$ structure to $C^{\infty}$ structure}
\label{subsec:Cinfinity}
This subsection is  similar to \cite[Section 26]{foootech} and \cite[Section 12]{foooconstr}.
\par
So far we have constructed a $G$-equivariant Kuranishi chart of 
$C^n$ class for any $n$.
In this subsection, we show how we obtain one in $C^{\infty}$
class. 
\par
We consider the embedding
$$
\mathcal R_{(1),m}
:
\mathcal V^{(1)}_{\rm map}(\epsilon) \times \mathcal V_{(1)}(\epsilon) \to L^2_{m+1}(\Sigma_1(\delta),X) 
\times \mathcal V_{(1)}(\epsilon)
$$
as in (\ref{753753}) and (\ref{75375322}). We put $m$ in the suffix to specify 
the Hilbert space $L^2_{m+1}$ we use.
We proved that this is a smooth embedding of $C^n$ class if 
$m > n+10$ and $\epsilon < \epsilon_m$.
We fix $\epsilon_0 < \epsilon_{10}$ and show the next lemma.
\begin{lem}\label{lem763}
The image of 
$$
\mathcal R_{(1),10}
:
\mathcal V^{(1)}_{\rm map}(\epsilon_0) \times \mathcal V_{(1)}(\epsilon_0) \to L^2_{11}(\Sigma_1(\delta),X) 
\times \mathcal V_{(1)}(\epsilon_0)
$$
is contained in 
$
C^{k}(\Sigma_1(\delta),X) 
\times \mathcal V_{(1)}(\epsilon_0)
$
for any $k$
and is 
a smooth submanifold of 
$L^2_{11}(\Sigma_1(\delta),X) 
\times \mathcal V_{(1)}(\epsilon_0)
$ of $C^{\infty}$ class.
\end{lem}
\begin{proof}
By elliptic regularity $u_{{\bf v},{\bf x}}$ is a smooth map.
Moreover $\Phi_{1,{\bf x},\delta}$ is a smooth map.
Therefore by definition the image of $\mathcal R_{(1),m}$ 
is contained in  
$C^{\infty}(\Sigma_1(\delta),X) 
\times \mathcal V_{(1)}(\epsilon_0)$.
\begin{rem}\label{rem768}
We remark that the image of $\mathcal R_{(1),m}$ coincides with
the image of $\mathcal R_{(1),m'}$ in a neighborhood of 
$((\Sigma_1,\vec z_1\cup \vec w_1),u_1)$.
\par
In fact they coincides with the set of 
all the pairs $(u'\circ \Phi_{a,{\bf x},\delta},{\bf x})$, where
$u' : \Sigma_1({\bf x}) \to X$ such that
$
\overline{\partial}u' \in E((\Sigma_1({\bf x}),\vec z_1({\bf x})),u')
$
and  $(\Sigma_1({\bf x}),\vec z_1({\bf x})\cup \vec w_1({\bf x}),u')$ 
is $\epsilon$-close to $(\Sigma_1,\vec z_1\cup \vec w_1,u_1)$.
(This is a consequence of Lemma \ref{lem71533}.)
\end{rem}
We put
$$
\widehat{\mathcal R}_{(1)}^k
:
\mathcal V^{(1)}_{\rm map}(\epsilon_0) \times \mathcal V_{(1)}(\epsilon_0) \to
C^{k}(\Sigma_1(\delta),X) 
\times \mathcal V_{(1)}(\epsilon_0)
$$
\par
Note the map
$$
L^2_{m}(\Sigma_1(\delta),X) 
\times \mathcal V_{(1)}(\epsilon_0)
\to 
C^{k}(\Sigma_1(\delta),X) 
\times \mathcal V_{(1)}(\epsilon_0)
$$
is a smooth embedding for $k > m + 10$.
In fact the first factor is a smooth 
embedding between Banach manifolds 
and the second factor is the identity map.
It implies that the image 
$
{\mathcal R}_{(1),10}
(\mathcal V^{(1)}_{\rm map}(\epsilon_m) \times \mathcal V_{(1)}(\epsilon_m))
$ 
is a submanifold of $C^n$ class if $n > m+10$.
The issue is $\epsilon_m \to 0$ as $m\to \infty$.
\footnote{In other words the Newton iteration 
we used in \cite{foooexp} converges in $L^2_{m}$ topology 
for  $({\bf v},{\bf x}) \in \mathcal V^{(1)}_{\rm map}(\epsilon_m) \times \mathcal V_{(1)}(\epsilon_m)$
where $\epsilon_m \to 0$.}
So to prove the lemma we consider also the charts 
centered at various points altogether.

Let $\frak p_2 \in \mathcal V^{(1)}_{\rm map}(\epsilon_0) \times \mathcal V_{(1)}(\epsilon_0)$.
We fixed a stabilization data ${\frak W}^{(1)}$ at $\frak p_1$.
We take a stabilization data ${\frak W}^{(2)}$ at $\frak p_2$ as in 
Case 5 of the proof of Proposition \ref{prop720}.
We have a commutative diagram similar to (\ref{diagram746}):
\begin{equation}\label{diagram74622}
\begin{CD}
\mathcal V^{(2)}_{\rm map}(\epsilon_m) \times \mathcal V_{(2)}(\epsilon_m) @ >{\mathcal R_{(2)}}>> L^2_{m+1}(\Sigma_2(\delta_m),X) 
\times \mathcal V_{(2)}(\epsilon_m)\\
@ VV{\tilde{\mathscr J}}V @ VV{\Phi}V\\
\mathcal V^{(1)}_{\rm map}(\epsilon_0) \times \mathcal V_{(1)}(\epsilon_0) 
@ >{\widehat{\mathcal R}^k_{(1)}}>> C^k(\Sigma_1(\delta),X) 
\times \mathcal V_{(1)}(\epsilon_0)
\end{CD}
\end{equation}
with 
\begin{equation}\label{form760rev00}
\Phi(F,{\bf x}) = (F\circ \phi \circ \Phi_{1,{\bf x}_2,\delta},\psi({\bf x})).
\end{equation}
(See (\ref{form773}) and (\ref{form774}).)
The map in the right vertical arrow is of $C^{\infty}$ class
since $F \mapsto F\circ \phi \circ \Phi_{1,{\bf x}_2,\delta}$ 
with $\phi \circ \Phi_{1,{\bf x}_2,\delta}$ is a smooth map independent of ${\bf x}$ and $\psi$ is smooth.
It follows that the image of 
${\mathcal R}_{(1),10}$ is of $C^{n}$ class at $\frak p_2$ 
for $m > n+k+10$.
Since this holds for any $\frak p_2$ and $m$, Lemma \ref{lem763} follows.
\end{proof}
We regard 
$\mathcal V^{(1)}_{\rm map}(\epsilon_0) \times \mathcal V_{(1)}(\epsilon_0)$ 
as a manifold of $C^{\infty}$ class
so that the embedding
$\mathcal R_{(1),10}$ becomes an embedding of $C^{\infty}$
class.
Note this $C^{\infty}$ structure is different from 
previously defined one, which is the direct product 
structure using Definition \ref{defn712}.
They coincide each other at the origin $\frak p_1$ and 
also the underlying $C^1$ structure coincides everywhere.
We call this $C^{\infty}$ structure the {\it new $C^{\infty}$ structure}.
\index{new $C^{\infty}$ structure}
\par
We remark 
that 
$$
V(\frak p_1;\epsilon_0,({\frak W}^{(1)},\vec{\mathcal N}_1))
= 
\{({\bf v},{\bf x})
\mid u^1_{{\bf v},{\bf x}}(w_{1,j}({\bf x})) \in \mathcal N_j^{(1)}, j=1,\dots,k_1\}.
$$
by definition.
We consider the next commutative diagram.
\begin{equation}\label{diagram777}
\xymatrix{
&
\mathcal V^{(2)}_{\rm map}(\epsilon_m) \times \mathcal V_{(2)}(\epsilon_m)
\ar[r]_{\widehat{\mathcal R}_{(2)}^{k}}\ar[dd]_{\tilde{\mathscr J}}
&C^{k}(\Sigma_2(\delta),X) 
\times \mathcal V_{(2)}(\epsilon_m)\ar[dd]_{\Phi}
\ar[rd]  \\
&
&&
X
\\
&
\mathcal V^{(1)}_{\rm map}(\epsilon_0) \times \mathcal V_{(1)}(\epsilon_0)\ar[r]_{{\mathcal R}_{(1),10}} & 
L^2_{11}(\Sigma_1(\delta),X)
\times \mathcal V_{(1)}(\epsilon_0)\ar[ru]_{}
}
\end{equation}
where $m > k+10$.
Here the two maps to $X$ appearing in 
Diagram (\ref{diagram777}) is given by 
$
(\hat u',{\bf x}) \mapsto \hat u'((\Phi_{i,{\bf x},\delta})^{-1}(w_{i,j}({\bf x})))
$.
Note this map
$$
C^{k}(\Sigma_2(\delta),X) 
\times \mathcal V_{(2)}(\epsilon_m)
\to 
X
$$
is of $C^{k}$ class.
Therefore 
the composition
$$
\mathcal V^{(1)}_{\rm map}(\epsilon_0) \times \mathcal V_{(1)}(\epsilon_0)
\to
L^2_{11}(\Sigma_1(\delta),X)
\times \mathcal V_{(1)}(\epsilon_0)
\to X
$$
(which is nothing but the map $({\bf v},{\bf x})
\mapsto u^1_{{\bf v},{\bf x}}(w_{1,j}({\bf x}))$
is of $C^{k}$ class with respect to the 
new $C^{\infty}$ structure of 
$\mathcal V^{(1)}_{\rm map}(\epsilon_0) \times \mathcal V_{(1)}(\epsilon_0)$
at 
$\frak p_2$. 
(Here we use the fact that $\Phi$ is of $C^{\infty}$ class,
$\tilde{\mathscr J}$ is an open embedding of $C^k$ class, and the 
commutativity of the Diagram (\ref{diagram777}).)
\par
Since this holds for any $\frak p_2$ and $k$,
the submanifold 
$V(\frak p_1;\epsilon_0,({\frak W}^{(1)},\vec{\mathcal N}_1))$
is a submanifold of $C^{\infty}$ 
class of 
$\mathcal V^{(1)}_{\rm map}(\epsilon_0) \times \mathcal V_{(1)}(\epsilon_0)$
equipped with the new $C^{\infty}$ structure.
\par
We thus defined a $C^{\infty}$
structure on 
$V(\frak p_1;\epsilon_0,({\frak W}^{(1)},\vec{\mathcal N}_1))$.
Here $\epsilon_0$ is $\frak p_1$ dependent.
So we write 
$V(\frak p_1;\epsilon_{\frak p_1},({\frak W}^{(1)},\vec{\mathcal N}_1))$
from now on.
\par
We next show that the coordinate change
$$
\mathscr J_{12;\epsilon_{\frak p_1},\epsilon_{\frak p_2}} : V(\frak p_2;\epsilon_{\frak p_2},({\frak W}^{(2)},\vec{\mathcal N}_2)) \to 
V(\frak p_1;\epsilon_{\frak p_1},({\frak W}^{(1)},\vec{\mathcal N}_1))
$$
is of $C^{\infty}$ class with respect to the new $C^{\infty}$ structure.
\par
Let $\frak p_3 \in V(\frak p_2;\epsilon_{\frak p_2},({\frak W}^{(2)},\vec{\mathcal N}_2))$
be an arbitrary point. It suffices to prove  
that
$\mathscr J_{12;\epsilon_{\frak p_1},\epsilon_{\frak p_2}}$ is of $C^{\infty}$ class at $\frak p_3$.
We take two strong stabilization data $({\frak W}^{(3,1)},\vec{\mathcal N}_{3,1})$ and $({\frak W}^{(3,2)},\vec{\mathcal N}_{3,2})$ at $\frak p_3$ as follows.
\begin{enumerate}
\item
$\frak p_2, ({\frak W}^{(2)},\vec{\mathcal N}_2)$ and $({\frak W}^{(3,2)},\vec{\mathcal N}_{3,2})$ are 
as in Case 5 of the proof of Proposition \ref{prop720}.
\item
$\frak p_1, ({\frak W}^{(1)},\vec{\mathcal N}_1)$ and $({\frak W}^{(3,1)},\vec{\mathcal N}_{3,1})$ are 
as in in Case 5 of the proof of Proposition \ref{prop720}. 
\end{enumerate}
We consider the next diagram.
\begin{equation}\label{diagram778}
\!\!\!\!\!\!\!\!\!\!\!\!\!\!\!\!\!\!\!\!\!\!\!\!
\begin{CD}
V(\frak p_3;\epsilon',({\frak W}^{(3,2)},\vec{\mathcal N}_{3,2})) @ >{\mathscr J_{23;\epsilon_{\frak p_2},\epsilon'}}>> 
V(\frak p_2;\epsilon_{\frak p_2},({\frak W}^{(2)},\vec{\mathcal N}_{2}))
\\
@ VV{\mathscr J_{33;\epsilon,\epsilon'}}V @ VV{\mathscr J_{12;\epsilon_{\frak p_1},\epsilon_{\frak p_2}}}V\\
V(\frak p_3;\epsilon,({\frak W}^{(3,1)},\vec{\mathcal N}_{3,1})) @ >{\mathscr J_{13;\epsilon_{\frak p_1},\epsilon}}>> 
V(\frak p_1;\epsilon_{\frak p_1},({\frak W}^{(1)},\vec{\mathcal N}_{1})).
\end{CD}
\end{equation}
The commutativity of Diagram (\ref{diagram778}) up to $\mathcal G_1$ action is a consequence of 
(\ref{formform740}).
\par
By the proof of Lemma \ref{lem763} and 
the definition of the $C^{\infty}$ structures, 
the horizontal arrows are smooth at the origin $\frak p_3$.
Therefore it suffices to prove that
the left vertical arrow is of $C^{\infty}$ class at $\frak p_3$.
\par
This is actually a consequence of the proof of Proposition \ref{prop720}.
To carry out the proof of Proposition \ref{prop720}
we take $L^2_{m+2n+1}$ space for ${\frak W}^{(3,2)}$ 
and $L^2_{m+n+1}$ space for ${\frak W}^{(3,1)}$. Then the coordinate change 
$\mathscr J_{33;\epsilon,\epsilon'}$
is of $C^n$ class if $m > n+10$.
Therefore 
$\mathscr J_{12;\epsilon_{\frak p_1},\epsilon_{\frak p_2}}$
is of $C^n$ class for any $n$ at $\frak p_3$.
So it is of $C^{\infty}$ class at $\frak p_3$.
(We use Remark \ref{rem768} here with ${\frak p_1}$ replaced by ${\frak p_3}$.)
\par
We thus proved that our 
smooth structures on $V(\frak p_1;\epsilon_0,({\frak W}^{(1)},\vec{\mathcal N}_{1}))$ 
can be glued to give a smooth structure 
on $U(((\Sigma,\vec z),u);\epsilon_2)$.
\par
The proof that we can define a smooth structure on 
our obstruction bundle 
$E(((\Sigma,\vec z),u);\epsilon_2)$
is similar.
Note the new $C^{\infty}$ structure on $V(\frak p_1;\epsilon_0,({\frak W}^{(1)},\vec{\mathcal N}_1))$ 
is defined so that the map 
$$
V(\frak p_1;\epsilon_0,({\frak W}^{(1)},\vec{\mathcal N}_1))
\to
\mathcal V^{(1)}_{\rm map}(\epsilon_0) \times \mathcal V_{(1)}(\epsilon_0)
\to
L^2_{11}(\Sigma_1(\delta),X)
\times \mathcal V_{(1)}(\epsilon_0)
$$
is an embedding of $C^{\infty}$ class.
Therefore we can define a smooth structure on 
$\frak E(\frak p_1;\epsilon;{\frak W}^{(1)})$
so that the trivialization 
(\ref{form772222}) is a trivialization of $C^{\infty}$ class.
\par
The proof that the map
$$
\widetilde{
\mathscr J}_{12;\epsilon_{\frak p_1},\epsilon_{\frak p_2}}
:
\frak E(\frak p_1;\epsilon;({\frak W}^{(1)},\vec{\mathcal N}_1))
\to 
\frak E(\frak p_2;\epsilon;({\frak W}^{(2)},\vec{\mathcal N}_2)).
$$
(\ref{maptildeJ}) is of $C^{\infty}$ class is the same
as the proof that 
${
\mathscr J}_{12;\epsilon_{\frak p_1},\epsilon_{\frak p_2}}$
is of $C^{\infty}$ class.
\par
The smoothness of the Kuranishi map is immediate from 
(\ref{shikifors}).
\par
The proof of Lemma \ref{lem735735} is now complete.
\qed
%\section{The case of bordered marked curve}
%\label{sec:bordered}

\section{Convex function and Riemannian center of Mass: Review}
\label{sec:centerofmass} 

This section is a review of 
convex function and center of mass technique, 
which are classical topics in Riemannian 
geometry. (See \cite{grokar}.) 
We include this review in this paper since this topic is not so familiar  
among the researchers of pseudo holomorphic curve, 
Gromov-Witten theory, or Floer homology.
(For example Proposition \ref{prop98} is hard to find in the literature though 
this proposition is certainly regarded as `obvious' by experts.)
\par
Let $M$ be a Riemannian manifold. We use Levi-Civita connection 
$\nabla$.
A geodesic is a map $\ell : [a,b] \to M$ such that $\nabla_{\Dot{\ell}}\Dot{\ell}
= 0$
and $\Vert \Dot{\ell}(t) \Vert$ is a nonzero constant.
\begin{defn}
A function $f : M \to \R$ is said to be {\it convex}\index{convex}
if for any geodesic $\ell : [a,b] \to M$ we have
\begin{equation}
\frac{d^2}{dt^2} (f\circ \ell) \ge 0.
\end{equation}
$f$ is said to be {\it strictly convex}\index{strictly convex} if the strict inequality $>$ holds.
\par
In case 
\begin{equation}\nonumber
\frac{d^2}{dt^2} (f\circ \ell) \ge c > 0
\end{equation}
for all geodesic $\ell$ with $\Vert \Dot{\ell}\Vert = 1$, 
we say $f$ is $c$-strictly convex.\index{00C2strictly@$c$-strictly convex}
\end{defn}
The usefulness of strict convex function for our purpose is 
the following:
\begin{lem}
Let $f : M \to \R$ be a strictly convex function.
Suppose $f$ assumes its local minimum at both $p,q \in M$.
We also assume that there exists a geodesic joining $p$ and $q$.
Then $p=q$.
\end{lem}
\begin{proof}
This is an immediate consequence from the fact if 
$h : [a,b] \to \R$ is a strictly convex function and $h$ assume 
local minimum at both $a,b$ then $a=b$.
\end{proof}
A typical example of a convex function is the Riemannian 
distance. 
We denote by $d_M : M \times M \to \R_{\ge 0}$ the 
Riemannian distance function.
Let $U \subset M$ be a relatively compact open subset.
\begin{lem}
There exists $\epsilon > 0$ such that
on
$$
\{(p,q) \mid p,q \in U, d_M(p,q) < \epsilon\}
$$
the function
$$
(p,q) \mapsto d_M(p,q)^2
$$
is smooth and convex.
\end{lem}
This is a standard fact in Riemannian geometry.
Note this function $(p,q) \mapsto d_M(p,q)^2$ 
is not strictly convex. However for a fixed $q$, the function
$p \mapsto d_M(p,q)^2$ is strictly convex. 
We need to restrict the direction for the positivity of the 
second derivative of $(p,q) \mapsto d_M(p,q)^2$.
See the proof of 
Proposition \ref{prop98} below.

We use the next lemma also.
\begin{lem}\label{lem94}
Let $N$ be an oriented manifold with 
volume form $\Omega_N$ and $f : M\times N \to \R$  a smooth
function. Suppose that for each $y \in N$ $x\mapsto f(x,y)$ is 
convex and for each $x_0 \in M$ there exists $y$ such that 
$x \mapsto f(x,y)$ is strictly convex in a neighborhood of $x_0$.
Then the function $F : M \to \R$
$$
F(x) = \int_{y \in N} f(x,y) \Omega_N
$$ 
is strictly convex.
\end{lem}
The proof is obvious.
\par
For our application we need to show 
convexity of certain functions induced by a 
distance function.
We use Proposition \ref{prop98} for this purpose.
We also need to ensure uniformity of various 
constants obtained.
We use a version of boundedness of geometry 
for this purpose.
The next definition is a bit extravagant
for our purpose. However the situation we use 
certainly satisfies this condition.
\begin{defn}\label{defn95}
A family $\{(N_{b},K_b) \mid b \in \mathcal B\}$ of a pair 
of Riemannian manifolds $N_b$ and its compact subsets $K_b$ 
is said to be of {\it bounded geometry in all degree}
\index{bounded geometry in all degree}
if there exists $\mu >0$ and $C_k$ ($k=0,1,\dots,$) with 
the following properties.
\begin{enumerate}
\item
The injectivity radius is greater than $\mu$ at all points 
$x \in K_b \subset N_b$. 
\item
Moreover the metric ball of radius $\mu$ centered at $x \in K_b \subset N_b$
are relatively compact in $N_b$.
\item
We have estimate
$$
\Vert \nabla^k R^{N_b} \Vert \le C_k,
$$
for $k=0,1,2,\dots$. Here $R^{N_b}$ is the 
Riemann curvature tensor of ${N_b}$ and $\nabla$ is 
the Levi-Civita connection.
The inequality holds everywhere (point-wise) on ${N_b}$.
\end{enumerate}
\par
When we need to specify $\mu$, $\{C_k\}$
we say bounded geometry in all degree by $\mu$, $\{C_k\}$.
\end{defn}
\begin{rem}
We consider a pair $(N_b,K_b)$ rather than 
a single Riemannian manifold $N_b$, in order to include the 
case when our Riemannian manifold is not complete.
\end{rem}
We use geodesic coordinate $\exp : B_{\mu/2,x}{N_b} \to {N_b}$.
Here $B_{\mu/2,x}{N_b}$ is the metric ball of radius 
$\mu/2$ centered at $0$ in the tangent space $T_x{N_b}$.
Item (3) implies that the coordinate change of 
geodesic coordinate has uniformly bounded $C^k$ norm for any $k$.
\begin{defn}
Let $\{(N_{b},K_b) \mid b \in \mathcal B\}$ be as in Definition \ref{defn95}
and $M$ a Riemannian manifold.
Let $\delta < \mu/2$.
A family of smooth maps $f_b : N_b \to M$ is said to have 
uniform $C^k$ norm on the $\delta$ neighborhood of $K_b$ 
if the composition
\begin{equation}\label{form92}
f_b\circ\exp : B_{\delta,x}{N_b} \to M
\end{equation}
has uniformly bounded $C^k$ norm for $x \in K_b$. 
Here we regard $B_{\delta,x}{N_b}
= \{V \in T_xN_b \mid \Vert V\Vert < \delta\}$ as an open subset of 
the Euclid space by isometry.
\par
When we specify the $C^k$ bound, we say 
has uniformly bounded $C^k$ norm $\le B_k$.
It means that the $C^k$ norm of (\ref{form92}) 
is not greater than $B_k$.
\end{defn}
\begin{prop}\label{prop98}
Given $\mu$, $\{C_k\}$, $B$, $\delta$, $\rho$
there exists $\epsilon$ with the following 
properties.
\par
Let $\{(N_{b},K_b) \mid b \in \mathcal B\}$
have bounded geometry in all degree by $\mu$, $\{C_k\}$
and 
$f_b,  g_b :  N_{b} \to M$
be a pair of smooth maps such that they have 
uniform $C^2$ bound by $B$ on $\delta$ neighborhood of $K_b$. 
Suppose 
$$
d_M(f_b(x),g_b(x)) \le \epsilon
$$
holds on $\delta$ neighborhoods of $K_b$.
Moreover we assume
$$
d_{TM}(D_Vf _b,D_Vg_b) \ge \rho
$$
for all $V \in T_xN_b$, $\Vert V\Vert =1$, $d(x,K_b)< \delta$.
(Here $d$ is the Riemannian distance in the tangent bundle of $M$.)
\par
Then the function
\begin{equation}\label{form93}
x \mapsto d_M(f_b(x),g_b(x))^2
\end{equation}
on $K_b$ is strictly convex.
Moreover there exists $\sigma$ depending only on 
$\mu$, $\{C_k\}$, $B$, $\delta$, $\rho$ such that 
(\ref{form93}) is $\sigma$-strictly convex.
\end{prop}
\begin{proof}
Let $\ell : [-c,c] \to B_{\delta}K_b
= \{x \in N_b \mid d_{N_b}(x,K_b) < \delta\}$ be a geodesic of 
unit speed.
We put $\gamma_b(t) = (f_b(\ell(t)),g_b(\ell(t))$.
Note 
$$
({\rm Hess} h)(V,V') = \nabla_V \nabla_V' h - \nabla_{\nabla_VV'}h
$$
is a symmetric 2 tensor. If $h = d_M^2$ then 
$$
({\rm Hess}\, d_M^2)(V,V)  \ge \sigma'
$$
if $V = (V_1,V_2) \in T_{(p,q)}M^2$, 
$\Vert V \Vert = 1$, $d_M(p,q) < \epsilon_1$ and 
$d_{TM}(V_1,V_2) \ge \rho$.
Therefore
$$
\frac{d^2}{dt} (d_M^2 \circ \gamma_b)
\ge
C \rho - C\nabla_{\nabla_{\Dot{\gamma_b}}\Dot{\gamma_b}}d_M^2.
$$
The second term can be estimated by
$$
C \vert d_M\vert \vert \nabla d_M\vert
\le C \epsilon_1.
$$
The proposition follows.
\end{proof}
We also use the following lemma in Subsection \ref{mainprop}.
\begin{lem}\label{lem999}
Let $\pi : M \to N$ be a smooth fiber bundle 
on an open subset of a Hilbert space.
We assume that the fibers are finite dimensional
and take a smooth family of Riemannian metrics 
of the fibers.
\par
Let $f: M \to \R$ be a smooth function.
We assume:
\begin{enumerate}
\item
The restriction of $f$ to the fibers are strictly convex.
\item
The  minimum of the restriction of $f$ to the fibers 
$\pi^{-1}(x)$ 
is attained at the unique point $g(x) \in \pi^{-1}(x)$
for each $x \in N$.
\end{enumerate}
\par
Then the map $g : N \to M$ is smooth.
\end{lem}
\begin{proof}
The proof of continuity of $g$ is an exercise of general 
topology, which we omit.
\par
Let ${\rm Ker D\pi} \subset TM$ be the 
subbundle consisting of the vectors of vertical direction.
We define a section of its dual ${\rm Ker D\pi}^*$ by
$$
\nabla^{\rm vert} f : 
y \mapsto (V \mapsto V(f)).
$$
Here $y \in M$, $V\in {\rm Ker D_y\pi} \subset T_yM$.
By assumption $(\nabla^{\rm vert} f)(y) = 0$ 
if and only if $y = g(x)$ for $x = \pi(y)$.
\par
The differential of $\nabla^{\rm vert} f$ at $g(x)$ 
is the Hessian of $f\vert_{\pi^{-1}(x)}$ at $x$
and so is non-degenerate by strict convexity of 
$f\vert_{\pi^{-1}(x)}$.
\par
The smoothness of $g$ now is a consequence of implicit 
function theorem.
\end{proof}

\include{index}
\printindex

\bibliographystyle{amsalpha}

\end{document}